\documentclass[reqno, 11pt]{amsart}
\usepackage{amsmath}
\usepackage{amssymb}
\usepackage[dvipsnames]{xcolor}
\usepackage{amsthm}
\usepackage[normalem]{ulem}
\usepackage[utf8]{inputenc}
\usepackage{indentfirst}
\usepackage{mathtools}
\usepackage{esint}
\usepackage{pdfpages}
\usepackage[hidelinks]{hyperref}
\usepackage{enumitem}
\usepackage{pgf, tikz}
\usetikzlibrary{arrows.meta}
\usepackage{caption}
\usepackage{subcaption}
\usepackage{pgfplots}
\pgfplotsset{compat=1.18}
\usepackage{adjustbox}
\usepackage{verbatim}
\usepackage{hyperref}
\def\hmath$#1${\texorpdfstring{{\rmfamily\textit{#1}}}{#1}}
\usepackage{bbm}
\usepackage{dsfont}
\usepackage{cancel}
\usepackage{setspace}
\usepackage[margin=2.4cm]{geometry}

\usepackage{graphicx}
\setlist{nosep}
\hypersetup{
    colorlinks,
    citecolor=MidnightBlue,
    filecolor=black,
    linkcolor=MidnightBlue,
    urlcolor=black,
    linktoc=page
}
\mathtoolsset{showonlyrefs}

\newtheorem{theorem}{Theorem}[section]

\newtheorem{lemma}[theorem]{Lemma}
\newtheorem{proposition}[theorem]{Proposition}
\newtheorem{corollary}[theorem]{Corollary}
\newtheorem{remark}[theorem]{Remark}

\theoremstyle{definition}

\newcommand{\R}{\mathbb{R}}
\newcommand{\sphere}{\mathbb{S}}
\newcommand{\N}{\mathbb{N}}
\newcommand{\Z}{\mathbb{Z}}

\newcommand{\m}{\mathfrak{m}}
\newcommand{\n}{\mathfrak{n}}

\newcommand{\F}{\mathcal{F}}

\newcommand{\B}{\mathcal{B}}

\newcommand{\norm}[1]{\left\lVert#1\right\rVert}
\newcommand{\snorm}[1]{\lVert#1\rVert}
\newcommand{\inftynorm}[1]{\left\lVert#1\right\rVert_{L^\infty}}
\newcommand{\sinftynorm}[1]{\lVert#1\rVert_{L^\infty}}
\newcommand{\ltwonorm}[1]{\left\lVert#1\right\rVert_{L^2}}
\newcommand{\sltwonorm}[1]{\lVert#1\rVert_{L^2}}
\newcommand{\lpnorm}[2]{\left\lVert#1\right\rVert_{L^{#2}}}

\newcommand{\abs}[1]{\left\lvert#1\right\rvert}
\newcommand{\sabs}[1]{\lvert#1\rvert}
\newcommand{\set}[2]{\left \{ #1  \mid #2 \right \}}

\DeclareMathOperator\supp{supp}
\newcommand{\phikp}{\varphi_{k,p}}

\newcommand{\ff}{\hat{f}}
\newcommand{\absxi}{\abs{\xi}}
\newcommand{\abseta}{\abs{\eta}}
\newcommand{\absxieta}{\abs{\xi-\eta}}

\newcommand{\ltwoscalarproduct}[2]{\langle #1, #2\rangle_{L^2}}

\newcommand{\twopone}{2^{p_1}}
\newcommand{\twoptwo}{2^{p_2}}
\newcommand{\eps}{\varepsilon}

\newcommand{\twoktwo}{2^{k_2}}

\newcommand{\sinlambda}{\sqrt{1-\Lambda^2}}
\newcommand{\Setaphi}{S_\eta\Phi}

\newcommand{\absnabla}{\abs{\nabla}}

\newcommand{\zp}{\mathcal{Z}_+}
\newcommand{\zm}{\mathcal{Z}_-}
\newcommand{\zz}{\mathcal{Z}}
\newcommand{\tchi}{\widetilde{\chi}}
\newcommand{\tl}{\tilde{L}}

\newcommand{\red}[1]{\textcolor{red}{#1}}
 \newcommand{\blue}[1]{#1}
\newtagform{blue}[\textcolor{MidnightBlue}]{\color{black}(}{)}

\let\oldtocsection=\tocsection
\let\oldtocsubsection=\tocsubsection

\renewcommand{\tocsection}[2]{\hspace{0em}\oldtocsection{#1}{#2}}
\renewcommand{\tocsubsection}[2]{\hspace{1em}\oldtocsubsection{#1}{#2}}
\usepackage{xpatch}
\makeatletter   
\xpatchcmd{\@tocline}
{\hfil\hbox to\@pnumwidth{\@tocpagenum{#7}}\par}
{\ifnum#1<0\hfill\else\dotfill\fi\hbox to\@pnumwidth{\@tocpagenum{#7}}\par}
{}{}
\makeatother    
\makeatletter
 \def\l@subsection{\@tocline{2}{0pt}{4pc}{6pc}{}}
\def\l@subsubsection{\@tocline{3}{0pt}{8pc}{8pc}{}}
 \makeatother

\allowdisplaybreaks

\numberwithin{equation}{section}

\title[Long-time stability in the inviscid 2D Boussinesq equation]{Long-time stability of a stably stratified rest state in the inviscid 2D Boussinesq equation}
\author[C. Jurja]{Catalina Jurja}
\address{Institute of Mathematics, University of Zurich}
\email{catalina.jurja@math.uzh.ch}
\author[K. Widmayer]{Klaus Widmayer}
\address{Faculty of Mathematics, University of Vienna \& Institute of Mathematics, University of Zurich}
\email{klaus.widmayer@univie.ac.at \& klaus.widmayer@math.uzh.ch}



\begin{document}
\begin{abstract}
    We establish the nonlinear stability on a timescale $O(\eps^{-2})$ of a linearly, stably stratified rest state in the inviscid Boussinesq system on $\R^2$. Here $\eps>0$ denotes the size of an initially sufficiently small, Sobolev regular and localized perturbation. A similar statement also holds for the related dispersive SQG equation.

    At the core of this result is a dispersive effect due to anisotropic internal gravity waves. At the linearized level, this gives rise to amplitude decay at a rate of $t^{-1/2}$, as observed in \cite{Elgindi_2015}. We establish a refined version of this, and propagate nonlinear control via a detailed analysis of nonlinear interactions using the method of partial symmetries developed in \cite{EC2022}.
\end{abstract}

\vspace*{-11pt}
\maketitle
\setcounter{tocdepth}{1}
\vspace{-1cm}
\tableofcontents   
\usetagform{blue}
\section{Introduction}
The main focus of this work is the study of stability of certain steady states of the 2D inviscid Boussinesq system
\begin{equation}\label{eqn: BQ}
\left\{
      \begin{aligned}
 &\partial_t v +v\cdot \nabla v=-\nabla p -\varrho\vec{e_2},\\
 &\partial_t\varrho+v\cdot \nabla\varrho=0, \\
    &\mathrm{div}v=0,
\end{aligned}
    \right.
\end{equation}
which models the dynamics of an incompressible fluid $v:\R^+\times \R^2\to \R^2$ with pressure $p:\R^+\times \R^2\to \R$ and scalar density $\varrho:\R^+\times \R^2\to \R$ under the influence of gravity. 
This is a widely used simplified model for geophysical flow: the system \eqref{eqn: BQ} arises from the \emph{Boussinesq approximation} of the inhomogeneous Euler system (see \cite[\textcolor{MidnightBlue}{\S 2.4}]{vallis2017atmospheric}), in which the variation of density is assumed to be small compared to the effects of gravity (described by the buoyant force term $-\varrho \vec{e}_2$). 

\par Due to parallels with the 3D axisymmetric Euler equations (see e.g.\ \cite[\textcolor{MidnightBlue}{\S 5.4.1}]{majda2002vorticity}, \cite{EJ2019,Elgindi-Jeong-SingFormBQ,Chen-Hou-Blowup2DBQ}), the system \eqref{eqn: BQ} has seen a lot of attention in recent years: while local well-posedness and blow-up criteria of Beale-Kato-Majda type for initial data in $H^s$, $s>2$, have been shown via classical methods e.g.\ in \cite{LWP_Chae}, the long-time dynamics of solutions to this system are in general not understood, and may include rapid growth or even blow-up scenarios (see e.g.\ \cite{CCW14,kiselev2022smallscaleformation2d} and \cite{Elgindi-Jeong-SingFormBQ,EP23,RksOnSmoothness2dBQ3dE-Chen,CH22,CH24}).
In view of this, the study of stable dynamics is a natural step towards a fuller understanding of the behavior of solutions to \eqref{eqn: BQ}. 

\par In this work, we focus on dynamics near the stratified steady state
\begin{equation}\label{eq:steady}
    (v_s,\varrho_s)=(0,-x_2),\qquad p_s(x_1,x_2)=x_2^2/2.
\end{equation}
That is, for $v=v_s+u=u,\varrho=\varrho_s+\rho=-x_2+\rho$ we consider solutions to
\begin{equation}\label{eqn: perturbed BQ}
\left\{
      \begin{aligned}
 &\partial_t u +u\cdot \nabla u=-\nabla p -\rho\vec{e_2},\\
 &\partial_t\rho+u\cdot \nabla\rho=u_2, \\
    &\mathrm{div}u=0.
\end{aligned}
    \right.
\end{equation}
The setting of the steady state \eqref{eq:steady} is a prototypical setting of a \emph{stably stratified}\footnote{In particular, such configurations are \emph{spectrally stable}, as the eigenvalues of the linearized operator of \eqref{eqn: perturbed BQ} around zero are purely imaginary \cite{gallay2019stabilityvorticesidealfluids}.} fluid, where the density of the fluid increases in the direction of gravity (i.e.\ $\blue{\partial_{x_2}}\varrho_s(x)<0$).
This is a natural setting for many atmospheric and oceanic flows (under appropriate averaging, see e.g.\ \cite[\textcolor{MidnightBlue}{Ch.\ III}]{dauxois-ChallaengesInEnvironmentalFM}, \cite[\textcolor{MidnightBlue}{Ch.\ II}]{vallis2017atmospheric}).
In particular, here buoyant forces give rise to internal gravity waves, which act as a restoring mechanism. More precisely, as shown in \cite{Elgindi_2015}, the linear dynamics in \eqref{eqn: perturbed BQ} are waves with dispersion relation given by the symbol of a Riesz transform, and feature dispersive amplitude decay at a rate of $t^{-1/2}$. Together with a basic blow-up criterion for the energy, this allowed the authors of \cite{Elgindi_2015} to show that the time of existence of solutions extends from the trivial local-wellposedness time scale $O(\eps^{-1})$ to $O(\eps^{-4/3})$, where $\eps\ll 1$ denotes the size of the initial data (see also \cite{Wan2020} for a lower regularity setting).

\par In this article, we use a refined analysis of nonlinear interactions to show that stability (and thus also existence) of solutions to \eqref{eqn: perturbed BQ} in fact holds on the longer timescale $O(\eps^{-2})$. As discussed further below, this is the natural timescale of energy estimates given the rate of amplitude decay, and corresponds to that of a cubic nonlinearity. We summarize our main result as follows:

\begin{theorem}\label{thm: main thm BQ}
    There exist a norm $Y$, $N_0\in \N$ and an $\eps_0>0$ such that if for some $0<\eps<\eps_0$
    \begin{align*}
       & \norm{u_0}_{H^{N_0}}+\norm{\rho_0}_{H^{N_0}}\leq \eps, && \norm{u_{0}}_Y+\norm{\rho_0}_Y\leq \eps,
    \end{align*}
     then there exist $T\gtrsim \eps^{-2}$ and a unique solution $(u,\rho)\in C([0,T],H^{N_0}(\R^2,\R^2))\times C([0,T],H^{N_0}(\R^2))$ of \eqref{eqn: perturbed BQ} with initial data $(u_0,\rho_0)$. \blue{Moreover, for $t\in [0,T]$ this solution remains small in the above norms and decays in amplitude:
     \begin{align}
\norm{u(t)}_{H^{N_0}}+\norm{\rho(t)}_{H^{N_0}}\lesssim \eps, \quad \norm{\nabla u(t)}_{L^\infty}+\norm{\nabla \rho(t)}_{L^\infty}\lesssim t^{-\frac{1}{2}}\eps.
     \end{align}} 
     In particular, the corresponding unique solution of \eqref{eqn: BQ} with initial data $(u_0,-x_2+\rho_0)$ exists on the same timescale.
\end{theorem}
To the best of our knowledge, this is the longest known timescale of existence for solutions to \eqref{eqn: perturbed BQ}. We give a detailed overview of the proof of Theorem \ref{thm: main thm BQ} in Section \ref{sec: outline of the proof} below, while a more precise version of our result is stated in Theorem \ref{thm: main result}.

We comment on some points of immediate relevance.
\begin{enumerate}
    \item \blue{(Assumptions on the initial data)} Our analysis proceeds in the spirit of quasilinear, dispersive partial differential equations, in particular as developed in the ``method of partial symmetries'' of \cite{EC2022}, and thus relies heavily on the precise structure of nonlinear interactions in \eqref{eqn: perturbed BQ}.

    The norm $Y$ in Theorem \ref{thm: main thm BQ} is a sum of norms $B$ and $X$ defined in \eqref{B norm}-\eqref{X norm} that capture anisotropic localization and regularity in frequency space, see Section \ref{sec: Localizations}.
    Moreover, they include enough regularity in terms of a natural scaling vector field of the system \eqref{eqn: perturbed BQ} and ensure the decay of solutions at the linear rate of $t^{-1/2}$, see also the discussion in Section \ref{sec: outline of the proof}. That a restriction on the class of initial data is necessary for Theorem \ref{thm: main thm BQ} to hold is clear from the work \cite{bianchini2024strongillposednesslinfty2d}, which shows that there exist $L^\infty$-small initial data producing $L^\infty$-norm inflation of $\partial_x\rho$ in arbitrarily short time.

    \item \blue{(Dispersive structure)} The linearization of \eqref{eqn: perturbed BQ} is an anisotropic dispersive system. In  \cite{Elgindi_2015} it is shown that its dispersion relation is given by $\pm i\Lambda(\xi):=\pm i\xi_1\absxi^{-1}$, $\xi\in \R^2$. This is degenerate and leads to the sharp decay rate $t^{-\frac12}$, see \cite{Elgindi_2015}, which together with the energy estimates is the key limiting factor for the timescale in our result. In fact, invoking the standard blow-up criterion shows that $L^2$-based energies can only be expected to remain small on a timescale $O(\eps^{-2})$, whereas our other nonlinear arguments (to bound the norms $B,X$) could go slightly beyond this timescale.

    We remark that anisotropy does not necessarily lead to degeneracy, as witnessed in another classical geophysical model: the $\beta$-plane equations, a tangent plane model for Eulerian flows on the surface of a rotating 2D sphere. Thereby, rotation gives rise to linear waves with an anisotropic dispersion relation $i\xi_1\abs{\xi}^{-2}$, which however leads to decay at the full rate $t^{-1}$. Thanks to the presence of strong cancellations in the nonlinearity (via a ``double null structure''), stability was shown to hold globally in time in this model, see \cite{EW_betaplane,Pusateri_2018}.

    \item \blue{(Good unknowns)} The system nature of \eqref{eqn: perturbed BQ} poses a challenge, and in particular the fluid variables $u,\rho$ are not convenient from a perturbative point of view. Noting that due to incompressibility, the system \eqref{eqn: perturbed BQ} has only two degrees of freedom, we will instead work with two scalar unknowns $Z_\pm$, which diagonalize the linearized evolution. Through a suitable choice the crucial energy structure, symmetry properties and a certain ``null structure'' of the equations can be preserved (see the discussion in Section \ref{sec: outline of the proof}).

    \item \blue{(Strong stratification)} There are many parallels between the effect of constant rotation in homogeneous three-dimensional fluids and that of linear, stable stratification with constant gravity in two- or three-dimensional inhomogeneous fluids. In particular, the dispersion relations in all these cases are zero-homogeneous, anisotropic and degenerate. 
    
    Moreover, similarly as one can investigate the effect of a fast speed of rotation on existence timescales (see e.g.\ \cite[\textcolor{MidnightBlue}{\S 5}]{Math_Geophysics_Gallagher-Chemin} for the 3D Navier-Stokes, or \cite{TAKADA} for the 3D Euler equations), one can also track the strength of the stratification-gravity coupling. This is relevant for (more) steeply stratified versions $(v_s,\varrho_s)=(0,-\alpha x_2)$, $\alpha>0$, of the steady state \eqref{eq:steady}, or when quantifying gravity through a constant $g>0$ in the buoyant force term $-g\varrho \vec{e}_2$ in the momentum equation of \eqref{eqn: BQ}. Taking for simplicity $\alpha=g$, hereby $\alpha^{-1}$ plays the role of a small parameter that can be used to prolong existence times. In close analogy to the aforementioned references, this has been carried out in the context of \eqref{eqn: perturbed BQ} in \cite{wan2016} (see also \cite{3DBQ_Widmayer,TakBQ3d} for the 3D setting): given initial data $(u_0,\rho_0)$ and a time $T>0$, the authors use Strichartz estimates to derive a lower bound for $\alpha$ that guarantees the existence of solutions until at least time $T$.
    Via the time-scaling symmetry\footnote{Observe that if $(u,\rho)$ solve \eqref{eqn: perturbed BQ} on a time interval $[0,T]$, then for $\lambda>0$, the rescaled functions $(\lambda u(\lambda t,x), \lambda \rho(\lambda t,x))$ solve \eqref{eqn: perturbed BQ} on $[0,\lambda^{-1}T]$ with an additional ``strength of gravity'' constant $\lambda$ in front of the linear terms.} of \eqref{eqn: perturbed BQ}, for initial data of size $\eps$ this agrees with the $O(\eps^{-4/3})$ timescale of \cite{Elgindi_2015,Wan2020}, albeit in lower regularity $H^s$, $s>3$. Here, our result should allow to quantitatively improve these arguments, but we do not pursue this here.

    \item \blue{(Related results)} Natural interest also concerns other steady states of \eqref{eqn: BQ}, in particular those including a shearing motion transversal to the direction of gravity and general gravity profiles (i.e.\ steady states of the form $v_s=f(x_2)\vec{e}_1,\varrho_s=g(x_2)\vec{e}_2$), as well as other domain geometries.
    However, in general not even linearized dynamics are fully understood.
    
    The prototypical example in this context is the ``stably stratified Couette flow'', a steady state of \eqref{eqn: BQ} with fluid velocity $v_s=x_2\vec{e}_1$ and stable stratification profile $\varrho_s=-x_2$. Here, linearized dynamics can be understood explicitly. In the case of a channel domain $\mathbb{T}\times \R$, the background shear flow plays a dominant, strongly stabilizing role via inviscid damping, a classical mixing mechanism. As demonstrated in \cite{bedrossian2021nonlinearinvisciddampingshearbuoyancy}, this guarantees the nonlinear stability of stably stratified Couette flow on a timescale $O(\eps^{-2})$, provided the initial perturbations are of size $\eps$ and Gevrey regular. Contrary to our setting without background flow, thereby the oscillatory effects of buoyant forces do not stabilize perturbations and instead lead to a slow growth, suggesting that the aforementioned timescale is optimal for the result in \cite{bedrossian2021nonlinearinvisciddampingshearbuoyancy}. This is also related to echo chains in the linearized equations, see \cite{Zil23-1,Zil23-2}. (It is only in the setting of a 3D channel $\mathbb{T}\times\mathbb{R}\times\mathbb{T}$ that the dispersive effects of internal gravity waves have been shown to improve stability of the stably stratified Couette flow, albeit in the presence of viscosity \cite{CZDZW}.)

    For the analogue of \eqref{eq:steady} on $\R^3$, a related dispersive structure has been uncovered in \cite{3DBQ_Widmayer} and used to establish a layered 2D Euler dynamic in the singular limit of strong gravity (see also \cite{TakBQ3d}), but stability beyond the basic $O(\eps^{-1})$ timescale remains an open problem.

    \item \blue{(Open question) The behavior of solutions beyond the ``cubic'' time scale $T\gtrsim \eps^{-2}$ as given by Theorem \ref{thm: main thm BQ} remains a challenging open problem (as mentioned above, $\eps^{-2}$ is the optimal time scale for bounded energy solutions given the sharp amplitude decay at rate $t^{-1/2}$ and the standard blow-up criterion). While it may seem natural to conjecture that solutions will eventually leave the perturbative realm described here, we are not able to give a concrete description of how this would happen. Similar questions have also been raised and are open for many classical dispersive equations, see e.g.\ \cite{AD15,HIT16,IT17,IP18,Wu20,DM22,IT24}. In particular, it may be that solutions become singular after finite time as e.g.\ in \cite{MP17}, or global regularity may follow from a refined energy estimate relying on new insights into the nonlinear structure as e.g.\ in \cite{DIPP17,DIP25}.}
    
\end{enumerate}

\medskip

\par In fact, our arguments also apply to a simpler, closely related setting, namely that of the dispersive surface quasi-geostrophic (SQG) equation 
\begin{equation}   \label{SQG}
\begin{cases}
   \partial_t \theta +u \cdot \nabla \theta = R_1\theta, \\
   u=\nabla^\perp(-\Delta)^{-1/2}\theta, \\
   \theta(0,x)=\theta_0(x),
\end{cases}\end{equation}
where $\theta:\R_+ \times \R^2\to \R$ is the temperature of the fluid and $\widehat{R_1f}(\xi):=-i{\xi_1}{\abs{\xi}^{-1}}\widehat{f}$ is the Riesz transform in the first coordinate. This model has been suggested for certain wave turbulence interactions \cite{Local_Nonlocal_Dispersive_Turbulence}, and adds to the classical inviscid SQG equation the linear right hand side term $R_1\theta$, which has exactly the same dispersive structure as the Boussinesq system. Due to other structural parallels with the 3D Euler equations (in particular a ``vortex stretching'' dynamic of $\nabla^\perp\theta$, see e.g.\ \cite{P_Constantin_1994}), the dynamics of the inviscid SQG equation are of natural interest, but only understood in few cases (see e.g.\ \cite{Carstro-Cordoba-Gomez-Serrano-VstatesSQG,Carstro-Cordoba-Gomez-Serrano-VstatesSQG,Kiselev-Nazarov-NormInflSQG,Gomez-Serrano-Ionescu-Park-gSQG} and references therein).
Even in the dispersive version \eqref{SQG}, the long-time behavior of initially small solutions remains to be understood. However, close parallels between \eqref{SQG} and \eqref{eqn: perturbed BQ} (in terms of both the dispersive and energy structure) have already been exploited in \cite{Elgindi_2015} to show that the basic existence timescale of solutions extends to $O(\eps^{-4/3})$. 

The structural features used to establish Theorem \ref{thm: main thm BQ} also include the setting of \eqref{SQG} -- more precisely, $\theta$ can be viewed as analogous to one of the Boussinesq unknowns $Z_\pm$, with the additional simplification of having only one single nonlinearity (with a similar null structure). Thus, we can extend the time of existence to the timescale $O(\eps^{-2})$:

\begin{theorem}\label{thm: main thm SQG}
     With $Y$, $N_0\in \N$ and $\eps_0>0$ as in Theorem \ref{thm: main thm BQ}, there holds that if for some $0<\eps<\eps_0$
    \begin{align*}
       & \norm{\theta_0}_{H^{N_0}}+\norm{\theta_0}_Y\leq \eps,
    \end{align*}
    then there exists a unique solution $\theta \in C([0,T],H^{N_0}(\R^2))$ of \eqref{SQG} with $T\gtrsim \eps^{-2}$. \blue{Moreover, for $t\in ]0,T]$ this solution remains small in the above norms and decays in amplitude:
     \begin{align}
\norm{\theta(t)}_{H^{N_0}}\lesssim \eps, \qquad \norm{\nabla\theta(t)}_{L^\infty}\lesssim t^{-\frac{1}{2}}\eps.
     \end{align}} 
\end{theorem}

\subsection{Outline of the proof}\label{sec: outline of the proof}
In the following, we give an overview of the proof of Theorem \ref{thm: main thm BQ} (and consequently also of Theorem \ref{thm: main thm SQG}), highlighting the key features of our approach while referring to the later sections containing the full mathematical details.
\par Our proof relies on and adapts the method of partial symmetries, as developed in \cite{EC2022} (see also \cite{ren2024globalsolutionseulercoriolis}), to the present 2D setting. This in turn builds on a long history of ideas and techniques used in the study of the long-time behaviour of quasilinear dispersive equations with small initial data, in particular as they originate in the method of space-time resonances \cite{germain2009globalsolutionsgravitywater, GP-DefocusingNLS-GNT} and many important further developments, e.g.\ \cite{GP11,Ionescu_Pausader_EP2011, Ionescu_Pausader_KG2012,GM14,IP15,Germain_Masmoudi_Shatah_CWW2015, EulerMaxwell3DIonescuPausader, DIPP17, Deng_2017, Pusateri_2018}, an adequate discussion of which goes beyond the scope of this article.

\subsection*{Structure of the equations} 
We discuss first the features inherent to the system \eqref{eqn: perturbed BQ} and the equation \eqref{SQG} that lay the foundation for our approach.
\par \emph{Dispersive structure.} To start with, we recall from \cite{Elgindi_2015} that \eqref{eqn: perturbed BQ} and \eqref{SQG} exhibit dispersion at the linearized level, the dispersion relation $\pm\Lambda$ of which is the symbol of the Riesz transform $R_1$, i.e.\
\begin{align*}
    \Lambda(\xi) = \frac{\xi_{1}}{\absxi}, \hspace{1.5cm} \xi \in \R^2.
\end{align*}
While for the dispersive SQG equation the dispersive operator is directly apparent through the Riesz transform on the right–hand side of \eqref{SQG}, for the Boussinesq system this requires a short computation. We note that $\Lambda$ is zero-homogeneous, anisotropic and degenerate, in the sense that $\blue{\det}\mathrm{Hess}\Lambda(\xi)=-{\xi_2^2}{\absxi^{-6}}$ vanishes along $\{\xi_2=0\}$, which also leads to the comparatively slow dispersive decay rate $t^{-1/2}$.

In order to facilitate a proper nonlinear analysis also in the Boussinesq system, it is useful to choose suitable dispersive unknowns $Z_\pm$ (see Section \ref{sec: choice of unknowns}). These diagonalize the linearized equation (see Proposition \ref{prop: profiles BQ}), and are moreover chosen such that energy balances remain intact, e.g.\
\begin{equation*}
    \norm{u}_{L^2}^2+\norm{\rho}_{L^2}^2=\frac{1}{2}\norm{Z_+}_{L^2}^2+\frac{1}{2}\norm{Z_-}_{L^2}^2.
\end{equation*}
The nonlinear equations \eqref{eqn: perturbed BQ} can then be recast as 
\begin{align}\label{eqn:generic_nonlin}
   & \partial_t Z_\pm + \mathcal{N}_\pm(Z_+,Z_-)=\pm R_1Z_\pm,
\end{align}
where $\mathcal{N}_\pm(Z_+,Z_-)$ are quadratically nonlinear terms. (This is naturally already the form of the dispersive SQG equation \eqref{SQG}.)

\par \emph{Scaling symmetry and vector fields.} In addition to a time scaling symmetry, the systems \eqref{eqn: perturbed BQ} and \eqref{SQG} have the following spatial scaling symmetry: if $(u,\rho)$ solves \eqref{eqn: perturbed BQ} (resp.\ $\theta$ solves \eqref{SQG}), then so do $(\lambda u(t,\lambda^{-1}x),\lambda \rho(t,\lambda^{-1}x))$ (resp.\ $\lambda\theta(t,\lambda^{-1}x)$) for $\lambda>0$ (see Section \ref{sec: Symmetries and vector fields}). In our approach, we take advantage of the natural derivative $S$ arising from this scaling symmetry,
\begin{align}\label{def:intro-vf}
    Sf(x)=x\cdot\nabla_xf(x).
\end{align} 
The vector field $S$ commutes in a favourable way with the equations, allowing us to propagate ``regularity'' in terms of many of copies $S$, in particular in the form of $L^2$-energies (see Sections \ref{sec: Energy Estimates}, \ref{sec: Energy Estimates SQG}). However, due to the anisotropy, $S$ is the only such natural derivative.

To span the full tangent space at any $x\in\R^2$, we complement $S$, a radial derivative in polar coordinates, with another vector field $W$, which in polar coordinates corresponds to an angular derivative, see \eqref{def: W}. This vector field however, \emph{does not} commute with the equations. As a result, one of the main difficulties of the article is to control sufficient regularity in this angular direction, i.e.\ to propagate certain bounds along $W$, as they are captured in the $X$-norm discussed below.

\par \emph{Null structure of the nonlinearity.} A key ingredient that allows us to control the nonlinear interactions is the presence of a null structure. Concretely, the symbols of the quadratic nonlinearities vanish (in a quantifiable fashion) for frequency configurations for which the dispersion of the output and that of the inputs is degenerate. More precisely, all Fourier symbols of the various quadratic nonlinearities contain a factor $\zeta_2\abs{\zeta}^{-1}$ for some $\zeta \in \{\xi,\xi-\eta,\eta\}$ (see Lemma \ref{lemma: multiplier bound}), which in turn is related to the degeneracy of the dispersion $\Lambda$. This can be seen directly in the case of the SQG nonlinearity, and follows with a short computation also for the Boussinesq system -- see \eqref{multiplier} and \eqref{eqn: bq multipliers}. 
\blue{This null structure  derives from the skew structure of 2D Eulerian nonlinearities of the type $u\cdot \nabla \omega$ with $u=\nabla^\perp\Delta^{-\alpha}\omega$ and $\alpha>0$. In our case, it is relatively weak as it leads to cancellations of nonlinear interactions only if all input and output frequencies are located in a degenerate region.}

\subsection*{Setup of the proof}
By considering the Duhamel formulation of equations of the form \eqref{eqn:generic_nonlin} and filtering out the linear evolution, it suffices to study bilinear terms of the form
\begin{align}\label{def: introduction bilinear term}
    \widehat{\B_\m(f,g)}(t,\xi)=\int_0^t\int_{\R^2}e^{is\Phi(\xi,\eta)}\m(\xi,\eta)\widehat{f}(\xi-\eta)\widehat{g}(\eta)d\eta ds,
\end{align}
where $\Phi=\pm\Lambda(\xi)\pm\Lambda(\xi-\eta)\pm\Lambda(\eta)$ is a phase function, $\m$ a Fourier multiplier that encodes the nonlinearity and $f,\,g$ are either the profiles $\mathcal{Z}_\pm:=e^{\pm i t\Lambda}Z_\pm$ of the dispersive unknowns for the Boussinesq system or the profile $\Theta:=e^{it\Lambda}\theta$ in the setting of the SQG equation -- see Section \ref{sec: choice of unknowns}.

We then prove Theorems \ref{thm: main thm BQ} and \ref{thm: main thm SQG} via a bootstrap argument involving a hierarchy of energy estimates with many ($N_0\gg 1$) derivatives and vector fields $S$ (of order $M\ll N_0$), and $B$- and $X$-norms of aforementioned profiles with fewer derivatives and vector fields (of order $ N\ll M$) -- see Proposition \ref{prop: bootstrap argument}. 

\medskip \par \textbf{Localizations.} Our norms quantify localization and regularity, and are $L^2$-based with suitable weights in terms of frequency localization parameters -- see Section \ref{sec: Localizations}. On one hand, in addition to the standard Littlewood-Paley projectors $P_k$ for the size $\abs{\xi}$ of a frequency $\xi \in \R^2$, we quantify the vertical components $\xi_2\absxi^{-1}$ of the interacting frequencies through Littlewood-Paley projections $P_{k,p}$, $k\in \Z,\; p\in \Z^-$. We highlight that these quantify exactly the degree of degeneracy of the dispersion relations $\Lambda$, as well as the aforementioned null structure. On the other hand, we introduce an \emph{angular Littlewood-Paley decomposition} $R_l, \; l\in \Z^+$, to capture the angular regularity along $W$. In particular, we show in Proposition \ref{prop: angular localization properties} that there holds  $\ltwonorm{WR_lf}\simeq 2^l\ltwonorm{R_lf}$. This approach parallels the setup introduced in \cite{EC2022}, and enables us to control and propagate fractional powers in the angular direction -- see below.

\medskip \par \textbf{Choice of norms.} We define in \eqref{B norm}, \eqref{X norm} the $B$- and $X$-norms for a function $f:\R^2\to\R$ as
\begin{align*}
    &\norm{f}_{B}=\sup_{k\in \Z,\, p\in \Z^-} 2^{4k^+}2^{-\frac{k^-}{2}}2^{-\frac{p}{2}}\ltwonorm{P_{k,p}f},  
    &\norm{f}_{X}=\sup_{\substack{k\in \Z,\, l\in \Z^+\\ p \in \Z^-, l+p\geq 0}}\hspace{-0.3cm}2^{4k^+}2^{(1+\beta)l}2^{(\frac{1}{2}+\beta) p}\ltwonorm{P_{k,p}R_lf},
\end{align*}
where $k^-=\min\{k,0\}$ and $k^+=\max\{k,0\}$. The $B$-norm weighs the parameters $p\in \Z^-$ negatively and scales like the Fourier transform in $L^\infty$, whereas the $X$-norm weighs the parameter $p$ positively and gives control of $(1+\beta)$-derivatives in $W$ (expressed in terms of the angular localization parameter $l$). While propagating higher powers of $W$ nonlinearly is more difficult, it is also clear that a certain minimal power is needed in order to have a chance to obtain optimal decay estimates: In particular, we note that slightly more than one order of $W$ is needed in order to ensure control of the Fourier transform in $L^\infty$, see Lemma \ref{lemma: control of Fourier transform}.

\medskip \par \textbf{Linear decay and choice of norms.}
 A first key step of our proof is a refined \emph{linear} decay estimate for the semigroup $e^{\pm it\Lambda}$ in terms of our norms, see Proposition \ref{proposition: linear decay}. In general, it is known that the sharp $L^\infty$ decay rate is $t^{-1/2}$ (see \cite{Elgindi_2015}, reflecting the degeneracy of the dispersion), and we capture this as
 \begin{align*}
        \sinftynorm{P_ke^{\pm it\Lambda}f}\lesssim t^{-\frac{1}{2}}\sup_{0\leq n\leq 2}{(\norm{S^nf}_B
+\norm{S^nf}_X)},
    \end{align*}
     \blue{where the vector field $S$ is defined in \eqref{def:intro-vf}}. However, here it is important to track more detailed information that in particular allows us to obtain \emph{faster decay away from the degeneracy} of $\Lambda$. More precisely, in Proposition \ref{proposition: linear decay} we split the action of the semigroup in two components: one corresponds to high angular frequencies and decays \emph{in $L^2$}, whereas the other gives an $L^\infty$ decay, both quantified in terms of time and the parameter $p$ relating to the degeneracy. In particular, for $p>-10$ we obtain the almost full $t^{-1}$ decay rate in $L^\infty$, while for small $p$ this degenerates to scale at worst as $t^{-1/2}$.

\medskip \par \textbf{Energy estimates.} In our bootstrap setting, the $L^\infty$ decay of solutions can be directly used to establish $H^{N_0}$ energy estimates as well as $L^2$ estimates for many vector fields $S^n$, $n\leq M$, applied to a solution $(u,\rho)$ of \eqref{eqn: perturbed BQ} (or solution $\theta$ to \eqref{SQG}, respectively), see Section \ref{sec: both energy estimates}. The proof is standard for the $H^{N_0}$ energies, and proceeds through an inductive argument building on the commutator rule $[S,\nabla]=-\nabla$ for the vector fields (see \eqref{eqn: commutator S^n nabla} for an iterated version). The corresponding blow-up criteria show that these energies grow with the exponential of the time integral of amplitudes, which in our bootstrap leads to a growth factor of the form $\exp(\int_0^t \eps (1+s)^{-1/2} ds)$. The natural timescale for this to be uniformly bounded is thus $t\lesssim \eps^{-2}$ (see also Corollaries \ref{cor: energy estimates of Z+-} resp.\ \ref{cor: energy estimates SQG}).

In what follows, the $H^{N_0}$ energy estimates are used chiefly to obtain the desired bounds for high frequencies (called ``simple cases'' below), whereas the $S^n$, $n\leq M$, energy estimates are a key tool for iterated integration by parts along $S$, see below.

\medskip \par \textbf{Oscillatory toolbox: integration by parts along $S$ and normal forms.}
To exploit oscillations in the bilinear terms \eqref{def: introduction bilinear term}, we develop a framework for repeated integration by parts along the vector field $S$. To that end, it is important to understand the iterated action of the vector fields $S,\; W$ on the objects involved, and in particular on the multipliers. To systematically treat these, in Section \ref{sec: multiplier mechanics} we introduce a class of symbols that includes the building blocks involved in the multipliers and the phases, and is closed under the action of the vector fields. For this class, in Lemmas \ref{lemma: S on basic multipliers}-\ref{lemma: S_eta on multiplier m} we establish bounds (in terms of our localization parameters $k_i,p_i,l_i$, $i=1,2$, corresponding to the variables $\xi-\eta$ and $\eta$ involved in \eqref{def: introduction bilinear term}) for the iterated action of the vector fields $S,\; W$.
As a simple yet important observation, we find a suitable algebraic skew structure (see \eqref{eqn: symmetry of sigma}) that shows that whenever there is a ``gap" in the localization parameters $p,p_i$, then $S\Phi$ is bounded from below. Moreover, we encounter a rich structure that links lower bounds of $S\Phi$ with smallness of the phase $\Phi$ itself (see Proposition \ref{prop: lower bound on sigma}). Roughly speaking, this implies -- in quantifiable terms -- that we either have a lower bound for $S\Phi$ and thus iterative integrations by parts, or the phase $\Phi$ is comparatively large -- see Section \ref{sec: vector fields and the phases}. Assuming a lower bound for the action of $S$ on the phases, we collect this information in Lemma \ref{lemma: ibp in bilinear expressions}, where we present bounds for iterated integration by parts along $S$.
A further version of this is presented in the subsequent Lemma \ref{lemma: ibp in D_eta}, and Lemma \ref{lemma: W_xi on bilinear expressions} follows along similar lines.

To complement these arguments, we show in Section \ref{sec: Normal forms} how largeness of the phase function $\Phi$ can be taken advantage of via normal forms (i.e.\ an integration by parts in time), in particular in combination with other restrictions on the frequency configurations (see also Section \ref{sec: set-size estimate}).

\par In the context of this framework, with a proper organization of cases (see also Lemma \ref{lemma: case organisation}), the rough overall structure of the proofs of the various bounds on bilinear terms \eqref{def: introduction bilinear term} can be sketched as follows:
\begin{enumerate}
    \item \emph{Simple cases:} We observe that for very large or very small frequencies, we obtain the desired bounds via energy estimates and set size bounds.
    \item \emph{Gap in $p$:} Here we can integrate by parts according to Lemma \ref{lemma: ibp in bilinear expressions} to obtain the desired bounds for certain ranges of the localization parameters. In the remaining cases, we can use a balance of the $B$- and $X$-norms, depending on the size of the parameters or else normal forms, accompanied by set size estimates.
    \item \emph{No gap:} The refined linear decay estimates allow us to take advantage of the comparability of localization parameters $p$. This already suffices to establish the corresponding $B$ norm bounds, but additional arguments are necessary for the $X$ norm. 
\end{enumerate}

\medskip \par \textbf{Improved decay for the time derivative of  profiles in $L^2$.} The first instance where the aforementioned tools are used is in Section \ref{sec: bounds on dtSf}, where we establish a decay rate of almost $t^{-3/4}$ for the $L^2$ norm of the time derivative of the profiles $F_i\in\{\zz_\pm,\Theta\}$ (contrast this with the simple direct estimate, which only yields a decay at rate $t^{-1/2}$). \blue{The improved decay of the time derivative is particularly useful when employing normal forms in the nonlinear analysis discussed below.} \blue{To prove this result,} we follow the scheme described above, and after dealing with the simple cases we localize the profiles inside the integrals \eqref{def: introduction bilinear term} as $f_i=R_{l_i}P_{k_i,p_i}S^{b_i}F_i$, $i=1,2$, $b_1+b_2\leq N$. In the \emph{gap in $p$} case, we integrate by parts 
when feasible. Otherwise, we are in the setting where angular parameters $-l_i$ yield the decay at the cost of parameters $-p,-p_i$. We note that the parameters $p\in \Z^-$ come with a negative sign and to compensate for these ``losses'' we invoke the null structure of the nonlinearity, see e.g. Lemma \ref{lemma: bounds on dtSf} Case B.1.1. The \emph{no gap} case is easily covered by the refined linear decay from Proposition \ref{proposition: linear decay}, Lemma \ref{lemma: linear decay many vector fields}. 

\medskip\par\blue{  
\textbf{Bounds on the $B$- and $X$-norm.} Finally, in Sections \ref{sec: bounds on B norm} and \ref{sec: bounds on the X-norm}, respectively, we control bilinear terms of the form \eqref{def: introduction bilinear term} in the  $B$- and $X$-norms: we show that
  \begin{align*}
       \norm{\B_\m(F_1,F_2)}_B\lesssim  t^{\frac{1}{6}+\delta}\eps^2,\qquad \norm{\B_\m(F_1,F_2)}_X \lesssim t^{\frac{1}{2}-\delta}\eps^2,
    \end{align*}
   where $F_i=S^{b_i}\zz_\pm$, or $F_i=S^{b_i}\Theta$, for $\delta\ll1, b_1+b_2\leq N$. (In particular, this shows that with the present arguments the $B$-norm bound itself could be propagated on a time interval of almost order $\eps^{-6}\gg \eps^{-2}$, provided suitable improvements for the the energy estimates and the $X$-norm are established.)
The proof of theses estimates follows the strategy outlined above, and refines the techniques employed already to establish the improved decay of the time derivative of profiles. In particular, for the $B$-norm bounds the \textit{gap in $p$} cases become more delicate and we complement integrating by parts along vector fields with a normal form transform in certain configurations (in particular when the phase satisfies $\abs{\Phi}\gtrsim 1$). The latter yields bounds of the form
 \begin{align*}
      \ltwonorm{P_{k,p}\B_{\m}(f_1,f_2)}\lesssim \ltwonorm{P_{k,p}\mathcal{Q}_{\m\Phi^{-1}}(f_1,f_2)}+\ltwonorm{P_{k,p}\B_{\m\Phi^{-1}}(\partial_tf_1,f_2)}+\ltwonorm{P_{k,p}\B_{\m\Phi^{-1}}(f_1,\partial_tf_2)},
    \end{align*}
    where $f_i=P_{k_i,p_i}R_{l_i}F_i$ are the localized profiles and $\mathcal{Q}$ is a bilinear term of the form \eqref{def: introduction bilinear term} but without time integral. The two last terms above are handled using the bound on the time derivative described above and are of cubic order. On the other hand, the first term contains one less time parameter and previous arguments suffice to estimate it. The proof of the $X$-norm bounds is yet more delicate, requiring refinements of the aforementioned tools and we refer to the beginning of Sections \ref{sec:X-norm-1}, \ref{sec:X-norm-2} for more details.}
\subsection{Plan of the article}
In Section \ref{sec: functional framework} we introduce the necessary background to proceed with the proof of Theorems \ref{thm: main thm BQ}, \ref{thm: main thm SQG}. We describe in detail the choice of dispersive unknowns for the Boussinesq system in Section \ref{sec: choice of unknowns} and present the natural vector fields arising from the scaling symmetry of the equations in Section \ref{sec: Symmetries and vector fields}. Moreover, we introduce the necessary localizations in Section \ref{sec: Localizations}. The detailed statements of the main results are presented in Theorems \ref{thm: main result}, \ref{thm: main result detailed SQG} and proven in Proposition \ref{prop: bootstrap argument} using tools from subsequent sections. The linear decay estimate is presented in Section \ref{sec: Linear Decay}. The available energy estimates are discussed in Section \ref{sec: both energy estimates}. 
\par The technical tools involving the vector fields and in particular iterated integration by parts along vector fields, set-size estimates and normal forms are presented in Section \ref{sec: integration by parts}. The improved decay of the time derivative of our unknowns is proved in Section \ref{sec: bounds on dtSf}. Estimates on the $B$- and $X$-norms are shown in Sections \ref{sec: bounds on B norm}, \ref{sec: bounds on the X-norm}. Appendix \ref{appendix} contains auxiliary results such as the control of the Fourier transform in $L^\infty$ and multiplier bounds.

\section{Functional Framework and Main Result}\label{sec: functional framework}
In this section, we introduce the basic framework for our arguments and present the main results Theorems \ref{thm: main thm BQ} and \ref{thm: main thm SQG} in more detail. In particular, with a suitable functional framework and through an adequate choice of scalar dispersive unknowns for the Boussinesq system, we will show that the proof of the main results reduces to the study of a bootstrap argument involving certain bilinear expressions, the essential features of which are common to both the Boussinesq and SQG systems.

\subsection{Choice of scalar unknowns}\label{sec: choice of unknowns}
Consider solutions to the Boussinesq system \eqref{eqn: perturbed BQ} written as a system for the two scalar unknowns of vorticity and density $\omega, \rho:\R^+\times \R^2\to \R$ as
\begin{equation}\label{eqn: bq}
\left\{
      \begin{aligned}
 &\partial_t\omega +u\cdot \nabla \omega=-\partial_{x_1}\rho,\\
 &\partial_t\rho+u\cdot \nabla\rho=\partial_{x_1}\Delta^{-1}\omega, \\
    &u=\nabla^\perp \Delta^{-1}\omega,
\end{aligned}
    \right.
\end{equation}
where by convention $\nabla^\perp=(-\partial_{x_2},\partial_{x_1})$. The following result provides a choice of scalar unknowns that diagonalize the associated linear system:

\begin{proposition}\label{prop: profiles BQ}
Let $(\omega,\rho)\in C([0,T],(H^{-1}\cap H^s)\times H^s)$ solve \eqref{eqn: bq}.
Define the dispersive unknowns $Z_\pm$ and their profiles $\mathcal{Z}_\pm$ by
\begin{equation}\label{def: Profiles BQ}
    Z_{\pm}:=\absnabla^{-1}\omega\pm\rho,\qquad \mathcal{Z}_\pm:=e^{\pm it\Lambda}Z_\pm,
\end{equation}
where the dispersive operator is given by \begin{align*}
    \Lambda(\xi):=\frac{\xi_1}{\abs{\xi}}.
\end{align*}
Then $\mathcal{Z}_\pm$ satisfy 
\begin{align}\label{eqn: bilin_profiles}
    &\zz_\pm(t)=\zz_\pm(0)+\sum_{\mu\in \{+,-\}}\int_0^t\mathcal{Q}_{\m^{\mu\mu}_\pm}(\mathcal{Z}_{\mu},\mathcal{Z}_{\mu})(s)ds+\int_0^t\mathcal{Q}_{\m^{+-}_\pm}(\mathcal{Z}_{+},\mathcal{Z}_{-})(s)ds,
\end{align}
where 
\begin{align}\label{eqn: bilin_form_Q}
    \F(\mathcal{Q}_{\m^{\mu\nu}_\pm}(\mathcal{Z}_{\mu},\mathcal{Z}_{\nu}))(s,\xi)=\int_{\R^2} e^{ is\Phi^{\mu\nu}_\pm(\xi,\eta)}\m^{\mu\nu}_\pm(\xi,\eta) {\widehat{\zz_{\mu}}}(s,\xi-\eta)\widehat{\zz_{\nu}}(s,\eta)d\eta,
\end{align}
with phase functions
\begin{align}\label{eqn: bq phases}
&\Phi^{\mu\nu}_\pm(\xi,\eta)=\pm\Lambda(\xi)-\mu\Lambda(\xi-\eta)-\nu\Lambda(\eta),
\end{align}
and multipliers
\begin{equation}\label{eqn: bq multipliers}
\begin{split}
    &\m_\pm^{\mu\mu}(\xi,\eta)=-\frac{1}{8}\frac{\xi(\xi-\eta)^\perp}{\absxi\absxieta}\Big(\frac{\abseta^2-\absxieta^2}{\abseta}\Big)\mp\mu\frac{1}{8}\frac{(\xi-\eta)^\perp\eta}{\absxieta\abseta}\big(\absxieta-\abseta\big),\quad \mu\in\{-,+\} \\
  &\m^{+-}_\pm(\xi,\eta)=-\frac{1}{4}\frac{\xi(\xi-\eta)^\perp}{\absxi\absxieta}\Big(\frac{\abseta^2-\absxieta^2}{\abseta}\Big)\pm \frac{1}{4}\frac{(\xi-\eta)^\perp\eta}{\absxieta\abseta}\big(\absxieta +\abseta\big).
\end{split}
\end{equation} 
\end{proposition}
Moreover, a direct computation using that 
\begin{equation}\label{eqn: urho}
 u=-\frac{1}{2}\nabla^\perp\absnabla^{-1}(Z_++Z_-),\quad \rho=\frac{1}{2}(Z_+-Z_-),   
\end{equation}
shows that this choice of unknowns preserves the energy structure in the sense that
\begin{align}\label{eqn: energy conservation Z+ Z-}    
\norm{u}_{\dot{H}^k}^2+\norm{\rho}_{\dot{H}^k}^2=\frac{1}{2}\norm{Z_+}_{\dot{H}^k}^2+\frac{1}{2}\norm{Z_-}_{\dot{H}^k}^2=\frac{1}{2}\norm{\zz_+}_{\dot{H}^k}^2+\frac{1}{2}\norm{\zz_-}_{\dot{H}^k}^2,\quad k\in\N_0.
    \end{align}
\begin{proof}
By a direct computation, the system \eqref{eqn: bq} is equivalent to
\begin{equation}\label{eqn: BQ on Z+, Z-}
    \begin{split}
         &\partial_tZ_\pm +\frac{1}{4}\absnabla^{-1}\mathrm{div}\Big(\nabla^\perp \absnabla^{-1}(Z_++Z_-)\cdot \absnabla(Z_++Z_-)\Big)\\
 &\hspace{3.2cm}\pm\frac{1}{4}\nabla^\perp\absnabla^{-1}(Z_++Z_-)\cdot\nabla(Z_+-Z_-)=\pm R_1Z_+.
    \end{split}
\end{equation}
This can be rewritten compactly as follows
\begin{equation}\label{eqn: bq on Z+ Z- compact}
   (\partial_t\mp R_1)Z_\pm = \mathcal{N}_{\n^{++}_\pm}(Z_+,Z_+)+\mathcal{N}_{\n^{+-}_\pm}(Z_+,Z_-)+\mathcal{N}_{\n^{-+}_\pm}(Z_-,Z_+)+\mathcal{N}_{\n^{--}_\pm}(Z_-,Z_-),
\end{equation}
where for $\mu,\nu \in \{+,-\}$
\begin{align*}
    \F(\mathcal{N}_{\n^{\mu\nu}_\pm}(f,g))(\xi):=\int_{\R^2} \n^{\mu\nu}_\pm \hat{f}(\xi-\eta)\hat{g}(\eta)d\eta,
\end{align*}
 with multipliers
\begin{equation}\label{eqn: multipliers n}
\begin{split}
      & \n^{++}_{\pm}=\n^{-+}_\pm=-\frac{1}{4}\frac{\xi(\xi-\eta)^\perp}{\absxi\absxieta}\abseta \mp \frac{1}{4}\frac{(\xi-\eta)^\perp\eta}{\absxieta\abseta}\abseta,\\
    & \n^{--}_{\pm} = \n^{+-}_\pm=-\frac{1}{4}\frac{\xi(\xi-\eta)^\perp}{\absxi\absxieta}\abseta \pm \frac{1}{4}\frac{(\xi-\eta)^\perp\eta}{\absxieta\abseta}\abseta.
\end{split}
\end{equation}
Observe that since $\F(R_1f)(\xi)=-i\Lambda(\xi)\widehat{f}$, $\Lambda(\xi)=\frac{\xi_1}{\absxi}$, for the profiles $\zz_\pm$ there holds
\begin{align*}
   & \widehat{\zp}=e^{it\Lambda}\widehat{Z_+}, && \widehat{\zm}=e^{-it\Lambda}\widehat{Z_-},
\end{align*}
and by the Duhamel formulation we obtain that
\begin{align}\label{eqn: Duhamel BQ Z+ Z-}
    &\zz_\pm(t)=\zz_\pm(0)+\sum_{\mu,\nu \in \{+,-\}}\B_{\n_\pm^{\mu\nu}}(\mathcal{Z_\mu},\mathcal{Z_\nu})(t),
\end{align}
with 
\begin{align*}
  & \F(\B_{\n_\pm^{\mu\nu}}(\mathcal{Z_\mu},\mathcal{Z_\nu}))(t,\xi): =\int_0^t\F(\mathcal{Q}_{\n^{\mu\nu}_\pm}(\mathcal{Z}_{\mu},\mathcal{Z}_{\nu}))(s,\xi)ds, \\
  &\F(\mathcal{Q}_{\n^{\mu\nu}_\pm}(\mathcal{Z}_{\mu},\mathcal{Z}_{\nu}))(s,\xi):=\int_{\R^2} e^{is\Phi_\pm^{\mu\nu}(\xi,\eta)}\n^{\mu\nu}_\pm(\xi,\eta) \widehat{\mathcal{Z}_{\mu}}(s,\xi-\eta)\widehat{\mathcal{Z}_{\nu}}(s,\eta)d\eta
\end{align*}
and phase functions as in \eqref{eqn: bq phases}.
To arrive at the further simplified expression in \eqref{eqn: bilin_profiles} we symmetrize and collect terms: Observe that by symmetry of $\n^{\mu\mu}_\pm$ and $\Phi_{\mu\mu}$ under the change of variables $\eta\leftrightarrow\xi-\eta$ there holds that
    \begin{align*}
        \F(\mathcal{Q}_{\n^{\mu\mu}_\pm}(\mathcal{Z}_{\mu},\mathcal{Z}_{\mu}))(s,\xi)&=\int_{\R^2} e^{ is\Phi_\pm^{\mu\mu}(\xi,\eta)}\n^{\mu\mu}_\pm(\xi,\eta) \widehat{\mathcal{Z}_{\mu}}(s,\xi-\eta)\widehat{\mathcal{Z}_{\mu}}(s,\eta)d\eta\\
        &=\int_{\R^2} e^{ is\Phi_\pm^{\mu\mu}(\xi,\xi-\eta)}\n^{\mu\mu}_\pm(\xi,\xi-\eta) \widehat{\mathcal{Z}_{\mu}}(s,\eta)\widehat{\mathcal{Z}_{\mu}}(s,\xi-\eta)d\eta\\
        &=\F(\mathcal{Q}_{\m^{\mu\mu}_\pm}(\mathcal{Z}_{\mu},\mathcal{Z}_{\mu}))(s,\xi).
    \end{align*}
On the other hand, with the same change of variables and the symmetry $\Phi^{+-}_\pm(\xi,\eta)=\Phi^{-+}_\pm(\xi,\xi-\eta)$ we compute that
    \begin{align*}
        & \F({\mathcal{Q}}_{\n^{+-}_\pm}(\zz_+,\zz_-))(s,\xi)+ \F(\mathcal{Q}_{\n^{-+}_\pm}(\zz_-,\zz_+))(s,\xi)\\
         &=\int_{\R^2} e^{is\Phi^{+-}_\pm(\xi,\eta)}\n^{+-}_\pm(\xi,\eta)\widehat{\zz_+}(s,\xi-\eta)\widehat{\zz_-}(s,\eta)d\eta +\int_{\R^2} e^{is\Phi^{-+}_\pm(\xi,\eta)}\n^{-+}_\pm(\xi,\eta)\widehat{\zz_-}(s,\xi-\eta)\widehat{\zz_+}(s,\eta)d\eta\\
         &=\int_{\R^2} e^{is\Phi^{+-}_\pm(\xi,\eta)}\n^{+-}_\pm(\xi,\eta)\widehat{\zz_+}(s,\xi-\eta)\widehat{\zz_-}(s,\eta)d\eta+\int_{\R^2} e^{is\Phi^{+-}_\pm(\xi,\eta)}\n^{-+}_\pm(\xi,\xi-\eta)\widehat{\zz_-}(s,\eta)\widehat{\zz_+}(s,\xi-\eta)d\eta\\
         &=\F(\mathcal{Q}_{\m^{+-}_\pm}(\mathcal{Z}_{+},\mathcal{Z}_{-}))(s,\xi).
    \end{align*}
\end{proof}
Similarly we can reformulate the problem for the dispersive SQG equation \eqref{SQG}. If $\theta(t)$ solves \eqref{SQG} and $\Theta(t):=e^{it\Lambda}\theta(t)$ is the associated profile, then
\begin{align}\label{eqn: Duhamel for the profile}
    \Theta(t)=\Theta(0)+\B_{\m_0}(\Theta,\Theta)(t),
\end{align}
where
\begin{align}
    \F\B_{\m_0}(\Theta,\Theta)(t,\xi)&=\int_0^t \int_{\R^2}e^{is\Phi_+^{++}(\xi, \eta)}\m_0(\xi,\eta)\widehat{\Theta}(\xi-\eta)\widehat{\Theta}(\eta)d\eta ds,  \label{nonlinearity SQG} \\
    \m_0(\xi,\eta)&:=\frac{1}{2} \frac{(\xi-\eta)\cdot\eta^\perp}{\absxieta \abseta}(\absxieta-\abseta). \label{multiplier}
\end{align}
\begin{proof}[Proof of \eqref{eqn: Duhamel for the profile}-\eqref{multiplier}]
 By Duhamel's formula we have that
 \begin{align*}
    \theta(t)=e^{tR_1}\theta_0+ \int_0^t e^{(t-s)R_1}u\cdot \nabla \theta(s)ds,
\end{align*}
and thus 
\begin{align*}
    \widehat{\Theta}(t,\xi)&=\widehat{\Theta}_0(\xi)+\int_0^t \int_{\R^2}e^{is\Phi^{++}_{+}(\xi, \eta)}\frac{\eta^\perp\cdot(\xi-\eta)}{\abs{\eta}}\widehat{\Theta}(s,\eta)\widehat{\Theta}(s,\xi-\eta)d\eta ds,
\end{align*}
and the change of variables $\eta\leftrightarrow\xi-\eta$ as above gives the claim.
\end{proof}

\begin{subsection}{Scaling symmetry and vector fields}\label{sec: Symmetries and vector fields}
In this section we discuss the presence of natural derivatives arising from a scaling symmetry. Observe that the perturbed Boussinesq system \eqref{eqn: perturbed BQ} (\eqref{eqn: bq} resp.\ ) has the following scaling symmetry for $\lambda>0$:
\begin{align*}
   & u_\lambda(t,x)=\lambda u(t,\lambda^{-1}x), && \rho_\lambda(t,x)=\lambda \rho(t,\lambda^{-1}x), \\
   & \omega_\lambda(t,x)=\omega(t,\lambda^{-1}x), 
     &&p_\lambda(t,x)=\lambda^2p(t,\lambda^{-1}x).
\end{align*}
That is, if $(u,\rho)$ solve \eqref{eqn: perturbed BQ} with pressure $p$, then $(u_\lambda,\rho_\lambda)$ solve \eqref{eqn: perturbed BQ} with pressure $p_\lambda$. Similarly, if $(\omega,\rho)$ solves \eqref{eqn: bq}, then so does $(\omega_\lambda,\rho_\lambda)$. Solutions of the dispersive SQG equation satisfy an analogous scaling: if $\theta$ solves \eqref{SQG}, then so does $\theta_\lambda(t,x)=\lambda\theta(t, \lambda^{-1} x)$ for $\lambda>0$. This symmetry group is generated by the vector field $\mathcal{S}$ acting on functions $f$ as
    \begin{align}\label{def: S W}
       \mathcal{S}f:= - f+ Sf, \hspace{1.5cm} Sf:=x\cdot \nabla_xf.
    \end{align}
In particular, (as can be verified also directly since $S\Lambda=0$) we have that $S$ commutes with the linear semigroup of the Boussinesq resp.\ SQG equations,
\begin{equation}\label{eqn: semigroup_comm}
 [S,e^{it\Lambda}]=0.
\end{equation}
  
In order to span the full tangent space at each point, we complement the natural vector field $S$ with 
    \begin{align}\label{def: W}
        Wf:=x^{\perp}\cdot\nabla_xf.
    \end{align}
In polar coordinates $x \mapsto (r\cos \tau,r\sin\tau)$ these derivatives are given as the radial and angular derivative respectively, $ S=r\partial_r,\; W=\partial_\tau.$ This will be useful in the following sections. 
\par Moreover, we observe that the decomposition \eqref{def: Profiles BQ} of the Boussinesq unknowns $(u,\rho)$ into dispersive unknowns $Z_\pm$ and profiles $\zz_\pm$ interfaces naturally with the vector fields $S$ in $L^2$: By direct computation and using \eqref{eqn: semigroup_comm} we have that
\begin{equation}\label{eqn: S conservation Z+ Z-}
 \snorm{S^ku}_{L^2}^2+\snorm{S^k\rho}_{L^2}^2=\frac12\snorm{S^kZ_+}_{L^2}^2+\frac12\snorm{S^kZ_-}_{L^2}^2=\frac12\snorm{S^k\zz_+}_{L^2}^2+\frac12\snorm{S^k\zz_-}_{L^2}^2,\quad k\in \N.
\end{equation}
 
\end{subsection}
\begin{subsection}{Localizations}\label{sec: Localizations}
In this section we introduce localizations in frequency and angle, which will allow us to quantify the nonlinear interactions.
\par To define the Littlewood-Paley projections, let $\psi \in C^\infty (\R,[0,1])$ a radially symmetric bump function with $\supp \psi \subset [-\frac{8}{5},\frac{8}{5}]$ and $\psi|_{[-\frac{4}{5},\frac{4}{5}]}\equiv 1$. Moreover, we let $\varphi(x):=\psi(x)-\psi(2x)$ and define for $a \in \Z, \; b,\; c \in \Z^-$ and $\Lambda$ as in \eqref{eqn: bq phases}
\begin{align*}
  \varphi_{a,b}(\zeta):=\varphi(2^{-a}\abs{\zeta})\varphi(2^{-b}\sqrt{1-\Lambda^2(\zeta)}), &&    \varphi_{a,b,c}(\zeta):=\varphi(2^{-a}\abs{\zeta})\varphi(2^{-b}\sqrt{1-\Lambda^2(\zeta)})\varphi(2^{-c}\Lambda(\zeta)).
\end{align*}
For $k\in \Z$, $p,q\in \Z^-$ we define the associated Littlewood-Paley projections by
\begin{align*}
    \F(P_{k,p}f)(\xi)=\varphi_{k,p}(\xi)\widehat{f}(\xi), && \F(P_{k,p,q}f)(\xi)=\varphi_{k,p,q}(\xi)\widehat{f}(\xi).
\end{align*}
In later sections, we will use the localization projections simultaneously for the variables $\xi, \, \xi-\eta$ and $\eta$, and thus introduce the following short-hand notation
\begin{align}\label{eqn: chi localizations}
    \chi(\xi,\eta)=\varphi_{k,p}(\xi)\varphi_{k_1,p_1}(\xi-\eta)\varphi_{k_2,p_2}(\eta), && \widetilde{\chi}(\xi,\eta)=\varphi_{k,p,q}(\xi)\varphi_{k_1,p_1,q_1}(\xi-\eta)\varphi_{k_2,p_2,q_2}(\eta).
\end{align}
\begin{remark}\label{rk: notation similar supp properties}
    Throughout this paper, we will denote by $\overline{\chi}$ (resp.\ $\overline{\tchi}, \overline{\varphi}$) a function with similar support properties as $\chi$ (resp.\ $\tchi, \varphi$). For simplicity of notation we do not distinguish the corresponding localization operators $P_{a,b}$, $P_{a,b,c}$ arising from $\varphi$ or $\overline{\varphi}$.
\end{remark}
\par Next we introduce Littlewood-Paley-type localizations in order to quantify regularity in the polar coordinate angle. To that end, let $f\in L^2$ and consider polar coordinates $x\mapsto (r\cos\tau,r\sin\tau)$. Then we can expand
\begin{align}\label{eqn: Fourier expansion}
    f(x)=\sum_{n\in \Z}f_n(r)e^{in\tau},\quad f_n(r)=\frac{1}{2\pi}\int_0^{2\pi} f(r\cos\tau,r\sin\tau)e^{-in\tau}d\tau.
\end{align}
We recall here that by Parseval's theorem there holds
\begin{align}\label{Parseval angular localization}
    \ltwonorm{f}^2=2\pi \sum_{n\in \Z}\norm{f_n}^2_{L^2(\R^+,rdr)}.
\end{align}
Changing back to Cartesian coordinates in \eqref{eqn: Fourier expansion}, for $l\in \Z$ we define angular projections as 
\begin{align*}
(\bar{R}_{\leq l}f)(x)&:=\sum_{n\in \Z} \psi(2^{-l}n) \int_{\sphere^1} f(\abs{x}y)e^{-in \arccos{(y\cdot \frac{x}{\abs{x}})}}d\text{vol}_{\sphere^1}(y),\\
   (\bar{R}_{l}f)(x)&:=\sum_{n\in \Z} \varphi (2^{-l}n) \int_{\sphere^1} f(\abs{x}y)e^{-in \arccos{(y\cdot \frac{x}{\abs{x}})}}d\text{vol}_{\sphere^1}(y). \end{align*}
\begin{proposition}\label{prop: angular localization properties} Let $f\in L^2$, $\bar{R}_{\leq l}$ and $\bar{R}_l$ defined as above with $l\in \Z$, and $W$ as in \eqref{def: W}. Then following properties hold:
    \begin{enumerate}
        \item\label{it1: prop: angular localization properties} $f=\sum_{l\geq 0} \bar{R}_lf$  \quad and \quad $\ltwonorm{f}^2 \sim \sum_{l\geq 0}\ltwonorm{\bar{R}_lf}^2$;
        \item\label{it2: prop: angular localization properties} The operators $\bar{R}_{\leq l}$ and $\bar{R}_l$ are bounded in $L^\ell$ for $1\leq \ell\leq \infty$;
        \item\label{it3: prop: angular localization properties} The Bernstein property reads:
        \[\lpnorm{W \bar{R}_lf}{\ell} \sim 2^l\lpnorm{\bar{R}_lf}{\ell}.\]
    \end{enumerate}
\end{proposition}
\begin{proof}
The first property in \eqref{it1: prop: angular localization properties} follows from \eqref{eqn: Fourier expansion} and the fact that $\sum_{l\geq 0} \varphi(2^{-l}\cdot)$ is a partition of unity. Moreover, with \eqref{Parseval angular localization} and the fact that $\varphi^2$ has similar support properties as $\varphi$, there holds: 
\begin{align*}
    \ltwonorm{f}^2&=\int_0^\infty\int_0^{2\pi} \abs{f(r\cos\tau,r\sin\tau)}^2d\tau rdr\\
    &=\int_0^\infty 2\pi \sum_{n\in \Z} \abs{f_n(r)}^2rdr\\
    &=2\pi \int_0^\infty \sum_{n\in \Z}\sum_{l\geq 0} \varphi^2(2^{-l}n)\abs{f_n(r)}^2rdr\\
    &\blue{\sim} \sum_{l\geq 0} \sum_{n\in \Z}\int_{\R^2}\varphi^2(2^{-l}n)\abs{\int_{\sphere^1}f(\abs{x}y)e^{-in\arccos{(y\cdot\frac{x}{\abs{x}})}}d\text{vol}_{\sphere^1}(y)}^2dx\\
    &=\sum_{l\geq 0}\ltwonorm{\bar{R}_l f}^2.
\end{align*}
We proceed with the proof of \eqref{it2: prop: angular localization properties} for $\bar{R}_l$ and the result for $R_{\leq l }$ follows similarly. We view the operator $\bar{R}_l$ as a singular integral operator with kernel ${K_l}(x,y)=\sum_{n\in\Z}\varphi(2^{-l}n)e^{-in\arccos{(y\cdot x})}$ as follows:
\begin{align*}
    \bar{R}_lf(x)&=\int_{\sphere^1}f(\abs{x}y)\sum_{n\in\Z}\varphi(2^{-l}n)e^{-in\arccos{(y\cdot\frac{x}{\abs{x}})}}d\text{vol}_{\sphere^1}(y)\\
    &=\int_{\sphere^1}f(\abs{x}y){K_l}\Big(\frac{x}{\abs{x}},y\Big)d\text{vol}_{\sphere^1}(y).
\end{align*}
Since $\sabs{e^{-in\arccos{(y\cdot\frac{x}{\abs{x}})}}}=1$ and the telescoping sum present in $K_l$ is bounded, there holds:
\begin{align*}
\sup_{x}\snorm{{K_l}(x,y)}_{L^1{(\sphere^1,d\text{vol}(y))}}+\sup_{y}\snorm{{K_l}(x,y)}_{L^1{(\sphere^1,d\text{vol}(x))}}\lesssim 1.
\end{align*}
The claim follows then by Young's inequality for integral operators.
\par As for the proof of \eqref{it3: prop: angular localization properties} recall that in polar coordinates $W=\partial_\tau$. Using the properties of the Fourier transform and the equivalent polar coordinate representation above we see
   \begin{align*}
\snorm{W\bar{R}_lf}_{L^\ell_{x}}&=\snorm{W\bar{R}_lf}_{L^\ell( r dr d\tau)}=2\pi \big \lVert{\partial_\tau\sum_{n\in \Z}\varphi(2^{-l}n)\int_{0}^{2\pi}f(r\cos\tau,r\sin\tau)e^{-in\tau}d\tau}\big \rVert_{L^\ell(\R^+, rdr)}\\
       &=2\pi\big \lVert{\sum_{n\in \Z}\varphi(2^{-l}n)in\int_0^{2\pi}f(r\cos\tau,r\sin\tau)e^{-in\tau}d\tau}\big\rVert_{L^\ell}\\
       &\sim 2^l \snorm{\bar{R}_lf}_{L^{\ell}}.
   \end{align*}
\end{proof} 
Throughout the paper we will use polar coordinates in frequency space \begin{align}\label{eqn: polar coordinates}
    \xi\mapsto(\rho \cos\tau,\rho \sin\tau)=(\rho\Lambda, \pm\rho\sinlambda),
\end{align} 
 and without loss of generality we consider the upper hemisphere $(\rho,\tau) \in \R_+\times [0,\pi],$ so that $\xi=(\rho\Lambda,\rho\sinlambda)$.  Then there holds
 \begin{align}\label{eqn: phi_kp in rho lambda}
     &\varphi_{k,p}(\xi)=\varphi_{k,p}(\rho,\tau)=\varphi(2^{-k}\rho)\varphi(2^{-p}\sinlambda),&&\Lambda(\xi)=\cos \tau.
 \end{align}

To understand the interplay of the various projections, we observe that with  $ \xi=\rho \partial_\rho\xi,\; \xi^\perp=-\sinlambda\partial_\Lambda \xi$, there holds 
\begin{align*}
    \lpnorm{[W,P_{k,p}]f}{\ell}&=\lVert \varphi(2^{-k}\rho)\sinlambda 2^{-p}\varphi'(2^{-p}\sinlambda)\partial_\Lambda(\sinlambda)f \\
    &\hspace{0.15cm}+\sinlambda \varphi(2^{-k}\rho)\varphi(2^{-p}\sinlambda)\partial_\Lambda f -\sinlambda \varphi(2^{-k}\rho)\varphi(2^{-p}\sinlambda)\partial_\Lambda f\rVert_{L^\ell}\\
    &\lesssim 2^{-p}.
\end{align*}
In particular, 
\begin{align*}
\snorm{WP_{k,p}\bar{R}_lf}_{L^\ell}&=\snorm{[W,P_{k,p}]\bar{R}_lf+P_{k,p}W\bar{R}_lf}_{L^\ell}\lesssim 2^{-p}\snorm{\bar{R}_lf}_{L^\ell}+2^l\snorm{\bar{R}_lf}_{L^\ell},
\end{align*}
and thus for simultaneous localizations in $k,p,l$ the analogue of the above Bernstein property in \ref{prop: angular localization properties}\eqref{it3: prop: angular localization properties} can only hold if $-p\leq l$. To automatically take this into account we define the operators
\[R_l^p:=\begin{cases}
    0, & p+l<0\\
    \bar{R}_{\leq l}, & p+l=0\\
    \bar{R}_l, & p+l>0.
\end{cases}\]
In the following, we will suppress the superscript $p$ and note that these operators satisfy properties analogous to those in Proposition \ref{prop: angular localization properties}, so that in particular
\begin{align*}
&P_kf=\sum_{\substack{l\in \Z^+, p \in \Z^- \\ l+p\geq 0}}P_{k,p}R_lf, &&P_kf=\sum_{\substack{l\in \Z^+, p \in \Z^-, q\in \Z^- \\ l+p\geq 0}}P_{k,p,q}R_lf.     
\end{align*}
 These projections satisfy favorable commutation relations with the vector field $S$:
\begin{align*}
   &[S,P_k]f=-{P}_kf, && [S,P_{k,p}f]=-{P}_{k,p}f, && [S,P_{k,p,q}f]=-{P}_{k,p,q}f, &&  [S,R_l]f=0.
\end{align*}
To see this we compute that 
\begin{align*}
     \widehat{SP_kf}(\xi)&=(-2-S_\xi)\widehat{P_kf}(\xi)=-2^{-k}\absxi\varphi'(2^{-k}\absxi)\widehat{f}(\xi)+\widehat{P_kSf}(\xi),
\end{align*}
 and upon using that $S\Lambda=0$, the claims for the projections $P_{k,p}$ and $P_{k,p,q}$ also follow. Finally, for the angular projections, the claim follows from the definition of $R_l$ by recalling that in polar coordinates $S=r\partial_r$.

\par To fix notation, in our analysis we make the following notational conventions for the sizes of relevant quantities in terms of the localization parameters:
\begin{align*}
     &\abs{\xi}\sim 2^k, &&\abs{\frac{\xi_2}{\abs{\xi}}}=\sqrt{1-\Lambda^2(\xi)}\sim 2^p, && \abs{\frac{\xi_1}{\absxi}}=\abs{\Lambda(\xi)}\sim 2^q,\\
    &\abs{\xi-\eta}\sim 2^{k_1}, && \abs{\frac{\xi_2-\eta_2}{\abs{\xi-\eta}}}=\sqrt{1-\Lambda^2(\xi-\eta)}\sim \twopone, &&\abs{\frac{\xi_1-\eta_1}{\absxieta}}=\abs{\Lambda(\xi-\eta)}\sim 2^{q_1},\\
     &\abs{\eta}\sim \twoktwo, && \abs{\frac{\eta_2}{\abs{\eta}}}=\sqrt{1-\Lambda^2(\eta)}\sim \twoptwo, &&  \abs{\frac{\eta_1}{\abseta}}=\Lambda(\eta)\sim 2^{q_2}.
\end{align*}
\end{subsection}
\begin{subsection}{Main result}\label{subsec: Main Result detailed}
For $\beta>0$ to be determined (see also Remark \ref{remark: Size of parameters}), we define the following weighted norms using the notation $k^+=\max \{0,k\}$ and $k^-=\min \{0,k\}$:
\begin{align}
    &\norm{f}_{B}:=\sup_{k\in \Z,\, p\in \Z^-} 2^{4k^+}2^{-\frac{k^-}{2}}2^{-\frac{p}{2}}\ltwonorm{P_{k,p}f},  \label{B norm} \\
    &\norm{f}_{X}:=\sup_{\substack{k\in \Z,\, l\in \Z^+,\, p \in \Z^- \\ l+p\geq 0}}2^{4k^+}2^{(1+\beta)l}2^{(\frac{1}{2}+\beta) p}\ltwonorm{P_{k,p}R_lf}.\label{X norm}
\end{align}
The $B$-norm captures the anisotropic localizations (with respect to the degeneracy of the phase, via the parameter $p$) and scales like the Fourier transform in $L^{\infty}$, whereas the $X$-norm accounts for a certain amount of angular regularity in $W$ (measured through the weight in $2^l$).
\par In this framework, Theorem \ref{thm: main thm BQ} for the Boussinesq system \eqref{eqn: perturbed BQ} can be stated for the corresponding dispersive unknowns in detail as follows:
\begin{theorem}\label{thm: main result}
    Let $N>5$. There exist $M, N_0\in \N$, $\beta, \delta>0$ satisfying $N_0\gg M\gg N+\beta^{-2}$, $\delta\ll \beta$ and an $\eps_0>0$ such that if for some $0<\eps<\eps_0$ we have
\begin{equation}\label{eqn: initial data assumption}
\begin{split}
\norm{Z_{\pm,\blue{0}}}_{H^{N_0}}+\norm{S^aZ_{\pm,0}}_{L^2}&\leq \varepsilon,\hspace{1.5cm}  0\leq a \leq M, \\
\lVert{S^bZ_{\pm,\blue{0}}}\rVert_B+\lVert{S^bZ_{\pm,\blue{0}}}\rVert_X&\leq \varepsilon, \hspace{1.5cm} 0\leq b\leq N,
\end{split}
\end{equation}
then there exist $T\gtrsim \eps^{-2}$ and a unique solution $(Z_{+},Z_{-} ) \in (C([0,T],H^{N_0}(\R^2))^2$ of \eqref{eqn: bq on Z+ Z- compact} with initial data $(Z_+(0),Z_-(0))=(Z_{+,\blue{0}},Z_{-,\blue{0}})$, and therefore a unique solution $(u,\rho)\in C([0,T],H^{N_0}(\R^2,\R^2))\times C([0,T],H^{N_0}(\R^2))$ of \eqref{eqn: perturbed BQ} with initial data $(u_0,\rho_0)=\frac{1}{2}(-\nabla^\perp\abs{\nabla}^{-1}(Z_{+,0}+Z_{-,0}), Z_{+,0}-Z_{-,0})$.
\end{theorem}
Analogously, Theorem \ref{thm: main thm SQG} for the dispersive SQG equation \eqref{SQG} is stated in detail as follows:
\begin{theorem}\label{thm: main result detailed SQG}
    Let $N>5$. There exist $M, N_0\in \N$, $\beta,\delta>0$ satisfying $N_0\gg M\gg N+\beta^{-2}$, $\delta\ll \beta$ and an $\eps_0>0$ such that if $\theta_0$ satisfies
    \begin{equation}\label{eqn: initial data assumption SQG}
    \begin{split}
\norm{\theta_0}_{H^{N_0}}+\norm{S^a\theta_0}_{L^2}&\leq \varepsilon,\hspace{1.5cm}  0\leq a \leq M, \\
\lVert{S^b\theta_0}\rVert_B+\lVert{S^b\theta_0}\rVert_X&\leq \varepsilon, \hspace{1.5cm} 0\leq b\leq N
\end{split}
\end{equation}
for some $0<\eps<\eps_0$, then there exist $T\gtrsim\eps^{-2}$ and a unique solution  $\theta \in C([0,T],\blue{H^{N_0}}(\R^2))$ of \eqref{SQG}.
\end{theorem}
\begin{remark}\label{remark: Size of parameters}
    \begin{enumerate}
    \item \blue{As part of the proof of Theorem \ref{thm: main result} (Theorem \ref{thm: main result detailed SQG} resp.) via the continuity method based on Proposition \ref{prop: bootstrap argument} below, the solutions to the corresponding problems remain small of order $\eps$ in the considered norms on the interval $[0,T]$ with $T$ as in Proposition \ref{prop: bootstrap argument}:
    \begin{equation}
 \|Z_\pm(t)\|_{H^{N_0}}+\sum_{a=0}^M\|S^aZ_\pm(t)\|_{L^2}+\sum_{b=0}^N\|S^b\zz_\pm(t)\|_{B}+ \|S^b\zz_\pm(t)\|_{X}\lesssim \eps,
\end{equation}
resp.\
\begin{equation}
 \|\theta(t)\|_{H^{N_0}}+\sum_{a=0}^M\|S^a\theta(t)\|_{L^2}+\sum_{b=0}^N\|S^b\Theta(t)\|_{B}+ \|S^b\Theta(t)\|_{X}\lesssim \eps, \quad 0\leq t\leq T.
    \end{equation} 
In particular, for $0<t\leq T$ the solutions decay as follows:
    \begin{equation}
        \|S^bZ_\pm(t)\|_{L^\infty}\lesssim t^{-\frac12}\eps, \quad \text{resp.} \quad \|S^b\theta(t)\|_{L^\infty}\lesssim t^{-\frac12}\eps, \qquad 0\leq b\leq N-2.
    \end{equation}}
        \item We can choose the parameters in the above theorems as $\beta=10^{-2}$, $N_0\sim 10^9$, and  $\delta=2M^{-\frac{1}{2}}$, such that $N_0\gg M\gg M^{\frac{1}{2}}\gg \beta^{-2}$. Moreover, $\delta_0=2N_0^{-1}$ is an useful parameter in subsequent Sections \ref{sec: bounds on dtSf}-\ref{sec: bounds on the X-norm}. These are convenient choices from a technical point of view (see the proofs of Propositions \ref{prop: bounds on B norm}, \ref{prop: X-norm bounds for l>(1+delta)m} and \ref{prop: X-norm bounds for l<(1+delta)m}), but no effort has been made at optimizing them.
    \end{enumerate}
\end{remark}
Theorems \ref{thm: main result}, \ref{thm: main result detailed SQG} follow via a continuity argument using the local well-posedness of the Boussinesq system \eqref{eqn: perturbed BQ} (SQG equation \eqref{SQG} respectively) and the following proposition. We recall that with the scalar unknowns $Z_\pm$ and their respective profiles $\zz_\pm$, the system \eqref{eqn: perturbed BQ} is equivalent to \eqref{eqn: bilin_profiles}, and the SQG equation \eqref{SQG} for $\theta$ is equivalent to \eqref{eqn: Duhamel for the profile} for the SQG profile $\Theta$.
\begin{proposition}\label{prop: bootstrap argument}
    Let $C>0$, and $T\leq C\eps^{-2}$. Assume $\zz_\pm\in C([0,T],H^{N_0}(\R^2))$ solve \eqref{eqn: bilin_profiles} resp.\ $\Theta\in C([0,T],H^{N_0}(\R^2))$ solves \eqref{eqn: Duhamel for the profile}  with initial data satisfying \eqref{eqn: initial data assumption} resp.\ \eqref{eqn: initial data assumption SQG}. If for $t\in [0,T]$ there holds that
    \begin{equation}\label{eqn: bootstrap assumption}
    \lVert{S^b\zz_\pm(t)}\rVert_B+\lVert{S^b\zz_\pm(t)}\rVert_X \leq 100\eps \quad\textnormal{resp.}\quad\lVert{S^b\Theta(t)}\rVert_B+\lVert{S^b\Theta(t)}\rVert_X \leq 100\eps, \quad 0\leq b\leq N,
\end{equation}
 then for $F\in\{\zz_+,\zz_-\}$ resp.\ $F=\Theta$ we have 
     \begin{equation}\label{eqn:bootstrap-prop-bded-energy}
\lVert{F(t)}\rVert_{H^{N_0}}+\sum_{a=0}^M\lVert{S^aF(t)}\rVert_{L^2}\lesssim \eps,
    \end{equation}
and in fact there holds the improved bound
\begin{align}\label{eqn: bootstrap improved bound}
    \lVert{S^bF(t)}\rVert_B+\lVert{S^bF(t)}\rVert_X &\leq 10\eps.
    \end{align}
\end{proposition}
We outline next the proof of Proposition \ref{prop: bootstrap argument} to show how it combines the remaining arguments of the paper. 
\begin{proof}
    Without loss of generality, we consider the setting of the Boussinesq system. Under the bootstrap assumption \eqref{eqn: bootstrap assumption} and by Corollary \ref{cor: linear decay semigroup} there holds
    \begin{align*}
        \sinftynorm{S^b{Z}_{\pm}(t)}\lesssim  t ^{-\frac{1}{2}}\eps, \hspace{1.5cm} 0\leq b<N-2.
    \end{align*}
    Together with the initial data assumption this implies the bound \eqref{eqn:bootstrap-prop-bded-energy} on the energy as shown in Corollary \ref{cor: energy estimates of Z+-}, as long as $T\lesssim \eps^{-2}$. In order to prove \eqref{eqn: bootstrap improved bound}, we note that from the Duhamel formula \eqref{eqn: bilin_profiles} and for $0\leq b\leq N$ we have 
    \begin{align*}
        \lVert{S^b\mathcal{Z}_{\pm}(t)}\rVert_B+\lVert{S^b\mathcal{Z}_{\pm}(t)}\rVert_X&\leq \lVert{S^b\mathcal{Z}_{\pm}(0)}\rVert_B+\lVert{S^b\mathcal{Z}_{\pm}(0)}\rVert_X
        +\lVert{S^b\B_{\m_{\pm}^{+-}}(\mathcal{Z}_{+},\mathcal{Z}_{-})}\rVert_B\\
        &\hspace{-1cm}+\lVert{S^b\B_{\m_{\pm}^{+-}}(\mathcal{Z}_{+},\mathcal{Z}_{-})}\rVert_X
        +\sum_{\mu \in \{+,-\}}\lVert{S^b\B_{\m_{\pm}^{\mu\mu}}(\mathcal{Z}_{\mu},\mathcal{Z}_{\mu})}\rVert_B+\lVert{S^b\B_{\m_{\pm}^{\mu\mu}}(\mathcal{Z}_{\mu},\mathcal{Z}_{\mu})}\rVert_X.
    \end{align*}
    Therefore, to prove \eqref{eqn: bootstrap improved bound} it suffices to show that under the bootstrap assumption \eqref{eqn: bootstrap assumption} and for $\m \in \set{\m^{\mu\nu}_{\pm}}{\mu,\nu\in\{-,+\}}$ there holds
    \begin{align*}
        \lVert{S^b\B_{\m}(\mathcal{Z}_{\mu},\mathcal{Z}_{\nu})}\rVert_B+\lVert{S^b\B_{\m}(\mathcal{Z}_{\mu},\mathcal{Z}_{\nu})}\rVert_X \leq 9\eps, &\hspace{1.5cm} 0\leq b\leq N.
    \end{align*}
    Since $S$ derives from a symmetry of the equation (see the below Lemma \ref{lemma: S^N on bilinear expression Qm} for an explicit computation), it suffices to show that for $b_1,\;b_2\geq 0$ with $b_1+b_2\leq N$ there holds
    \begin{align}\label{eqn: bootstrap bilinear estimate reduced}
        \lVert{\B_{\m}(S^{b_1}\mathcal{Z}_{\mu},S^{b_2}\mathcal{Z}_{\nu})}\rVert_B+\lVert{\B_{\m}(S^{b_1}\mathcal{Z}_{\mu},S^{b_2}\mathcal{Z}_{\nu})}\rVert_X\lesssim 9\eps.
    \end{align}
    To handle such expressions, we also localize the time variable: for $t\in[0,T]$ we decompose the indicator function  $\mathds{1}_{[0,t]}$ in functions $\tau_0,...\tau_{L+1}:\R\to [0,1]$ with $\abs{L-\log_2(2+t)}\leq 2$ such that
    \begin{align*}
      &  \supp\tau_0\subset [0,2],\hspace{0.5cm} \supp \tau_m \subset [2^{m-1},2^{m+1}], \; m\in \{1,..., L\}, \hspace{0.5cm} \supp \tau_{L+1}\subset [t-2,t],\\
     &   \sum_{m=0}^{L+1}\tau_m(s)=\mathds{1}_{[0,t]}, \hspace{0.5cm}\tau_m(s)\in C^1(\R), \hspace{0.5cm}\int_0^t\abs{\tau_m(s)}ds\lesssim 1, \; m\in \{1,..., L\}.
    \end{align*}
    Then for a bilinear expression with multiplier $\m$ as in \eqref{eqn: bq multipliers} there holds
    \begin{align}\label{eqn: bootstrap B_m time decomposition}
        \B_{\m}(f,g)=\int_0^t\mathcal\mathcal{Q}_{\m}(f,g)ds=\sum_m\int_0^t\tau_m(s)\mathcal{Q}_{\m}(f,g)ds=\sum_m\B_{\m}^m(f,g),
    \end{align}
    where $\B_{\m}^m(f,g):=\int_0^t\tau_m(s)\mathcal{Q}_{\m}(f,g)ds$. Bounds on such time-localized bilinear terms are shown in the subsequent Sections \ref{sec: bounds on B norm} and \ref{sec: bounds on the X-norm}: In Proposition \ref{prop: bounds on B norm} we prove
    \begin{align*}
         \snorm{\B_{\m}^m(S^{b_1}\mathcal{Z}_{\mu},S^{b_2}\mathcal{Z}_{\nu})}_B\lesssim 2^{(\frac{1}{6}+\delta) m}\eps^2,
    \end{align*}
    whereas Propositions \ref{prop: X-norm bounds for l>(1+delta)m} and \ref{prop: X-norm bounds for l<(1+delta)m} show that
    \begin{align*}
        \snorm{\B_{\m}^m(S^{b_1}\mathcal{Z}_{\mu},S^{b_2}\mathcal{Z}_{\nu})}_X\lesssim 2^{(\frac{1}{2}-\frac{\delta}{8})m}\eps^2,
    \end{align*}
    where $\delta=2M^{-\frac{1}{2}}$. Therefore, with $C_1>0$ and $t\in [0,T]$ with $T\leq C\eps^{-2}$ we obtain
    \begin{align*}
          \lVert{S^b\B_{\m}(\mathcal{Z}_{\mu},\mathcal{Z}_{\nu})}\rVert_B+\lVert{S^b\B_{\m}(\mathcal{Z}_{\mu},\mathcal{Z}_{\nu})}\rVert_X \leq C_1t^{\frac{1}{2}-\frac{\delta}{8}}\eps^2\leq C_1C^{\frac12-\frac{\delta}{8}}\eps^{{\delta^2}/{16}}\eps.
    \end{align*} 
    Choosing $\eps_0>0$ such that $C_1C^{\frac12-\frac{\delta}{8}}\eps_0^{{\delta^2}/{16}}<9$ yields \eqref{eqn: bootstrap improved bound}.
\end{proof}

We conclude this section with a short lemma that records the interplay of the scaling vector field $S$ and bilinear terms.
\begin{lemma}\label{lemma: S^N on bilinear expression Qm}
    Let $N\in \N$, $S$ be the vector field defined in \eqref{def: S W}, and $\mathcal{Q}_{\m}(f,g)$ a bilinear expression as in \eqref{eqn: bilin_form_Q}, $\m\in\{\m_0,\m_\pm^{\mu\nu}\}$. Then there holds that
    \begin{align*}
        S^N\mathcal{Q}_{\m}(f,g)=\sum_{\substack{b_1,b_2\in\N_0,\\0\leq b_1+b_2\leq N}}c_{b_1b_2}\mathcal{Q}_{\m}(S^{b_1}f,S^{b_2}g),
    \end{align*}
    for universal constants $c_{b_1b_2}\in\Z$.
\end{lemma}
\begin{proof}
We begin by observing that $S_\xi\Lambda(\xi)=0$, and since $S_\eta\Lambda(\xi-\eta)=-S_\xi\Lambda(\xi-\eta)$ it follows that $(S_{\xi}+S_\eta)\Phi=0$. Furthermore, by a direct computation we have that $(S_\xi+S_\eta)\m=0$ for $\m\in\{\m_0,\m_\pm^{\mu\nu}\}$. Integration by parts in $S_\eta$ then gives
    \begin{align*}
       S_\xi\F{(\mathcal{Q}_{\m}(f,g))}(\xi)&=\int_{\R^2}e^{it\Phi}(S_\xi+S_\eta)(\m(\xi,\eta))\widehat{f}(\xi-\eta)\widehat{g}(\eta)d\eta\\&\quad +\int_{\R^2}e^{it\Phi}\m(\xi,\eta)(S_\xi+S_\eta)\widehat{f}(\xi-\eta)\widehat{g}(\eta)d\eta+\int_{\R^2}e^{it\Phi}\m(\xi,\eta)\widehat{f}(\xi-\eta)S_\eta\widehat{g}(\eta)d\eta\\
        &=\int_{\R^2}e^{it\Phi}\m(\xi,\eta)(S\widehat{f})(\xi-\eta)\widehat{g}(\eta)d\eta+\int_{\R^2}e^{it\Phi}\m(\xi,\eta)\widehat{f}(\xi-\eta)(S\widehat{g})(\eta)d\eta,
    \end{align*}
    and the claim follows by iteration.
\end{proof}

\end{subsection}

\section{Linear Decay} \label{sec: Linear Decay}
In this section we establish amplitude decay estimates for the semigroup $e^{it\Lambda}$ that build on our choice of norms. In particular, we collect the relevant information in a ``decay norm''\footnote{The relevance of including at least two copies of $S$ in this norm in order to obtain the linear decay can be seen in \textbf{Case B} in the proof of Proposition \ref{proposition: linear decay}, for example.}
\begin{align}\label{norm: decay norm}
    \norm{f}_D:=\sup_{0\leq n\leq 2}{(\norm{S^nf}_B
+\norm{S^nf}_X)}.
\end{align}
As a basic ingredient, this norm allows us to control the $L^\infty$ norm of the Fourier transform of suitably localized versions of $f$, i.e.
\begin{align*}
    \sinftynorm{\widehat{P_{k,p}f}}&\lesssim 2^{-4k^+}2^{-k} \norm{f}_D.
\end{align*}
This can be seen directly from the following lemma:
\begin{lemma}\label{lemma: control of Fourier transform} For any $f\in L^2$ there holds
    \[\inftynorm{\widehat{P_{k,p}f}} \lesssim2^{-4k^+} 2^{-k}\Big[\norm{P_kf}_B+\norm{SP_kf}_B+\norm{P_kf}_X+\norm{SP_kf}_X \Big].\]
\end{lemma}
The proof of this statement follows from the fundamental theorem of calculus and is detailed in Appendix \ref{ssec:Linftybd}.

The following establishes a decomposition of the action of the semigroup $e^{it\Lambda}$ and gives precise decay estimates in relation to the degeneracy of the corresponding linear phase. 
\begin{proposition}[Linear decay]\label{proposition: linear decay} Let $f:\R^2\to \R$ and consider the decay norm defined as in \eqref{norm: decay norm}. For $0<\beta^\prime<\beta$, we can decompose 
\[P_{k,p}e^{it\Lambda}f= I_{k,p}(f)+II_{k,p}(f)\]
such that the following bounds hold: for $I_{k,p}$ we have
\begin{align}
&p\leq -10:     &&\inftynorm{I_{k,p}(f)}\lesssim 2^{\frac{3}{4}k}2^{-\frac{15}{4}k^+}\min{\{2^p,2^{-p}\abs{t}^{-1}\}}\norm{f}_D, \label{eqn: linfty decay of the profile}\\
&p\geq -10:   &&\inftynorm{I_{k,p}(f)}\lesssim 2^{\frac{3}{4}k}2^{-\frac{15}{4}k^+}\log(\abs{t})\abs{t}^{-1}\norm{f}_D,
\end{align}
while the term $II_{k,p}(f)$ satisfies
\begin{align}
     \ltwonorm{II_{k,p}(f)}&\lesssim 2^{-4k^+}2^{-(\frac{1}{2}+2\beta^\prime)p}\abs{t}^{-(\frac{1}{2}+\beta^\prime)} \mathds{1}_{2^p\gtrsim \abs{t}^{-1/2}}\norm{f}_D. \label{eqn: l2 decay of the profile}
\end{align}
\end{proposition}
In particular, since $\inftynorm{II_{k,p}(f)}\lesssim\ltwonorm{\phikp}\ltwonorm{II_{k,p}(f)}$ the $L^\infty$ bound for $II_{k,p}(f)$ is given by
\begin{align*}
    \inftynorm{II_{k,p}(f)}\lesssim 2^k2^{-4k^+}2^{-2\beta^\prime p}\abs{t}^{-(\frac{1}{2}+\beta^\prime)} \mathds{1}_{2^p\gtrsim \abs{t}^{-1/2}}\norm{f}_D.
\end{align*} 
Before we proceed with the proof, we record the following useful corollary, which shows that Proposition \ref{proposition: linear decay} entails the sharp linear decay rate (see \cite[\textcolor{MidnightBlue}{\S 2.2}]{Elgindi_2015}).
\begin{corollary}\label{cor: linear decay semigroup}
    For $t>0$, the semigroup $e^{\pm it\Lambda}$ satisfies
    \begin{align*}
        \sinftynorm{P_ke^{\pm it\Lambda}f}\lesssim 2^{\frac{3}{4}k}2^{-\frac{15}{4}k^+} t^{-\frac{1}{2}}\norm{f}_D.
    \end{align*}
    In particular, under the bootstrap assumption \eqref{eqn: bootstrap assumption}, for the Boussinesq system \eqref{eqn: perturbed BQ} there holds
    \begin{align*}
        &\inftynorm{\nabla u(t)}+\inftynorm{Su(t)}+\inftynorm{\nabla\rho(t)}\lesssim  t^{-\frac{1}{2}}\eps,
    \end{align*}
    and analogously for the SQG equation \eqref{SQG}:
    \begin{align*}
        &\inftynorm{\nabla \theta(t)}+\inftynorm{u}+\inftynorm{Su(t)}\lesssim  t^{-\frac{1}{2}}\eps.
    \end{align*}
\end{corollary}
\begin{proof}
A direct set size estimate (see \eqref{eqn: P_k,p bound for small p} below) shows that
\begin{equation}
 \inftynorm{P_{k,p}e^{\pm it\Lambda}f}\lesssim 2^{\frac{3}{4}k}2^{-\frac{15}{4}k^+}2^p\norm{f}_D.
\end{equation}
 Together with Proposition \ref{proposition: linear decay} it follows that
     \begin{align*}
        \inftynorm{P_ke^{\pm it\Lambda}f}&\leq \sum_{p\in \Z^-}\inftynorm{P_{k,p}e^{\pm it\Lambda}f}=\sum_{ 2^{p}\lesssim t^{-1/2}}\inftynorm{P_{k,p}e^{\pm it\Lambda}f}+\sum_{ 2^{p}\gtrsim t^{-1/2}}\inftynorm{P_{k,p}e^{\pm it\Lambda}f}\\
        &\lesssim \sum_{2^{p}\lesssim t^{-1/2}}2^{\frac{3}{4}k}2^{-\frac{15}{4}k^+}2^p\norm{f}_D +\sum_{ 2^{p}\gtrsim t^{-1/2}}(\inftynorm{I_{k,p}}+\inftynorm{II_{k,p}})\\
        &\lesssim 2^{\frac{3}{4}k}2^{-\frac{15}{4}k^+} t^{-\frac{1}{2}}\norm{f}_D+ \sum_{2^{p}\gtrsim t^{-1/2}, p\leq -10}\inftynorm{I_{k,p}} +\sum_{2^{p}\gtrsim t^{-1/2}, p\geq -10}\inftynorm{I_{k,p}} \\
        &\hspace{1.5cm} +\sum_{ 2^{p}\gtrsim t^{-{1}/{2}}}2^k2^{-4k^+}2^{-2\beta'p}t^{-(\frac{1}{2}+\beta')}\norm{f}_D\\
        &\lesssim 2^{k}2^{-4k^+} t^{-\frac{1}{2}}\norm{f}_D+\hspace{-0.4cm}\sum_{p\leq -10, 2^p\gtrsim t^{-1/2}}\hspace{-0.7cm}2^{\frac{3}{4}k-\frac{15}{4}k^+}\min \{2^p,2^{-p}t^{-1}\}\norm{f}_D\\
        &\hspace{1.5cm}+\sum_{p\geq -10}2^{\frac{3}{4}k-\frac{15}{4}k^+}\log(t)t^{-1}\norm{f}_D\\
        &\lesssim 2^{\frac{3}{4}k}2^{-\frac{15}{4}k^+}  t ^{-\frac{1}{2}}\norm{f}_D.
    \end{align*}
    As for the SQG equation, recall that $\theta(t)=e^{-it\Lambda}\Theta(t),u(t)=e^{-it\Lambda}\nabla^\perp(-\Delta)^{-\frac{1}{2}}\Theta(t)=:e^{-it\Lambda}u_\Theta(t)$. Moreover observe that $[S,\nabla^\perp]=-\nabla^\perp$, $[S,\abs{\nabla}^{-1}]=-\abs{\nabla}^{-1}$ and $\sltwonorm{\nabla^\perp(-\Delta)^{-\frac{1}{2}}g}\lesssim \ltwonorm{g}$. Then there holds
    \begin{align*}
        \inftynorm{\nabla \theta(t)}+\inftynorm{u}+\inftynorm{Su(t)}&\lesssim \sum_{k\in\Z}\inftynorm{P_k\nabla e^{-it\Lambda} \Theta}+\inftynorm{P_ke^{-it\Lambda}u_\Theta}+\inftynorm{P_ke^{-it\Lambda}Su_\Theta}\\
        &\lesssim \sum_{k\in \Z} 2^{\frac{3}{4}k}2^{-\frac{15}{4}k^+}  t^{-\frac{1}{2}}(2^k\norm{\Theta}_D+ \norm{\Theta}_D+\norm{S\Theta}_D)\\
        &\lesssim  t^{-\frac{1}{2}}\eps.
    \end{align*}
    The bound for the Boussinesq system follows analogously by recalling the definition of the dispersive unknowns $Z_\pm$ and their respective profiles $\mathcal{Z}_\pm$ in \eqref{def: Profiles BQ}. Indeed, by \eqref{eqn: urho} we have that
    \begin{equation}
     u(t)=-\frac{1}{2}\nabla^
\perp \abs{\nabla}^{-1}(e^{-it\Lambda}\mathcal{Z}_+(t)+e^{it\Lambda}\mathcal{Z}_-(t)),\quad    \rho(t)=\frac{1}{2}(e^{-it\Lambda}\mathcal{Z}_+(t)-e^{it\Lambda}\mathcal{Z}_-(t)),
    \end{equation}
and for
\begin{align*}
    \tilde{A}(t):= \inftynorm{\nabla u(t)} +\inftynorm{Su(t)}+ \inftynorm{\nabla\rho(t)},
\end{align*}
we obtain as above using the commuting properties between derivatives and $S$ that
\begin{align*}
  \tilde{A}(t)& \lesssim \sum_{k\in \Z}\sinftynorm{P_k\nabla \nabla^
\perp \abs{\nabla}^{-1}(e^{-it\Lambda}\mathcal{Z}_+(t)+e^{it\Lambda}\mathcal{Z}_-(t))}+\inftynorm{P_k \nabla(e^{-it\Lambda}\mathcal{Z}_+(t)-e^{it\Lambda}\mathcal{Z}_-(t))}\\
&\qquad+\sinftynorm{P_kS\nabla^
\perp \abs{\nabla}^{-1}(e^{-it\Lambda}\mathcal{Z}_+(t)+e^{it\Lambda}\mathcal{Z}_-(t))}\\
&\lesssim \sum_{k\in \Z} \inftynorm{P_ke^{\mp it \Lambda}\mathcal{Z}_\pm}+2^k\inftynorm{P_ke^{\mp it \Lambda}\mathcal{Z}_\pm}+\inftynorm{P_ke^{\mp it \Lambda}S\mathcal{Z}_\pm}\\
&\lesssim t^{-\frac{1}{2}}\eps.
\end{align*}
\end{proof}

\begin{proof}[Proof of Proposition \ref{proposition: linear decay}]
Without loss of generality let $t>0$, and consider the semigroup given by
 \begin{align}
P_{k,p}e^{it\Lambda}f(x)&=\int_{\R^2}e^{it\Lambda(\xi)+ix\cdot\xi }\widehat{P_{k,p}f}(\xi)d\xi \nonumber\\
    &=\int_0^\infty \int_{-1}^1 e^{i\Psi(\rho,\Lambda)}\varphi(2^{-k}\rho)\varphi(2^{-p}\sinlambda)\hat{f}(\rho,\Lambda)\frac{\rho}{\sinlambda}d\Lambda d\rho,\label{eqn: localized semigroup}\\
    \Psi&:=t\Lambda +x_1\rho\Lambda +x_2\rho\sinlambda,\nonumber
 \end{align}
 where we have used the polar coordinates notation \eqref{eqn: polar coordinates}.
 \par To begin with, assume that for some $C>0$
 \begin{align*}
   t^{\frac{1}{2}}2^p\leq C, \quad \text{ or } \quad t2^{-k}\leq 1.
  \end{align*}
   Observe that if $\sqrt{1-\Lambda^2}\sim 2^p\leq C t^{-\frac{1}{2}}\ll 1$, on the support of $\varphi_{k,p}$ there holds $\abs{\Lambda}\geq \frac{1}{2}$. Letting $\overline{\varphi}_{k,p}$ be a function with similar support properties as $\varphi_{k,p}$, by a change of variables $\Lambda\mapsto 2^{-p}\sinlambda=y$ and Lemma \ref{lemma: control of Fourier transform} we obtain
        \begin{equation}\label{eqn: P_k,p bound for small p}
        \begin{split}
        \abs{P_{k,p}e^{it\Lambda}f} &\lesssim \int_0^\infty \int_{-1}^1 \sabs{\varphi(2^{-k}\rho)\varphi(2^{-p}\sqrt{1-\Lambda^2})\ff}\frac{\rho}{\sqrt{1-\Lambda^2}}d\Lambda d\rho\\
        &\lesssim \int_0^\infty \sabs{\overline{\varphi}(2^{-k}\rho)}\rho d\rho\int_{-1}^1\abs{\overline{\varphi}(y)}2^pdy\sinftynorm{\widehat{P_{k,p}f}}\\
        &\lesssim 2^{2k}2^{p}\sinftynorm{\widehat{P_{k,p}f}}\\
        &\lesssim 2^{k-4k^+}2^p\norm{f}_D.
        \end{split}
    \end{equation}
 From now on we assume 
    \begin{equation}\label{assumption: away form degenracy.}
     t^{\frac{1}{2}}2^p> C \iff 2^{-p}< C^{-1}t^{\frac{1}{2}}, \quad \text{and}\quad t2^{-k}>1.
   \end{equation}
    We decompose
    \[f=R_{\leq l_0}f+(\mathrm{Id}-R_{\leq l_0})f,\]
    where $l_0$ is the largest integer such that the following inequality holds
    \[2^{l_0}\leq t2^p(t2^{2p})^{-\kappa},\hspace{1cm} 0<\kappa<\frac{(\beta-\beta^\prime)}{1+\beta}\]
    and $0<\beta'<\beta$. We then let
     \begin{align*}
         P_{k,p}e^{it\Lambda}f=P_{k,p}R_{\leq l_0}e^{it\Lambda}f+P_{k,p}(\mathrm{Id}-R_{\leq l_0})e^{it\Lambda}f=:I_{k,p}(f)+II_{k,p}(f).
     \end{align*} 
We can estimate the high angular frequencies using the $X$-norm \eqref{X norm} to obtain claim \eqref{eqn: l2 decay of the profile}:
    \begin{align*}
        \ltwonorm{II_{k,p}(f)}&\lesssim \sum_{l> l_0,\; p+l\geq 0}\ltwonorm{P_{k,p}R_lf}\lesssim \sum_{l>l_0, \; p+l\geq 0}2^{-4k^+}2^{-(1+\beta)l}2^{-\frac{p}{2}}2^{-\beta p}\norm{f}_X\\
        &\lesssim 2^{-4k^+}2^{-(1+\beta)(l_0+1)}2^{-\frac{p}{2}}2^{-\beta p}\norm{f}_X\\
        &\lesssim 2^{-4k^+}[t2^p(t2^{2p})^{-\kappa}]^{-(1+\beta)}2^{-\frac{p}{2}}2^{-\beta p}\norm{f}_X\\
        &\lesssim 2^{-4k^+} t^{-(1+\beta)}2^{-(1+\beta)p}(t2^{2p})^{(\beta-\beta')}2^{-\frac{p}{2}-\beta p}\norm{f}_X\\
        &\lesssim 2^{-4k^+}t^{-\frac{1}{2}-\beta'}2^{-\frac{p}{2}-2\beta'p}\norm{f}_X,
    \end{align*}
    where we have used $\beta^\prime<\beta$ and $t^{\frac{1}{2}}2^p\geq C$.

    From now we assume that $f=R_{\leq l_0}f$ and note that by the Bernstein property Proposition \ref{prop: angular localization properties}\eqref{it3: prop: angular localization properties} for any $a,b\in \N_0$ there holds
    \begin{equation}\label{estimate: mixed derivatives L infinity}
        \sinftynorm{S^b\partial^a_\Lambda \hat{f}}\lesssim t^a2^{ap}(t2^{2p})^{-\kappa a}\sinftynorm{S^b\ff}.
    \end{equation}
    In the following we will integrate by parts in the expression \eqref{eqn: localized semigroup} in different directions. To that end, we compute the derivatives
    \begin{align}
         &  \partial_\Lambda \Psi=t+ x_1\rho- x_2\rho\frac{\Lambda}{\sinlambda}, && \partial^2_\Lambda\Psi= -x_2\rho\frac{1}{(1-\Lambda^2)^{\frac{3}{2}}}, \label{eqn: lin decay derivatives of psi} \\
      &\partial_\Lambda\partial_\rho\Psi=x_1-x_2\frac{\Lambda}{\sinlambda},
      &&\partial_\rho\Psi=x_1\Lambda+x_2\sinlambda, &&\partial_\rho^2\Psi=0.   \nonumber 
    \end{align}
    \noindent\textbf{Part 1:} Let $p\leq -10$. In particular, on the support of $\varphi_{k,p}$ there holds $\abs{\Lambda}\geq \frac{1}{2}$.
  \par  \textbf{Case A:} For some $c>2$
    \begin{align*}
       \abs{x_1}<c^{-1}t2^{2p-k}, \hspace{1.5cm} \abs{x_2}\leq c^{-2}t2^{p-k}.
    \end{align*}
    With \eqref{eqn: lin decay derivatives of psi} this implies the following bounds on derivatives of $\Psi$:
    \[\abs{\partial_\Lambda \Psi}\geq \abs{t+x_1\rho}-\abs{x_2}\rho\frac{\Lambda}{\sinlambda}\gtrsim t,\hspace{0.5cm}\abs{\partial_\Lambda^2\Psi} \leq c^{-2}t2^{p-k}2^k2^{-3p}.\]  
    Next we integrate by parts in the expression \eqref{eqn: localized semigroup} $N$ times until $N\kappa\geq 1$. To that end, let  \begin{align*}
        h(\rho,\Lambda):=\varphi(2^{-k}\rho)\varphi(2^{-p}\sinlambda)\frac{\rho}{\sinlambda}
    \end{align*}
    and compute
 \begin{align*}
    &\partial_\Lambda h(\rho,\Lambda)=\varphi(2^{-k}\rho)\rho[2^{-p}\overline{\varphi}(2^{-p}\sinlambda)\frac{-\Lambda}{1-\Lambda^2}+\overline{\varphi}(2^{-p}\sinlambda)\frac{\Lambda}{(1-\Lambda^2)^{3/2}}],
 \end{align*}
  and in particular
 \begin{align}\label{eqn: l1 norm of derivatives of h}
    & \int_0^\infty\int_{-1}^1\abs{h}d\Lambda d\rho\lesssim 2^{2k}2^p, && \int_0^\infty\int_{-1}^1\abs{\partial_\Lambda h}d\Lambda d\rho\lesssim 2^{2k}2^{-p}.
 \end{align}
 Using that $e^{i\Psi}=\frac{1}{i\partial_\Lambda \Psi}\partial_\Lambda e^{i\Psi}$ we integrate by parts repeatedly in \eqref{eqn: localized semigroup} and observe that boundary terms vanish because of the condition \eqref{assumption: away form degenracy.} and $\varphi(0)=0$, we obtain
   \begin{align*}
        I_{k,p}(f)&=\int_0^\infty \int_{-1}^1 e^{i\Psi}h\ff  d\Lambda d\rho=-\int_0^\infty \int_{-1}^1 e^{i\Psi}\partial_\Lambda\left(\frac{1}{i\partial_\Lambda\Psi}h\ff\right)d\Lambda d\rho\\
        &=-\iint e^{i\Psi}\left[ \partial_\Lambda\left(\frac{1}{i\partial_\Lambda \Psi}\right)h-\frac{1}{i\partial_\Lambda \Psi}\partial_\Lambda h\right]\ff d\Lambda d\rho + \iint e^{i\Psi}\frac{1}{i\partial_\Lambda\Psi}h\partial_\Lambda \ff d\Lambda d\rho\\
        &=-\iint e^{i\Psi}\left[\frac{-i\partial_\Lambda^2\Psi}{(i\partial_\Lambda\Psi)^2}h+\frac{1}{i\partial_\Lambda \Psi}\partial_\Lambda h \right]\ff d\Lambda d\rho  - \iint e^{i\Psi}\frac{1}{i\partial_\Lambda\Psi}h\partial_\Lambda \ff d\Lambda d\rho.
\end{align*}
We integrate by parts again in the second integral and obtain
        \begin{align*} 
I_{k,p}(f)&=-\iint e^{i\Psi}\left[\frac{-i\partial_\lambda^2\Psi}{(i\partial_\Lambda\Psi)^2}h+\frac{\partial_\Lambda h}{i\partial_\Lambda \Psi} \right]\ff d\Lambda d\rho+\iint e^{i\Psi}\left[ \frac{-2i\partial_\Lambda^2\Psi}{(i\partial_\Lambda\Psi)^3}h+ \frac{\partial_\Lambda h}{(i\partial_\Lambda\Psi)^2}\right]\partial_\Lambda \hat{f} d\Lambda d\rho\\
& \hspace{3cm} +\iint e^{i\Psi}\frac{h}{(i\partial_\Lambda\Psi)^2}\partial_\Lambda^2\hat{f}d\Lambda d\rho.
    \end{align*}
 Continuing $N$-times this integration by parts in the integral with the highest order derivative of $\hat{f}$ until $N\kappa\geq 1$ yields
    \begin{align*}
        I_{k,p}(f)&=\int_0^\infty \int_{-1}^1 e^{i\Psi}h(\rho,\Lambda)\ff  d\Lambda d\rho\\
        &= \int_0^\infty \int_{-1}^1 e^{i\Psi}\Bigg[\sum_{j=0}^{N-1}\frac{(-1)^j}{(i\partial_\Lambda \Psi)^j}\Bigg( \frac{\partial_\Lambda h}{i\partial_\Lambda \Psi} -\frac{(j+1)ih \partial_\Lambda^2\Psi}{(i\partial_\Lambda\Psi)^2}\Bigg)\partial_\Lambda^j \ff + \frac{(-1)^{j+1}h}{(i\partial_\Lambda\Psi)^N} \partial_\Lambda^N \ff\Bigg] d\Lambda d\rho\\
        &= \sum_{j=0}^{N} I_j.
    \end{align*} 
 For $0\leq j\leq N-1$, with \eqref{assumption: away form degenracy.}, \eqref{estimate: mixed derivatives L infinity} and \eqref{eqn: l1 norm of derivatives of h} we can estimate the terms $I_j$ in the sum above as
    \begin{align*}
        \abs{I_j}&\lesssim \abs{i\partial_\Lambda \Psi}^{-j} \int_0^\infty \int_{-1}^1 \Bigg( \abs{\frac{\partial_\Lambda h}{i\partial_\Lambda \Psi}}+ \abs{\frac{h\partial_\Lambda^2 \Psi}{(i\partial_\Lambda\Psi)^2}} \Bigg)  d\Lambda d\rho \inftynorm{\widehat{P_{k,p}\partial_\Lambda^jf}}\\
        &\lesssim t^{-j}\Bigg[t^{-1}\int_0^\infty \int_{-1}^1 \abs{\partial_\Lambda h}d\Lambda d\rho +t^{-2}c^{-2}t2^{p-k}2^k 2^{-3p} \int_0^\infty \int_{-1}^1\abs{h}d\Lambda d\rho\Bigg] t^{j}2^{jp}(t2^{2p})^{-j\kappa}\sinftynorm{\widehat{P_{k,p}f}}\\
        &\lesssim 2^{2k}t^{-1}2^{-p}\sinftynorm{\widehat{P_{k,p}f}}\\
        &\lesssim 2^{-4k^+}2^{k}t^{-1}2^{-p}\norm{f}_D,
    \end{align*}
    where we have used Lemma \ref{lemma: control of Fourier transform} in the last estimate.
On the other hand, when $j=N$ we obtain
    \begin{align*}
        \abs{I_N}&\lesssim t^{-N}2^{2k}2^pt^N2^{pN}(t2^{2p})^{-N\kappa}\sinftynorm{\widehat{P_{k,p}\partial_\Lambda^Nf}}\lesssim 2^{-4k^+}2^kt^{-1}2^{-p}\norm{f}_D.
    \end{align*}

 \textbf{Case B:} Assume
    \begin{align*}
 \abs{x_1}\geq c^{-1}t2^{2p-k} \textnormal{ and }    \abs{x_2}\leq c^{-2}t2^{p-k}.
    \end{align*} 
    \par 
    We have the following bound on the radial derivative of $\Psi$
    \begin{align}\label{estimate: lower bound on r-derivative of psi}
        \abs{\partial_\rho\Psi}\gtrsim t2^{2p-k},
    \end{align}
    as can be seen from \eqref{eqn: lin decay derivatives of psi}: 
    Using that $p\leq -10$ (and thus $\abs{\Lambda}\geq \frac{1}{2}$) it follows that $\abs{\partial_\rho\Psi}\geq \abs{x_1}\Lambda-\abs{x_2}\sinlambda\gtrsim t2^{2p-k}.$
    This allows us to integrate by parts in $\rho$ with $e^{i\Psi}=\frac{1}{i\partial_\rho\Psi}\partial_\rho(e^{i\Psi})$ and obtain
    \begin{align}\label{eqn: Part 1 Case B decomposition of semigroup}
       P_{k,p}e^{it\Lambda}f(x)&=\int_0^\infty \int_{-1}^1e^{i\Psi}h(\rho,\Lambda)\ff d\Lambda d\rho=-\int_0^\infty \int_{-1}^1e^{i\Psi} \left[\frac{\partial_\rho h}{i\partial_\rho\Psi}\ff+\frac{h}{i\partial_\rho\Psi}\partial_\rho\ff \right] d\Lambda d\rho.
    \end{align}
    Note that 
    \[\partial_\rho h=\varphi(2^{-p}\sinlambda)\frac{1}{\sinlambda}[(\partial_\rho\varphi)(2^{-k}\rho)2^{-k}\rho+\varphi(2^{-k}\rho)],\]
    and therefore 
    \begin{align*}
        \int_0^\infty \int_{-1}^1\abs{\partial_\rho h}d\Lambda d\rho&\lesssim \int_{-1}^1\frac{|\varphi(2^{-p}\sinlambda)|}{\sinlambda}d\Lambda \int_0^\infty \Big[ 2^{-k}|\overline{\varphi}(2^{-k}\rho)\rho|+ |\varphi(2^{-k}\rho)|\Big]d\rho \lesssim 2^{p}2^{k}.
    \end{align*}
    With this estimate, Lemma \ref{lemma: control of Fourier transform} and \eqref{estimate: lower bound on r-derivative of psi}, the first term in \eqref{eqn: Part 1 Case B decomposition of semigroup} can be bounded as 
    \begin{align*}
        \abs{\int_0^\infty \int_{-1}^1e^{i\Psi}\frac{\partial_\rho h}{i\partial_\rho\Psi}\ff d\Lambda d\rho}&\lesssim ct^{-1}2^{-2p+k}\int_0^\infty \int_{-1}^1\abs{\partial_\rho h}d\Lambda d\rho \sinftynorm{\widehat{P_{k,p}f}}\\
        &\lesssim 2^{2k}t^{-1}2^{-p}\sinftynorm{\widehat{P_{k,p}f}}\\
        &\lesssim 2^{-4k^+}2^kt^{-1}2^{-p}\norm{f}_D.
    \end{align*}
    For the second term in \eqref{eqn: Part 1 Case B decomposition of semigroup} we recall that $S=\rho\partial_\rho$ and obtain with Lemma \ref{lemma: control of Fourier transform}
    \begin{align*}
         \abs{\int_0^\infty \int_{-1}^1e^{i\Psi} \frac{h}{i\partial_\rho\Psi}\partial_\rho\ff d\Lambda d\rho} & \lesssim ct^{-1}2^{-2p+k}\int_{-1}^1 \frac{|\varphi(2^{-p}\sinlambda)|}{\sinlambda}d\Lambda\int_0^\infty |\varphi(2^{-k}\rho)|d\rho\sinftynorm{\widehat{P_{k,p}Sf}}\\
         &\lesssim 2^{2k}t^{-1}2^{-p}\sinftynorm{\widehat{P_{k,p}Sf}}\\
         &\lesssim 2^{-4k^+}2^kt^{-1}2^{-p}\norm{f}_D.
    \end{align*}
    Here we note that the decay norm $\norm{\cdot}_D$ also bounds  $\sinftynorm{S\hat{f}}$ by its definition \eqref{norm: decay norm}.
        \par\textbf{Case C:} \blue{Here we treat the remaining case}
   \blue{
\begin{equation}
    \abs{x_2}> c^{-2}t2^{p-k}.
\end{equation}
   }
We have the following lower bound:
    \begin{align}\label{eqn: mixed_derivs_lbd}
    \abs{\partial_\rho\Psi}+ \abs{\partial_\rho\partial_\Lambda\Psi}\gtrsim 2^{-k}t.
    \end{align}
    \blue{Indeed,
     from \eqref{eqn: lin decay derivatives of psi} there holds $\Lambda\partial_\Lambda \partial_\rho \Psi- \partial_\rho\Psi=-\frac{x_2}{\sinlambda}$ and it follows
    \begin{align*}
        \abs{\partial_\rho\Psi}+ \abs{\partial_\rho\partial_\Lambda\Psi}\geq  \abs{\partial_\rho\Psi}+ \abs{\Lambda\partial_\rho\partial_\Lambda\Psi}\geq \abs{\partial_\rho\Psi-\Lambda\partial_\rho\partial_\Lambda\Psi}>c^{-2}t2^{-k}.
    \end{align*}}
    With this we can integrate by parts in $\rho$ or use the set-size gain when integrating in $\Lambda$. To formalize this, we decompose
    \begin{align*}
        P_{k,p}e^{it\Lambda}f=\int_0^\infty \int_{-1}^1e^{i\Psi}\varphi_{k,p}(\rho,\Lambda)\hat{f}\frac{\rho}{\sinlambda}d\Lambda d\rho=\sum_{n\geq 0}I_n,
    \end{align*}
    where $I_n=\int_0^\infty \int_{-1}^1e^{i\Psi}\varphi_{k,p}(\rho,\Lambda)\varphi(2^{-n}\partial_\rho\Psi)\hat{f}\frac{\rho}{\sinlambda}d\Lambda d\rho$. On the support of $I_0$, we have $\abs{\partial_\rho\Psi}\sim 1$, hence $\abs{\partial_\Lambda\partial_\rho\Psi}\gtrsim 2^{-k}t$ and with a change variables $y=\partial_\rho\Psi$ we obtain the decay:
    \begin{align*}
        \abs{I_0}&\lesssim \sinftynorm{\widehat{P_{k,p}f}} \iint\varphi_{k,p}(\rho,\Lambda)\varphi(\partial_\rho\Psi)\frac{\rho}{\sinlambda}d\Lambda d\rho\\
        &\lesssim \sinftynorm{\widehat{P_{k,p}f}}\snorm{\varphi_{p}(\Lambda)\sinlambda^{-1}}_{L^{\infty}_{\Lambda}} \iint \varphi(\partial_\rho\Psi)\varphi(2^{-k}\rho)\rho d\Lambda d\rho\\
        &\lesssim 2^{-p}\sinftynorm{\widehat{P_{k,p}f}} \iint\varphi(y)\abs{\partial_\Lambda \partial_\rho\Psi}^{-1}\varphi(2^{-k}\rho)\rho dy d\rho\\
        &\lesssim 2^{2k-p}t^{-1}\sinftynorm{\widehat{P_{k,p}f}} \\
        &\lesssim 2^{k-4k^+}2^{-p}t^{-1}\norm{f}_D.
    \end{align*}
   The summation for $n\geq 1$ will be split according to \eqref{eqn: mixed_derivs_lbd}.  We observe that $ \partial_\rho^2\Psi=0$. Thus, integrating by parts in $\partial_\rho$ once (respectively twice) gives
    \begin{equation}
        I_n=-I_n^{(1)}=I_n^{(2)},
    \end{equation}
    where
    \begin{align*}
     I_n^{(1)}&=\iint e^{i\Psi}\frac{\varphi(2^{-n}\partial_\rho\Psi)}{i\partial_\rho\Psi}\partial_\rho(\varphi_{k,p}\ff \frac{\rho}{\sinlambda})d\Lambda d\rho,\\
     I_n^{(2)}&=\iint e^{i\Psi}\frac{\varphi(2^{-n}\partial_\rho\Psi)}{(\partial_\rho\Psi)^2}\partial_\rho^2 (\varphi_{k,p}\hat{f}\frac{\rho}{\sinlambda})d\Lambda d\rho.
    \end{align*}
    For a function $\overline{\varphi}_{k,p}$ with similar support properties as $\varphi_{k,p}$ we thus have the bounds
    \begin{align*}
        \sabs{I_n^{(1)}}&\lesssim 2^{-n}2^{-p}(\sinftynorm{\widehat{P_{k,p}f}}+\sinftynorm{\widehat{P_{k,p}Sf}})\iint\varphi(2^{-n}\partial_\rho\Psi)\overline{\varphi}_{k,p}(\rho,\Lambda)d\Lambda d\rho\\
        &\lesssim 2^{-n}2^{-p}2^{-k}2^{-4k^+}\norm{f}_D\iint\varphi(2^{-n}\partial_\rho\Psi)\overline{\varphi}_{k,p}(\rho,\Lambda)d\Lambda d\rho
    \end{align*}
    and
    \begin{equation}\label{eqn: I_n^2 bound}
    \begin{split}
         \sabs{I_n^{(2)}}&\lesssim 2^{-2n}\iint \varphi(2^{-n}\partial_\rho \Psi)\overline{\varphi}_{k,p}[\sabs{\hat{f}}2^{-k}+\sabs{\partial_\rho\hat{f}}+2^{k}\sabs{\partial_\rho^2\hat{f}}]\sinlambda^{-1}d\Lambda d\rho\\
        &\lesssim 2^{-2n}2^{-p}2^{-k}\sum_{a=0}^2\sinftynorm{\widehat{P_{k,p}S^af}} \iint  \varphi(2^{-n}\partial_\rho \Psi)\overline{\varphi}_{k,p}(\rho,\Lambda)d\Lambda d\rho\\
        &\lesssim 2^{-2n}2^{-p}2^{-2k}2^{-4k^+}\norm{f}_D\iint  \varphi(2^{-n}\partial_\rho \Psi)\overline{\varphi}_{k,p}(\rho,\Lambda)d\Lambda d\rho.
    \end{split}
    \end{equation}
    Now note that when $\abs{\partial_\rho\Psi}\sim 2^n\leq t2^{-k}$, by \eqref{eqn: mixed_derivs_lbd} there must hold that $\abs{\partial_\Lambda\partial_\rho\Psi}\gtrsim 2^{-k}t$. By changing variables $y=\partial_\rho\Psi$, the set size of integration in $\Lambda$ gives
    \begin{equation}
    \begin{aligned}
    \iint\varphi(2^{-n}\partial_\rho\Psi)\overline{\varphi}_{k,p}(\rho,\Lambda)d\Lambda d\rho &\lesssim \iint\varphi(2^{-n}\partial_\rho\Psi)\varphi(2^{k}t^{-1}\partial_\Lambda\partial_\rho\Psi)\overline{\varphi}_{k,p}(\rho,\Lambda)d\Lambda d\rho\lesssim 2^{n}2^{2k}t^{-1},
    \end{aligned}
    \end{equation}
    so that
     \begin{equation}\label{eqn: I_n^(1) bound}
          \begin{split}
         \sabs{I_n^{(1)}}&\lesssim 2^{k}2^{-4k^+}2^{-p}t^{-1}\norm{f}_D.
    \end{split}
     \end{equation}
    Similarly, as long as $2^n\leq t2^{-k}$ we obtain from \eqref{eqn: I_n^2 bound} that 
    \begin{align}\label{eqn: I_n^(2) bound}
         \sabs{I_n^{(2)}}&\lesssim 2^{-n}2^{-4k^+}2^{-p}t^{-1}\norm{f}_D.
    \end{align}
On the other hand, a simple set size estimate yields
    \begin{align}\label{eqn: I_n^(2) bound-2}
         \sabs{I_n^{(2)}}&\lesssim 2^{-2n}2^{-p}2^{-2k}2^{-4k^+}\norm{f}_D.
    \end{align} 
    Together, \eqref{eqn: I_n^(1) bound}--\eqref{eqn: I_n^(2) bound-2} show that
    \begin{align*}
        \sum_{n\geq 1}\abs{I_n}&\lesssim \sum_{1\leq n\leq \log(t)-k}\min\{\sabs{I_n^{(1)}},\sabs{I_n^{(2)}}\}+\sum_{n\geq \log(t)-k}\sabs{I_n^{(2)}}\\
        &\lesssim 2^{-4k^+}2^{-p}t^{-1} \norm{f}_D\sum_{n\leq \log(t)-k}\min \{2^{k},2^{-n}\}+2^{-p}2^{-2k}2^{-4k^+}\norm{f}_D\sum_{n\geq \log(t)-k}2^{-2n}\\
        &\lesssim 2^{\frac{3}{4}k} 2^{-4k^+}2^{-p}t^{-1}\norm{f}_D+2^{\frac{3}{4}k}2^{-\frac{15}{4}k^+}2^{-p}t^{-1}\norm{f}_D.
    \end{align*}

\noindent\textbf{Part 2:} Fix $p\geq -10$. \blue{We first observe that \textbf{Cases A} and
and \textbf{C} follow exactly as above, since we don't use the largeness of $\Lambda$.}
\par\textbf{Case B.} First note that in \textbf{Part 1} we explicitly used the smallness of the parameter $p$ and the fact that $\abs{\Lambda}$ was bounded from below. In the current setting, we need to invoke the horizontal localization $\abs{\Lambda}\sim 2^q$, $q\in\Z^-$ introduced in Section \ref{sec: Localizations}. Moreover, we will use the $\varphi_{k,p,q}(\rho,\Lambda)$ functions, as well as the following fact for fixed $k,\,p$:
\begin{align*}
    \abs{P_{k,p}e^{it\Lambda}f}\leq\sum_{q\in \Z^-}\abs{P_{k,p,q}e^{it\Lambda}f}= \sum_{q\leq -\log(t)}\abs{P_{k,p,q}e^{it\Lambda}f}+\sum_{-\log(t)\leq q\leq 0}\abs{P_{k,p,q}e^{it\Lambda}f}
\end{align*}
First of all, we observe that if $q\leq -\log(t)$ (and since $p\geq -10$ is fixed), 
    \begin{align*}
        \abs{P_{k,p,q}e^{it\Lambda}f}&\lesssim \iint \varphi_{k,p,q}(\rho,\Lambda)\frac{\rho}{\sinlambda}\ff d\Lambda d\rho\\
        &\lesssim \sinftynorm{\widehat{P_{k,p}f}}\iint \overline{\varphi}(2^{-k}\rho)\overline{\varphi}(2^{-q}\Lambda)\rho d\Lambda d\rho\\
        &\lesssim  2^{2k}2^q\sinftynorm{\widehat{P_{k,p}f}},
    \end{align*}
     which implies using Lemma \ref{lemma: control of Fourier transform} that
     \begin{align*}
         \sum_{q\leq -\log(t)}\abs{P_{k,p,q}e^{it\Lambda}f}\lesssim 2^{k}2^{-4k^+}t^{-1}\norm{f}_D.
     \end{align*}
    To deal with the summation for $q\geq-\log(t)$, we proceed similarly to \textbf{Case C} above, however with a $q$-dependence. \blue{Observe that in 
    the setting of \textbf{Case B} there holds
      \begin{align*}
    2^q\abs{\partial_\rho\Psi}+ \abs{\partial_\rho\partial_\Lambda\Psi}\gtrsim 2^{-k}t.
    \end{align*}
    This can be seen from \eqref{eqn: lin decay derivatives of psi}, which implies that
\begin{align*}
        \partial_\Lambda \partial_\rho\Psi+\frac{\Lambda}{1-\Lambda^2}\partial_\rho \Psi=\frac{x_1}{1-\Lambda^2} \hspace{0.5cm} \Longrightarrow \hspace{0.5cm} (1-\Lambda^2)\partial_\Lambda \partial_\rho \Psi+\Lambda \partial_\rho\Psi=x_1,
    \end{align*}
    so that in particular
    \begin{align*}
        \abs{\partial_\Lambda \partial_\rho \Psi}+2^q\abs{\partial_\rho\Psi}\geq\abs{ (1-\Lambda^2)\partial_\Lambda \partial_\rho \Psi+\Lambda {\partial_\rho\Psi}}= \abs{x_1}\geq c^{-1}t2^{-k}.
    \end{align*}}
  We decompose the semigroup 
    \begin{align}\label{eqn: P_k,p,q decomposition in I_n}
        P_{k,p,q}e^{it\Lambda}f=\int_0^\infty \int_{-1}^1e^{i\Psi}\varphi_{k,p,q}(\rho,\Lambda)\hat{f}\frac{\rho}{\sinlambda}d\Lambda d\rho=\sum_{n\geq 0}I_n,
    \end{align}
    where $I_n=\int_0^\infty \int_{-1}^1e^{i\Psi}\varphi_{k,p,q}(\rho,\Lambda)\varphi(2^{-n}\partial_\rho\Psi)\hat{f}\frac{\rho}{\sinlambda}d\Lambda d\rho$. Again, we want to either integrate by parts in $\partial_\rho$ and make use of the fact that $\abs{\partial_\Lambda\partial_\rho\Psi}^{-1}\lesssim 2^{k}t^{-1}$, or employ a set size bound. To that end, observe first that on the support of $I_0$ we have $\abs{\partial_\rho\Psi}\sim 1$ and thus
    \begin{align*}
        1+\abs{\partial_\rho\partial_\Lambda\Psi}\gtrsim 2^q\abs{\partial_\rho\Psi}+\abs{\partial_\rho\partial_\Lambda\Psi} \gtrsim 2^{-k}t \hspace{0.5cm} \Longrightarrow \hspace{0.5cm} \abs{\partial_\Lambda\partial_\rho\Psi}\gtrsim 2^{-k}t.
    \end{align*}
    Therefore, by a change of variables $y=\partial_\rho\Psi$, Lemma \ref{lemma: control of Fourier transform} and $2^{-p}\lesssim 2^{10}$, we obtain
    \begin{align*}
        \abs{I_0}&\lesssim \sinftynorm{\widehat{P_{k,p}f}} \iint\overline{\varphi}_{k,p,q}(\rho,\Lambda)\varphi(\partial_\rho\Psi)\frac{\rho}{\sinlambda}d\Lambda d\rho\\
        &\lesssim \sinftynorm{\widehat{P_{k,p}f}}\snorm{\overline{\varphi}_{k,p,q}(\Lambda)\sinlambda^{-1}}_{L^{\infty}_\Lambda}\iint\varphi(\partial_\rho\Psi) \varphi(2^{-k}\rho) \rho d\Lambda d\rho\\
        &\lesssim 2^{2k}\sinftynorm{\widehat{P_{k,p}f}} \int\varphi(y)\abs{\partial_\Lambda \partial_\rho \Psi}^{-1} dy\\
        &\lesssim 2^{k-4k^+}t^{-1}\norm{f}_D.
    \end{align*}
    We highlight here that we use the $P_{k,p}$ projections to bound the Fourier transform with Lemma \ref{lemma: control of Fourier transform}. For $n\geq 1$, we integrate by parts either once or twice in $\partial_\rho$ using $\partial_\rho^2\Psi=0$ and obtain as in \textbf{Part 1}
    \begin{align*}
        \sabs{I_n^{(1)}}&\lesssim 2^{-n}2^{-k}2^{-4k^+}\norm{f}_D\iint   \varphi(2^{-n}\partial_\rho \Psi)\overline{\varphi}_{k,p,q}(\rho,\Lambda)d\Lambda d\rho,\\
        \sabs{I_n^{(2)}}&\lesssim 2^{-2n}2^{-2k}2^{-4k^+}\norm{f}_D\iint   \varphi(2^{-n}\partial_\rho \Psi)\overline{\varphi}_{k,p,q}(\rho,\Lambda)d\Lambda d\rho.
    \end{align*}
Similarly as in \textbf{Part 1}, we decompose the sum over all $n\geq1$ into the part where $q+n\leq \log(t)-k$ and $q+n\geq \log(t)-k$. 
\par If $q+n\leq \log(t)-k$ then $2^q\abs{\partial_\rho\Psi}\sim 2^{q+n}\lesssim 2^{-k}t$ and so necessarily $\abs{\partial_\Lambda\partial_\rho \Psi}\gtrsim 2^{-k}t$. As before we obtain the bounds \eqref{eqn: I_n^(1) bound} and \eqref{eqn: I_n^(2) bound} by the change of variables $y=\partial_\rho\Psi$ and Lemma \ref{lemma: control of Fourier transform}. Summing in $n$ yields
    \begin{align*}
        \sum_{q+n\leq \log(t)-k}\abs{I_n}&\lesssim 2^{-4k^+}t^{-1}\norm{f}_D\sum_{q+n\leq \log(t)-k}\min\{2^k,2^{-n}\}\lesssim 2^{\frac{3}{4}k}2^{-4k^+}t^{-1}\norm{f}_D
    \end{align*}  
    On the other hand, for $q+n\geq \log(t)-k$, we obtain with \eqref{eqn: I_n^(2) bound} that
        \begin{align*}
         \sum_{q+n\geq \log(t)-k}\abs{I_n}& \lesssim \sum_{q+n\geq \log(t)-k}2^{-2n}2^{-2k}2^{-4k^+}\norm{f}_D\iint \overline{\varphi}_{k,p,q}(\rho,\Lambda)d\Lambda d\rho\\
             &\lesssim   2^q 2^{-k}2^{-4k^+}\norm{f}_D \sum_{q+n\geq \log(t)-k}2^{-2n}   \\
             &\lesssim 2^k2^{-4k^+}t^{-2}\norm{f}_D.
        \end{align*}     
Thus from \eqref{eqn: P_k,p,q decomposition in I_n} we have
\begin{align*}
    \abs{P_{k,p,q}e^{it\Lambda}f}\lesssim 2^{\frac{3}{4}k}2^{-4k^+}t^{-1}\norm{f}_D\max \{2^{\frac{k}{4}},1\}\lesssim 2^{\frac{3}{4}k}2^{-\frac{15}{4}k^+}t^{-1}\norm{f}_D,
\end{align*}
    and altogether there holds
    \begin{align*}
       \sum_{-\log(t)\leq q} \hspace{-0.3cm}\abs{P_{k,p,q}e^{it\Lambda}f}&=  \sum_{-\log(t)\leq q}\sum_{n\geq 0}\abs{I_n}\lesssim\hspace{-0.3cm} \sum_{-\log(t)\leq q}\hspace{-0.3cm} 2^{\frac{3}{4}k}2^{-\frac{15}{4}k^+}t^{-1}\norm{f}_D\lesssim  2^{\frac{3}{4}k}2^{-\frac{15}{4}k^+}\log(t)t^{-1}\norm{f}_D.
    \end{align*}
\end{proof}

In the bootstrap setting \eqref{eqn: bootstrap assumption}, the above proposition can be applied directly to $S^b\mathcal \zz_\pm(t)$ and $S^b \Theta(t)$ where $0\leq b\leq N-2$, as is clear from the two copies of $S$ required in the decay norm \eqref{norm: decay norm}. Thanks to interpolation, we can furthermore obtain some decay also for the remaining powers of vector fields on the profiles:
\begin{lemma}\label{lemma: linear decay many vector fields}
    Let $F\in \{\zz_\pm,\Theta\}$ and assume the bootstrap condition \eqref{eqn: bootstrap assumption} holds. Moreover, assume the number of vector fields $M>0$ in \eqref{eqn: initial data assumption} (\eqref{eqn: initial data assumption SQG} resp.) is sufficiently large and let $0<\kappa\ll \beta$. Then the weaker decay holds:
    \begin{align*}
        \sinftynorm{P_ke^{it\Lambda}S^bF}\lesssim 2^{\frac{3}{4}k-3k^+}t^{-\frac{1}{2}+\kappa}\eps, \hspace{1cm} 0\leq b\leq N.
    \end{align*}
\end{lemma}
\begin{proof}
   For $b\leq N-2$, the decay norm $\snorm{S^b F}_D$ is bounded since the bootstrap assumption \eqref{eqn: bootstrap assumption} holds. Therefore by Proposition \ref{proposition: linear decay} there holds:
    \begin{align*}
         \sinftynorm{P_ke^{it\Lambda}S^bF}\lesssim 2^{\frac{3}{4}k-\frac{15}{4}k^+}t^{-\frac{1}{2}}\eps.
    \end{align*}
    For $N-2< b\leq N$ we use interpolation: For integers $r,K\geq 1,$ $a,b\geq 0$ and $ \snorm{S^{\leq b}F}_{L^r}:=\sup_{0\leq \alpha\leq b}\norm{S^\alpha F}_{L^r}$ a standard convexity argument (see e.g.\ \cite[\textcolor{MidnightBlue}{Lemma A.6}]{EC2022}) gives that
    \begin{align*}
    \snorm{S^{\leq b}g}_{L^{2r}}\lesssim_{K,r,b}\norm{g}_{L^{2r}}^{1-\frac{1}{K}}\snorm{S^{\leq Kb}g}_{L^{2r}}^{\frac{1}{K}}.
    \end{align*}
    Applying this with $g=P_kS^{b-2}F$, Proposition \ref{proposition: linear decay} and the energy estimates \eqref{eqn:bootstrap-prop-bded-energy} with $M\gg 1$ sufficiently large (in particular such that $N+2(K-1)<M$), there holds with $r\sim K\gg \kappa^{-1}$ that
    \begin{align*}
        \sinftynorm{P_ke^{it\Lambda}S^bF}&\lesssim 2^{\frac{k}{r}}\snorm{P_ke^{it\Lambda}S^bF}_{L^{2r}}\\
        &\lesssim 2^{\frac{k}{r}}2^{k\frac{r-1}{rK}}\sinftynorm{P_ke^{it\Lambda}S^{b-2}F}^{(1-\frac{1}{r})(1-\frac{1}{K})}\sltwonorm{P_kS^{b-2}F}^{\frac{1}{r}(1-\frac{1}{K})}\sltwonorm{P_kS^{\leq b+2(K-1)}F}^{\frac{1}{K}}\\
        &\lesssim 2^{\frac{k}{r}}2^{k\frac{r-1}{rK}}(2^{\frac{3}{4}k-\frac{15}{4}k^+}t^{-\frac{1}{2}}\eps)^{(1-\frac{1}{r})(1-\frac{1}{K})}\eps^{\frac{1}{r}(1-\frac{1}{K})}\eps^{\frac{1}{K}}\\
        &\lesssim 2^{\frac{3}{4}k-3k^+}t^{-\frac{1}{2}+\frac{1}{K}}\eps.
    \end{align*}
\end{proof}

\section{Energy Estimates}\label{sec: both energy estimates}
In this section we establish energy estimates for the systems \eqref{eqn: perturbed BQ} and
\eqref{SQG}. We give first the full details for the Boussinesq system \eqref{eqn: perturbed BQ}, where also the appropriate choice of scalar unknowns plays an important role, see Corollary \ref{cor: energy estimates of Z+-}. The SQG setting \eqref{SQG} can then be dealt with analogously.

\subsection{Energy estimates Boussinesq system}\label{sec: Energy Estimates}
We establish energy estimates for the Boussinesq system \eqref{eqn: perturbed BQ}. On the one hand we have the classical $H^n$ estimates on $(u,\rho)$. On the other hand, we provide $L^2$ estimates for arbitrarily many vector fields $S$. Recall the scaling vector field $\mathcal{S}=-\mathrm{id}+S$ from \eqref{def: S W}. If $(u,\rho)$ is a solution to \eqref{eqn: perturbed BQ}, then $(\mathcal{S}^n u,\mathcal{S}^n\rho)$ satisfies
\begin{equation}\label{eqn: S^n BQ}
\left\{
      \begin{aligned}
 &\partial_t\mathcal{S}^nu +\mathcal{S}^n(u\cdot\nabla u)=-\mathcal{S}^n\nabla p -\mathcal{S}^n\rho \Vec{e_2}\\ 
 &\partial_t\mathcal{S}^n\rho+ \mathcal{S}^n( u\cdot \nabla\rho)=\mathcal{S}^nu_2 \\
    & \mathrm{div}u=0.
\end{aligned}
    \right.
\end{equation}
\begin{remark}
    \begin{enumerate}
        \item We obtain energy estimates first on the vector field $\mathcal{S}$ using \eqref{eqn: S^n BQ}. These yield estimates on $S=x\cdot\nabla_x$ since  $S^n=\sum_{k=0}^nc_k\mathcal{S}^k$ for binomial constants $c_k>0$.
        \item We note that with a similar proof as below we obtain $\dot{H}^{-1}$ estimates on arbitrarily many vector fields $S$ applied to a solution ($u$,$\rho$) to \eqref{eqn: perturbed BQ}, provided that the initial data $(u_0,\rho_0)\in (H^{-1}\cap H^n)\times (H^{-1}\cap H^n)$.
    \end{enumerate}
\end{remark}
\begin{proposition}\label{prop: energy estimates BQ}
    Let $(u,\rho)\in  C([0,T],H^n(\R^2))\times C([0,T],H^n(\R^2))$ solve \eqref{eqn: perturbed BQ} with initial data $(u_0,\rho_0)\in C([0,T],H^n(\R^2))\times C([0,T],H^n(\R^2))$ for $ T\geq 0$ and some $n\in \N$. Then for $0\leq t\leq T$ there holds
    \begin{align*}
        &\norm{u(t)}_{H^{n}}^2+\norm{\rho(t)}_{H^{n}}^2- \norm{u_0}_{H^{n}}^2-\norm{\rho_0}_{H^{n}}^2\lesssim \int_0^tA(s)(\norm{u(s)}_{H^n}^2+\norm{\rho(s)}_{H^{n}}^2)(1+s)^{-\frac{1}{2}}ds,\\
        &\sltwonorm{{S}^nu(t)}^2+\sltwonorm{{S}^n\rho(t)}^2-\sltwonorm{{S}^nu_0}^2-\sltwonorm{{S}^n\rho_0}^2 \\
        &\hspace{2.4cm}\lesssim\int_0^t A_1(s)\big(\norm{u(s)}_{H^n}^2+\norm{\rho(s)}_{H^n}^2+\sum_{j=0}^n\sltwonorm{{S}^ju(s)}^2+\sltwonorm{{S}^j\rho(s)}^2\big)(1+s)^{-\frac{1}{2}}ds,
    \end{align*}
    where for $0\leq s\leq t$
    \begin{align*}
        &A_0(s):=(\inftynorm{\nabla u}+\inftynorm{\nabla \rho})(1+s)^{\frac{1}{2}}, && A_1(s):=(\inftynorm{\nabla u}+\sinftynorm{{S}u}+\inftynorm{\nabla \rho})(1+s)^{\frac{1}{2}}.
    \end{align*}
\end{proposition}
\begin{proof}
    The first statement follows by standard Sobolev energy estimates, see \cite[\textcolor{MidnightBlue}{Proposition 2.5}]{LWP_Chae}, hence we only give the details for the energy estimate involving $S$. 
 \par Observe that $[\mathcal{S},\nabla]=\nabla$ and by iteration, we have the following commutator rule
\begin{align}\label{eqn: commutator S^n nabla}
    [\mathcal{S}^n,\nabla]=\sum_{k=0}^{n-1}c_k\nabla \mathcal{S}^k=\sum_{k=0}^{n-1}\tilde{c}_k \mathcal{S}^k\nabla,&& c_k,\tilde{c}_k\in \Z.
\end{align}
  By taking a scalar product with $\mathcal{S}^nu$ and $\mathcal{S}^n\rho$ respectively in \eqref{eqn: S^n BQ}, we obtain
    \begin{equation*}
\left\{
      \begin{aligned}
 &\frac{1}{2}\frac{d}{dt}\sltwonorm{\mathcal{S}^nu}+\ltwoscalarproduct{\mathcal{S}^n(u\cdot \nabla u)}{\mathcal{S}^nu}=-\ltwoscalarproduct{\mathcal{S}^n \nabla p}{\mathcal{S}^nu}-\ltwoscalarproduct{\mathcal{S}^n\rho \vec{e_2}}{\mathcal{S}^nu}\\
 &\frac{1}{2}\frac{d}{dt}\sltwonorm{\mathcal{S}^n\rho }+\ltwoscalarproduct{\mathcal{S}^n(u\cdot \nabla \rho)}{\mathcal{S}^n\rho}=\ltwoscalarproduct{\mathcal{S}^nu_2}{\mathcal{S}^n\rho}.
\end{aligned}
    \right.
\end{equation*}
Adding the two equations, we observe that the terms $-\ltwoscalarproduct{\mathcal{S}^n\rho \vec{e_2}}{\mathcal{S}^nu}+\ltwoscalarproduct{\mathcal{S}^nu_2}{\mathcal{S}^n\rho}$ cancel. Moreover, since $\mathcal{S}\nabla p=(\mathrm{id}-S)\nabla p$ an
integration by parts yields
\begin{align*}
    -\ltwoscalarproduct{\mathcal{S}^n \nabla p}{\mathcal{S}^nu}=-\ltwoscalarproduct{\sum_{k=0}^{n-1}c_k\nabla\mathcal{S}^k  p}{\mathcal{S}^nu}=\ltwoscalarproduct{\sum_{k=0}^{n-1}c_k\mathcal{S}^k  p}{\nabla\cdot\mathcal{S}^nu}=\sum_{k,j=0}^{n-1}\ltwoscalarproduct{c_k\mathcal{S}^k  p}{\tilde{c}_j\mathcal{S}^j\nabla\cdot u}=0.
\end{align*}
It remains to bound the terms $\ltwoscalarproduct{\mathcal{S}^n(u\cdot \nabla u)}{\mathcal{S}^nu}$ and $\ltwoscalarproduct{\mathcal{S}^n(u\cdot \nabla \rho)}{\mathcal{S}^n\rho}$. We bound the first scalar product. By the Lebniz rule  and commutator rule \eqref{eqn: commutator S^n nabla} we have
\begin{align}\label{eqn: bound s^n(u grad u)}
    \ltwoscalarproduct{\mathcal{S}^n(u\cdot \nabla u)}{\mathcal{S}^nu}=\ltwoscalarproduct{\sum_{k=0}^n\mathcal{S}^ku\sum_{j=0}^{n-k}\nabla\mathcal{S}^{j}u}{\mathcal{S}^nu}.
\end{align}
Observe that for $k=0$ and $k=n$ there holds by incompressibility of $u$
\begin{align*}
    &\ltwoscalarproduct{u\mathcal{S}^n\nabla u}{\mathcal{S}^nu}=\sum_{j=0}^{n-1}\ltwoscalarproduct{u\nabla \mathcal{S}^ju}{\mathcal{S}^nu}+\ltwoscalarproduct{u\nabla \mathcal{S}^nu}{\mathcal{S}^nu}\lesssim \inftynorm{u}\sum_{j=0}^{n-1}\sltwonorm{\nabla \mathcal{S}^ju}\sltwonorm{\mathcal{S}^nu},\\
    &\ltwoscalarproduct{\mathcal{S}^nu\nabla u}{\mathcal{S}^nu}\lesssim \inftynorm{\nabla u}\sltwonorm{\mathcal{S}^nu}.
\end{align*}
Thus in order to prove the claim, with \eqref{eqn: bound s^n(u grad u)} it remains to establish the following bounds for $1\leq j\leq n-1$
\begin{align}
     \sltwonorm{\nabla \mathcal{S}^ju}&\lesssim\sum_{k=0}^n \sltwonorm{\mathcal{S}^ku}^2+\norm{u}_{H^n}^2, \label{eqn: nabla Sn-1 bound} \\ 
     \ltwoscalarproduct{\mathcal{S}^ju \mathcal{S}^{n-j}\nabla u}{\mathcal{S}^nu}&\lesssim (\sinftynorm{\mathcal{S}u}+\sinftynorm{\nabla u})(\sum_{k=0}^n\sltwonorm{\mathcal{S}^ku}^2+\norm{u}_{H^n}^2) \label{eqn: mixed Sn terms claim}.
\end{align}
These follow from integration by parts using $\ltwoscalarproduct{(\mathcal{S}-2\mathrm{id})f}{g}=-\ltwoscalarproduct{f}{\mathcal{S}g}$ and the commutator rule \eqref{eqn: commutator S^n nabla} and standard interpolation -- see e.g.\ \cite[\textcolor{MidnightBlue}{proof of Proposition 5.1, Lemma 5.3}]{guo2020stabilizing}.
\end{proof}

\begin{corollary}\label{cor: energy estimates of Z+-}
    Let $Z_\pm$, $\zz_\pm$ be the dispersive unknowns resp.\ profiles \eqref{def: Profiles BQ} of \eqref{eqn: perturbed BQ} and $t\in [0,T]$ with $T\lesssim \eps^{-2}$, $M$ as in Theorem \ref{thm: main result}. Then under the bootstrap assumptions \eqref{eqn: bootstrap assumption} there holds that
    \begin{align*}
\norm{Z_\pm(t)}_{H^{N_0}}+\sum_{a=0}^M\snorm{S^aZ_\pm(t)}_{L^2}=\norm{\zz_\pm(t)}_{H^{N_0}}+\sum_{a=0}^M\snorm{S^a\zz_\pm(t)}_{L^2}\lesssim \eps.
    \end{align*}
\end{corollary}
\begin{proof}
     With Corollary \ref{cor: linear decay semigroup} and under the bootstrap assumption \eqref{eqn: bootstrap assumption}, we observe that for $s>0$
    \begin{align*}
        A_j(s)\lesssim  \abs{s}^{-\frac{1}{2}}\sum_{\mu \in \{\pm\}}(\norm{{\zz}_\mu}_D+\norm{S{\zz}_{\mu}}_D) (1+s)^{\frac{1}{2}}\lesssim \eps , && j=0,1.
    \end{align*}
   The claim then follows from \eqref{eqn: energy conservation Z+ Z-} resp.\ \eqref{eqn: S conservation Z+ Z-} and Gronwall's lemma: we obtain for $t\lesssim \eps^{-2}$ that
    \begin{align*}
        \norm{Z_\pm(t)}_{H^{N_0}}^2+\sum_{a=0}^M\snorm{S^aZ_\pm(t)}_{L^2}^2\lesssim \eps^2\exp\Big(\int_0^t c \eps(1+s)^{-\frac{1}{2}}ds \Big)\lesssim \eps^2.
    \end{align*}
\end{proof}
\subsection{Energy estimates SQG equation}\label{sec: Energy Estimates SQG}
For the SQG equation \eqref{SQG} we have by \eqref{eqn: semigroup_comm} that
\begin{equation}\label{equation for Stheta}
    \partial_t \mathcal{S}^n\theta+\mathcal{S}^n(u\cdot\nabla\theta)=R_1\mathcal{S}^n\theta.
\end{equation}
\par The energy estimates for $\mathcal{S}^n\theta$ then yield estimates for $S^n\theta$ as in the previous section:
\begin{proposition}\label{prop: energy estimates SQG}
    Let $\theta$ be a solution to \eqref{SQG} on $0\leq t\leq T$ and $n\in \N$. Then there holds
  \begin{align*}
  \norm{\theta(t)}_{H^n}^2-\norm{\theta_0}_{H^n}^2&\lesssim \int_0^t \tilde{A}_0(s)\norm{\theta(s)}_{H^n}^2\frac{1}{(1+\abs{s})^{\frac{1}{2}}}ds,\\
        \sltwonorm{{S}^n\theta(t)}^2-\sltwonorm{{S}^n\theta_0}^2&\lesssim \int_0^t \tilde{A}_1(s)(\norm{\theta(s)}^2_{H^n}+\sum_{j\leq n}\sltwonorm{{S}^j\theta(s)}^2)\frac{1}{(1+\abs{s})^{\frac{1}{2}}}ds, 
  \end{align*}
where
$\tilde{A_0}(s)=(\inftynorm{\nabla \theta}+\inftynorm{\nabla u})(1+\abs{s})^{\frac{1}{2}},\; \tilde{A_1}(s)=(\inftynorm{\nabla \theta}+\sinftynorm{u}+\sinftynorm{\mathcal{S} u})(1+\sabs{s})^{\frac{1}{2}}.$
\end{proposition} 
\begin{corollary}\label{cor: energy estimates SQG}
    Under the bootstrap assumption \eqref{eqn: bootstrap assumption} and with $M$ as in Theorem \ref{thm: main result detailed SQG}, for $t\in[0,T]$ and $T\lesssim \eps^{-2}$ there holds
   \begin{align*}
        \norm{\theta(t)}_{H^{N_0}}+\sum_{a=0}^M \ltwonorm{S^a\theta(t)}=\norm{\Theta(t)}_{H^{N_0}}+\sum_{a=0}^M \ltwonorm{S^a\Theta(t)}\lesssim \eps,
   \end{align*} 
   where $\Theta(t)=e^{it\Lambda}\theta(t)$ is the profile of $\theta$.
\end{corollary}

\section{Oscillatory Toolbox: Integration by Parts along Vector Fields and Normal Forms}\label{sec: integration by parts}
In this section we present the technical tools used to establish the main results.  After some preliminary computations for vector fields in Section \ref{sec: vector field lemmas}, in Section \ref{sec: multiplier mechanics} we construct a class of multipliers that contains those of our bilinear terms and is closed under the action of $S$. Moreover, for these we can track bounds in terms of our localisation parameters $k,p,l,k_i,p_i,l_i$ for $i=1,2$ when iteratively applying $S$. The action of the vector fields $S,W$ on the phases $\Phi$ are discussed in Section \ref{sec: vector fields and the phases}. In particular, we present a result that guarantees either largeness of $\Phi$ or lower bounds for $S\Phi$, see Proposition \ref{prop: lower bound on sigma}. A robust method for integrating by parts along $S$ in bilinear expressions involving multipliers of the class previously defined is presented in Section \ref{subsection: Integration by parts in bilinear expressions}, and combines the multiplier mechanics and some basic vector field algebra, quantified via the localizations introduced in Section \ref{sec: Localizations}. In particular, here the angular projectors $R_l$ are used to precisely capture under which conditions repeated integration by parts is feasible. The action of the vector field $W$ on bilinear expressions is also discussed in Section \ref{subsection: Integration by parts in bilinear expressions}. In Section \ref{sec: Case organization} we present a lemma that serves to organize the proofs in the sections to follow according to the relative size of the localization parameters involved. In Section \ref{sec: set-size estimate} we discuss possible gains due to small sets of integration, and finally, in Section \ref{sec: Normal forms}, we present normal forms. 
\subsection{Vector field lemmas}\label{sec: vector field lemmas}
We will now discuss the action of $S$ and $W$. Since different variables are involved we will sometimes highlight the variables on which these vector fields are acting explicitly, recalling from \eqref{def: S W} that
\begin{equation}
 S_x=x\cdot\nabla_x,\quad W_x=x^\perp\cdot\nabla_x.   
\end{equation}
To integrate by parts in bilinear forms such as \eqref{eqn: bilin_form_Q}, we further define
\begin{align}
            &S_{\xi-\eta}:=(\xi-\eta)\nabla_\eta, && W_{\xi-\eta}:=(\xi-\eta)^\perp\nabla_\eta.
\end{align}
When there is no risk of confusion, we will suppress the explicit dependence of $S,W$.

\begin{lemma}\label{lemma: first vf lemma}
    For $x,y\in \R^2$ there holds that
    \begin{enumerate}
        \item  $\partial_{x_1}=\frac{x^\perp}{\abs{x}^2}\cdot(W_x,S_x)^T$ and $ \partial_{x_2}=\frac{x}{\abs{x}^2}\cdot(W_x,S_x)^T,$
   \item $y_1\partial_{x_1}+y_2\partial_{x_2}=\frac{y\cdot x}{\abs{x}^2}S+\frac{y\cdot x^\perp}{\abs{x}^2}W_x$.
    \end{enumerate}
\end{lemma}
\begin{proof}
    We compute directly that
    \begin{align*}
      &  \frac{x^\perp}{\abs{x}^2}\cdot(W,S)^T=\frac{1}{\abs{x}^2}[-x_2(-x_2\partial_{x_1}+x_1\partial_{x_2})+x_1(x_1\partial_{x_1}+x_2\partial_{x_2})]=\partial_{x_1},\\
      &\frac{x}{\abs{x}^2}\cdot(W,S)^T=\frac{1}{\abs{x}^2}[x_1(-x_2\partial_{x_1}+x_1\partial_{x_2})+x_2(x_1\partial_{x_1}+x_2\partial_{x_2})]=\partial_{x_2},
    \end{align*}
    and the second statement also follows by direct computation.
    \end{proof}

    Since $S,W$ span the tangent space at any point, they allow us to resolve any derivative as follows:
    \begin{lemma}\label{lemma: S_eta= S_xi-eta W_xi-eta}
    There holds
    \begin{align}\label{eqn: S_eta= S_xi-eta W_xi-eta}
    \begin{split}
         &S_\eta=\frac{\eta(\xi-\eta)}{\absxieta^2}S_{\xi-\eta}-\frac{\eta(\xi-\eta)^\perp}{\absxieta^2}W_{\xi-\eta}, 
         \hspace{1cm}S_{\xi-\eta}=\frac{(\xi-\eta)\eta}{\abs{\eta}^2}S_\eta+\frac{(\xi-\eta)\eta^\perp}{\abs{\eta}^2}W_\eta,\\
         & W_\xi=\frac{(\xi-\eta)\xi}{\absxieta^2}W_{\xi-\eta}-\frac{(\xi-\eta)^\perp\xi}{\absxieta^2}S_{\xi-\eta}.
    \end{split}
    \end{align}
    \end{lemma}
    \begin{proof}
     The first statement in \eqref{eqn: S_eta= S_xi-eta W_xi-eta} follows from Lemma \ref{lemma: first vf lemma} with $x=\xi-\eta$ and $y=\eta$,
  \begin{align*}
    \partial_{\eta_1}(f(\xi-\eta))= \frac{(\xi-\eta)^\perp}{\absxieta^2}(W_{\xi-\eta},S_{\xi-\eta})^T, &&   \partial_{\eta_2}(f(\xi-\eta))= \frac{(\xi-\eta)}{\absxieta^2}(W_{\xi-\eta},S_{\xi-\eta})^T.
  \end{align*}
  so that
  \begin{align*}
      S_\eta=\eta\nabla_\eta=\frac{\eta(\xi-\eta)}{\absxieta^2}S_{\xi-\eta}-\frac{\eta(\xi-\eta)^\perp}{\absxieta^2}W_{\xi-\eta}.
  \end{align*}
      The second statement in \eqref{eqn: S_eta= S_xi-eta W_xi-eta} follows from Lemma \ref{lemma: first vf lemma} with $x=\eta$ and $y=\xi-\eta$, while to obtain the third statement take $x=\xi-\eta$ and $y=\xi$.
    \end{proof}

\subsection{Multiplier mechanics}\label{sec: multiplier mechanics}
For repeated integration by parts in bilinear terms, it is important to understand how the vector field $S$ acts on the multipliers and phase functions present. In our framework, this is quantified in terms of the localization parameters defined in Section \ref{sec: Localizations}. 

We begin by defining the set of elementary multipliers
\begin{align*}
    E:=\set{\frac{\zeta\cdot \theta^\perp}{\abs{\zeta}\abs{\theta}},\frac{\zeta\cdot \theta}{\abs{\zeta}\abs{\theta}}, \Lambda(\zeta), \sqrt{1-\Lambda^2(\zeta)}}{\zeta, \theta \in \{\xi, \xi-\eta,\eta\}}.
\end{align*}
Elements $e\in E$ satisfy $\abs{e}\leq 1$ and we will show that the set of linear combinations of products of such elements is closed under the action of $S_\eta$, $S_{\xi-\eta}$. We define the following sets to track the order of the multipliers:
\begin{align*}
    E_0 :=\text{span}_\R\set{\prod_{i=1}^Ne_i}{e_i \in E,\, N\in \N},\hspace{1cm}
    E_a^b :=\text{span}_\R\set{\abseta^a\absxieta^{b}e}{e\in E_0}, \; a,\;b\in \Z.
\end{align*}
As mentioned in the introduction, the nonlinearity in our problem has a skew-symmetric structure. It turns out that the following quantity plays a central role:
\begin{align} \label{eqn: symmetry of sigma}
    \sigma(\xi,\eta):=(\xi-\eta)\cdot\eta^\perp, \quad \sigma(\xi,\eta)=\sigma(\xi-\eta,\eta)= -\sigma(\xi,\xi-\eta).
\end{align}
\begin{lemma} \label{lemma: S on basic multipliers}

    Let $e \in E_a^b$, and consider localizations $\chi,\;\tchi$ as in \eqref{eqn: chi localizations}. Then there holds
    \begin{equation*}
     S_\eta e \in E_{a}^b\cup E_{a+1}^{b-1},\qquad S_{\xi-\eta}e \in E_a^b\cup E_{a-1}^{b+1},
    \end{equation*}
    and we have the bounds 
    \begin{align}
       \abs{S_\eta e}\chi(\xi,\eta)&\lesssim (1+2^{k_2-k_1}2^{p_{\max}})\inftynorm{e\chi}, \label{eqn: S_eta e bound} \\
     \abs{S_\eta e}\tchi(\xi,\eta)&\lesssim (1+2^{k_2-k_1}(2^{q_{\max}}+2^{p_{\max}}))\inftynorm{e\tchi},\nonumber
     \end{align}
     and symmetrically
     \begin{align}
       \abs{S_{\xi-\eta} e}\chi(\xi,\eta) &\lesssim (1+2^{k_1-k_2}2^{p_{\max}}) \inftynorm{e\chi}, \nonumber \\
        \abs{S_{\xi-\eta} e}\tchi(\xi,\eta) &\lesssim (1+2^{k_1-k_2}(2^{q_{\max}}+2^{p_{\max}})) \inftynorm{e\tchi}.\nonumber
    \end{align}

\end{lemma}
\begin{proof}
By symmetry and the product rule it suffices to consider $S_\eta e$, with $e\in E$ and prove \eqref{eqn: S_eta e bound}. 
We have four types of elementary multipliers in $E$. First observe that with $\sigma$ as in \eqref{eqn: symmetry of sigma} there holds
\begin{align}\label{eqn: Seta Lambda}
   & S_\eta \Lambda(\eta)=0,&& S_\eta \Lambda(\xi-\eta)=-\frac{\xi_2-\eta_2}{\absxieta^3}\sigma(\xi,\eta).
\end{align}
Thus $S_\eta \Lambda(\zeta)\in E_{1}^{-1}$ for $\zeta\in \{\xi,\eta,\xi-\eta\}$. Similarly with \eqref{eqn: Seta Lambda} there holds
\begin{align*}
&S_\eta\sqrt{1-\Lambda^2(\xi-\eta)})=\Lambda(\xi-\eta)\frac{(\xi-\eta)\cdot\eta^\perp}{\absxieta\abseta}\frac{\abseta}{\absxieta}, && S_\eta \sqrt{1-\Lambda^2(\eta)})=0,
\end{align*}
which are also elements of $ E_{1}^{-1}$. Thus the bounds hold:
\begin{align*}
  &  \abs{S_\eta\Lambda(\xi-\eta)}\chi(\xi,\eta)\lesssim 2^{k_2-k_1}2^{p_{\max}}, &&  \abs{S_\eta\sqrt{1-\Lambda^2(\xi-\eta)}}\chi\lesssim 2^{k_2-k_1}2^{p_{\max}}.
\end{align*}
Next we have the following computations
\begin{align}
     &S_\eta\abseta=\abseta, && S_\eta\absxieta=-\frac{\eta\cdot(\xi-\eta)}{\absxieta}, &&S_\eta\sigma =-\eta\cdot\xi^\perp, && S_\eta((\xi-\eta)\cdot\eta)=\eta\cdot(\xi-\eta)-\abseta^2,\label{eqn: S_eta on abs, sigma}
     \end{align}
With this we prove the claim for multipliers of the form  $\frac{\zeta\cdot \theta^\perp}{\abs{\zeta}\abs{\theta}}$ and $\frac{\zeta\cdot \theta}{\abs{\zeta}\abs{\theta}}$, where by symmetry it suffices to consider $\zeta=\xi-\eta, \, \theta=\eta$:
\begin{align*}
   & S_\eta\Big(\frac{(\xi-\eta)\cdot\eta^\perp}{\absxieta\abseta}\Big)
    =\frac{(\xi-\eta)\cdot\eta^\perp}{\absxieta\abseta}\frac{\eta\cdot(\xi-\eta)}{\abseta\absxieta}\frac{\abseta}{\absxieta},
   \;  S_{\eta}\Big(\frac{(\xi-\eta)\cdot\eta}{\absxieta\abseta} \Big)
    =\frac{-\abseta}{\absxieta}\left(1-\left(\frac{(\xi-\eta)\cdot\eta}{\absxieta\abseta}\right)^2\right).
\end{align*}
These are elements of $E_{1}^{-1}$ and satisfy the bounds:
\begin{align*}
  & \abs{ S_\eta\Big(\frac{(\xi-\eta)\cdot\eta^\perp}{\absxieta\abseta}\Big)}\lesssim 2^{k_2-k_1}2^{p_{\max}},
  &&\abs{ S_{\eta}\Big(\frac{(\xi-\eta)\cdot\eta}{\absxieta\abseta} \Big)}\lesssim 2^{k_2-k_1}.
\end{align*}
Since $1\in E$ we obtain altogether for $e\in E_0$:
\begin{align}\label{eqn: S^n of e_0}
   & \abs{S_\eta e}\chi\lesssim 1+2^{k_2-k_1}2^{p_{\max}}.
\end{align}
Finally let $e \in E_a^b$. Then $e=\abseta^a \absxieta^{b}e_0$ for an $e_0 \in E_0$, and with \eqref{eqn: S_eta on abs, sigma} and for a suitable $e_1\in E_0$ we have
\begin{align*}
    S_\eta e
    &=a\abseta^{a}\absxieta^be_0+b\abseta^{a+1}\absxieta^{b-1}e_1+\abseta^a\absxieta^bS_\eta e_0,
\end{align*}
which is an element of $E_a^b\cup E_{a+1}^{b-1}$. The bound follows directly using the claim \eqref{eqn: S^n of e_0}. An analogous computation gives the result for $S_{\xi-\eta}$.
\end{proof}
\begin{lemma}\label{lemma: W_xi on elementary multipliers}
    Let  $\tilde{E}_a^b:=\set{\absxi^{a}\absxieta^{b}e_0}{e_0 \in E_0}$ and $e\in E_0$. Then $W_\xi e \in \tilde{E}_1^{-1}$ and
    \begin{align}
        \abs{W_\xi e}\chi\lesssim 1+ 2^{k-k_1}2^{p_{\max}}.
    \end{align}
    Moreover, if $e_{ab}\in \tilde{E}_a^b$ then $W_\xi e_{ab}\in \tilde{E}_{a}^{b}\cup \tilde{E}_{a+1}^{b-1}$, and 
    \begin{align}
         \abs{W_\xi e_{ab}}\chi\lesssim (1+ 2^{k-k_1}2^{p_{\max}}) \inftynorm{e_{ab}\chi}.
    \end{align}
\end{lemma}
\begin{proof}
    The claim follows by similar computations as in the lemma above.
\end{proof}

As discussed in the introduction of the paper, the null-structure of the nonlinearity is encoded in the multipliers \eqref{eqn: bq multipliers}, \eqref{multiplier} of the bilinear terms. More precisely, the relevant bounds for us are the following:
\begin{lemma}\label{lemma: multiplier bound}
    Let $\m\in \{\m_0,\m_\pm^{\mu\nu}\}$, where $\m_0$ is the multiplier in \eqref{multiplier}, $\m_\pm^{\mu\nu}$ one of the multipliers in \eqref{eqn: bq multipliers}. Then  $\m\in E_1^0\cup E_0^1$ and the following bound holds:
    \begin{align*}
        &\abs{\m(\xi,\eta)}\chi\lesssim 2^{k}2^{p_{\max}}, &&\abs{\m(\xi,\eta)}\tchi\lesssim 2^{k}2^{p_{\max}+q_{\max}} ,
    \end{align*}
     where $\chi,\, \tchi$ as in \eqref{eqn: chi localizations}.
\end{lemma}

\begin{proof}
We prove the first bound and note that the second one follows analogously when additionally localizing in $q,q_1,q_2\in \Z^-$. 
Since 
\begin{equation}
 \abs{\m_0}\chi=\frac{1}{2} \frac{\abs{(\xi-\eta)\cdot\eta^\perp}}{\absxieta \abseta}\abs{\absxieta-\abseta}\leq \frac12 (2^{p_1}+2^{p_2})2^k, 
\end{equation}
the claim is direct for $\m_0$.

For $\m^{\mu\nu}_\pm$ we will establish the following bound, which implies the claim:
 \begin{align}\label{eqn: bq multiplier first bound}
     \abs{\m^{\mu\nu}_\pm}\chi\lesssim 2^{p_{\max}}\min\{2^{k}\max\{1+2^{k_1-k_2},1+2^{k_2-k_1}\},\max\{2^{k_1},2^{k_2}\}\}.
 \end{align} 
 To see this, we bound the additional terms in \eqref{eqn: bq multipliers} separately. On the one hand, we have the direct bounds
    \begin{equation}\label{eqn: bound each component multiplier}
    \begin{split}
         \abs{\frac{\xi(\xi-\eta)^\perp}{\absxi\absxieta}\Big( \frac{\abseta^2-\absxieta^2}{\abseta} \Big)}\chi&\lesssim (2^{p}+2^{p_1})2^{k}(1+2^{k_1-k_2}),\\
        \abs{\frac{\xi(\xi-\eta)^\perp}{\absxi\absxieta}(\absxieta+\abseta)}\chi&\lesssim(2^{p_1}+2^{p_2})2^{k_1}+(2^{p_1}+2^{p_2})2^{k_2},
    \end{split}
    \end{equation}
    while we can alternatively use \eqref{eqn: symmetry of sigma} to obtain that
    \begin{equation}\label{eqn: bound each component multiplier2}
    \begin{split}
     \abs{\frac{\xi(\xi-\eta)^\perp}{\absxi\absxieta}\Big( \frac{\abseta^2-\absxieta^2}{\abseta} \Big)}\chi&=\abs{\frac{\xi(\xi-\eta)^\perp}{\absxi\absxieta}\abseta+\frac{\xi\eta^\perp}{\absxi\abseta}\absxieta}\chi\lesssim (2^{p}+2^{p_1})2^{k_2}+(2^p+2^{p_2})2^{k_1},\\
     \abs{\frac{(\xi-\eta)^\perp\eta}{\absxieta\abseta}\big(\absxieta +\abseta\big)}\chi&=\abs{\frac{\xi(\xi-\eta)^\perp}{\absxi\absxieta}\absxi\Big(\frac{\abseta+\absxieta}{\abseta}\Big)}\chi\lesssim (2^p+2^{p_1})2^{k}(1+2^{k_1-k_2}).
    \end{split}
    \end{equation}
\end{proof}
For repeated integration by parts we also want to understand how the vector fields act on a multiplier $\m$. For a set $A$, we let $\abseta A=\set{\abseta\cdot a}{a\in A}$.
\begin{lemma}\label{lemma: S_eta on multiplier m}
    Let $\m\in \{\m_0, \m^{\mu\nu}_\pm\}$ be one of the multiplies defined in \eqref{multiplier}, \eqref{eqn: bq multipliers}. Then  for $N\geq 1$ there holds $S_\eta^N\m\in\bigcup_{i=0}^{N}E_{i}^{1-i}\cup E_{1+i}^{-i}$ and $W_\xi^N\m \in \bigcup_{i=0}^{N} \abseta \tilde{E}_i^{-i}\cup \tilde{E}_i^{1-i}$. Moreover, the following bounds hold
    \begin{align*}
        &\sabs{S_\eta^N\m}\chi\lesssim 2^{k_{\max}}[1+2^{k_2-k_1}2^{p_{\max}}]^N\chi,
        &\sabs{W_{\xi}^N\m}\chi\lesssim 2^{k_{\max}}[1+2^{k-k_1}2^{p_{\max}}]^N\chi.
    \end{align*}
    Analogous bounds hold true when localizing in $\tchi$.
\end{lemma} 

\begin{proof}
   By Lemma \ref{lemma: multiplier bound}, there holds $\m \in E_0^1 \cup E_1^0$ and thus,  by Lemma \ref{lemma: S on basic multipliers}, we obtain by repeatedly applying $S_\eta$
\begin{align*}
    \abs{S_\eta^N\m\chi}=2^{k_{\max}}(1+2^{k_2-k_1}(1+2^{p_2-p_1}))^N\chi.
\end{align*}
The second claim follows similarly from Lemma \ref{lemma: W_xi on elementary multipliers} by noting that $\m \in \abseta E_0\cup \tilde{E}_0^1$.
\end{proof}

\subsection{Vector fields and the phases}\label{sec: vector fields and the phases}
Next we discuss how the vector field $S$ acts on the phase $\Phi^{\pm}_{\mu\nu}$. Recall from \eqref{eqn: bq phases} the definition
\begin{align}\label{def: Phase phi as Lambda}
\Phi^{\mu\nu}_\pm(\xi,\eta)=\pm\Lambda(\xi)-\mu\Lambda(\xi-\eta)-\nu\Lambda(\eta),
\end{align}
and note that by \eqref{eqn: Seta Lambda} we have that
\begin{equation}
    S_\eta \Phi^{\mu\nu}_\pm=\mu S_\eta\Lambda(\xi-\eta).
\end{equation}
Without loss of generality we will thus only consider $\Phi(\xi,\eta):=\Phi_{+}^{++}(\xi,\eta)$. 
\begin{lemma} \label{lemma: Seta Phi}
Let $\sigma$ as in \eqref{eqn: symmetry of sigma} and $\chi$ as in \eqref{eqn: chi localizations}. Then $S_\eta\Phi\in E_{1}^{-1}$ and $S_{\xi-\eta}\Phi\in E_{-1}^{1}$ on the support of $\chi$ there holds that
    \begin{align*}
           & S_\eta \Phi(\xi,\eta)= \frac{\xi_2-\eta_2}{\absxieta^3}\sigma(\xi,\eta), && S_{\xi-\eta}\Phi(\xi,\eta)=\frac{\eta_2}{\abseta^3}\sigma(\xi,\eta),
    \end{align*}
    and 
    \begin{align*}
        \abs{S_\eta\Phi}\chi\sim 2^{-2k_1}2^{p_1}\abs{\sigma(\xi,\eta)}\chi, \hspace{1cm} \abs{S_{\xi-\eta}\Phi}\chi\sim 2^{-2k_2}2^{p_2}\abs{\sigma(\xi,\eta)}\chi.
    \end{align*}
   The analogous bounds hold for the $\tchi$ localizations.
\end{lemma}
\begin{proof}
    With \eqref{def: Phase phi as Lambda}, the definition \eqref{eqn: symmetry of sigma} of $\sigma$ and \eqref{eqn: Seta Lambda} there holds that
    \[S_\eta\Phi=-S_\eta \Lambda(\eta)-S_\eta\Lambda(\xi-\eta)=\frac{\xi_2-\eta_2}{\absxieta^3} \sigma(\xi,\eta). \]
    Therefore, on the support of $\chi$ we have
    \[\abs{S_\eta\Phi}\chi\sim 2^{-2k_1}2^{p_1}\abs{\sigma(\xi,\eta)}\chi.\]
    Similarly by symmetry of $S$ and $\sigma$, we have
    $S_{\xi-\eta}\Phi=-S_{\xi-\eta}\Lambda(\eta)=-\frac{\eta_2}{\abseta^3}\sigma(\xi,\eta)$
    and the size estimate follows directly.
\end{proof}
Together with \eqref{eqn: symmetry of sigma}, this lemma provides an important insight: whenever on the support of $\chi$ we have a size gap between any pair of the parameters $p,p_j$, $j=1,2$, we have a lower bound for $\sigma$ and thus for $S_\eta\Phi$ resp.\ $S_{\xi-\eta}\Phi$.

However, this can be further refined when taking also the size of the phase function (with respect to the localizations in $\Lambda$) into account. The following proposition is a key ingredient of our analysis and tells us roughly speaking that either we have a lower bound on $\abs{\sigma}$ (and thus integrate by parts along $S$ with Lemma \ref{lemma: ibp in bilinear expressions}) or the phase is large (and one can employ normal forms as in Section \ref{sec: Normal forms}).
\begin{proposition}\label{prop: lower bound on sigma} 
    Let $\Phi\in \set{\Phi_{\pm}^{\mu\nu}}{\mu,\nu\in\{+,-\}}$. Then either $\abs{\Phi}\tchi\geq 2^{q_{\max}-10}$ or $\sigma$ satisfies the lower bound
    \begin{align*}
    &  \abs{\sigma}\tchi\gtrsim 2^{k_{\min}+k_{\max}}2^{p_{\max}+q_{\max}}.
    \end{align*}
\end{proposition}
\begin{proof} Let $q_\alpha=\max\{q,q_1,q_2\}$ and $q_\beta=\min\{q,q_1,q_2\}$, and denote correspondingly $p_{{\max}}=p_{\beta}$ and $p_{{\min}}=p_{\alpha}$. Moreover let
\begin{align*}
    \Lambda_\alpha=\max\set{ \abs{\Lambda(\zeta)}}{\zeta\in \{\xi,\xi-\eta,\eta\}}, \qquad \Lambda_\beta=\min\set{ \abs{\Lambda(\zeta)}}{\zeta\in \{\xi,\xi-\eta,\eta\}}.
\end{align*} Assume that $\abs{\Phi}\tchi<2^{q_\alpha-10}$. 
\par    \textbf{Case A:} Assume we have a gap in $p$, i.e.\ $\abs{p_{\alpha}-p_{\beta}}> 10$, then there holds $2^{p_{\alpha}}<2^{-10}$. Moreover, since $ \Lambda^2_\alpha+1-\Lambda^2_\alpha=1$ implies $ 2^{2q_\alpha}\gtrsim 1-2^{2p_\alpha}\gtrsim 1-2^{-20}$, it follows that $2^{q_{\max}}=2^{q_\alpha}\sim 1$.
Then it follows directly from \eqref{eqn: symmetry of sigma}
    \begin{align*}
        \abs{\sigma}\tchi&\gtrsim 2^{k_\alpha+k_\beta}\sabs{\Lambda_\alpha\sqrt{1-\Lambda_\beta^2}-\sqrt{1-\Lambda_\alpha^2}\Lambda_\beta}  \gtrsim 2^{k_\alpha+k_\beta} 2^{p_{\beta}}\sim2^{k_{\max}+k_{\min}}2^{p_{\max}}.
    \end{align*}
\par    \textbf{Case B:} $\abs{p_{\alpha}-p_{\beta}}\leq 10$. \\
    We claim that $\frac{\Lambda_\beta}{\Lambda_\alpha}<\frac{2}{3}$. Otherwise, if $\Lambda_\alpha\leq \frac{3}{2}\Lambda_\beta$, there holds
    \[\abs{\Phi}\tchi=\abs{\Lambda(\xi)\pm\Lambda(\xi-\eta)\pm\Lambda(\eta)}\geq \frac{1}{3}\Lambda_\alpha\geq 2^{q_\alpha-2},\]
    which contradicts the assumption $\abs{\Phi}\tchi\leq 2^{q_{\alpha}-10}$. Hence there holds $\frac{\Lambda_\beta}{\Lambda_\alpha}<\frac{2}{3}$, which implies
    \[\Lambda_\alpha \sqrt{1-\Lambda_\beta^2}>\frac{3}{2}\Lambda_\beta\sqrt{1-\Lambda_\beta^2}>\frac{3}{2}\Lambda_\beta\sqrt{1-\Lambda_\alpha^2}. \]
    From this it follows with \eqref{eqn: symmetry of sigma} that
    \begin{align*}
        \abs{\sigma}\tchi&\gtrsim 2^{k_\alpha+k_\beta}\abs{\Lambda_\alpha\sqrt{1-\Lambda_\beta^2}-\sqrt{1-\Lambda_\alpha^2}\Lambda_\beta}\gtrsim 2^{k_\alpha+k_\beta}\Lambda_\alpha\sqrt{1-\Lambda_\beta^2}\gtrsim 2^{k_{\min}+k_{\max}}2^{p_{\max}} 2^{q_{\max}}.
    \end{align*}
\end{proof}

We also record some basic bounds for the action of $W$ on the phases $\Phi$: 
\begin{lemma}\label{lemma: W_xi Phi}
    For the vector field $W_\xi=\xi^\perp\cdot\nabla_\xi$ and a phase $\Phi\in\set{\Phi_\pm^{\mu\nu}}{\mu,\nu\in \{+,-\}}$ as in \eqref{eqn: bq phases} there holds
    \begin{align*}
        &W_\xi\Phi_{\pm}^{\mu\nu}=\mp\frac{\xi_2}{\absxi}-\mu\frac{\xi_2-\eta_2}{\absxieta^3}\xi\cdot(\xi-\eta)=\mp\sqrt{1-\Lambda^2(\xi)}-\mu\sqrt{1-\Lambda^2(\xi-\eta)}\frac{\xi\cdot(\xi-\eta)}{\absxi\absxieta}\frac{\absxi}{\absxieta},\\
        &\abs{W_\xi\Phi_\pm^{\mu\nu}}\chi\lesssim  2^p+2^{p_1}2^{k-k_1}.
    \end{align*}
\end{lemma}
\begin{proof}
    We compute with $\mu,\nu\in\{+,-\}$
    \begin{align*}
        W_\xi \Phi_\pm^{\mu\nu}&=\pm W_\xi \Lambda(\xi)-\mu W_\xi\Lambda(\xi-\eta)=\mp\frac{\xi_2}{\absxi}-\mu \frac{(\xi_2-\eta_2)}{\absxieta}\frac{(\xi-\eta)\cdot\xi}{\absxieta\absxi}\frac{\absxi}{\absxieta}.
    \end{align*}
    Then the bound follows
    \begin{align*}
        \abs{W_\xi\Phi_\pm^{\mu\nu}}\chi&\lesssim 2^p+2^{p_1}2^{k-k_1}.
    \end{align*}
\end{proof}

\subsection{Integration by parts in bilinear expressions}\label{subsection: Integration by parts in bilinear expressions}
The main goal of this section is to establish bounds for repeated integration by parts along $S$ in bilinear terms of the form
\begin{align}\label{eqn: bilinear form Qm}
    \F (\mathcal{Q}_{\m}(f_1,f_2))(t,\xi)=\int_{\R^2}e^{it\Phi(\xi,\eta)}\m(\xi,\eta)\widehat{f_1}(t,\xi-\eta)\widehat{f_2}(t,\eta) d\eta,
\end{align}
where $\m\in\{\m_0,\m_\pm^{\mu\nu}; \mu,\nu\in\{+,-\}\}$ is one of the multipliers and $\Phi\in \{\Phi_\pm^{\mu\nu};\mu,\nu\in\{+,-\}\}$ is one of the phases associated with the Boussinesq resp.\ SQG equations, and $f_1,f_2$ are corresponding profiles.

\subsubsection{Integration by parts along $S$.}
We present next the main lemma for iterated integration by parts along $S_\eta$ and $S_{\xi-\eta}$. Let $f\in L^2$ and $N\in\N$, then 
     \begin{align*}
         \ltwonorm{(1,S)^Nf}:=\sum_{i=0}^N\ltwonorm{S^if}.
     \end{align*}

\begin{lemma} \label{lemma: ibp in bilinear expressions}
\begin{enumerate}
    \item\label{it1: lemma: ibp in bilinear expressions} Assume that $\abs{\sigma}\chi\gtrsim L\gtrsim 2^{k_{\max}+k_{\min}+p_{\max}}$. Then for $N\in \N$ there holds:
\begin{align*}
    \inftynorm{\F(\mathcal{Q}_{\m\chi}(R_{l_1}f_1,R_{l_2}f_2))}&\lesssim 2^{k_{\max}}[t^{-1}2^{2k_1}2^{-p_1}L^{-1}(1+2^{k_2-k_1}2^{l_1})]^N\\
    &\quad\cdot\ltwonorm{P_{k_1,p_1}R_{l_1}(1,S)^Nf_1}\ltwonorm{P_{k_2,p_2}R_{l_2}(1,S)^Nf_2}, \\
     \inftynorm{\F(\mathcal{Q}_{\m\chi}(R_{l_1}f_1,R_{l_2}f_2))}&\lesssim2^{k_{\max}} [t^{-1}2^{2k_2}2^{-p_2}L^{-1}(1+2^{k_1-k_2}2^{l_2})]^N\\
    &\quad\cdot\ltwonorm{P_{k_1,p_1}R_{l_1}(1,S)^Nf_1}\ltwonorm{P_{k_2,p_2}R_{l_2}(1,S)^Nf_2}.
\end{align*}
\item\label{it2: lemma: ibp in bilinear expressions}  Assume that $\abs{\sigma}\tchi\gtrsim \tilde{L}\gtrsim2^{k_{\max}+k_{\min}+p_{\max}+q_{\max}}$. Then for 
 $N\in \N$ there holds:
\begin{align*}
    \inftynorm{\F(\mathcal{Q}_{\m\tchi}(R_{l_1}f_1,R_{l_2}f_2))}&\lesssim 2^{k_{\max}}[t^{-1}(2^{k_2-k_1-p_1-q_1}+2^{2k_1}2^{-p_1}\tl^{-1}(1+2^{k_2-k_1}(2^{q_2-q_1}+2^{l_1})))]^N\\
    &\quad\cdot \ltwonorm{P_{k_1,p_1,q_1}R_{l_1}(1,S)^Nf_1}\ltwonorm{P_{k_2,p_2,q_2}R_{l_2}(1,S)^Nf_2}, \\
     \inftynorm{\F(\mathcal{Q}_{\m\tchi}(R_{l_1}f_1,R_{l_2}f_2))}&\lesssim 2^{k_{\max}}[t^{-1}(2^{k_1-k_2-p_2-q_2}+2^{2k_2}2^{-p_2}\tl^{-1}(1+2^{k_1-k_2}(2^{q_1-q_2}+2^{l_2})))]^N\\
    &\quad\cdot \ltwonorm{P_{k_1,p_1,q_1}R_{l_1}(1,S)^Nf_1}\ltwonorm{P_{k_2,p_2,q_2}R_{l_2}(1,S)^Nf_2}.
\end{align*}
\end{enumerate}
\end{lemma}
\begin{proof} We start by proving the first bound in \eqref{it1: lemma: ibp in bilinear expressions}, noting that the second one follows by symmetry and the analogous bounds for $S_{\xi-\eta}$. Let $F=R_{l_1}f_1$ and $G=R_{l_2}f_2.$ With $e^{it\Phi}=S_\eta e^{it\Phi}\frac{1}{itS_\eta\Phi}$ and Lemma \ref{lemma: S_eta= S_xi-eta W_xi-eta}, integrating by parts once in $S_\eta$ yields:
    \begin{equation}\label{eqn: decompostion Q1 Q2 Q3}
    \begin{split}
       &\F(\mathcal{Q}_{\m\chi}(F,G))\\    &=\int_{\R^2}e^{it\Phi(\xi,\eta)}\m(\xi,\eta)\chi(\xi,\eta)\widehat{F}(\xi-\eta)\widehat{G}(\eta) d\eta\\
        &=it^{-1}\int_{\R^2}e^{it\Phi}S_\eta\Big[\frac{1}{S_\eta\Phi}\m(\xi,\eta)\chi(\xi,\eta)\widehat{F}(\xi-\eta)\widehat{G}(\eta)\Big]d\eta\\
        &=it^{-1}\Big(\int_{\R^2}e^{it\Phi}S_{\eta}\Big(\frac{\m\chi}{\Setaphi}\Big)\widehat{F}(\xi-\eta)\widehat{G}(\eta)d\eta+\int_{\R^2}e^{it\Phi}\frac{\m\chi}{\Setaphi}\widehat{F}(\xi-\eta)S_\eta\widehat{G}(\eta)d\eta\\
        &\hspace{1cm}+\int_{\R^2}e^{it\Phi}\frac{\m\chi}{\Setaphi}\left(\frac{\eta(\xi-\eta)}{\absxieta^2}S_{\xi-\eta}\widehat{F}(\xi-\eta)-\frac{\eta(\xi-\eta)^\perp}{\absxieta^2}W_{\xi-\eta}\widehat{F}(\xi-\eta)\right)\widehat{G}(\eta)d\eta\Big)\\
        &=it^{-1}\left(\mathcal{Q}_{S_\eta(\m\chi(S_\eta\Phi)^{-1})}(F,G)+\mathcal{Q}_{\m_1\chi (S_\eta \Phi)^{-1}}(SF,G)+\mathcal{Q}_{\m_2\chi (S_\eta \Phi)^{-1}}(WF,G)+\mathcal{Q}_{\m\chi(S_\eta\Phi)^{-1}}(F,SG)\right),
    \end{split}
    \end{equation}
    where $\m_1,\m_2\in E_2^{-1}\cup E_1^0$. We demonstrate the first bound in (1) for $N=1$. Observe that since $\abs{\sigma}\chi\gtrsim L$, by Lemma \ref{lemma: Seta Phi} there holds 
    \begin{align*}
        {\abs{\Setaphi}}^{-1}\chi\lesssim 2^{2k_1}2^{-p_1}L^{-1}.
    \end{align*}
    With $\m\in E_0^1\cup E_1^0$ by Lemma \ref{lemma: multiplier bound} there holds
    \begin{align*}
    \abs{\mathcal{Q}_{\m\chi(S_\eta\Phi)^{-1}}(F,SG)}&\lesssim\int_{\R^2}\abs{\frac{\m\chi}{\Setaphi}\widehat{F}(\xi-\eta)S_\eta\widehat{G}(\eta)}d\eta\\
        &\lesssim 2^{2k_1}2^{-p_1}L^{-1}\int_{\R^2}\sabs{\m\chi(\xi,\eta)\widehat{F}(\xi-\eta)S_\eta\widehat{G}(\eta)}d\eta\\
        &\lesssim2^{2k_1}2^{-p_1}L^{-1}\norm{\m\chi}_{L^\infty}\int_{\R^2}\chi(\xi,\eta)\sabs{\widehat{F}(\xi-\eta)S_\eta\widehat{G}(\eta)}d\eta,
    \end{align*}
    which leads to
    \begin{align*}
        \inftynorm{\mathcal{Q}_{\m\chi(S_\eta\Phi)^{-1}}}\lesssim 2^{2k_1}2^{-p_1}L^{-1}\norm{\m\chi}_{L^\infty}\ltwonorm{P_{k_1,p_1}R_{l_1}f_1}\ltwonorm{P_{k_2,p_2}R_{l_2}Sf_2}.
    \end{align*}
    For the second term on the right-hand side of \eqref{eqn: decompostion Q1 Q2 Q3}, we note that with \begin{align*}
      \m_1=\m\frac{\eta(\xi-\eta)}{\abseta\absxieta}\frac{\abseta}{\absxieta}, \quad \m_2=-\m\frac{\eta(\xi-\eta)^\perp}{\abseta\absxieta}\frac{\abseta}{\absxieta},\quad \m_1,\m_2\in E_2^{-1}\cup E_1^0,  
    \end{align*}
    there holds
    \begin{align*}
      &\abs{\mathcal{Q}_{\m_1\chi (S_\eta \Phi)^{-1}}(SF,G)+\mathcal{Q}_{\m_2\chi (S_\eta \Phi)^{-1}}(WF,G)}\\
      &\quad\lesssim \int_{\R^2}\abs{\frac{1}{\Setaphi}\left[\m_1\chi S_{\xi-\eta}\widehat{F}(\xi-\eta)+\m_2\chi W_{\xi-\eta}\widehat{F}(\xi-\eta)\right]\widehat{G}(\eta)}d\eta\\
        &\quad\lesssim 2^{2k_1}2^{-p_1}L^{-1}2^{k_2-k_1}\int_{\R^2}\abs{\m\chi(\xi,\eta)\left[S_{\xi-\eta}\widehat{F}(\xi-\eta)+W_{\xi-\eta}\widehat{F}(\xi-\eta)\right]\widehat{G}(\eta) }d\eta.
    \end{align*}
 Altogether, using Lemma \ref{prop: angular localization properties}\eqref{it3: prop: angular localization properties} we have:
 \begin{align*}
     &\inftynorm{\mathcal{Q}_{\m_1\chi (S_\eta \Phi)^{-1}}(SF,G)}+\inftynorm{\mathcal{Q}_{\m_2\chi (S_\eta \Phi)^{-1}}(WF,G)}\\
     &\quad \lesssim 
      2^{2k_1}2^{-p_1}L^{-1}2^{k_2-k_1}\inftynorm{\m\chi}(\ltwonorm{P_{k_1,p_1}R_{l_1}Sf_1}+2^{l_1}\ltwonorm{P_{k_1,p_1}R_{l_1}f_1})\ltwonorm{P_{k_2,p_2}R_{l_2}f_2}.
 \end{align*}
 Finally we estimate the first term on the right-hand side of \eqref{eqn: decompostion Q1 Q2 Q3}, which can also be broken down in two parts:
 \begin{equation}\label{eqn: Q1 ibp}
 \begin{split}
      \mathcal{Q}_{S_\eta(\m\chi(S_\eta\Phi)^{-1})}(F,G)&=\int_{\R^2}e^{it\Phi}\frac{S_\eta(\m\chi)}{\Setaphi}\widehat{F}(\xi-\eta)\widehat{G}(\eta)d\eta-
     \int_{\R^2}e^{it\Phi}\frac{\m\chi S^{2}_\eta\Phi}{(\Setaphi)^2}\widehat{F}(\xi-\eta)\widehat{G}(\eta)d\eta\\
     &=:\mathcal{Q}_{1}(\xi,t)+\mathcal{Q}_{2}(\xi,t).
 \end{split}
 \end{equation}
 For the second term on the right-hand side above we obtain using Lemma \ref{lemma: S on basic multipliers} on $\Phi\in E_0$:
 \begin{align*}
\abs{\mathcal{Q}_{2}(\xi,t)}&\lesssim\int_{\R^2}\abs{\frac{\m\chi S^{2}_\eta\Phi}{(\Setaphi)^2}\widehat{F}(\xi-\eta)\widehat{G}(\eta)}d\eta\\&
     \lesssim 2^{k_{\max}}(1+2^{k_2-k_1}2^{p_{\max}})2^{2k_1}2^{-p_1}L^{-1}\int_{\R^2}\sabs{\chi(\xi,\eta)\widehat{F}(\xi-\eta)\widehat{G}(\eta)}d\eta.
 \end{align*}
 Now we handle $\mathcal{Q}_{1}(\xi,t)$.
     Recalling the definition \eqref{eqn: chi localizations} of $\chi$, we have
     \begin{align*}
         S_\eta \chi(\xi,\eta)=\varphi_{k,p}(\xi)[S_\eta(\varphi_{k_1,p_1}(\xi-\eta))\varphi_{k_2,p_2}(\eta)+\varphi_{k_1,p_1}(\xi-\eta)S_\eta(\varphi_{k_2,p_2}(\eta))].
     \end{align*}
     Using Lemma \ref{lemma: S on basic multipliers} we find
     \begin{align*}
         S_\eta(\varphi_{k_1,p_1}(\xi-\eta))&=-2^{-k_1}\frac{\eta(\xi-\eta)}{\absxieta}\overline{\varphi}^1_{k_1}(\xi-\eta)\varphi_{p_1}(\xi-\eta)\\
         &\quad+2^{-p_1}\frac{\eta(\xi-\eta)^\perp}{\absxieta^2}\Lambda(\xi-\eta){\varphi}_{k_1}(\xi-\eta)\overline{\varphi}^2_{p_1}(\xi-\eta),
     \end{align*}
     where $\overline{\varphi}^{i}$, $i=1,2$ are functions with similar support properties as $\varphi$ (see also Remark \ref{rk: notation similar supp properties}).
     By abusing the notation slightly, we obtain similarly that $ S_\eta\varphi_{k_2,p_2}(\eta)=2^{-k_2}\abseta\overline{\varphi}_{k_2,p_2}(\eta)$.
     Altogether this gives
     \begin{align*}
       \abs{S_\eta\chi}\lesssim (1+2^{k_2-k_1}(1+2^{-p_1}))\chi.
     \end{align*}
     This, together with Lemma \ref{lemma: S_eta on multiplier m}  implies the following bound on $\mathcal{Q}_{1}^{1}(\xi,t)$:
     \begin{align*}
         \abs{\mathcal{Q}_{1}(\xi,t)}&\lesssim \int_{\R^2} \sabs{S_\eta\Phi}^{-1}\sabs{(S_\eta\m \chi+\m S_\eta\chi)\widehat{F}(\xi-\eta)\widehat{G}(\eta)}d\eta\\
         &\lesssim 2^{k_{\max}}2^{2k_1}2^{-p_1}L^{-1}(1+2^{k_2-k_1}2^{-p_1})\int_{\R^2}\sabs{\chi\widehat{F}(\xi-\eta)\widehat{G}(\eta) }d\eta.
     \end{align*}    
   Hence with \eqref{eqn: Q1 ibp}, $\mathcal{Q}_{S_\eta(\m\chi(S_\eta\Phi)^{-1})}(F,G)$ satisfies the following bound
   \begin{align*}
       \inftynorm{\mathcal{Q}_{S_\eta(\m\chi(S_\eta\Phi)^{-1})}(F,G)}\lesssim 2^{k_{\max}}2^{2k_1}2^{-p_1}L^{-1}(1+2^{k_2-k_1}2^{-p_1})\ltwonorm{P_{k_1,p_1}R_{l_1}f_1}\ltwonorm{P_{k_2,p_2}R_{l_2}f_2}.
   \end{align*}
     Finally, since $l_1+p_1\geq 0$ and $l_1\geq 0$, and $\m\in E_0^1\cup E_1^0$ we obtain from \eqref{eqn: decompostion Q1 Q2 Q3}:
     \begin{align*}
         \inftynorm{\F(\mathcal{Q}_{\m\chi}(F,G))}& \lesssim 2^{k_{\max}}t^{-1}2^{2k_1-p_1}L^{-1}[1+2^{k_2-k_1}(1+2^{l_1})+2^{k_2-k_1}2^{-p_1}]\\
         &\quad\cdot\ltwonorm{P_{k_1,p_1}(1,S)F}\ltwonorm{P_{k_2,p_2}(1,S)G}\\
         &\lesssim 2^{k_{\max}}t^{-1}2^{2k_1-p_1}L^{-1}(1+2^{k_2-k_1}2^{l_1})\ltwonorm{P_{k_1,p_1}(1,S)F}\ltwonorm{P_{k_2,p_2}(1,S)G}.
     \end{align*}

   For $N\geq2$ we proceed iteratively from \eqref{eqn: decompostion Q1 Q2 Q3}, where we observe that the multipliers obtained due to  integration by parts are in the admissible classes defined in Section \ref{sec: multiplier mechanics}. Thus, Lemmas \ref{lemma: S on basic multipliers}, \ref{lemma: S_eta on multiplier m} and  \ref{lemma: Seta Phi} can be applied iteratively.
   \par As for the claim \eqref{it2: lemma: ibp in bilinear expressions}, the proof follows similarly using the bounds with the $\tchi$ localizations in Lemmas \ref{lemma: S on basic multipliers} and \ref{lemma: S_eta on multiplier m}. The only difference arises when the vector field $S$ falls on $\tchi(\xi,\eta)$ (see the term $\mathcal{Q}_1$ in \eqref{eqn: Q1 ibp} for the first iteration). Here, using the fact that $S_\eta\Lambda(\eta)=0$ and $S_\eta\Lambda(\xi-\eta)=-S_\eta\Phi$ we obtain:
   \begin{align*}
  \frac{ {S_\eta\tchi}}{{S_\eta\Phi}}=\frac{1}{S_\eta\Phi} \left(2^{-k_1}\frac{\eta(\xi-\eta)}{\abseta\absxieta}\abseta+2^{-p_1}\Lambda(\xi-\eta)\frac{(\xi-\eta)\eta^\perp}{\absxieta\abseta}\frac{\abseta}{\absxieta}+2^{-k_2}\abseta\right){\overline{\tchi}^1}+2^{-q_1} {\overline{\tchi}}^2,
   \end{align*}
   where $\overline{\tchi}^1,\overline{\tchi}^2$ have similar support properties as $\tchi$. The arising multipliers are again in the admissible class defined in Section \ref{sec: multiplier mechanics} and the iteration follows as above.
\end{proof}
\subsubsection{Integration by parts along $D$.}\label{sec: ibp along D} 
We also present a result on a zero-homogeneous horizontal derivative that will be useful in the proof of Proposition \ref{prop: X-norm bounds for l<(1+delta)m}, see \textbf{Case B.2(b)}.
Define 
\begin{align*}
    &D_{\eta}:=\abseta \partial_{\eta_1}=\Lambda(\eta)S_\eta-\sqrt{1-\Lambda^2(\eta)} W_\eta, && D_{\xi-\eta}:=\absxieta\partial_{\eta_1}.
\end{align*}
\begin{lemma} \label{lemma: ibp in D_eta}
     Assume that $\abs{D_\eta\Phi}\chi\gtrsim L$. Then for $N\in \N$ there holds:
\begin{align*}
    \inftynorm{\F(\mathcal{Q}_{\m\chi}(R_{l_1}f_1,R_{l_2}f_2))}&\lesssim 2^{k_{\max}}[t^{-1}L^{-1}(2^{l_1+p_1}+2^{l_2+p_2})]^N\\
    &\quad\cdot\ltwonorm{P_{k_1,p_1}R_{l_1}(1,S)^Nf_1}\ltwonorm{P_{k_2,p_2}R_{l_2}(1,S)^Nf_2}. 
    \end{align*}
\end{lemma}
\begin{proof}
The proof follows the same scheme as the proof of Lemma \ref{lemma: ibp in bilinear expressions} with $e^{it\Phi}=(it)^{-1}\frac{D_\eta e^{it\Phi}}{D_\eta\Phi}$, hence we just record the necessary computations to proceed as above. There holds
\begin{align*}
    &D_\eta(\Lambda(\eta))=\frac{\eta_2^2}{\abseta^2}=1-\Lambda^2(\eta), \hspace{0.5cm} D_\eta(\sqrt{1-\Lambda^2(\eta)})=-\Lambda(\eta)\sqrt{1-\Lambda^2(\eta)},\\
    & D_\eta(\Lambda(\xi-\eta))=-\frac{\abseta}{\absxieta}(1-\Lambda^2(\xi-\eta)), \hspace{0.4cm} D_\eta(\sqrt{1-\Lambda^2(\xi-\eta)})=\frac{\abseta}{\absxieta}\Lambda(\xi-\eta)\sqrt{1-\Lambda^2(\xi-\eta)},\\
    &D_\eta\abseta=\abseta\Lambda(\eta), \hspace{0.5cm} D_\eta\absxieta=-\abseta\Lambda(\xi-\eta), \hspace{0.5cm} D_\eta=\frac{\abseta}{\absxieta}D_{\xi-\eta}.
\end{align*}
With these computations and $\overline{\varphi}^1,\overline{\varphi}^2$ functions with similar support properties as $\varphi$, we have
\begin{align*}
&D_\eta\chi(\xi,\eta)\\
&\quad=\left(-2^{-k_1}\abseta\Lambda(\xi-\eta)+2^{-p_1}\frac{\abseta}{\absxieta}\Lambda(\xi-\eta)\sqrt{1-\Lambda^2(\xi-\eta)}\right)\varphi_{k,p,q}(\xi)\overline{\varphi}^1_{k_1,p_1}(\xi-\eta)\varphi_{k_2,p_2}(\eta)\\
&\qquad+\left(2^{-k_2}\abseta\Lambda(\eta)-2^{-p_2}\Lambda(\eta)\sqrt{1-\Lambda^2(\eta)}\right)\varphi_{k,p}(\xi)\varphi_{k_1,p_1}(\xi-\eta)\overline{\varphi}^2_{k_2,p_2}(\eta).
\end{align*}
Together with the Bernstein property Proposition \ref{prop: angular localization properties}\eqref{it3: prop: angular localization properties}, this implies that:
\begin{align*}
    &D_\eta\F(P_{k_1,p_1}R_{l_1}F_1)(\xi-\eta)\\
    &\hspace{9em}\sim 2^{k_2-k_1}[\Lambda(\xi-\eta)\F(P_{k_1,p_1}R_{l_1}(1,S)F_1)(\xi-\eta)+2^{l_1+p_1}\F(P_{k_1,p_1}R_{l_1}F_1)(\xi-\eta)],\\
    &D_\eta \F(P_{k_2,p_2}R_{l_2}F_2)(\eta)\sim \Lambda(\eta)\F(P_{k_2,p_2}R_{l_2}(1,S)F_2)(\eta)+2^{l_2+p_2}\F(P_{k_2,p_2}R_{l_2}F_2)(\eta).
\end{align*}
Moreover, in order to control $D_\eta^M\Phi$ we compute
\begin{align*}
   \abs{D_\eta^M\Lambda(\xi)}=0,\hspace{0.5cm}  \abs{D_\eta^M\Lambda(\eta)}\lesssim 1-\Lambda^2(\eta),\hspace{0.5cm}  \abs{D_\eta^M\Lambda(\xi-\eta)}\lesssim \frac{\abseta^M}{\absxieta^M}(1-\Lambda^2(\xi-\eta)).
\end{align*}
\end{proof}

\subsubsection{Towards finite speed of propagation}\label{sec: finite speed of propagation}
In the proof of Proposition \ref{prop: X-norm bounds for l>(1+delta)m}, where we bound the $X$-norm in the case that the parameter $l$ is large, we also need to understand how the vector field $W_\xi$ acts on bilinear expressions \eqref{eqn: bilinear form Qm}.
\begin{lemma}\label{lemma: W_xi on bilinear expressions} Let $\mathcal{Q}_\m$ be a bilinear expression as in \eqref{eqn: bilinear form Qm}. Then for $N\in \N$ there holds:
        \begin{align*}
                \ltwonorm{R_l \mathcal{Q}_\m(f_1,f_2)}&\lesssim 2^{k_{\min}}2^{k+p_{\max}}2^{-Nl}[t(2^p+2^{k-k_1}2^{p_1})+2^{-p}+2^{k-k_1+l_1}]^N\ltwonorm{(1,S)^2f_1}\ltwonorm{f_2}\\
         &\quad+2^{k_{\min}}2^{k+p_{\max}}2^{-3l}\ltwonorm{S^3f_1}\ltwonorm{f_2}.
        \end{align*}
        The analogous bound holds with the roles of $f_1$ and $f_2$ (and their respective localizations) interchanged.
    \end{lemma}

\begin{proof} The core of the proof is the Bernstein property for the vector field $W$ from Proposition \ref{prop: angular localization properties}\eqref{it3: prop: angular localization properties}:
    \begin{align*}
        \ltwonorm{R_l \mathcal{Q}_\m(f_1,f_2)}\lesssim 2^{-l}\ltwonorm{W_{\xi}\mathcal{Q}_{\m}(f_1,f_2)}.
    \end{align*}
    By changing variables we can assume w.l.o.g.\ that $k_2\leq k_1$.
    We begin by proving that
     \begin{align}\label{eqn: W_xi first iteration}
        W_\xi R_l \F(\mathcal{Q}_{\m}(f_1,f_2))=\F(\mathcal{Q}_{\m_1^{(1)}}(f_1,f_2))-\F(\mathcal{Q}_{\m_2^{(1)}}(Sf_1,f_2))+\F(\mathcal{Q}_{\m_3^{(1)}}(Wf_1,f_2)),
    \end{align}
    where $\m_i^{(1)} \in \abseta \tilde{E}_a^b\cup\tilde{E}_c^d$, for some $a,b,c,d\in \Z$ for $i=1,2,3$. 
    We compute using Lemma \ref{lemma: S_eta= S_xi-eta W_xi-eta}:
    \begin{align*}
        W_\xi \F(\mathcal{Q}_{\m}(f_1,f_2))&=\int_{\R^2} W_\xi(e^{-it\Phi}\m)\widehat{f_1}(\xi-\eta)\widehat{f_2}(\eta)d\eta +\int_{\R^2} e^{-it\Phi}\m W_\xi \widehat{f_1}(\xi-\eta)\widehat{f_2}(\eta)d\eta\\
        &=\int_{\R^2}W_\xi(e^{-it\Phi}\m)\widehat{f_1}(\xi-\eta)\widehat{f_2}(\eta)d\eta +\int_{\R^2} e^{-it\Phi}\m \frac{(\xi-\eta)\xi}{\absxieta^2}W_{\xi-\eta}\widehat{ f_1}(\xi-\eta)\widehat{f_2}(\eta)d\eta\\
        &\quad-\int_{\R^2} e^{-it\Phi}\m \frac{(\xi-\eta)^\perp\xi}{\absxieta^2}S_{\xi-\eta} \widehat{f_1}(\xi-\eta)\widehat{f_2}(\eta)d\eta\\
        &=\F(\mathcal{Q}_{\m^{(1)}_1}(f_1,f_2))-\F(\mathcal{Q}_{\m^{(1)}_2}(Sf_1,f_2))+\F(\mathcal{Q}_{\m^{(1)}_3}(Wf_1,f_2)),
    \end{align*}
    where 
    \begin{align*}
        &\m_1^{(1)}=-itW_\xi\Phi \m+ W_\xi\m, &&\m_2^{(1)}=\frac{(\xi-\eta)^\perp\xi}{\absxieta^2}\m,
        &&\m_3^{(1)}=\frac{(\xi-\eta)\xi}{\absxieta^2}\m.
    \end{align*}
   Using Lemmas \ref{lemma: S_eta on multiplier m}, \ref{lemma: W_xi Phi}, the multipliers satisfy 
   
     \begin{equation}\label{eqn: m_is bounds Wxi ibp}
    \begin{split}
         \sinftynorm{\m_1^{(1)}\chi}&\lesssim 2^m(2^p+2^{k-k_1}2^{p_1})\abs{\m}+2^k(1+2^{k-k_1}(1+2^{p-p_1})\\
         &\lesssim (2^m(2^p+2^{k-k_1}2^{p_1})+2^{-p}+2^{k-k_1-p}(1+2^{p-p_1}))\inftynorm{\m\chi},\\
        \sinftynorm{\m_2^{(1)}\chi}&\lesssim 2^{k-k_1}\inftynorm{\m\chi},\\
        \sinftynorm{\m_3^{(1)}\chi}&\lesssim 2^{k-k_1}\inftynorm{\m\chi}.
    \end{split}
    \end{equation} Hence from \eqref{eqn: W_xi first iteration}, the following bound holds
    \begin{align*}
           \ltwonorm{W_\xi R_l\F(\mathcal{Q}_{\m}(f_1,f_2))}&\lesssim 2^{k_{\min}}2^{k+p_{\max}}[(t(2^p+2^{k-k_1}2^{p_1})+2^{-p}+ 2^{k-k_1}(2^{-p_1}+2^{l_1}))\ltwonorm{f_1}\ltwonorm{f_2} \\
        &\hspace{3.5cm}+\ltwonorm{Sf_1}\ltwonorm{f_2}],
    \end{align*}
       Iterating this process, we see that taking $W^j\mathcal{Q}_{\m}(f_1,f_2)$ generates $3^j$ bilinear expressions. Inductively it follows
    \begin{align*}
        W\mathcal{Q}_{\m^{(j)}}(f_1,f_2)= \mathcal{Q}_{\m^{(j+1)}_1}(f_1,f_2)+\mathcal{Q}_{\m^{(j+1)}_2}(Sf_1,f_2)+\mathcal{Q}_{\m^{(j+1)}_3}(Wf_1,f_2),
    \end{align*}
    where the multipliers are 
    \begin{align*}
        \m_1^{(j+1)}=-itW_\xi\Phi \m^{(j)}+W_\xi \m^{(j)}, && \m_2^{(j+1)}=\frac{(\xi-\eta)^\perp\xi}{\absxieta^2}\m^{(j)}, && \m_3^{(j+1)}=\frac{(\xi-\eta)\xi}{\absxieta^2}\m^{(j)}.
    \end{align*}
    Furthermore, with Lemmas \ref{lemma: S_eta on multiplier m}, \ref{lemma: W_xi Phi} and \ref{lemma: W_xi on bilinear expressions}, we see by induction that the multipliers satisfy the following bounds
    \begin{align*}
      &  \lVert{\m_1^{(j+1)}\chi}\rVert_{L^{\infty}}\lesssim (2^m(2^p+2^{k-k_1}2^{p_1})+2^{-p}+2^{k-k_1-p_1})^j2^{k+p_{\max}}, \\
      &  \lVert{\m_2^{(j+1)}\chi}\rVert_{L^{\infty}}\lesssim \lVert{\m^{(j)}\chi}\rVert_{L^{\infty}}, \qquad
        \lVert{\m_3^{(j+1)}\chi}\rVert_{L^{\infty}}\lesssim \lVert{\m^{(j)}\chi}\rVert_{L^{\infty}}.
    \end{align*}
    At each step we have a bound on the $L^2$ norm:
    \begin{equation}\label{eqn: L2 bound 2^-Kl}
      \begin{split}
           \ltwonorm{ W\mathcal{Q}_{\m^{(j)}}(f_1,f_2)}&\lesssim\sltwonorm{\mathcal{Q}_{\m^{(j+1)}_1}(f_1,f_2)}+\sltwonorm{\mathcal{Q}_{\m^{(j+1)}_2}(Sf_1,f_2)}+\sltwonorm{\mathcal{Q}_{\m^{(j+1)}_3}(Wf_1,f_2)}\\
        & \lesssim 2^{k+p_{\max}}(t(2^p+2^{k-k_1}2^{p_1})+2^{-p}+2^{k-k_1-p_1})^j\ltwonorm{f_1}\ltwonorm{f_2}\\
        &\quad+\lVert{\m^{(j)}\chi}\rVert_{L^{\infty}}\ltwonorm{Sf_1}\ltwonorm{f_2}+2^{l_1}\lVert{\m^{(j)}\chi}\rVert_{L^{\infty}}\ltwonorm{f_1}\ltwonorm{f_2}.
      \end{split} 
    \end{equation}
   Observe that when the vector field $W_\xi$ produces an $Sf_1$ term (multipliers of the type $\m_2^{(j+1)}$), we have no additional losses in $m,p,l_1$. Hence, for such terms we stop after three iterations, while for the rest we can continue the iteration as above. Altogether with the Bernstein property we obtain
    \begin{align*}
         \ltwonorm{R_l \mathcal{Q}_\m(f_1,f_2)}&\lesssim 2^{k_{\min}}2^{k+p_{\max}}2^{-Nl}[t(2^p+2^{k-k_1}2^{p_1})+2^{-p}+2^{k-k_1+l_1}]^N\ltwonorm{S^{\leq 2}f_1}\ltwonorm{f_2}\\
&\quad+2^{k_{\min}}2^{k+p_{\max}}2^{-3l}\ltwonorm{S^3f_1}\ltwonorm{f_2}.
    \end{align*}

\end{proof}  
\subsection{Case organisation and a reduction lemma}\label{sec: Case organization}
The following lemma gives an overview of the relation between different localisation parameters depending on their relative size to one another.

\begin{lemma}\label{lemma: case organisation}
    Assume $p_{\min}=p\leq p_{\max}-10$. Then on the support of $\chi$ the following configurations are possible:
    \begin{enumerate}
        \item $\abs{k_1-k_2}\leq 4$ then $p\leq \min\{p_1,p_2\}-3$ and $\abs{p_1-p_2}\leq 5$,
        \item $k_1<k_2-4$, then $\abs{k-k_2}\leq 2$ and $p_{\max}=p_1$; moreover there holds either $p\leq p_2-10\leq p_1-12$ and $p_2+k_2-2\leq p_1+k_1\leq p_2+k_2+2$, or $\abs{p-p_2}\leq 10$ and ${p_1+k_1}\leq p_2+k_2+3$,
        \item $k_2<k_1-4$, then $\abs{k-k_1}\leq 2$ and $p_{\max}=p_2$; moreover there holds either $p\leq p_1-10\leq p_2-12$ and $p_2+k_2-2\leq p_1+k_1\leq p_2+k_2+2$, or $\abs{p-p_1}\leq 2$ and $p_2+k_2\leq p_1+k_1+3$.
    \end{enumerate}
\end{lemma}
\begin{remark}
\begin{enumerate}
    \item The analogous result holds with the roles of $p$, $p_i$ for $i=1,2$ interchanged.
    \item The analogous result holds for the localization parameters $q,q_i$ for $i=1,2$ in the ``gap in $q$" case, that is when $q_{\min}\leq q_{\max}-10$.
    \item In the following we will use the notation $\ll$, $\sim$, $\lesssim$ that includes both multiplicative bounds on the dyadic scale $2^n$ and additive constants at the level of the parameter $n\in \Z$. For example $2^p\ll 2^{p_1}$ implies there exist constants $C,C_1>0$ such that $2^{p}\leq C_12^{p_1-C}$. Similarly, $2^{p}\sim 1$ (equivalently $p\sim 0$) implies $-C<p\leq 0$ for a constant $C\in \N$.
\end{enumerate}
\end{remark}
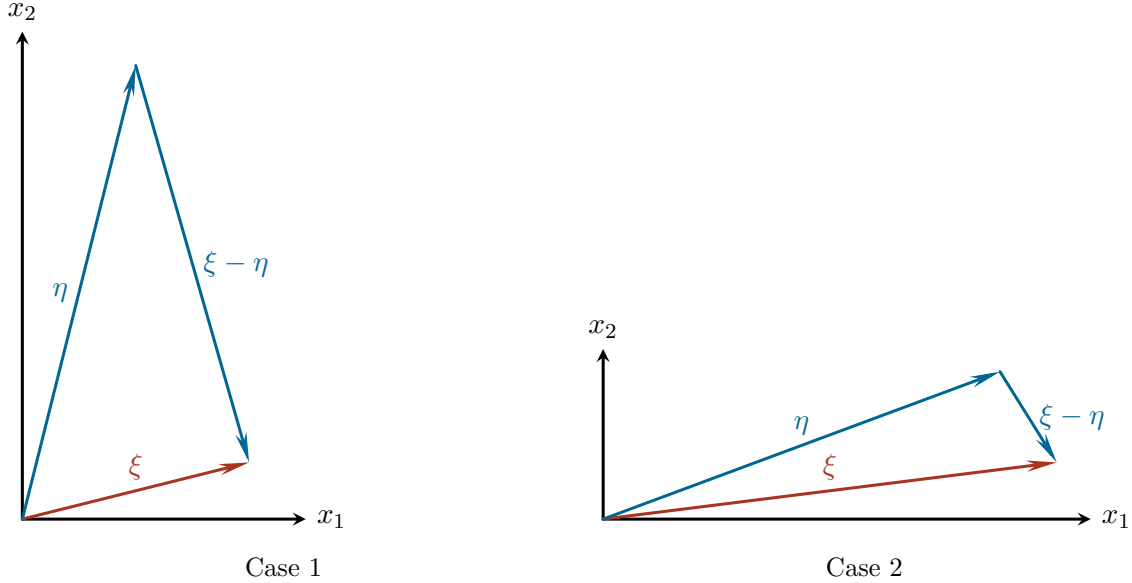
\begin{figure}[htbp]
    \centering
\begin{subfigure}[b]{0.45\textwidth}
    \begin{tikzpicture}[>=stealth,scale=1.5,line cap=round,
    bullet/.style={circle,inner sep=1.5pt,fill}]    
 \draw[line width=0.4mm, ->] (0,0) -- (2.5,0) node[right]{$x_1$};
 \draw[line width=0.4mm, ->] (0,0) -- (0,4.3) node[above]{$x_2$};
 \draw [line width=0.4mm, -{Stealth[length=4mm, width=1.5mm]}, Mahogany] (0,0) -- (2,0.5) node [midway, above] {$\xi$};
 \draw [line width=0.4mm,-{Stealth[length=4mm, width=1.5mm]}, MidnightBlue]  (0,0) -- (1,4) node [midway, left] {$\eta$};
 \draw[line width=0.4mm,-{Stealth[length=4mm, width=1.5mm]}, MidnightBlue] (1,4) --  (2,0.5) node [midway, right] {$\xi-\eta$};
\end{tikzpicture} 
\caption*{Case 1}
\end{subfigure}
  \begin{subfigure}[b]{0.45\textwidth}
    \begin{tikzpicture}[>=stealth,scale=1.5,line cap=round,
    bullet/.style={circle,inner sep=1.5pt,fill}]   
 \draw[line width=0.4mm, ->] (0,0) -- (4.3,0) node[right]{$x_1$};
 \draw[line width=0.4mm, ->] (0,0) -- (0,1.5) node[above]{$x_2$};
 \draw [line width=0.4mm, -{Stealth[length=4mm, width=1.5mm]}, Mahogany] (0,0) -- (4,0.5) node [midway, above] {$\xi$};
 \draw [line width=0.4mm,-{Stealth[length=4mm, width=1.5mm]}, MidnightBlue]  (0,0) -- (3.5,1.3) node [midway, above] {$\eta$};
 \draw[line width=0.4mm,-{Stealth[length=4mm, width=1.5mm]}, MidnightBlue] (3.5,1.3) -- (4,0.5) node [midway, right] {$\xi-\eta$};
\end{tikzpicture} 
\caption*{Case 2}
\end{subfigure}
\caption{Two possible scenarios from Lemma \ref{lemma: case organisation} in Cartesian coordinates}
    \label{fig:enter-label}
\end{figure}
\begin{proof}[Proof of Lemma \ref{lemma: case organisation}]
   We prove each case separately.
   \par\emph{(1)} Observe that by triangle inequality and $\xi=(\xi-\eta)+\eta$ there holds $2^{k}\leq 2^{k_1}+2^{k_2}\leq 2^{k_2+5}$. Assume w.l.o.g.\ that $p_2=p_{\max}$. Then if $p_1<p+3$ there holds that $p_1<p_2-7$, and since $\eta_2=\xi_2-(\xi_2-\eta_2)$ we obtain 
    \begin{align*}
        2^{p_2+k_2}\leq 2^{p+k}+2^{p_1+k_1}< 2^{p_2-10+k_2+5}+2^{p_2-7+4+k_2}\leq 2^{p_2+k_2}(2^{-5}+2^{-3}),
    \end{align*}
    which leads to a contradiction. Hence $p_1\geq p+3$. Moreover, since $\eta_2=\xi_2-(\xi_2-\eta_2)$ and thus 
    \begin{equation}
        2^{p_2+k_2}\leq 2^{p+k}+2^{p_1+k_1}\leq 2^{p_1-3+k_2+5}+2^{p_1+k_2+4}\leq 2^{p_1+k_2+5},
    \end{equation}
   and thus $p_1\leq p_2\leq p_1+5$.

     \par\emph{(2)} First observe that $2^{k}\leq 2^{k_1}+2^{k_2}\leq 2^{k_2+2}$. 
 If $p\leq p_2-10$ then 
    \begin{align*}
       & 2^{p_2+k_2}\leq 2^{p+k}+2^{p_1+k_1}\leq 2^{p_2-10+k_2+2}+2^{p_1+k_1}, &&   2^{p_1+k_1}\leq 2^{p+k}+2^{p_2+k_2}\leq 2^{p_2+k_2+2}.
    \end{align*}
      Moreover, there holds $p_{\max}=p_1$ and $p\leq p_2-10\leq p_1-12$. If on the other hand $\abs{p-p_2}\leq 10$ then only the following inequality holds $2^{p_1+k_1}\leq 2^{p+k}+2^{p_2+k_2}\leq 2^{p_2+k_2+3}$. 
By symmetry, the proof of \emph{(3)} is analogous to that of \emph{(2)}.
\end{proof}

\subsection{Set-size estimates}\label{sec: set-size estimate}
\par In this section we present a key ingredient in the proofs of bounds for bilinear estimates Lemma \ref{lemma: bounds on dtSf}, Propositions \ref{prop: bounds on B norm}, \ref{prop: X-norm bounds for l>(1+delta)m}, \ref{prop: X-norm bounds for l<(1+delta)m}. In particular, in bounding localized bilinear expressions in $L^2$, we can ``gain" the smallest of the parameters $\frac{p}{2},\frac{p_i}{2},\frac{q}{2}, \frac{q_i}{2}$, $i=1,2$.  
\begin{lemma}\label{lemma: set size estimate}
    Let $f,g\in L^2$ and $\m $ our multiplier. Then for a bilinear expression
    \[\widehat{\mathcal{Q}_{\m}(f,g)}(\xi)=\int_{\R^2} e^{-it\Phi}\tchi(\xi,\eta)\m(\xi,\eta)\hat{f}(\xi-\eta)\hat{g}(\eta)d\eta,\]
    where $\tchi(\xi,\eta)$ as in \eqref{eqn: chi localizations}
    there holds
    \begin{align*}
        \ltwonorm{\mathcal{Q}_\m(f,g)}\lesssim \abs{S}\norm{m}_{L^\infty_{\xi,\eta}}\ltwonorm{P_{k_1,p_1,q_1}f}\ltwonorm{P_{k_2,p_2,q_2}g}
    \end{align*}
    with 
    \begin{align*}
        \abs{S}:= \min \{ 2^{\frac{k}{2}+\frac{p}{2}},2^{\frac{k_1}{2}+\frac{p_1}{2}},2^{\frac{k_2}{2}+\frac{p_2}{2}} \}\cdot \min \{ 2^{\frac{k}{2}+\frac{q}{2}},2^{\frac{k_1}{2}+\frac{q_1}{2}},2^{\frac{k_2}{2}+\frac{q_2}{2}} \}.
    \end{align*}
\end{lemma}
\begin{proof}
    Without loss of generality, we localize in the $\xi$ and $\xi-\eta$ variable and assume $2^{k+p}\lesssim 2^{k_1+p_1}$ and $2^{k_1+q_1}\lesssim 2^{k+q}$. For $h\in L^2$ we have
    \begin{align*}
        \abs{\langle \mathcal{Q}_\m(f,g),h\rangle}&\lesssim \norm{\m}_{L^\infty_{\xi,\eta}}\lVert\hat{h}(\xi)\hat{f}(\xi-\eta)\rVert_{L^2_{\xi,\eta}}\norm{\varphi_{k,p,q}(\xi)\varphi_{k_1,p_1,q_1}(\xi-\eta)\hat{g}(\eta)}_{L^2_{\xi,\eta}}\\
        &\lesssim \norm{\m}_{L^\infty_{\xi,\eta}} \ltwonorm{f}\ltwonorm{h}2^{\frac{k}{2}+\frac{p}{2}}2^{\frac{k_1}{2}+\frac{q_1}{2}}\ltwonorm{g},
    \end{align*}
    where we have used the support properties of $\varphi_{k,p,q}$.
    The claim follows by exchanging the variables.
    \end{proof}
 \begin{remark}
        We note that we use the $q$-localization only in once instance in the proof of Proposition \ref{prop: X-norm bounds for l<(1+delta)m}, Case D. Otherwise we will use the set-size estimate above with $\chi(\xi,\eta)=\varphi_{k,p}(\xi)\varphi_{k_1,p_1}(\xi-\eta)\varphi_{k_2,p_2}(\eta)$.
    \end{remark}

\subsection{Normal forms}\label{sec: Normal forms}
In this section we present how to use normal forms in in proving bounds on bilinear expression. In particular, we discuss a way to split the analysis into two regions depending on the size of the phase $\Phi$. In the so-called non-resonant part, where a positive lower bound on $\abs{\Phi}$ is available, one can integrate by parts in time and use the improved decay of the time derivative from Section \ref{sec: bounds on dtSf}, see \eqref{eqn: nonresonant multiplier bilinear estimate}. On the other hand, in the resonant part set-size estimates are available, see Lemma \ref{lemma: normal forms}\eqref{it3: lemma: normal forms}-\eqref{it4: lemma: normal forms}. 
\par Let $\Phi$ be one of the phases defined in \eqref{eqn: bq phases} and $\psi$ as in Section \ref{sec: Localizations}. For $\lambda>0$ to be appropriately chosen, we split the multiplier in a resonant and non-resonant part as follows: 
\begin{align*}
    \m(\xi,\eta)=\psi(\lambda^{-1}\Phi)\m(\xi,\eta)+(1-\psi(\lambda^{-1}\Phi))\m(\xi,\eta)=: \m^{res}(\xi,\eta)+\m^{nr}(\xi,\eta).
\end{align*}
This yields the following decomposition:
\begin{align*}
    \B_{\m}(f,g)=\B_{\m^{res}}(f,g)+\B_{\m^{nr}}(f,g).
\end{align*}
Furthermore, after integrating by parts, the non-resonant part can be written as follows: 
 \begin{align}\label{eqn: nonresonant multiplier bilinear estimate}
    \B_{\m^{nr}}(f,g)=\mathcal{Q}_{\m^{nr}\Phi^{-1}}(f,g)+\B_{\m^{nr}\Phi^{-1}}(\partial_tf,g)+\B_{\m^{nr}\Phi^{-1}}(f,\partial_tg).
\end{align}

The following lemma provides useful set-size estimates for the terms in the decomposition above.
\begin{lemma}\label{lemma: normal forms}
    Let $\lambda>0$, functions $\chi,\; \tchi$ as in \eqref{eqn: chi localizations} and $f_j$ be localized profiles for $j=1,2$. Assume we have a splitting as in \eqref{eqn: nonresonant multiplier bilinear estimate}, then there holds
    \begin{enumerate}
        \item \label{it1: lemma: normal forms} The boundary term satisfies
        \begin{align*}
            \ltwonorm{P_{k,p}\mathcal{Q}_{\m^{nr}\Phi^{-1}}(f_1,f_2)}\lesssim  2^{k+p_{\max}}\lambda^{-1}\abs{S}\ltwonorm{f_1}\ltwonorm{f_2}.
        \end{align*}
        \item  \label{it2: lemma: normal forms}  If $\lambda>0$ is chosen such that $\abs{\Phi\chi}\geq \lambda\gtrsim 1$ then $\m^{res}=0$ and $\m=\m^{nr}$ with the following bound
        \begin{align*}
            \ltwonorm{P_{k,p}\mathcal{Q}_{\m^{nr}\Phi^{-1}}(f_1,f_2)}\lesssim 2^{k+p_{\max}}\min \{\inftynorm{e^{it\Lambda}f_1}\ltwonorm{f_2},\ltwonorm{f_1}\inftynorm{e^{it\Lambda}f_2}\}.
        \end{align*}
        \item  \label{it3: lemma: normal forms} If $\abs{\partial_{\eta_1}\Phi\chi}\gtrsim K>0$ or $\abs{\partial_{\xi_1}\Phi\chi}\gtrsim K>0$, then there holds
        \begin{align*}
            &\ltwonorm{P_{k,p}\mathcal{Q}_{\m^{res}}(f_1,f_2)}\lesssim 2^{k+p_{\max}}\lambda^{\frac{1}{2}}K^{-\frac{1}{2}}\min \{2^{\frac{k_1+p_1}{2}},2^{\frac{k_2+p_2}{2}} \}\ltwonorm{f_1}\ltwonorm{f_2}.
        \end{align*}
         \item  \label{it4: lemma: normal forms}  If $\abs{\partial_{\eta_2}\Phi\tchi}\gtrsim K>0$ or $\abs{\partial_{\xi_2}\Phi\tchi}\gtrsim K>0$, then there holds
        \begin{align*}
            &\ltwonorm{P_{k,p,q}\mathcal{Q}_{\m^{res}}(f_1,f_2)}\lesssim 2^{k+p_{\max}+q_{\max}}\lambda^{\frac{1}{2}}K^{-\frac{1}{2}}\min \{2^{\frac{k_1+q_1}{2}},2^{\frac{k_2+q_2}{2}} \}\ltwonorm{f_1}\ltwonorm{f_2}.
        \end{align*}
    \item  \label{it5: lemma: normal forms} The analogous bounds \eqref{it1: lemma: normal forms}-\eqref{it3: lemma: normal forms} hold when additionally localizing in $q, q_i,$ $i=1,2,$ with $\tchi$ as in \eqref{eqn: chi localizations}.
    \end{enumerate}
\end{lemma}
\begin{proof}
    The claims \eqref{it1: lemma: normal forms} and \eqref{it2: lemma: normal forms} follow by the set-size estimate Lemma \ref{lemma: set size estimate} and multiplier bounds Lemmas \ref{lemma: multiplier bound W norm}, \ref{lemma: m phi^-1 multiplier bound}. We prove the third claim.  Assume  $\abs{\partial_{\eta_1}\Phi}\chi\gtrsim K$ and without loss of generality that $2^{k_2+p_2}\lesssim 2^{k_1+p_1}$ (otherwise exchange $h$ and $\hat{f}$). For any $h\in L^2$ there holds:
             \begin{align*}
                 \sabs{\langle{\mathcal{Q}}_{\m^{res}} (f,g),h\rangle}&\lesssim \int_{\R^2}\int_{\R^2} \abs{\m(\xi,\eta)}\chi(\xi,\eta)\varphi(\lambda^{-1}\Phi)\sabs{\hat{f}(\xi-\eta)\hat{g}(\eta)}\abs{h(\xi)}d\xi d\eta\\
                 &\lesssim \norm{\m}_{L^\infty_{\xi,\eta}}\snorm{\hat{f}(\xi-\eta)\hat{g}(\eta)}_{L^{2}_{\xi,\eta}}\norm{\chi(\xi,\eta)\varphi(\lambda^{-1}\Phi) h(\xi)}_{L^{2}_{\xi,\eta}}.
             \end{align*}
             It remains to bound the last term and the claim. Observe that
             \begin{align*}
                 \norm{\chi(\xi,\eta)\varphi(\lambda^{-1}\Phi) h(\xi)}_{L^{2}_{\xi,\eta}}^2\lesssim \sup_{\xi}\int_{\eta\in\R^2} \sabs{\chi(\xi,\eta)\varphi(\lambda^{-1}\Phi)}^2d\eta\norm{h}_{L^2_\xi}^2.
             \end{align*}
             It remains to prove
             \begin{align*}
                 \sup_{\xi} \int_{\eta\in\R^2} \sabs{\chi(\xi,\eta)\varphi(\lambda^{-1}\Phi)}^2d\eta\lesssim2^{k_2+p_2}\lambda K^{-1}. 
             \end{align*}
             To that end, we do a change of variables $\eta\mapsto (\Phi(\xi,\eta),\eta_2)=:\zeta$ and use the fact that $\sabs{\det {\frac{\partial \zeta}{\partial \eta}}}=\abs{\partial_{\eta_1}\Phi}$ to obtain for a fixed $\xi$:
             \begin{align*}
                 \int_{\R^2}\sabs{\chi(\xi,\eta)\varphi(\lambda^{-1}\Phi)}^2d \eta &\lesssim\int_{\R^2} \abs{\chi(\xi,\zeta)\varphi(\lambda^{-1}\zeta_1)}^2  \abs{\det {\frac{\partial \zeta}{\partial \eta}}}^{-1}d\zeta
                 \lesssim 2^{k_2+p_2} \lambda K^{-1}.
             \end{align*}
             In case $\abs{\partial_{\xi_1}\Phi}\gtrsim K$, use the change of variables $\xi\mapsto (\Phi(\xi,\eta),\xi_2)$ instead. The proof of the fourth statement is analogous using a change of variables $\eta\mapsto (\eta_1,\Phi(\xi,\eta))$ if $\abs{\partial_{\eta_2}\Phi}\gtrsim K$ and $\xi\mapsto (\xi_1, \Phi(\xi,\eta))$ if $\abs{\partial_{\xi_2}\Phi}\gtrsim K$.
\end{proof}

\section{Bounds on \texorpdfstring{$\partial_tS^NF$}{dtSnf} in \texorpdfstring{$L^2$}{Ltwo}} \label{sec: bounds on dtSf}
 In this section we prove time decay of the time derivative of profiles in $L^2$, which will be used in subsequent sections when performing normal forms. 
\begin{lemma}\label{lemma: bounds on dtSf}
    Let $F\in\{\zz_\pm,\Theta\}$ and assume the bootstrap assumption \eqref{eqn: bootstrap assumption} holds. Then for $\delta=2M^{-1}>0$,  $m\in\N$ and $t\in[2^m,2^{m+1}[\cap[0,T] $ there holds
    \begin{align*}
        \sltwonorm{\partial_tP_kS^bF(t)}\lesssim 2^{\frac{3}{4}k}2^{-2k^+}2^{(-\frac{3}{4}+2\delta) m}\eps^2.
    \end{align*}
\end{lemma}
The lemma is proved following the scheme discussed in Section \ref{sec: outline of the proof}:
Using the energy estimates in \textbf{Case A}, we can reduce to proving the claim \eqref{eqn: claim time derivative} for parameter-localized interactions.
\textbf{Case B} deals with the ``gap in $p$" setting when we can integrate by parts using Lemma \ref{lemma: ibp in bilinear expressions}. When this is not feasible, we take advantage of the smallness of the set over which we integrate to obtain the claim via Lemma \ref{lemma: set size estimate}. The ``no gap" setting is handled in \textbf{Case C} using the linear decay estimate from Section \ref{sec: Linear Decay}. Throughout the proof we employ the multiplier bound from Lemma \ref{lemma: multiplier bound}.

\begin{proof}
By the Duhamel formulations \eqref{eqn: Duhamel BQ Z+ Z-}, \eqref{eqn: Duhamel for the profile} and Lemma \ref{lemma: S^N on bilinear expression Qm}, it suffices to suitably bound the sums 
$\sum_{k_1,k_2\in \Z}\sltwonorm{P_k\mathcal{Q}_\m(P_{k_1}S^{b_1}F_1,P_{k_2}S^{b_2}F_2)}$, where, $b_1+b_2\leq N$ and $F_i\in \{\mathcal{Z}_\pm,\Theta\},\, i=1,2$.

  \par\textbf{Case A: Simple cases.}  The set size estimate Lemma \ref{lemma: set size estimate} and the multiplier bound Lemma \ref{lemma: multiplier bound} yield
    \begin{align*}\sltwonorm{P_k\mathcal{Q}_\m(P_{k_1}S^{b_1}F_1,P_{k_2}S^{b_2}F_2)}&\lesssim \abs{S}\norm{\m}_{L^\infty_{\xi,\eta}}\sltwonorm{P_{k_1}S^{b_1}F_1}\sltwonorm{P_{k_2}S^{b_2}F_2}\\
        &\lesssim 2^{k_{\min}}2^{\frac{p_{\min}}{2}}2^{k}2^{p_{\max}}2^{-N_0(k_1^++k_2^+)}\snorm{P_{k_1}S^{b_1}F_1}_{H^{N_0}}\snorm{P_{k_2}S^{b_2}F_2}_{H^{N_0}}.
    \end{align*}
    Hence if $k_{\max}\geq\delta_0m:= 2N^{-1}_0m$ or $k_{\min}\leq -2m$ and since $N_0>4$, we obtain with the bootstrap assumption \eqref{eqn: bootstrap assumption} that
    \begin{equation}\label{eqn: sum in k_i}
    \begin{split}
        \sum_{\substack{k_1,k_2\in \Z\\ k_{\max}\geq\delta_0m \text{ or } k_{\min}\leq -2m} }\sltwonorm{P_k\mathcal{Q}_\m(P_{k_1}S^{b_1}F_1,P_{k_2}S^{b_2}F_2)}\lesssim 2^k2^{-2k^+}2^{-m}\eps^2,
    \end{split}
    \end{equation}
    and it remains to bound
    \begin{equation}
        \sum_{\substack{k_1,k_2 \in \Z\\ -2m<k_1,k_2<\delta_0m}}\sltwonorm{P_k\mathcal{Q}_\m(P_{k_1}S^{b_1}F_1,P_{k_2}S^{b_2}F_2)}.
    \end{equation}
Localizing further in $p,\,p_i$ and $l_i$, $i=1,2$, and writing $f_i=P_{k_i,p_i}R_{l_i}S^{b_i}F_i$ we have
 \begin{align}\label{eqn: sum in pi,li}
\sltwonorm{P_k\mathcal{Q}_\m(P_{k_1}S^{b_1}F_1,P_{k_2}S^{b_2}F_2)}=\sum_{p\in \Z^-}\sum_{\substack{p_1\in \Z^-, l_1\in \Z^+\\ p_1+l_1\geq 0}}\sum_{\substack{p_2\in \Z^-, l_2\in \Z^+\\ p_2+l_2\geq 0}} \ltwonorm{P_{k,p}\mathcal{Q}_{\m}(f_1,f_2)}.
 \end{align}
Observe that for $\max\{l_1,l_2\}\geq 2m$
    \begin{align*}
       \sum_{\substack{p_i\in \Z^-, l_i\in \Z^+\\ \max\{l_1,l_2\}\geq 2m}}\ltwonorm{P_{k,p}\mathcal{Q}_{\m}(f_1,f_2)}&\lesssim  \hspace*{-9pt}\sum_{\substack{p_i\in \Z^-, l_i\in \Z^+\\ \max\{l_1,l_2\}\geq 2m}}\hspace*{-9pt}2^{k_{\min}}2^{\frac{p_{\min}}{2}}2^k2^{-4k_1^+}2^{-l_1}2^{-\frac{p_1}{2}}\norm{f_1}_X2^{-4k_2^+}2^{-l_2}2^{-\frac{p_2}{2}}\norm{f_2}_X\\
        &\lesssim 2^{k}2^{-2k^+}2^{(-1+\delta)m}\eps^2.
    \end{align*}
   Therefore, it remains to bound bilinear terms for the following localization parameters:
    \begin{align}\label{eqn: localization paramters reduction dtSNf}
     &  -2m<k,\;k_1,\;k_2<\delta_0m, && -2m< p_1,\;p_2\leq 0, && 0\leq l_1,\;l_2< 2m.
   \end{align}
   Observe that each sum on the right-hand side of \eqref{eqn: sum in k_i}, \eqref{eqn: sum in pi,li} ranges over an interval of order $m\lesssim 2^{\gamma m}$ for all $\gamma>0$. Thus, it suffices to prove
   \begin{align}\label{eqn: claim time derivative}
        \ltwonorm{P_{k,p}\mathcal{Q}_{\m}(f_1,f_2)}\lesssim 2^{\frac{3}{4}k}2^{-2k^+}2^{(-\frac{3}{4}+\delta)m}\eps^2,
   \end{align}
  for the localization parameters as in \eqref{eqn: localization paramters reduction dtSNf} and $\delta=2 M^{-1}$.
    \par \textbf{Case B: Gap in $p$ with $p_{\min}\ll p_{\max}$.} In this part we assume without loss of generality that $k_1\leq k_2$. Then $k_{\min}\in \{k,k_1\}$ and $k_{\max}\in \{k,k_2\}$ and so in particular $2^{k_{\max}+k_{\min}}\sim 2^{k+k_1}.$ In this case by Proposition \ref{prop: lower bound on sigma} there holds $\abs{\sigma}\gtrsim 2^{p_{\max}}2^{k_1+k}$. With the condition  $k_1\leq k_2$, we have two further subcases to cover, namely $k_{\min}=k_1$ and $k_{\min}=k$.  
    \par \textbf{Case B.1:} $k_{\min}=k_1$, then there holds $2^{k}\sim 2^{k_2}$. If 
    \begin{align}\label{eqn: dtS^nf A.1 l_1}
        -p_{\max}+2l_1\leq (1-\delta)m,
    \end{align} 
    we obtain the claim by integrating by parts $M\gg N$ times along $S_{\eta}$. Indeed, by Lemma \ref{lemma: ibp in bilinear expressions}\eqref{it1: lemma: ibp in bilinear expressions} and the energy estimates \eqref{eqn:bootstrap-prop-bded-energy} there holds
    \begin{align*}
        \ltwonorm{P_{k,p}\mathcal{Q}_{\m}(f_1,f_2)}&\lesssim 2^k  \inftynorm{\F(\mathcal{Q}_\m(f_1,f_2))}\\
        &\lesssim 2^k2^{k_{\max}} [2^{-m}2^{-p_1}2^{k_{1}-k-p_{\max}}(1+2^{k_2-k_1}2^{l_1})]^M \\
        &\hspace{1cm} \cdot \sltwonorm{P_{k_1,p_1}R_{l_1}(1,S)^Mf_1}\sltwonorm{P_{k_2,p_2}R_{l_2}(1,S)^Mf_2}\\
        &\lesssim 2^{k}2^{-{3}k^+}[2^{-m}2^{-p_1-p_{\max}}2^{l_1}]^M\eps^2\\
        &\lesssim 2^{k}2^{-{3}k^+}2^{-2m}\eps^2,
    \end{align*}
    where $\delta:=2M^{-1}\ll 1$. Similarly, integration by parts along $S_{\xi-\eta}$ as per Lemma \ref{lemma: ibp in bilinear expressions}\eqref{it1: lemma: ibp in bilinear expressions} yields the claim \eqref{eqn: claim time derivative} if 
    \begin{align}\label{eqn: dtS^nf A.1 l_2}
        \max \{ k-k_1-p_{\max}+l_2, 2l_2-p_{\max} \}<(1-\delta)m.
    \end{align}
    Indeed, from Lemma \ref{lemma: ibp in bilinear expressions}\eqref{it1: lemma: ibp in bilinear expressions} with $2^{k_1-k_2}\lesssim 1$, $2^{k_2}\sim 2^k$ and $2^{-p_2}\lesssim 2^{l_2}$, the contribution at each iteration is
    \begin{align*}
        2^{-m}2^{2k_2}2^{-p_2}2^{-p_{\max}-{k_1}-{k}}(1+2^{k_1-k_2}2^{l_2})\lesssim 2^{-m}2^{k-k_1}2^{-p_{\max}}2^{l_2}+2^{-m}2^{-p_{\max}}2^{2l_2}.
    \end{align*}
    Assume now that \eqref{eqn: dtS^nf A.1 l_1} and \eqref{eqn: dtS^nf A.1 l_2} don't hold, then we consider two cases depending on which term on the right-hand side of \eqref{eqn: dtS^nf A.1 l_2} is larger.
    \par\textbf{1.} Assume first that  $-2l_1<-(1-\delta)m-p_{\max}$ and $-l_2<-(1-\delta)m-p_{\max}+k-k_1$. By the set size estimate Lemma \ref{lemma: set size estimate} with $\abs{S}\lesssim 2^{k_1+\frac{p_1}{2}}$ we obtain:
    \begin{align*}
        \ltwonorm{P_{k,p}\mathcal{Q}_{\m}(f_1,f_2)}&\lesssim 2^{k}2^{p_{\max}} 2^{k_{\min}+\frac{p_{\min}}{2}}2^{-4k_1^+}2^{-{l_1}}2^{-\frac{p_1}{2}}\norm{f_1}_X2^{-4k^+}2^{-\frac{l_2}{2}}\norm{f_2}_X\\
        &\lesssim 2^{k+p_{\max}}2^{k_1}2^{-4k_1^+-4k^+}2^{(-1+\delta)m}2^{-p_{\max}}2^{\frac{k-k_1}{2}}\eps^2\\
        &\lesssim 2^k2^{-3k^+}2^{(-1+\delta)m}\eps^2.
    \end{align*}
   \par \textbf{2.} Assume  $\max\{-2l_1,-2l_2\}<-(1-\delta)m-p_{\max}$. Again using the size set estimate Lemma \ref{lemma: set size estimate} we obtain:
    \begin{align*}
        \ltwonorm{P_{k,p}\mathcal{Q}_{\m}(f_1,f_2)}&\lesssim 2^{k}2^{p_{\max}} 2^{k_{\min}+\frac{p_{\min}}{2}}2^{-4k_1^+}2^{-l_1-\frac{p_1}{2}}\norm{f_1}_X2^{-4k^+}2^{-\frac{l_2}{2}}\norm{f_2}_X\\
        &\lesssim 2^k2^{-2k^+}2^{p_{\max}}2^{(-\frac{3}{4}+\frac{3}{4}\delta)m}2^{-\frac{3}{4}p_{\max}}\eps^2\\
        &\lesssim 2^k2^{-2k^+}2^{(-\frac{3}{4}+\delta)m}\eps^2.
    \end{align*}
    \par \textbf{Case B.2:} $k_{\min}=k$ and therefore $2^{k_2}\sim 2^{k_1}$. We obtain the claim \eqref{eqn: claim time derivative} through integration by parts along $S_{\eta}$ if $ k_1-k-p_{\max}+2l_1\leq (1-\delta)m.$
    Indeed, there holds
    \begin{align*}
         \ltwonorm{P_{k,p}\mathcal{Q}_{\m}(f_1,f_2)}&\lesssim 2^k \inftynorm{\F(\mathcal{Q}_\m(f_1,f_2))}\\
        &\lesssim 2^{k+k_{\max}} [2^{-m}2^{-p_1}2^{k_1-k-p_{\max}}(1+2^{k_2-k_1}2^{l_1})]^M \\
        &\hspace{1cm} \cdot \ltwonorm{P_{k_1,p_1}R_{l_1}(1,S)^Mf_1}\ltwonorm{P_{k_2,p_2}R_{l_2}(1,S)^Mf_2}\\
        &\lesssim 2^{k}2^{-3k^+}[2^{-m}2^{k_1-k}2^{2l_1}2^{-p_{\max}}]^M\eps^2\\
        &\lesssim 2^{k}2^{-3k^+}2^{-2m}\eps^2.
    \end{align*}
    Similarly, by integrating by parts in $S_{\xi-\eta}$, we obtain the claim \eqref{eqn: claim time derivative} if $k_1-k-p_{\max}+2l_2\leq (1-\delta)m$. 
    Otherwise if $\max\{-2l_1,-2l_2\}< -(1-\delta)m+k_1-k-p_{\max}$, we estimate as in \textbf{Case B.1}:
   \begin{align*}
        \ltwonorm{P_{k,p}\mathcal{Q}_{\m}(f_1,f_2)}&\lesssim 2^{k}2^{p_{\max}} 2^{k_{\min}+\frac{p_{\min}}{2}}2^{-8k_1^+}2^{-l_1-\frac{p_1}{2}}\norm{f_1}_X2^{-\frac{l_2}{2}}\norm{f_2}_X\\
        &\lesssim 2^{k+p_{\max}}2^{\frac{k_1+k}{2}}2^{-8k_1^+}2^{(-\frac{3}{4}+\delta)m}2^{-\frac{3}{4}p_{\max}}2^{\frac{3}{4}(k_1-k)}\eps^2\\
        &\lesssim 2^{\frac{3}{4}k}2^{-2k^+}2^{(-\frac{3}{4}+\delta)m}\eps^2.
   \end{align*} 
   \par \textbf{Case C: No gaps with $p\sim{p_1}\sim{p_2}$.} Assume without loss of generality that $f_1$ has fewer vector fields than $f_2$, i.e. $b_1\leq b_2$. Then we can use Proposition \ref{proposition: linear decay} on $f_1$ such that for $0<\beta'<\beta$ we have the decomposition
   \begin{align*}
       P_{k_1,p_1}e^{it\Lambda}f_1=I_{k_1,p_1}(f_1)+II_{k_1,p_1}(f_1).
   \end{align*}
Moreover, we can apply Lemma \ref{lemma: linear decay many vector fields} with $\kappa\ll \beta'$ on $f_2$, so that the following bounds hold
   \begin{align*}
      & \inftynorm{I_{k_1,p_1}(f_1)}\lesssim 2^{\frac{3}{4}k_1}2^{-\frac{15}{4}k_1^+}2^{-p}2^{(-1+\delta)m}\norm{f_1}_D,&&
        \ltwonorm{II_{k_1,p_1}(f_1)}\lesssim  2^{-4k_1^+}2^{-\frac{p}{2}}2^{-\frac{m}{2}} \norm{f_1}_D,\\
        &\sinftynorm{P_{k_2,p_2}e^{it\Lambda}f_2}\lesssim 2^{\frac{3}{4}k_2}2^{-3k_2^+}2^{(-\frac{1}{2}+\kappa)m}\eps^2.
   \end{align*}
   With these, we obtain the claim \eqref{eqn: claim time derivative}:
   \begin{align*}
         \ltwonorm{P_{k,p}\mathcal{Q}_{\m}(f_1,f_2)}&\lesssim  \ltwonorm{\m I_{k_1,p_1}(f_1)e^{it\Lambda}P_{k_2,p_2}f_2}+\ltwonorm{\m II_{k_1,p_1}(f_1)e^{it\Lambda}P_{k_2,p_2}f_2}\\
         &\lesssim 2^{k+p}(\inftynorm{I_{k_1,p_1}(f_1)}\ltwonorm{P_{k_2,p_2}e^{it\Lambda}f_2}+\ltwonorm{II_{k_1,p_1}(f_1)}\inftynorm{P_{k_2,p_2}e^{it\Lambda}f_2})\\
         &\lesssim 2^{k+p}[2^{\frac{3}{4}k_1}2^{-\frac{15}{4}k_1^+}2^{-p}2^{(-1+\delta)m}\norm{f_1}_D2^{-4k_2^+}2^{\frac{p}{2}}\norm{f_2}_B\\
         &\hspace{1cm}+ 2^{-4k_1^+}2^{-(\frac{1}{2}+2\beta')p}2^{(-\frac{1}{2}-\beta')m}\norm{f_1}_D2^{\frac{3}{4}k_2-3k_2^+}2^{(-\frac{1}{2}+\kappa)m}\eps]\\
         &\lesssim 2^k 2^{-2k_1^+-2k_2^+}[2^{(-1+\delta)m}+2^{(-\frac{1}{2}-\beta')m}2^{(-\frac{1}{2}+\kappa)m}]\eps^2\\
         &\lesssim 2^k2^{-2k^+}2^{(-1+\delta)m}\eps^2.
   \end{align*}
   This concludes the proof of the lemma.
\end{proof}

\section{Bounds on the \texorpdfstring{$B$}{B}-norm}\label{sec: bounds on B norm} In this section we prove bounds on the $B$-norm of bilinear terms needed in the proof of Proposition \ref{prop: bootstrap argument}. As explained in \eqref{eqn: bootstrap bilinear estimate reduced}, \eqref{eqn: bootstrap B_m time decomposition} it suffices to suitably bound $\F\B_\m(F_1,F_2)(t,\xi)=\int_0^t \mathcal{Q}_\m(F_1,F_2)(s,\xi)ds, $
where $\mathcal{Q}_\m(F_1,F_2)$ is a bilinear expression as in \eqref{eqn: bilinear form Qm} and  $\B_\m$ is localized on a time interval $t\in [2^m,2^{m+1}[$, $m\in \N$.
\begin{proposition}\label{prop: bounds on B norm}
In the setting of Proposition \ref{prop: bootstrap argument}, and in particular under the bootstrap assumptions \eqref{eqn: bootstrap assumption}, the following holds true: For $\m\in \set{\m_0,\m_\pm^{\mu\nu}}{\mu,\nu\in\{+,-\}}$, $t\in [2^m,2^{m+1}[\cap [0,T]$ and $\delta=2M^{-\frac{1}{2}}$, there holds that
   \begin{align}
        \norm{\B_\m(F_1,F_2)}_B\lesssim  2^{(\frac{1}{6}+2\delta) m}\eps^2 +2^{(\frac{1}{4}+3\delta) m}\eps^3,
    \end{align}
    where $F_i\in \{S^{b_i}\mathcal{Z}_{\pm},S^{b_i}\Theta\}$, $0\leq b_1+b_2\leq N$, $i=1,2$.
\end{proposition}

 The proof follows the outline presented in Section \ref{sec: outline of the proof} and expands on the arguments already employed in Section \ref{sec: bounds on dtSf}. \blue{The \emph{simple cases} and the \emph{no gaps} cases (\textbf{Cases A, D} below) follow along similar lines to the proof in Section \ref{sec: bounds on dtSf}. The \emph{gap in $p$} case is split into two parts: if $\max\{p,p_1,p_2\}\sim 0$ (\textbf{Case B} below), the claim follows using integration by parts or else the $B$- and $X$-norms and set-size estimates. On the other hand, if (\textbf{Case C} below) $\max\{p,p_1,p_2\}\ll 0$ and thus $\abs{\Phi}\gtrsim 1$, we can employ normal forms, a new and central feature compared to the proof of Lemma \ref{lemma: bounds on dtSf}. This allows for the localized bilinear expressions to be bounded by (see \eqref{eqn: nonresonant multiplier bilinear estimate}):
    \begin{align*}
           \ltwonorm{P_{k,p}\B_{\m}(f_1,f_2)}\lesssim \ltwonorm{P_{k,p}\mathcal{Q}_{\m\Phi^{-1}}(f_1,f_2)}+\ltwonorm{P_{k,p}\B_{\m\Phi^{-1}}(\partial_tf_1,f_2)}+\ltwonorm{P_{k,p}\B_{\m\Phi^{-1}}(f_1,\partial_tf_2)},
    \end{align*} where $f_i,$ $i=1,2$ are profiles localized both in frequency and  space. The first boundary term is relatively easy to estimate as we have one time parameter less to ``gain". The other two terms are handled using the bound on the time derivative Lemma \ref{lemma: bounds on dtSf}.} 
\begin{proof}
    By the definition of the $B$-norm in \eqref{B norm} we have after localizing in $k_i$, $i=1,2$:
    \begin{align*}
        \norm{\B_\m(F_1,F_2)}_B=\sup_{k\in \Z,p\in\Z^-}2^{4k^+}2^{-\frac{k^-}{2}}2^{-\frac{p}{2}}\sum_{k_1,k_2\in\Z}\ltwonorm{P_{k,p}\B_\m(P_{k_1}F_1,P_{k_2}F_2)}.
    \end{align*}
   \par \textbf{Case A: Simple cases.}  
   If for $\delta_0:=2N_0^{-1}$ there holds that $k_{\max}\geq\delta_0m$ or $k_{\min}\leq -4m$, then the claim is obtained using Lemma \ref{lemma: set size estimate}, the multiplier bound Lemma \ref{lemma: multiplier bound} and the bootstrap assumption \eqref{eqn: bootstrap assumption} together with the energy estimates \eqref{eqn:bootstrap-prop-bded-energy}, by summing the following bound over $k_1,\;k_2$ within this range:
    \begin{align*}
        2^{4k^+}2^{-\frac{k^-}{2}}2^{-\frac{p}{2}}\ltwonorm{P_{k,p}\B_\m(P_{k_1}F_1,P_{k_2}F_2)}&\lesssim 2^m2^{\frac{9}{2}k^+}2^{-\frac{k}{2}}2^{-\frac{p}{2}}\abs{S}\inftynorm{\m\chi}\sltwonorm{P_{k_1}F_1}\sltwonorm{P_{k_2}F_2}\\
        &\lesssim 2^m2^{\frac{k_{\min}}{2}+k+\frac{9}{2}k^+}2^{-N_0k_1^+}2^{-N_0k_2^+}\snorm{P_{k_1}F_1}_{H^{N_0}}\snorm{P_{k_2}F_2}_{H^{N_0}}\\
        &\lesssim 2^m2^{\frac{k_{\min}}{2}}2^{-(N_0-6)k_1^+}2^{-(N_0-6)k_2^+}\eps^2.
    \end{align*}
   So from now on we can assume $-4m<k,k_1,k_2<\delta_0m$. We localize further in $p_i,\,l_i$ with $p_i+l_i\geq0$, $i=1,2$ and let $f_i=P_{k_i,p_i}R_{l_i}F_i$. If $\max\{l_1,l_2\}\geq 4m$ we obtain with the set size estimate Lemma \ref{lemma: set size estimate} and the bootstrap assumption:
    \begin{align*}
        {2^{4k^+}}2^{-\frac{k^-}{2}}2^{-\frac{p}{2}}\ltwonorm{P_{k,p}\B_\m(f_1,f_2)}
        &\lesssim 2^{m+2\delta_0m}2^{-l_1-l_2}2^{-\frac{p_1}{2}}2^{-\frac{p_2}{2}}\norm{f_1}_X\norm{f_2}_X\\
        &\lesssim 2^{(-1+2\delta_0)m}2^{-\frac{l_1+p_1}{2}}2^{-\frac{l_2+p_2}{2}}\eps^2.
    \end{align*}
 Thus, as explained in \textbf{Case A} in the proof of Lemma \ref{lemma: bounds on dtSf}, see also \eqref{eqn: sum in k_i}-\eqref{eqn: sum in pi,li}, it suffices to establish the claim
\begin{align}\label{eqn: B norm claim}
    \sup_{k,p}2^{4k^+}2^{-\frac{k^-}{2}}2^{-\frac{p}{2}}\ltwonorm{P_{k,p}\B_{\m}(f_1,f_2)}\lesssim 2^{(\frac{1}{6}+\delta) m}\eps^2 +2^{(\frac{1}{4}+\frac{5}{2}\delta) m}\eps^3
\end{align}for the following localization parameters:
\begin{align} \label{eqn: localisation parameter assumptions B norm}
 &   -4m<k,k_i<\delta_0 m , && -p_i\leq l_i\leq 4m, && -4m\leq p_i\leq 0, &&i=1,2.
\end{align}
\par\textbf{Case B: Gap in $p$ with $p_{\min}\ll p_{\max}\sim 0$.} We assume w.l.o.g.\ that $p_1\leq p_2$. Observe that by Proposition \ref{prop: lower bound on sigma} there holds $\abs{\sigma}\sim 2^{k_{\min}+k_{\max}}$. 
Moreover the multiplier bound $\inftynorm{\m \chi}\lesssim 2^{k}$ holds by Lemma \ref{lemma: multiplier bound}. Repeated integration by parts in $S_\eta$ or $S_{\xi-\eta}$ as per Lemma \ref{lemma: ibp in bilinear expressions}\eqref{it1: lemma: ibp in bilinear expressions} yields the claim if
\begin{equation}\label{eqn: ibp is working}
    \begin{split}
      S_{\eta}:\hspace{1cm}  2^{2k_1}2^{-p_1}2^{-k_{\max}-k_{\min}}(1+2^{k_2-k_1}2^{l_1})&\leq 2^{(1-\delta)m}, \\
      S_{\xi-\eta}:\hspace{1cm}   2^{2k_2}2^{-p_2}2^{-k_{\max}-k_{\min}}(1+2^{k_1-k_2}2^{l_2})&\leq 2^{(1-\delta)m},
    \end{split}
\end{equation} 
where $\delta:=2M^{-\frac{1}{2}}\gg 2M^{-1}\gg 2N_0^{-1}=\delta_0$. Indeed, if the first condition above holds, integration by parts in $S_\eta$ with $\ltwonorm{\varphi_{k,p}}\lesssim 2^{k}2^{\frac{p}{2}}$ and Lemma \ref{lemma: ibp in bilinear expressions}\eqref{it1: lemma: ibp in bilinear expressions} give 
\begin{equation}\label{eqn: claim for B norm after ibp}
\begin{split}
     2^{4k^+-\frac{k^-}{2}}2^{-\frac{p}{2}}\ltwonorm{P_{k,p}\B_\m(f_1,f_2)}&\lesssim 2^{\frac{9}{2}k^+}2^{-\frac{k}{2}}2^{-\frac{p}{2}}\ltwonorm{\varphi_{k,p}}\sinftynorm{\widehat{\B_\m(f_1,f_2)}}\\
   & \lesssim 2^{5k^+}2^{k_{\max}}2^m [2^{-m}2^{2k_1}2^{-p_1}2^{-k_{\max}-k_{\min}}(1+2^{k_2-k_1}2^{l_1})]^M\\
&\quad \cdot \ltwonorm{P_{k_1,p_1}R_{l_1}(1,S)^Mf_1}\ltwonorm{P_{k_2,p_2}R_{l_2}(1,S)^Mf_2}\\
&\lesssim  2^{-m}\eps^2.
\end{split}
\end{equation}
With a similar computation, we obtain the claim when integrating by parts along $S_{\xi-\eta}$. 
\par Otherwise, if neither condition in \eqref{eqn: ibp is working} holds, we consider several cases that are organized according to Lemma \ref{lemma: case organisation}.
\par\textbf{Case B.1: $p_{\min}= p\ll p_{\max}$.} From Lemma \ref{lemma: case organisation} and under the constraint $p_1\leq p_2$, there are two further geometrical settings to consider.
\par\textbf{Case B.1(a): $2^{k_1}\sim 2^{k_2}$}. Then $p\ll p_1\sim p_2$, $k_{\max}+k_{\min}\sim k+k_1$ and moreover  $\min\{l_1,l_2\}>(1-\delta)m+k-k_1$. From Lemma \ref{lemma: set size estimate} with $\abs{S}\lesssim 2^{k+\frac{p}{2}}$ and the bootstrap assumption \eqref{eqn: bootstrap assumption} we obtain
\begin{align*}
    2^{4k^+}2^{-\frac{k^-}{2}}2^{-\frac{p}{2}}\ltwonorm{P_{k,p}\B_\m(f_1,f_2)}&\lesssim 2^{\frac{9}{2}k^+}2^m 2^{\frac{3}{2}k}2^{-8k_1^+}2^{-l_1}2^{-\frac{l_2}{2}}\norm{f_1}_X\norm{f_2}_X
    \lesssim 2^{(-\frac{1}{2}+4\delta)m}\eps^2.
    \end{align*}
    \par\textbf{Case B.1(b): $2^{k_2}\ll 2^{k_1}\sim 2^{k}$.} Then $p\leq p_1\ll p_2=p_{\max}\sim 0$ and $k_{\max}+k_{\min}\sim k_2+k$. In this setting \eqref{eqn: ibp is working} for $S_{\xi-\eta}$ doesn't hold if
   $l_2>(1-\delta)m$. The claim follows from Lemma \ref{lemma: set size estimate} with $\abs{S}\lesssim 2^{k+\frac{p}{2}}$ and \eqref{eqn: bootstrap assumption}:
    \begin{align*}
        2^{4k^+-\frac{k^-}{2}}2^{-\frac{p}{2}}\ltwonorm{P_{k,p}\B_\m(f_1,f_2)}
        &\lesssim 2^{4k^+-\frac{k^-}{2}}2^m2^{2k}2^{-4k_1^++\frac{k_1^-}{2}}\norm{f_1}_B 
        2^{-(1+\beta)l_2}\norm{f_2}_X
        \lesssim 2^{-\frac{\beta}{2}m}\eps^2.
    \end{align*}
     The last estimate follows since $\beta\gg \delta= 2M^{-1/2}$.
    \par \textbf{Case B.2: $p_{\min}= p_1\ll p_{\max}$.} By Lemma \ref{lemma: case organisation}, we have three subcases to consider.
    \par \textbf{Case B.2(a): $2^{k_1}\lesssim 2^{k}\sim 2^{k_2}$.} Then $p_1\ll p\sim p_2\sim 0$ and $k_{\max}+k_{\min}\sim k_1+k$. If neither condition in \eqref{eqn: ibp is working} holds, we can assume  
    \begin{align*}
     &   l_1-p_1>(1-\delta)m && \text{and}&& \max\{k_2-k_1,l_2\}> (1-\delta)m.
    \end{align*}      
     First let $\max\{k_2-k_1,l_2\}=l_2> (1-\delta)m$. Then it follows from Lemma \ref{lemma: set size estimate} with $\abs{S}\lesssim 2^{k_1+\frac{p_1}{2}}$:
    \begin{align*}
        2^{4k^+}2^{-\frac{k^-}{2}}2^{-\frac{p}{2}}\ltwonorm{P_{k,p}\B_\m(f_1,f_2)}&\lesssim 2^{4k^+}2^{-\frac{k^-}{2}}2^m2^{k_1+\frac{p_1}{2}}2^k2^{-4k_1^+}2^{-\frac{l_1}{2} } \norm{f_1}_X 2^{-4k_2^+}2^{-(1+\beta)l_2}\norm{f_2}_X\\
        &\lesssim 2^{k^+}2^m2^{-(\frac{3}{2}+\beta)(1-\delta)m}\eps^2\\
       &\lesssim 2^{(-\frac{1}{2}-\frac{\beta}{2})m}\eps^2,
    \end{align*}
    since $\beta\gg \delta$. Now assume $\max\{k_2-k_1,l_2\}=k_2-k_1> (1-\delta)m$ and $l_1-p_1>(1-\delta)m$, then the claim \eqref{eqn: B norm claim} follows from Lemma \ref{lemma: set size estimate} with $\abs{S}\lesssim 2^{k_1+\frac{p_1}{2}}$:
    \begin{align*}
        2^{4k^+}2^{-\frac{k^-}{2}}2^{-\frac{p}{2}}\ltwonorm{P_{k,p}\B_\m(f_1,f_2)}&\lesssim 2^{4k^+}2^{-\frac{k^-}{2}}2^m 2^{k_1+\frac{p_1}{2}}2^{k_2}2^{-\frac{l_1}{2}}\norm{f_1}_X  2^{-4k_2^+}2^{\frac{k_2^-}{2}}\norm{f_2}_B\\
        &\lesssim2^{m}2^{-\frac{3}{2}(1-\delta)m}\eps^2\\
        &\lesssim 2^{(-\frac{1}{2}+2\delta)m}\eps^2.
    \end{align*}
        \par\textbf{Case B.2(b): $2^k\ll 2^{k_2}\sim2^{k_1}.$} Then there holds $p_1\leq p_2 \ll p\sim 0$, $k_{\max}+k_{\min}\sim k+k_2$. Moreover, by Lemma \ref{lemma: case organisation} there holds $2^{k}\lesssim 2^{k_2+p_2}$. Assume the second condition in \eqref{eqn: ibp is working} doesn't hold, that is $-l_2<-(1-\delta)m -p_2+k_2-k.$  Then we obtain the claim \eqref{eqn: B norm claim} from Lemma \ref{lemma: set size estimate} with $\abs{S}\lesssim 2^{\frac{k}{2}+\frac{k_2+p_2}{2}}$ and $2^{k}\lesssim 2^{k_2+p_2}$:
    \begin{align*}
         2^{4k^+}2^{-\frac{k^-}{2}}2^{-\frac{p}{2}}\ltwonorm{P_{k,p}\B_\m(f_1,f_2)}&\lesssim 2^{\frac{9}{2}k^+}2^m 2^{k}2^{\frac{k_2+p_2}{2}}2^{-8k_2^+}2^\frac{p_{1}}{2}\norm{f_1}_X2^{-\frac{5}{6}l_2}2^{-\frac{1}{3}p_2}\norm{f_2}_X\\
         &\lesssim 2^{-3k_2^+}2^m2^{k}2^{\frac{2}{3}p_2}2^{-\frac{5}{6}(1-\delta)m}2^{-\frac{5}{6}(p_2+k-k_2)}\eps^2\\
         &\lesssim 2^{(\frac{1}{6}+\delta)m}\eps^2.
    \end{align*}
    \par\textbf{Case B.2(c): $2^{k_2}\ll 2^{k}\sim 2^{k_1}$.} Then $p_1\leq p\ll p_2\sim 0$ and $k_{\max}+k_{\min}\sim k+k_2$.
     If $l_2>(1-\delta)m$ (cf. \eqref{eqn: ibp is working}) it follows from Lemma \ref{lemma: set size estimate} with $\abs{S}\lesssim 2^{k+\frac{p}{2}}$:
      \begin{align*}
         2^{4k^+}2^{-\frac{k^-}{2}}2^{-\frac{p}{2}}\ltwonorm{P_{k,p}\B_\m(f_1,f_2)}&\lesssim2^{-\frac{p}{2}}2^m 2^{2k+\frac{p}{2}}\norm{f_1}_B2^{-(1+\beta)l_2}\norm{f_2}_X\lesssim 2^{-\frac{\beta}{2}m}\eps^2.
    \end{align*}
This concludes \textbf{Case B}.
    \par \textbf{Case C: Gap in $p$ with $p_{\min}\ll p_{\max}\ll 0$}. Then $\abs{\Phi}\geq \frac{1}{10}$ and we can do a decomposition $\m=\m^{res}+\m^{nr}$ as presented in Section \ref{sec: Normal forms} with  $\lambda=\frac{1}{100}$. In this case, $\m^{res}=0$, and thus $\m=\m^{nr}$ with
      \begin{align*}
        \ltwonorm{P_{k,p}\B_\m(f_1,f_2)}\lesssim \ltwonorm{P_{k,p}\mathcal{Q}_{\m\Phi^{-1}}(f_1,f_2)}+\ltwonorm{P_{k,p}\B_{\m\Phi^{-1}}(\partial_tf_1,f_2)}+\ltwonorm{P_{k,p}\B_{\m\Phi^{-1}}(f_1,\partial_tf_2)}.
    \end{align*}
    We prove the claim \eqref{eqn: B norm claim} for the last two terms with Lemma \ref{lemma: set size estimate} with $\abs{S}\lesssim 2^{k+\frac{p}{2}}$ and Lemmas \ref{lemma: m phi^-1 multiplier bound} and \ref{lemma: bounds on dtSf}:
    \begin{align*}
        2^{4k^+} 2^{-\frac{k^-}{2}}2^{-\frac{p}{2}}\ltwonorm{P_{k,p}\B_{\m\Phi^{-1}}(\partial_tf_1,f_2)}&\lesssim 2^{\frac{9}{2}k^+}2^{-\frac{k}{2}}2^{-\frac{p}{2}}2^{k+\frac{p}{2}}2^{k}2^m\ltwonorm{\partial_tf_1}\ltwonorm{f_2}\\
        &\lesssim 2^{\frac{9}{2}k^+}2^{\frac{3}{2}k}2^m2^{(-\frac{3}{4}+2\delta) m}\eps^3\\
        &\lesssim 2^{(\frac{1}{4}+\frac{5}{2}\delta)m}\eps^3.
    \end{align*}
  The term containing $\partial_tf_2$ is bounded analogously. For the boundary term, by Lemma \ref{lemma: set size estimate} with $\abs{S}\lesssim 2^{k+\frac{p}{2}}$ we obtain
   \begin{align*}
        2^{4k^+}2^{-\frac{k^-}{2}}2^{-\frac{p}{2}}\ltwonorm{P_{k,p}\mathcal{Q}_{\m\Phi^{-1}}(f_1,f_2)}&\lesssim2^{\frac{9}{2}k^+}2^{-\frac{k}{2}}2^{-\frac{p}{2}}2^{k+\frac{p}{2}}2^{-4k_1^+}\norm{f_1}_B 2^{-4k_2^+}\norm{f_2}_B\\
        &\lesssim 2^{\delta m}\eps^2.
    \end{align*}
\par \textbf{Case D: No gaps with $p_1\sim p_2\sim p$.} Assume without loss of generality that $b_1\leq b_2$, that is $f_2$ has more vector fields and we can apply Proposition \ref{proposition: linear decay} on $f_1$ and Lemma \ref{lemma: linear decay many vector fields} on $f_2$. With the decomposition $P_{k_1,p_1}e^{it\Lambda}f_1=I_{k_1,p_1}(f_1)+II_{k_1,p_1}(f_1)$ the following decay bounds hold
\begin{align*}
    &\inftynorm{I_{k_1,p_1}(f_1)}\lesssim 2^{\frac{3}{4}k_1}2^{-\frac{15}{4}k_1^+}2^{-p}2^{(-1+\delta)m}\norm{f_1}_D, && \ltwonorm{II_{k_1,p_1}(f_1)}\lesssim 2^{-4k_1^+}2^{-\frac{p}{2}}2^{-\frac{m}{2}}\norm{f}_D,\\
    &\sinftynorm{P_{k_2,p_2}e^{it\Lambda}f_2}\lesssim 2^{\frac{3}{4}k_2}2^{-3k_2^+}2^{(-\frac{1}{2}+\kappa)m}\eps.
\end{align*}
Then for the $B$-norm there holds
\begin{align*}
    &2^{4k^+}2^{-\frac{k^-}{2}}2^{-\frac{p}{2}}\ltwonorm{P_{k,p}\B_\m(f_1,f_2)}\\
    &\hspace{6em}\lesssim 2^{\frac{9}{2}k^{+}}2^m2^{\frac{p}{2}+\frac{k}{2}}\big[\inftynorm{I_{k_1,p_1}(f_1)}\ltwonorm{P_{k_2,p_2}f_2}+\ltwonorm{II_{k_1,p_1}(f_1)}\inftynorm{e^{it\Lambda}P_{k_2,p_2}f_2}\big]\\
    &\hspace{6em}\lesssim 2^{5k^{+}}2^{\frac{p}{2}}2^m[2^{-3k_1^+}2^{-p}2^{(-1+\delta)m}2^{-4k_2^+}2^{\frac{k_2^-}{2}}2^{\frac{p}{2}} +2^{-4k_1^+}2^{-\frac{p}{2}}2^{-\frac{m}{2}}2^{\frac{3}{4}k_2}2^{-3k_2^+}2^{(-\frac{1}{2}+\kappa) m}]\eps^2\\
    &\hspace{6em}\lesssim 2^{4k^+}(2^{\delta m}+2^{\kappa m})\eps^2\\
    &\hspace{6em}\lesssim 2^{\beta m}\eps^2,
\end{align*}
where $\kappa\ll \beta$ in Lemma \ref{lemma: linear decay many vector fields}. This is an admissible contribution and finishes the proof of the proposition.

\end{proof}
\section{Bounds on the \texorpdfstring{$X$}{X}-norm}\label{sec: bounds on the X-norm}
We prove the $X$-norm bounds in two steps depending on the size of $l$ relative to $m$.
\subsection{\texorpdfstring{$X$}{X}-norm bounds for \texorpdfstring{$l>(1+\delta)m$}{l-bigger-than-m}}\label{sec:X-norm-1} 
\begin{proposition}\label{prop: X-norm bounds for l>(1+delta)m}
   In the setting of Proposition \ref{prop: bootstrap argument}, and in particular under the bootstrap assumptions \eqref{eqn: bootstrap assumption}, the following holds true: For $\m\in \set{\m_0,\m_\pm^{\mu\nu}}{\mu,\nu\in\{+,-\}}$, $t\in [2^m,2^{m+1}[\cap [0,T]$ and $\delta=2M^{-\frac{1}{2}}$, there holds that
    \begin{align*}
        \sup_{\substack{k,l,p\\ l+p\geq 0, l>(1+\delta)m}}2^{4k^+}2^{(1+\beta)l}2^{\beta p}2^{\frac{p}{2}}\ltwonorm{P_{k,p}R_l\B_\m(F_1,F_2)}\lesssim 2^{(\frac{1}{2}-\frac{\delta}{10}) m }\eps^2+2^{(\frac{1}{4}+7\beta)m}\eps^3,
    \end{align*}
     where $F_i\in \{S^{b_i}\mathcal{Z}_{\pm},S^{b_i}\Theta\}$, $0\leq b_1+b_2\leq N$, $i=1,2$.
\end{proposition}
\blue{This result corresponds to an approximate ``finite speed of propagation'' and implies that the time parameter $m$ is the largest one in the problem. The proof of Proposition \ref{prop: X-norm bounds for l>(1+delta)m} is structured in two parts. In \textbf{Part 1}, we assume additionally that $l+p<\delta m$. Here we note that the weight in the norm is bounded by a factor $2^{-p/2}$, which is similar to the $B$-norm. Thus, the claim \eqref{eqn: X norm claim} follows via energy estimates, integration by parts along $S$ and normal forms. In \textbf{Part 2}, where $l+p>\delta m$, to overcome the ``large'' parameter $l$ we invoke the properties of $W$ on bilinear terms, as established in Lemmas \ref{lemma: W_xi on elementary multipliers} and \ref{lemma: W_xi Phi}. Combined with the iterated action of $W$ on bilinear terms as developed in Lemma \ref{lemma: W_xi on bilinear expressions} to ``gain'' negative $l$ parameters in the bilinear terms $\mathcal{B}_{\m}$, we can restrict the possible size of $l$ in terms of other parameters. In the remaining cases we can rely on the well-established tools: integration by parts, normal forms and linear decay.}

\begin{proof} We split the proof in two main parts.
\par    \textbf{Part 1: $l+p<\delta m$.} Similarly to the proof of Proposition \ref{prop: bounds on B norm} we want to show that the energy estimates that we get from the bootstrap assumption allow us to restrict the range of the localisation parameters.
    \par \textbf{Case A: Simple cases.} Using the set size estimate Lemma \ref{lemma: set size estimate} with $\abs{S}\lesssim 2^{\frac{k_{\min}+k+p}{2}}$ and the bootstrap assumption \eqref{eqn: bootstrap assumption} we have
 \begin{align*}
        2^{4k^+}2^{(1+\beta)l}2^{\beta p}2^{\frac{p}{2}}\ltwonorm{P_{k,p}R_l\B_{\m}(F_1,F_2)}&\lesssim 2^{4k^+}2^{(1+\beta)l}2^{\beta p}2^{\frac{p}{2}}2^m\abs{S}\inftynorm{\m\chi}\ltwonorm{P_{k_1,p_1}F_1}\ltwonorm{P_{k_2,p_2}F_2}\\
        &\lesssim 2^{(1+2\delta )m}2^{\frac{k_{\min}}{2}}2^{-(N_0-5)(k_1^++k_2^+)}\norm{P_{k_1}F_1}_{H^{N_0}}\norm{P_{k_2}F_2}_{H^{N_0}}
    \end{align*}
Therefore, we can assume $-4m\leq k,k_1,k_2\leq \delta_0 m$, with $\delta_0:=2N_0^{-1}\ll \delta$. Localizing further in $p_i,\; l_i$ and letting $f_i=P_{k_i,p_i}R_{l_i}F_i$, we can restrict the $l_i,\; p_i$ parameters using Lemma \ref{lemma: set size estimate} with $\abs{S}\lesssim 2^{k+\frac{p}{2}}$, and the bootstrap assumption \eqref{eqn: bootstrap assumption}:
    \begin{align*}
        2^{4k^+}2^{(1+\beta)l}2^{\beta p}2^{\frac{p}{2}}\ltwonorm{P_{k,p}R_l\B_\m(F_1,F_2)}
        &\lesssim 2^{(1+(1+\beta)\delta+5\delta_0)m}2^{-\frac{l_1}{2}-\frac{l_2}{2}}2^{-\frac{l_1+p_1}{2}}2^{-\frac{l_2+p_2}{2}}\norm{f_1}_X\norm{f_2}_X\\
       &\lesssim 2^{(1+2\delta)m}2^{-\frac{l_1}{2}-\frac{l_2}{2}}2^{-\frac{l_1+p_1}{2}}2^{-\frac{l_2+p_2}{2}}\eps^2.
    \end{align*}
    Hence, the $X$-norm remains bounded if $\max\{l_1,l_2\}\geq 4m$. Thus, analogous to \textbf{Case A} in the proof of Lemma \ref{prop: bounds on B norm} it suffices to prove 
    \begin{align}\label{eqn: X norm claim}
        \sup_{k,l+p<\delta m,l>(1+\delta)m} 2^{4k^+}2^{(1+\beta)l}2^{\beta p}2^{\frac{p}{2}}\ltwonorm{P_{k,p}R_l\B_{\m}(f_1,f_2)}\lesssim 2^{3\delta m}\eps^2+ 2^{(\frac{1}{4}+4\delta)m}\eps^3,
    \end{align}
    for the following localisation parameters
    \begin{align}\label{eqn: localisation parameter assumptions X norm}
 &   -4m<k,k_i<\delta_0 m , && -p_i\leq l_i\leq 4m, && -4m\leq p_i\leq 0, && i=1,2. 
    \end{align} 
    Observe that $2^p\ll 1$ since $l+p<\delta m$ and $l>(1+\delta)m$. Hence we have the following two cases to consider. 
    \par \textbf{Case B: $2^{p_1}+2^{p_2}\ll 1$. }
        Then $\abs{\Phi}> \frac{1}{10}$ and we can do a splitting of the multiplier in the resonant and non-resonant parts as in Section \ref{sec: Normal forms} with $\lambda=\frac{1}{100}$. Observe that $\m^{res}=0$ and so $\m=\m^{nr}$ with 
       \begin{align*}
              \ltwonorm{P_{k,p}R_l\B_{\m}(f_1,f_2)}&\lesssim \ltwonorm{P_{k,p}R_l\mathcal{Q}_{\m\Phi^{-1}}(f_1,f_2)}+\ltwonorm{P_{k,p}R_l\B_{\m\Phi^{-1}}(\partial_tf_1,f_2)}\\
              &\quad+ \ltwonorm{P_{k,p}R_l\B_{\m\Phi^{-1}}(f_1,\partial_tf_2)}.
       \end{align*}
We bound each term in the decomposition with Lemmas \ref{lemma: multiplier bound}, \ref{lemma: bounds on dtSf}, \ref{lemma: m phi^-1 multiplier bound} and \ref{lemma: set size estimate} with $\abs{S}\lesssim 2^{k+\frac{p}{2}}$:
    \begin{align*}
        2^{4k^+}2^{(1+\beta)l}2^{\beta p}2^{\frac{p}{2}}\ltwonorm{P_{k,p}R_l\B_{\m\Phi^{-1}}(\partial_tf_1,f_2)}&\lesssim 2^{4k^+}2^{(1+(1+\beta)\delta) m}2^{2k}\ltwonorm{\partial_tf_1}\ltwonorm{f_2}\\
        &\lesssim 2^{(1+(1+\beta)\delta+6\delta_0)m}2^{(-\frac{3}{4}+2\delta)m}\eps^3\\
        &\lesssim 2^{(\frac{1}{4}+4\delta) m}\eps^3.
    \end{align*}
      The same holds by symmetry for the last term in the splitting above. Now it remains to estimate the boundary term. From Lemma \ref{lemma: set size estimate} with $\abs{S}\lesssim 2^{k+\frac{p}{2}}$ and \eqref{eqn: bootstrap assumption} we obtain
    \begin{align*}
        2^{4k^+}2^{(1+\beta)l}2^{\beta p}2^{\frac{p}{2}}\ltwonorm{P_{k,p}\mathcal{Q}_{\m\Phi^{-1}}(f_1,f_2)}&\lesssim2^{4k^+}2^{(1+\beta)(l+p)}2^{2k}2^{-4k_1^+-4k_2^+}\norm{f_1}_B\norm{f_2}_B\lesssim 2^{3\delta m}\eps^2,
    \end{align*}
    which gives the claim and closes \textbf{Case B}.
    \par \textbf{Case C: $\max\{2^{p_1},2^{p_2}\}\sim 1$.} Assume w.l.o.g.\ that $p_2\leq p_1$, so that $p_1=p_{\max}$ and then $k_{\max}\in \{k_2,k\}$. Thus by Proposition \ref{prop: lower bound on sigma} there holds $\abs{\sigma}\sim 2^{k_1+k}$. Integration by parts along $S_\eta$ gives the claim analogously to \eqref{eqn: ibp is working}-\eqref{eqn: claim for B norm after ibp} if
   $l_1+k_2-k\leq (1-\delta)m$ and $\delta=2M^{-\frac{1}{2}}$. Otherwise, if
    \begin{align}\label{eqn: Case A.3 X norm}
        -l_1-k_2+k<-(1-\delta)m,
    \end{align}
    by Lemma \ref{lemma: set size estimate} with $\abs{S}\lesssim 2^{k+\frac{p}{2}}$ and the bootstrap assumption \eqref{eqn: bootstrap assumption} there holds:
    \begin{align*}
        2^{4k^+}2^{(1+\beta)l}2^{\beta p}2^{\frac{p}{2}}\ltwonorm{P_{k,p}R_l\B_{\m}(f_1,f_2)}&\lesssim 2^{4k^+}2^{(1+\beta)(l+p)}2^m2^{2k}2^{-4k_1^+}2^{-(1+\beta)l_1}\norm{f_1}_X2^{-4k_2^+}\norm{f_2}_B\\
        &\lesssim 2^{(-\beta+(1+\beta)\delta )m}2^{2k}2^{4(k^+-k_1^+-k_2^+)}2^{(1+\beta)(k_2-k)}\eps^2\\
        &\lesssim 2^{-\frac{\beta}{2}m}\eps^2,
    \end{align*}
   since $\delta_0\ll\delta\ll \beta$. This finishes the proof of \textbf{Part 1}.
    \par \textbf{Part 2: $l+p>\delta m$.} 
    \par \textbf{Case A: Simple cases.} As in \textbf{Part 1}, we can restrict the localisation parameters.
        Using Lemma \ref{lemma: set size estimate}, the energy estimates and $l>(1+\delta)m$, we can bound
        \begin{align*}
          2^{4k^+}2^{(1+\beta)l}2^{\beta p}2^{\frac{p}{2}}\ltwonorm{P_{k,p}R_l\B_\m (f_1,f_2)}&\lesssim 2^{5\max\{k_1^+,k_2^+\}}2^{(2+\beta)l}2^{-\delta m}2^{k_{\min}}2^{-N_0k_1^+-N_0k_2^+}\eps^2.
        \end{align*} 
        Hence we obtain the claim if $k_{\min}\leq -3l$ or if $k_{\max}\geq \delta_0 l $, with $\delta_0=3N_0^{-1}$. Localising further in $p_i,l_i$, $i=1,2$ and estimating using the $X-$norm, we obtain the claim if $\max\{l_1,l_2\}\geq (4+4\beta)l$:
        \begin{align*}
            2^{4k^+}2^{(1+\beta)l}2^{\beta p}2^{\frac{p}{2}}\ltwonorm{P_{k,p}R_l\B_\m (f_1,f_2)}&\lesssim 2^{(2+\beta)l}2^{-\delta m}2^{k+k_{\min}}2^{-\frac{l_1+l_2}{2}}\norm{f_1}_X\norm{f_2}_X\lesssim 2^{-\delta m}2^{(-\beta+2\delta_0) l}\eps^2.
        \end{align*}
        In the following, we will prove 
        \begin{align*}
            \sup_{k,l+p>\delta m,l>(1+\delta)m} 2^{4k^+}2^{(1+\beta)l}2^{\beta p}2^{\frac{p}{2}}\ltwonorm{P_{k,p}R_l\B_\m (f_1,f_2)}\lesssim 2^{(\frac{1}{2}-\frac{\delta}{8}) m}\eps^2+2^{(\frac{1}{4}+6\beta)m}\eps^3,
        \end{align*}
   for the parameters
    \begin{align*}
        -3l\leq k,k_i\leq \delta_0l, && -(4+4\beta)l\leq p_i \leq 0, && -p_i\leq l_i\leq (4+4\beta)l, && i=1,2.
    \end{align*}
    \par \textbf{Case B.} 
    We employ Lemma \ref{lemma: W_xi on bilinear expressions} with $N\in \N$ such that $N\delta^2>(2+\beta)$. From this we see that if
        \begin{align}\label{eqn: condition so that W_xi on bilinear expression works}
        2^m(2^p+2^{k-k_i}2^{p_i})+2^{k-k_i+l_i}< 2^{(1-\delta^2)l}, \qquad i=1 \text{ or }i=2,
    \end{align}
 we obtain an acceptable bound on the $X$-norm:
    \begin{align*}
        2^{4k^+}2^{(1+\beta)l}2^{\beta p}2^{\frac{p}{2}}\ltwonorm{P_{k,p}R_l\B_\m (f_1,f_2)}&\lesssim 2^{6\delta_0m}2^{(2+\beta)l}[ 2^{-N\delta^2l}2^{-N\delta m} +2^{-3l}]\eps^2\lesssim 2^{-\delta^2l}2^{5\delta m}\eps^2.
    \end{align*}
     Assume w.l.o.g.\ that $k_2\leq k_1$ and recall $2^{m}\lesssim 2^{l-\delta m}$, then condition \eqref{eqn: condition so that W_xi on bilinear expression works} (and thus the claim) holds if 
    \begin{align}\label{eqn: Bernestein works}
        l_1\leq (1-\delta^2)l \quad \text{or} \quad k-k_2 +\max\{m+p_2,l_2\}\leq (1-\delta^2)l.
    \end{align}
    Otherwise if \eqref{eqn: Bernestein works} doesn't hold, we distinguish two further cases: $m+p_2\leq l_2$ or $m+p_2>l_2$.
    \par \textbf{B.1: $m+p_2\leq l_2$, $l_1>(1-\delta^2)l$ and $k-k_2+l_2>(1-\delta^2)l$.} We proceed with the by now standard scheme of proof.
    \par \textbf{B.1(a): Gap in $p$: $p_{\min}\ll p_{\max}$.} Based on Lemma \ref{lemma: case organisation} and with the constraint $k_2\leq k_1$, we have the following cases:
    \par\textbf{B.1(a.1): $p_{\min}= p\ll p_{\max}$.} From Lemma \ref{lemma: case organisation} we have two settings for the $k,k_1,k_2$ parameters. 
    If $2^{k_1}\sim 2^{k_2}$, then $p\ll p_1\sim p_2=p_{\max}$ and $2^{k}\lesssim 2^{k_2}$. From Lemma \ref{lemma: set size estimate} with $\abs{S}\lesssim 2^{k_2+\frac{p_2}{2}}$ we bound
 \begin{align*}
      2^{4k^+}2^{(1+\beta)l}2^{\beta p}2^{\frac{p}{2}}\ltwonorm{P_{k,p}R_l\B_\m (f_1,f_2)}&\lesssim 2^{-\delta m}2^{(2+\beta)l}2^{(2+\beta)p_2}2^{2k_2}2^{-(1+\beta)(l_1+l_2)}2^{-(1+2\beta)p_2}\norm{f_1}_X\norm{f_2}_X\\
      &\lesssim 2^{-\delta m}2^{(-\beta+3\delta^2)l}2^{(1-\beta)k_2}2^{(1+\beta)k}\eps^2\\
      &\lesssim 2^{-\frac{\delta}{2}m}2^{-\frac{\beta}{2}l}\eps^2.
 \end{align*}
 If on the other hand $2^{k_2}\ll 2^{k_1}\sim 2^{k}$, then $p\leq p_1\ll p_2$ and moreover $2^{m}\lesssim2^{\frac{l_2-p_2}{2}+\frac{l-\delta m}{2}}$. Thus from Lemma \ref{lemma: set size estimate} with $\abs{S}\lesssim 2^{k_2+\frac{p_2}{2}}$ and using the $X$-norms on $f_1,\,f_2$ we have 
    \begin{align*}
        2^{4k^+}2^{(1+\beta)l}2^{\beta p}2^{\frac{p}{2}}\ltwonorm{P_{k,p}R_l\B_\m (f_1,f_2)}&\lesssim 2^{-\frac{\delta}{2}m}2^{(\frac{3}{2}+\beta)l}2^{k+k_2}2^{\frac{3}{2}p_{2}} 2^{-(1+\beta)l_1}2^{-4k_2^+}2^{-(\frac{1}{2}+\beta)l_2}2^{-(1+\beta) p_2}\eps^2\\
        &\lesssim 2^{-\frac{\delta}{2}m}2^{(-\beta+2\delta^2)l}2^{(\frac{1}{2}-\beta)k_2}2^{(\frac{1}{2}+\beta)k}\eps^2\\
        &\lesssim 2^{-\frac{\delta}{4}m}2^{-\frac{\beta}{2}l}\eps^2,
    \end{align*}
    since $\delta_0\ll \delta\ll \beta$.
 
    \par \textbf{B.1(a.2): $p_{\min}\sim p_1 \ll p_{\max}$.} If
    $2^{k}\sim 2^{k_2}$, then $2^p\sim 2^{p_2}$ and so $p_1\ll p \sim p_2$. Moreover, ${-l_2}< {-(1-\delta^2)l}$ and by Lemma \ref{lemma: set size estimate} with $\abs{S}\lesssim 2^{k_1+\frac{p_1}{2}}$ we obtain :
 \begin{align*}
         &2^{4k^+}2^{(1+\beta)l}2^{\beta p}2^{\frac{p}{2}}\ltwonorm{P_{k,p}R_l\B_\m (f_1,f_2)}\\
         &\hspace{10em}\lesssim2^m 2^{(1+\beta)l}2^{(\frac{3}{2}+\beta)p_2}2^{k+k_1+\frac{p_1}{2}}2^{-4k_1^+}2^{-l_1-\frac{p_1}{2}}\norm{f_1}_X  2^{-(1+\beta)l_2}2^{-(\frac{1}{2}+\beta)p_2}\norm{f_2}_X\\
         &\hspace{10em}\lesssim 2^{(\beta+\delta_0) m}2^{\beta p_2}2^{(-\beta+2\delta^2)l}\eps^2.
    \end{align*}
      Next, if $2^{k}\ll 2^{k_2}\sim 2^{k_1}$, then $p_1\leq p_2\ll p$ and $2^{p+k}\lesssim 2^{p_2+k_2}$. Moreover, ${-l_2}<{-(1-\delta^2)l}$ and we estimate with $\abs{S}\lesssim 2^{k_2+\frac{p_1}{2}}$:
    \begin{align*}
         &2^{4k^+}2^{(1+\beta)l}2^{\beta p}2^{\frac{p}{2}}\ltwonorm{P_{k,p}R_l\B_\m (f_1,f_2)}\\
         &\hspace{10em}\lesssim 2^{ m}2^{(1+\beta)l}2^{(\frac{3}{2}+\beta)p_2}2^{k}2^{k_2+\frac{p_1}{2}} 2^{-4k_2^+}2^{-l_1}2^{-\frac{p_1}{2}} \norm{f_1}_X2^{-(1+\beta)l_2}2^{-(\frac{1}{2}+\beta)p_2}\norm{f_2}_X\\
         &\hspace{10em}\lesssim  2^{m}2^{(-1+3\delta^2)l}2^{p_2}\eps^2\\
         &\hspace{10em}\lesssim 2^{(\beta-\frac{\delta}{2}) m}2^{(-\beta+3\delta^2)l}\eps^2.
         \end{align*}
 Finally, if $2^{k_2}\ll 2^{k}\sim 2^{k_1}$, then  $p_1\leq p\ll p_2$ and we obtain with $\abs{S}\lesssim 2^{\frac{k+k_2}{2}+\frac{p_1}{2}}$:
    \begin{align*}
        &2^{4k^+}2^{(1+\beta)l}2^{\beta p}2^{\frac{p}{2}}\ltwonorm{P_{k,p}R_l\B_\m (f_1,f_2)}\\
        &\hspace{10em}\lesssim 2^{(\frac{1}{2}-\frac{\delta}{2})m}2^{(\frac{3}{2}+\beta)l}2^{\frac{3}{2}k}2^{k_{2}+\frac{p_{1}}{2}}2^{-l_1}2^{-\frac{p_1}{2}}\eps2^{-\frac{2}{3}l_2}2^{-\frac{1}{3}p_2}\norm{f_2}_X^{\frac{2}{3}} 2^{\frac{k_2^-}{6}}2^{\frac{p_2}{6}}\norm{f_2}_B^{\frac{1}{3}}\\
        &\hspace{10em}\lesssim 2^{(\frac{1}{2}-\frac{\delta}{2})m}2^{(-\frac{1}{6}+\beta+2\delta^2)l}2^{\frac{3}{2}k}2^{\frac{2}{3}k_2}2^{-\frac{2}{3}(k_2-k)}\eps^2\\
        &\hspace{10em}\lesssim 2^{(\frac{1}{2}-\frac{\delta}{4})m}2^{(-\frac{1}{6}+2\beta)l}\eps^2,
    \end{align*}
    which is an acceptable contribution.
         \par\textbf{B.1(a.3): $p_{\min}= p_2\ll p_{\max}$.}
         Then by Lemma \ref{lemma: case organisation} we have two possibilities for the parameters $k,k_1,k_2$. First, if $2^k\sim 2^{k_1}$, then  $p_2\ll p\sim p_1$ and $2^{-l_2}\lesssim 2^{-(1-\delta^2)l}$. Using Lemma \ref{lemma: set size estimate} with $\abs{S}\lesssim 2^{k_2+\frac{p_2}{2}}$ and the $X$-norms of $f_1,f_2$, we have
         \begin{align*}
              2^{4k^+}2^{(1+\beta)l}2^{\beta p}2^{\frac{p}{2}}\ltwonorm{P_{k,p}R_l\B_\m (f_1,f_2)}&\lesssim 2^{(\beta -\frac{\delta}{2})m}2^{2l}2^{(\frac{1}{2}+\beta)p_1}2^{\frac{p_{2}}{2}} 2^{-(1+\beta)l_1}2^{-(\frac{1}{2}+\beta)p_1}2^{-l_2}2^{-\frac{p_2}{2}}\eps^2\\
              &\lesssim 2^{(\beta-\frac{\delta}{2}) m}2^{(-\beta+3\delta^2)l}\eps^2.
         \end{align*}
          If on the other hand $2^k\ll 2^{k_1}\sim 2^{k_2}$, then $2^{p_1}\ll 2^p$ and hence $p_2\leq p_1\ll p$ and $2^{p+k}\lesssim 2^{p_1+k_2}$. With $\abs{S}\lesssim 2^{k_2+\frac{p_2}{2}}$ there holds
          \begin{align*}
              2^{4k^+}2^{(1+\beta)l}2^{\beta p}2^{\frac{p}{2}}\ltwonorm{P_{k,p}R_l\B_\m (f_1,f_2)}&\lesssim 2^{(\beta-\frac{\delta}{2})m}2^{2l}2^{(\frac{1}{2}+\beta)p}2^{k+k_2}2^{-(1+\beta)l_1}2^{-(\frac{1}{2}+\beta)p_1}2^{-4k_2^+}2^{-l_2}\eps^2\\
              &\lesssim 2^{(\beta-\frac{\delta}{2})m}2^{(-\beta+3\delta^2)l}2^{(\frac{1}{2}+\beta)(k_2-k)}2^{2k}2^{-4k_2^+}\\
              &\lesssim 2^{(\beta -\frac{\delta}{2}) m}2^{(-\beta+3\delta^2)l}\eps^2.
          \end{align*}
            \par \textbf{B.1(b): No gap in $p$: $p\sim p_1\sim p_2$.} With Lemma \ref{lemma: set size estimate} and $m\leq l_2-p$ we obtain
          \begin{align*}
              2^{4k^+}2^{(1+\beta)l}2^{\beta p}2^{\frac{p}{2}}\ltwonorm{P_{k,p}R_l\B_\m (f_1,f_2)}&\lesssim 2^{m}2^{(1+\beta)l}2^{(2+\beta)p}2^{k+k_{\min}} 2^{-(1+\beta)l_1}2^{-(1+\beta)l_2}2^{-(1+2\beta)p}\eps^2\\
              &\lesssim 2^{\beta m}2^{(1-\beta)p}2^{2\delta^2l}2^{(1-\beta)p}2^{k+k_{\min}}2^{-2\beta l_2}\eps^2\\
              &\lesssim 2^{2\beta m}2^{-\beta l }\eps^2.
          \end{align*}
          This finishes \textbf{Case B.1}.
          \par \textbf{B.2: $l_2<m+p_2$, $l_1>(1-\delta^2)l$ and $k-k_2+m+p_2>(1-\delta^2)l$.} This is the other possibility if \eqref{eqn: Bernestein works} doesn't hold. Here we have
          \begin{align}\label{eqn: l+p/2>delta m conditions on l2 l1 in the last case}
              (1+\delta)m<l<(1+2\delta^2)(k-k_2+m+p_2),\; -l_1<-(1-\delta^2)l, \; -p_2-k+k_2<-\frac{\delta}{2}m.
          \end{align}
         \par \textbf{B.2(a): No gaps with $p\sim p_1\sim p_2$.}
          An $L^2-L^\infty$ estimate using Lemma \ref{lemma: linear decay many vector fields} on $f_2$ and \eqref{eqn: l+p/2>delta m conditions on l2 l1 in the last case} gives:
        \begin{align*}
            2^{4k^+}2^{(1+\beta)l}2^{\beta p}2^{\frac{p}{2}}\ltwonorm{P_{k,p}R_l\B_\m (f_1,f_2)}&\lesssim 2^{m}2^{(1+\beta)l}2^{(\frac{3}{2}+\beta)p}2^{k}2^{-(1+\beta)l_1}2^{-(\frac{1}{2}+\beta)p}2^{\frac{3}{4}k_2}2^{(-\frac{1}{2}+\kappa)m}\eps^2\\
            &\lesssim 2^{(\frac{1}{2}+\kappa)m}2^{2\delta^2(1+2\delta^2)(k-k_2+m+p)}2^{p}2^{k+\frac{3}{4}k_2}\eps^2\\
            &\lesssim 2^{(\frac{1}{2}+\kappa+3\delta^2)m}2^{(\frac{3}{4}-3\delta^2)k_2}2^{(1+3\delta^2)k}\eps^2\\
            &\lesssim 2^{(\frac{1}{2}-\frac{\delta}{8})m}\eps^2,
        \end{align*}
     since $\delta_0\ll \delta^2\ll \delta$ and we can choose $\kappa\ll \delta^2\ll \beta$. Note that in the last step we used the third condition in \eqref{eqn: l+p/2>delta m conditions on l2 l1 in the last case} on $k_2$.
        \par\textbf{B.2(b): Gap in $p$: $p_{\min}\ll p_{\max}$.}
         \par We can integrate by parts along $S_{\xi-\eta}$ via Lemma \ref{lemma: ibp in bilinear expressions}\eqref{it1: lemma: ibp in bilinear expressions} and using \eqref{eqn: l+p/2>delta m conditions on l2 l1 in the last case}, obtain the claim if  
          \begin{align}\label{eqn: ibp last step l+p>delta m}
              \max\{k_2-k-p_2-p_{\max} ,-p_2-p_{\max}+l_2\}\leq (1-\delta)m.
          \end{align}
        Assume now that \eqref{eqn: ibp last step l+p>delta m} doesn't hold. Still \eqref{eqn: l+p/2>delta m conditions on l2 l1 in the last case} holds and we proceed with two cases depending on which term on the right-hand side of \eqref{eqn: ibp last step l+p>delta m} is the largest.
       If  $k_2-k-p_2-p_{\max}> (1-\delta)m$, then we obtain with $\abs{S}\lesssim 2^{\frac{k+k_2}{2}+\frac{p_2}{2}}$ in Lemma \ref{lemma: set size estimate}:
        \begin{align*}
            2^{4k^+}2^{(1+\beta)l}2^{\beta p}2^{\frac{p}{2}}\ltwonorm{P_{k,p}R_l\B_\m (f_1,f_2)}&\lesssim 2^{m}2^{(1+\beta)l}2^{(\frac{3}{2}+\beta)p_{\max}}2^{\frac{3}{2}k}2^{\frac{k_2}{2}+\frac{p_{2}}{2}}2^{-\frac{l_1}{2}}\norm{f_1}_X2^{\frac{k_2}{2}}2^{\frac{p_2}{2}}\norm{f_2}_B\\
            &\lesssim 2^{m}2^{(\frac{1}{2}+\beta+\frac{\delta^2}{2})(1+2\delta^2)(k-k_2+m+p_2)}2^{\frac{3}{2}k+k_2}2^{\frac{3}{2}p_{\max}+p_2}\eps^2\\
            &\lesssim 2^{(\frac{3}{2}+\beta+2\delta^2)m}2^{(2+\beta+2\delta^2)k}2^{(\frac{1}{2}-\beta-2\delta^2)k_2}2^{(\frac{3}{2}+\beta)(p_2+p_{\max})}\eps^2\\
            &\lesssim 2^{\delta m}\eps^2.
        \end{align*}
        If on the other hand there holds $-p_2-p_{\max}+l_2> (1-\delta)m$, we have two settings.
       
         \par \textbf{B.2(b.1): Gap in $p$ with $p_{\max}\sim 0$.}
 If $2^{p_1}\sim 1$, then by Lemma \ref{lemma: set size estimate} with $\abs{S}\lesssim 2^{k_2+\frac{p_2}{2}}$ and \eqref{eqn: l+p/2>delta m conditions on l2 l1 in the last case} we obtain:
        \begin{align*}
             2^{4k^+}2^{(1+\beta)l}2^{\beta p}2^{\frac{p}{2}}\ltwonorm{P_{k,p}R_l\B_{\m} (f_1,f_2)}&\lesssim 2^m2^{(1+\beta)l}2^{k}2^{k_2+\frac{p_{2}}{2}} 2^{-(1+\beta)l_1}\norm{f_1}_X2^{-l_2}2^{-\frac{p_2}{2}}\norm{f_2}_X\\
             &\lesssim 2^{\delta m} 2^{3\delta^2(k-k_2+m+p_2)}2^{k+k_2}2^{-p_2}\eps^2\\
             &\lesssim2^{(\delta+3\delta^2)m}2^{(1+3\delta^2)k}2^{(1-3\delta^2)k_2}2^{-(1-3\delta^2)p_2}\eps^2\\
             &\lesssim 2^{2\delta m}\eps^2.
        \end{align*}
     Now let $2^{p_2}\sim 1$, then from \eqref{eqn: l+p/2>delta m conditions on l2 l1 in the last case} we have $-l_2<-(1-\delta)m$. Using Lemma \ref{lemma: set size estimate} with $\abs{S}\lesssim 2^{\frac{k_2+k_1+p_1}{2}}$ we obtain:
        \begin{align*}
             2^{4k^+}2^{(1+\beta)l}2^{\beta p}2^{\frac{p}{2}}\ltwonorm{P_{k,p}R_l\B_{\m} (f_1,f_2)}&\lesssim 2^m2^{(1+\beta)l}2^{k}2^{\frac{k_2+k_1+p_{1}}{2}} 2^{-l_1}2^{-\frac{p_1}{2}}2^{-(1+\beta)l_2}\eps^2\\
             &\lesssim 2^{(-\beta+2\delta)m}2^{(\beta+2\delta^2)(k-k_2+m)}2^{k+\frac{k_1+k_2}{2} } \eps^2\\
             &\lesssim 2^{5\delta m}\eps^2.
        \end{align*}
        Finally, if $2^{p}\sim 1$, with Lemma \ref{lemma: case organisation} and with the constraint $k_2\leq k_1$ we have just two possibilities: either $2^{k_2}\ll 2^{k}\sim 2^{k_1}$ and $p_2\ll p\sim p_1\sim 0$ which was handled above, or $2^{k}\ll 2^{k_2}\sim 2^{k_1}$ which implies $p_1\leq p_2\ll p\sim 0$ and in particular $2^{k}\lesssim 2^{k_2+p_2}$. From this the claim follows using Lemma \ref{lemma: set size estimate} with $\abs{S}\lesssim 2^{k_2+\frac{p_1}{2}}$ and \eqref{eqn: l+p/2>delta m conditions on l2 l1 in the last case}:
        \begin{align*}
            2^{4k^+}2^{(1+\beta)l}2^{\beta p}2^{\frac{p}{2}}\ltwonorm{P_{k,p}R_l\B_{\m} (f_1,f_2)}&\lesssim 2^m2^{(1+\beta)l}2^{k}2^{k_2+\frac{p_1}{2}}2^{-l_1}2^{-\frac{p_1}{2}}\norm{f_1}_X2^{-l_2} 2^{-\frac{p_2}{2}}\norm{f_2}_X\\
            &\lesssim 2^m2^{(\beta+3\delta^2)(k-k_2+m+p_2)}2^{2k_2+p_2}2^{-(1-\delta)m}2^{-\frac{3}{2}p_2}\eps^2\\
            &\lesssim 2^{(\beta+\delta+3\delta^2)m}2^{(\beta+2\delta^2)k}2^{(2-\beta-2\delta^2)k_2}2^{-\frac{p_2}{2}}\eps^2\\
            &\lesssim 2^{2\beta m}\eps^2.
        \end{align*}
        \par \textbf{B.2(b.2): Gap in $p$ with $p_{\min}\ll p_{\max}\ll 0$.} Then $\abs{\Phi}>\frac{1}{10}$ and we split the analysis in the resonant and non-resonant parts as presented in Section \ref{sec: Normal forms}. By choosing $\lambda=\frac{1}{100}$, we have $\m^{res}=0$ and so we can do a normal form as in Lemma \ref{lemma: normal forms} with $\m^{nr}=\m$. For the boundary term on the right-hand side of \eqref{eqn: nonresonant multiplier bilinear estimate} there holds using \eqref{eqn: l+p/2>delta m conditions on l2 l1 in the last case} and $\abs{S}\lesssim 2^{\frac{k_2+k_1+p_1}{2}}$
         \begin{align*}
             2^{4k^+}2^{(1+\beta)l}2^{\beta p}2^{\frac{p}{2}}\ltwonorm{P_{k,p}R_l\mathcal{Q}_{\m} (f_1,f_2)}&\lesssim 2^{(1+\beta)l}2^{k}2^{\frac{k_1+k_2+p_{1}}{2}}2^{-l_1}2^{-\frac{p_1}{2} }\norm{f_1}_X\norm{f_2}_B\\
            &\lesssim 2^{(\beta+2\delta^2)(k-k_2+m+p_2)}2^{k+\frac{k_1+k_2}{2}}\eps^2\\
            &\lesssim 2^{2\beta m}\eps^2,
         \end{align*}
    since $\delta_0\ll \delta^2\ll \beta$. Next we estimate the remaining terms using Lemma \ref{lemma: bounds on dtSf}, $\abs{S}\lesssim 2^{k_2+\frac{p_2}{2}}$, condition \eqref{eqn: l+p/2>delta m conditions on l2 l1 in the last case} and the setting that \eqref{eqn: ibp last step l+p>delta m} doesn't hold. We note that here we balance the $X-$ and $B-$norms on $f_2$ to overcome the loss in $k_2$ and obtain a bounded $X-$norm:
    \begin{align*}
        &2^{4k^+}2^{(1+\beta)l}2^{\beta p}2^{\frac{p}{2}}\ltwonorm{P_{k,p}R_l\B_{\m} (\partial_tf_1,f_2)}\\
         &\qquad\lesssim 2^m2^{(1+\beta)l}2^{k+p_{\max}}2^{k_{2}+\frac{p_{2}}{2}}\ltwonorm{\partial_t f_1}\ltwonorm{f_2}\\
         &\qquad\lesssim 2^{(\frac{1}{4}+2\delta)m}2^{(1+\beta+3\delta^2)(k-k_2+m+p_2)}
        2^{k+p_{\max}}2^{k_2+\frac{p_2}{2}}\eps^2 2^{-(1-4\beta)l_2}2^{-(1-4\beta)\frac{p_2}{2}}\norm{f_2}_{X}^{1-4\beta}2^{2k_2}2^{2\beta p_2}\norm{f_2}_B^{4\beta}\\
         &\qquad\lesssim 2^{(\frac{1}{4}+5\beta+3\delta)m}2^{9\beta p_2}2^{4\beta p_{\max}} 2^{(2+\beta+3\delta^2)k}2^{(\beta-3\delta^2)k_2}\eps^3\\
         &\qquad\lesssim 2^{(\frac{1}{4}+6\beta)m}\eps^3.
    \end{align*}
    And the last term:
    \begin{align*}
        2^{4k^+}2^{(1+\beta)l}2^{\beta p}2^{\frac{p}{2}}\ltwonorm{P_{k,p}R_l\B_{\m} (f_1,\partial_tf_2)}&\lesssim 2^m2^{(1+\beta)l}2^{k+\frac{k_2+k_1+p_1}{2}}\ltwonorm{f_1}\ltwonorm{\partial_t f_2}\\
        &\lesssim 2^{(\frac{1}{4}+2\delta)m}2^{(1+\beta)l}2^{k+\frac{k_2+k_1+p_1}{2}}2^{-l_1}2^{-\frac{p_1}{2}}\norm{f_1}_X\eps^2\\
        &\lesssim 2^{(\frac{1}{4}+2\beta)m}\eps^3.
    \end{align*}  
 This concludes all the cases and the proof of \textbf{Part 2}, and thus the proof of the proposition.
\end{proof}

\subsection{\texorpdfstring{$X$}{X}-norm bounds for \texorpdfstring{$l<(1+\delta)m$}{l-smaller-than-m}}\label{sec:X-norm-2}
\begin{proposition}\label{prop: X-norm bounds for l<(1+delta)m}
  In the setting of Proposition \ref{prop: bootstrap argument}, and in particular under the bootstrap assumptions \eqref{eqn: bootstrap assumption}, the following holds true: For $\m\in \set{\m_0,\m_\pm^{\mu\nu}}{\mu,\nu\in\{+,-\}}$, $t\in [2^m,2^{m+1}[\cap [0,T]$ and $\delta=2M^{-\frac{1}{2}}$, there holds that
    \begin{align*}
        \sup_{\substack{k,l,p\\ l+p\geq 0, l<(1+\delta)m}}2^{4k^+}2^{(1+\beta)l}2^{\beta p}2^{\frac{p}{2}}\ltwonorm{P_{k,p}R_l\B_\m(F_1,F_2)}\lesssim 2^{(\frac{1}{2}-\frac{3}{4}\delta) m }\eps^2+2^{(1-\beta)m}\eps^3,
    \end{align*}
     where $F_i\in \{S^{b_i}\mathcal{Z}_{\pm},S^{b_i}\Theta\}$, $0\leq b_1+b_2\leq N$, $i=1,2$.
\end{proposition}
This is the most challenging result of our article. Similarly to the proofs of Lemma \ref{lemma: bounds on dtSf} and Propositions \ref{prop: bounds on B norm}, \ref{prop: X-norm bounds for l>(1+delta)m}, we can use the energy estimates to treat very large or small frequencies. Otherwise, alongside previously used tools such as integration by parts along $S$, set-size estimates and normal forms, we need a more refined analysis in certain settings. The most delicate part of the proof concerns \textbf{Case B.2} when there holds that $p_1\leq p_2\ll p$. This leads to large losses for integration by parts along $S$, while normal forms are not generally beneficial since $\abs{\Phi}$ may be very small. To handle this, we use refined versions of the aforementioned tools adapted to the precise geometry of frequency interactions at hand, in particular also through set-size estimates as in Lemma \ref{lemma: normal forms}\eqref{it3: lemma: normal forms}. Moreover, in the ``no gap'' case (\textbf{Case D} below), the linear decay estimates alone do not suffice to obtain the claim, and instead we need to introduce additional localizations $q,q_1,q_2$ in the horizontal direction \blue{$\xi_1\absxi^{-1}$ (see Section \ref{sec: Localizations}). This allows for integration by parts in the vector field $S$ and the use of normal forms that include the aforementioned parameters to conclude the proof.}

\begin{proof}
    \par \textbf{Case A: Simple cases.} As in the \textbf{Cases A} in the proofs of  Propositions \ref{prop: bounds on B norm}, \ref{prop: X-norm bounds for l>(1+delta)m}, and with additional localizations $f_i=P_{k_i,p_i}R_{l_i}F_i$, we can treat most of the frequencies using the energy bounds obtained from the bootstrap assumption \eqref{eqn: bootstrap assumption}.
   Thus it suffices to prove
    \begin{align*}
       \sup_{k,l<(1+\delta)m, l+p\geq 0} 2^{4k^+}2^{(1+\beta+2\delta)m}2^{\beta p}2^{\frac{p}{2}}\ltwonorm{P_{k,p}R_l\B_\m(f_1,f_2)}\lesssim 2^{(\frac{1}{2}-\frac{4}{5}\delta) m} \eps^2+2^{(1-\frac{\beta}{2})m}\eps^3,
    \end{align*}
    for the following parameters:
    \begin{align*}
    &    -4m<k,k_i<\delta_0m, && -p_i \leq l_i\leq 4m, && -4m\leq p_i\leq 0, && i=1,2.
    \end{align*}
    Note that in this setting there holds $2^{(1+\beta)l}\lesssim 2^{(1+\beta+2\delta)m}$. We proceed with several cases.
    \par \textbf{Case B: Gap in $p$ with $p_{\max}\sim 0$ and $p_{\min}\ll p_{\max}\sim 0$.} Here there holds $\abs{\sigma}\sim 2^{k_{\max}+k_{\min}}$. 
    Integration by parts along $S$ via Lemma \ref{lemma: ibp in bilinear expressions}\eqref{it1: lemma: ibp in bilinear expressions} yields the claim if
     \begin{equation}\label{eqn: ibp is working X norm} 
    \begin{split}
       S_\eta: \hspace{1cm}  2^{2k_1}2^{-p_1}2^{-k_{\max}-k_{\min}}(1+2^{k_2-k_1}2^{l_1})&\leq 2^{(1-\delta)m}, \\
       S_{\xi-\eta}: \hspace{1cm} 2^{2k_2}2^{-p_2}2^{-k_{\max}-k_{\min}}(1+2^{k_1-k_2}2^{l_2})&\leq 2^{(1-\delta)m},
    \end{split}
\end{equation}
 where $\delta=2M^{-\frac{1}{2}}$. Assume now that \eqref{eqn: ibp is working X norm} doesn't hold and that w.l.o.g. $p_1\leq p_2$ and treat several cases based on Lemma \ref{lemma: case organisation}.
\par \textbf{Case B.1: $p_{\min}\sim p\ll p_{\max}\sim 0$.} By Lemma \ref{lemma: multiplier bound} the multiplier bound reads $\sinftynorm{\m\chi}\lesssim 2^k$. By Lemma \ref{lemma: case organisation} and under the constraint $p_1\leq p_2$, we have two further cases to consider.
\par \textbf{Case B.1(a): $2^{k_1}\sim 2^{k_2}$}. Then $p\ll p_1\sim p_2\sim 0$ and condition \eqref{eqn: ibp is working X norm} doesn't hold if $\max\{-l_1,-l_2\}< -(1-\delta)m-k+k_1$. We use the set size estimate Lemma \ref{lemma: set size estimate} with $\abs{S}\lesssim 2^{k}$ and the bootstrap assumption \eqref{eqn: bootstrap assumption}:
\begin{align*}
    2^{4k^+}2^{(1+\beta+2\delta)m}2^{\beta p}2^{\frac{p}{2}}\ltwonorm{P_{k,p}R_l\B_\m(f_1,f_2)}&\lesssim  2^{4k^+}2^{(2+\beta+2\delta)m}2^{2k}\ltwonorm{P_{k_1,p_1}f_1}\ltwonorm{P_{k_2,p_2}f_2}\\
    &\lesssim 2^{-4k_1^+}2^{(2+\beta+2\delta)m}2^{2k}2^{-(l_1+l_2)}\norm{f_1}_X\norm{f_2}_X\\
    &\lesssim 2^{(\beta+4\delta)m}2^{2k}2^{-3k_1^+}2^{2(-k+k_1)}\eps^2\\
    &\lesssim 2^{2\beta m}\eps^2.
\end{align*}
\par \textbf{Case B.1(b): $2^{k_2}\ll2^{k_1}\sim 2^k $}. Then $p \leq p_1\ll p_2\sim 0$, ${k_{\max}+k_{\min}}\sim {k_1+k_2}$ and $2^{k_2}\lesssim 2^{p_1+k}$. In this case we have $-l_1<-(1-\delta)m-p_1+k-k_2$ and $-l_2<-(1-\delta)m$ (cf.\ \eqref{eqn: ibp is working X norm}). We obtain the claim from Lemma \ref{lemma: set size estimate} with $\abs{S}\lesssim 2^{k_2+\frac{p_2}{2}}\lesssim 2^{k_2}$:
\begin{align*}
    2^{4k^+}2^{(1+\beta+2\delta)m}2^{\beta p}2^{\frac{p}{2}}\ltwonorm{P_{k,p}R_l\B_\m(f_1,f_2)}&\lesssim  2^{(2+\beta+2\delta)m}2^{(\frac{1}{2}+\beta) p_1}2^{k+k_2}\ltwonorm{P_{k_1,p_1}f_1}\ltwonorm{P_{k_2,p_2}f_2}\\
    &\lesssim 2^{(2+\beta+2\delta)m}2^{(\frac{1}{2}+\beta) p_1}2^{k+k_2}2^{-\frac{5}{8}l_1}2^{-\frac{p_1}{8} }2^{-(1+\beta)l_2}\eps^2\\
    &\lesssim 2^{(\frac{3}{8}+3\delta)m}2^{\frac{13}{8}k+\frac{3}{8}k_2}2^{(-\frac{1}{4}+\beta) p_1} \eps^2\\
    &\lesssim 2^{(\frac{3}{8}+4\delta)m}\eps^2.
\end{align*}
\par \textbf{Case B.2: $p_{\min}\sim p_1\ll p_{\max}\sim 0$.} By Lemma \ref{lemma: case organisation} we have the following three cases to consider.
\par \textbf{Case B.2(a): $2^{k}\sim 2^{k_2}$.} Then $p_1\ll p\sim p_2\sim 0$ and $k_{\max}+k_{\min}\sim {k_1+k_2}$. Condition \eqref{eqn: ibp is working X norm} doesn't hold if
\begin{align*}
   & l_1-p_1>(1-\delta)m &&\text{ and }&& \max \{ k_2-k_1,l_2\}>(1-\delta)m.
\end{align*}
\par\textbf{1.} If $\max \{ k_2-k_1,l_2\}=k_2-k_1>(1-\delta)m$, we obtain an admissible bound from Lemma \ref{lemma: set size estimate} with $\abs{S}\lesssim 2^{k_1+\frac{p_1}{2}}$:
\begin{align*}
     2^{4k^+}2^{(1+\beta+2\delta)m}2^{\beta p}2^{\frac{p}{2}}\ltwonorm{P_{k,p}R_l\B_\m(f_1,f_2)}&\lesssim  2^{(2+\beta+2\delta)m}2^{k_1+\frac{p_1}{2}}2^{k}2^{-\frac{l_1}{2}} 2^{\frac{k_1^-}{4}}\norm{f_1}_X^{\frac{1}{2}}\norm{f_1}_{B}^{\frac{1}{2}}\norm{f_2}_B\\
     &\lesssim 2^{(\frac{3}{2}+\beta+3\delta)m}2^{\frac{5}{4}k_1}2^{k}\eps^2\\
     &\lesssim 2^{(\frac{1}{4}+2\beta)m}\eps^2.
\end{align*}
\par \textbf{2.} If on the other hand $\max \{ k_2-k_1,l_2\}=l_2>(1-\delta)m$ we compute with  $\abs{S}\lesssim 2^{k_1+\frac{p_1}{2}}$
\begin{align*}
    2^{4k^+}2^{(1+\beta+2\delta)m}2^{\beta p}2^{\frac{p}{2}}\ltwonorm{P_{k,p}R_l\B_\m(f_1,f_2)}&\lesssim  2^{(2+\beta+2\delta)m}2^{k+k_1+\frac{p_1}{2}}2^{-(1+\beta)(l_1+l_2)} 2^{-(\frac{1}{2}+\beta) p_1}\norm{f_1}_X\norm{f_2}_X\\
    &\lesssim 2^{-4k_1^+}2^{(-\beta+4\delta)m}2^{k_1+k}2^{-2(\frac{1}{2}+\beta)p_1}\eps^2\\
    &\lesssim 2^{-4k_1^+}2^{(-\beta+5\delta)m}2^{(\frac{1}{2}+\beta)(k_1-k-2p_1)}\eps^2.
\end{align*}
The claim follows if $k_1-k-2p_1\leq (1-2\beta)m $. Otherwise, if \begin{align}\label{eqn: case B.2(c) condition normal form}
    k\leq -(1-2\beta)m-2p_1+k_1,
\end{align} we do a splitting as presented in Section \ref{sec: Normal forms} with $\lambda=2^{-4\beta m}$. Thus, we have the following decomposition 
\begin{align*}
    \ltwonorm{P_{k,p}R_l\B_{\m}(f_1,f_2)}\lesssim     \ltwonorm{P_{k,p}R_l\B_{\m^{res}}(f_1,f_2)}+    \ltwonorm{P_{k,p}R_l\B_{\m^{nr}}(f_1,f_2)}.
\end{align*}
Observe that with \eqref{eqn: case B.2(c) condition normal form} and the definition of the phase \eqref{eqn: bq phases} with $\mu,\nu\in \{+,-\}$, we have:
\begin{align*}
    \abs{\partial_{\eta_1}\Phi_{\pm}^{\mu\nu}(\xi,\eta)}&=\abs{\mu\frac{(\xi_2-\eta_2)^2}{\abs{\xi-\eta}^3}+\nu\frac{\eta_2^2}{\abs{\eta}^3}}\gtrsim \abs{2^{2p_2}2^{-k_2}-2^{2p_1}2^{-k_1}}\gtrsim 2^{-k}=:K.
\end{align*}
The resonant term can be treated via Lemma \ref{lemma: normal forms}\eqref{it3: lemma: normal forms} with $\lambda=2^{-4\beta m}$ and $K=2^{-k}$:
\begin{align*}
    2^{4k^+}2^{(1+\beta+2\delta)m}2^{\beta p}2^{\frac{p}{2}}\ltwonorm{P_{k,p}R_l\B_{\m^{res}}(f_1,f_2)}&\lesssim  2^{(2+\beta+2\delta)m}(\lambda K^{-1})^{\frac{1}{2}}2^{k+\frac{p_1}{2}}2^{\frac{p_1}{2}}2^{-(1+\beta)l_2}\norm{f_1}_B\norm{f_2}_X\\
    &\lesssim 2^{(\frac{1}{2}-\frac{\beta}{2})m}\eps^2.
\end{align*}
On the non-resonant part, we can do a normal form as \eqref{eqn: nonresonant multiplier bilinear estimate} and bound the $L^2$-norm of each term using Lemma \ref{lemma: normal forms}.
For the boundary term in \eqref{eqn: nonresonant multiplier bilinear estimate} we have:
\begin{align*}
    2^{4k^+}2^{(1+\beta+2\delta)m}2^{\beta p}2^{\frac{p}{2}}\ltwonorm{P_{k,p}R_l\mathcal{Q}_{\m^{nr}\Phi^{-1}}(f_1,f_2)}
    &\lesssim 2^{(1+\beta+2\delta)m}\lambda^{-1}2^{2k+\frac{p_1}{2}}2^{\frac{p_1}{2}}\norm{f_1}_B2^{-(1+\beta)l_2}\norm{f_2}_X\\
    &\lesssim 2^{(-\frac{1}{2}+2\beta)m }\eps^2.
\end{align*}
 For the second term in the splitting \eqref{eqn: nonresonant multiplier bilinear estimate}, we use Lemma \ref{lemma: bounds on dtSf} and $\abs{S}\lesssim 2^{k_1+\frac{p_1}{2}}$:
\begin{align*}
     &2^{4k^+}2^{(1+\beta+2\delta)m}2^{\beta p}2^{\frac{p}{2}}\ltwonorm{P_{k,p}R_l\B_{\m^{nr}\Phi^{-1}}(\partial_tf_1,f_2)}\\
      &\hspace{15em}\lesssim   2^{4k^+}2^{(2+\beta+2\delta)m}2^{k}\lambda^{-1}\abs{S}\ltwonorm{\partial_tf_1}\ltwonorm{f_2}\\
      &\hspace{15em}\lesssim 2^{(2+\beta+2\delta)m}2^{k+k_1}\lambda^{-1}2^{\frac{p_1}{2}}2^{(-\frac{3}{4}+2\delta)m}\eps^22^{-(1+\beta)(1-\delta)m}\norm{f_2}_X\\
      &\hspace{15em}\lesssim 2^{5\beta m}\eps^3.
\end{align*}
Similarly, for the last term in \eqref{eqn: nonresonant multiplier bilinear estimate} with Lemma \ref{lemma: bounds on dtSf} and $\abs{S}\lesssim 2^{k_1+\frac{p_1}{2}}$ we have
\begin{align*}
     2^{4k^+}2^{(1+\beta+2\delta)m}2^{\beta p}2^{\frac{p}{2}}\ltwonorm{P_{k,p}R_l\B_{\m^{nr}\Phi^{-1}}(f_1,\partial_tf_2)}&\lesssim   2^{4k^+}2^{(2+\beta+2\delta)m}2^{k}\lambda^{-1}\abs{S}\ltwonorm{f_1}\ltwonorm{\partial_tf_2}\\
     &\lesssim 2^{(2+\beta+2\delta)m}2^{2k}\lambda^{-1}2^{{p_1}}\norm{f_1}_B2^{(-\frac{3}{4}+2\delta)m}\eps^2\\
     &\lesssim 2^{(\frac{3}{4}+6\beta)m}\eps^3.
\end{align*}

\par\textbf{Case B.2(b): $2^{k}\ll 2^{k_1}\sim 2^{k_2}$ and $p_1\leq p_2\ll p=p_{\max}\sim 0.$} By Lemma \ref{lemma: case organisation} there holds $2^{k}\lesssim 2^{p_2+k_2}\sim 2^{p_2+k_1}$. We obtain the claim via integration by parts if $l_1\leq (1-\delta)m+p_1+k-k_1$ or $l_2\leq (1-\delta)m+p_2+k-k_1$, see \eqref{eqn: ibp is working X norm}. Hence we may assume
\begin{align}\label{eqn: l_i bound X norm p1<=p_2}
   & -l_1<-(1-\delta)m -p_1-k+k_1 &&\text{ and }&& -l_2<-(1-\delta)m -p_2-k+k_1.
\end{align}
In this setting, we treat two different parts based on the signs in the phase and on the relative size of $p_1$ to $p_2$. 
Recall the definition of the phases \eqref{eqn: bq phases}, i.e.\
\begin{align*}
    \Phi_\pm^{\mu\nu}=\pm\Lambda(\xi)-\mu\Lambda(\xi-\eta)-\nu\Lambda(\eta), &\hspace{1.5cm} \mu,\nu \in \{+,-\}.
\end{align*}
\par\textbf{Case B.2(b.1):} Assume $\mu=\nu$ and $p_1\sim p_2\ll p\sim 0$. 
\par\textbf{1.} If $\Lambda(\xi-\eta)\Lambda(\eta)>0$, since $2^{p_1}\sim 2^{p_2}\ll 1$ there holds:
\begin{align*}
    \abs{\Lambda(\xi-\eta)+\Lambda(\eta)}\geq \frac{3}{2}.
\end{align*}
This implies in particular that the phase is large:
\begin{align*}
    \abs{\Phi_\pm^{\mu\mu}}=\abs{\pm \Lambda(\xi)-\mu\Lambda(\xi-\eta)-\mu\Lambda(\eta)}\geq
    \abs{\abs{\Lambda(\xi-\eta)+\Lambda(\eta)}-\abs{\Lambda(\xi)}}\geq \frac{1}{2}.
\end{align*}
With this observation and $\lambda=10^{-2}$, we note that a splitting as in Section \ref{sec: Normal forms} contains only the non-resonant part. That is $\B_{\m}(f_1,f_2)=\B_{\m^{nr}}(f_1,f_2)$ and we can apply Lemma \ref{lemma: normal forms} to bound each term in \eqref{eqn: nonresonant multiplier bilinear estimate} in $L^2$.
We proceed with the boundary term and with Lemmas \ref{lemma: linear decay many vector fields}, \ref{lemma: normal forms}\eqref{it2: lemma: normal forms} and \eqref{eqn: l_i bound X norm p1<=p_2} obtain that
\begin{align*}
    2^{4k^+}2^{(1+\beta+2\delta)m}\ltwonorm{P_{k,p}R_l\mathcal{Q}_{\m^{nr}\Phi^{-1}}(f_1,f_2)}&\lesssim 2^{4k^+}2^{(1+\beta+2\delta)m}2^{k}\inftynorm{e^{it\Lambda}f_1}\ltwonorm{f_2}\\
    &\lesssim 2^{(\frac{1}{2}+\beta+\kappa+2\delta)m}2^k2^{\frac{l_2}{2}}\eps^2\\
    &\lesssim 2^{(\beta+\kappa+3\delta)m}2^{\frac{k_1}{2}}\eps^2.
\end{align*}
For the terms in \eqref{eqn: nonresonant multiplier bilinear estimate} containing the time derivative we use Lemmas \ref{lemma: normal forms}\eqref{it1: lemma: normal forms}, \ref{lemma: bounds on dtSf} and with $\abs{S}\lesssim 2^{\frac{k+k_1+p_2}{2}}$ obtain that
\begin{align*}
    2^{4k^+}2^{(1+\beta+2\delta)m} \ltwonorm{P_{k,p}R_l\B_{\m^{nr}\Phi^{-1}}(\partial_tf_1,f_2)}&\lesssim  2^{4k^+}2^{(2+\beta+2\delta)m}2^{\frac{3k+k_1+p_2}{2}}\ltwonorm{\partial_tf_1}\ltwonorm{f_2}\\
    &\lesssim 2^{(\frac{5}{4}+\beta+2\delta)m}2^{\frac{3k+k_1+p_2}{2}}2^{-\frac{l_2}{2}}\eps^3\\
    &\lesssim 2^{(\frac{3}{4}+2\beta)m}\eps^3.
\end{align*}
The third term is bounded similarly by symmetry using \eqref{eqn: l_i bound X norm p1<=p_2} on $l_1$ instead. 
\par \textbf{2.} Assume $\Lambda(\xi-\eta)\Lambda(\eta)<0$. We assume w.l.o.g.\ that $\Lambda(\xi-\eta)<0$ and $\Lambda(\eta)>0$ (see Figure \ref{fig:X} for illustration).
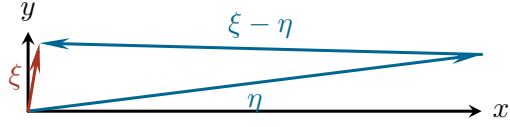
\begin{figure}[h]
    \centering
    \begin{tikzpicture}[>=stealth,scale=1.5,line cap=round,
    bullet/.style={circle,inner sep=1.5pt,fill}]    
 \draw[line width=0.4mm, ->] (0,0) -- (4,0) node[right]{$x$};
 \draw[line width=0.4mm, ->] (0,0) -- (0,0.7) node[above]{$y$};
 \draw [line width=0.4mm, -{Stealth[length=4mm, width=1.5mm]}, MidnightBlue] (0,0) -- (4,0.5) node [midway, below] {$\eta$};
 \draw [line width=0.4mm,-{Stealth[length=4mm, width=1.5mm]}, Mahogany]  (0,0) -- (0.1,0.6) node [midway, left] {$\xi$};
 \draw[line width=0.4mm,-{Stealth[length=4mm, width=1.5mm]}, MidnightBlue] (4,0.5) --  (0.1,0.6) node [midway, above] {$\xi-\eta$};
\end{tikzpicture} 
\caption{A sample setting of Case 2.}\label{fig:X}
\end{figure}\\
First we observe that
\begin{align*}
    \abs{\partial_{\eta_1}\Phi_{\pm}^{\mu\mu}}\chi=\abs{\Pi(\xi-\eta)-\Pi(\eta)}\chi,\qquad \Pi(\zeta):={\zeta_2^2}/{\abs{\zeta}^3}, \qquad \zeta \in \R^2.
\end{align*} 
Moreover, there holds that
\begin{align*}
    \nabla_{\zeta}\Pi(\zeta)=\frac{\zeta_2}{\abs{\zeta}^3}\Big(-\frac{3\zeta_1\zeta_2}{\abs{\zeta}^2}, \frac{-\zeta_2^2+2\zeta_1^2}{\abs{\zeta}^2} \Big)^T,
\end{align*}
and by the mean value theorem together with the condition $2^{k}\lesssim 2^{p_1+k_1}$ we obtain
\begin{align*}
   \abs{\partial_{\eta_1}\Phi_{\pm}^{\mu\mu}}\chi=  \abs{\Pi(\xi-\eta)-\Pi(\eta)}\chi(\xi,\eta)\geq \inf_{(\zeta,\upsilon) \in \supp \chi}\abs{(\nabla\Pi)(\zeta)\cdot \upsilon}\gtrsim 2^{p_1-2k_1+k}.
\end{align*}
 Then we can integrate by parts in $D_{\eta}=\abseta\partial_{\eta_1}$ using Lemma \ref{lemma: ibp in D_eta} with $L=2^{p_1-2k_1+k}$ and obtain the claim if
\begin{align*}
    2^{-p_1-k+k_1}(2^{l_1+p_1}+2^{l_2+p_2})<2^{(1-\delta)m}.
\end{align*}
That is, we may assume 
\begin{align*}
    \max\{2^{l_1-k+k_1},2^{l_2-k+k_1}\}>2^{(1-\delta)m}.
\end{align*}
Without loss of generality, we assume $2^{l_1-k+k_1}>2^{(1-\delta)m}$ and that \eqref{eqn: l_i bound X norm p1<=p_2} holds for $l_2$. The other case is treated by symmetry. To handle this case, we do a splitting $\m=\m^{nr}+\m^{res}$ as in Section \ref{sec: Normal forms} with $\lambda=2^{4\beta k}2^{(1+8\beta)p_1}$. On the non-resonant part we bound each term in \eqref{eqn: nonresonant multiplier bilinear estimate} and using $2^{k}\lesssim 2^{p_1+k_1}$ and the $X$-norm on $f_1$, $B$-norm on $f_2$, we obtain 
\begin{align*}
     2^{4k^+}2^{(1+\beta+2\delta)m}\ltwonorm{P_{k,p}R_l\mathcal{Q}_{\m^{nr}\Phi^{-1}}(f_1,f_2)}&\lesssim 2^{4k^+}2^{(1+\beta+2\delta)m}2^{\frac{3k+k_1+p_1}{2}}\lambda^{-1}\ltwonorm{f_1}\ltwonorm{f_2}\\
     &\lesssim 2^{(1+\beta+2\delta)m}2^{\frac{3k+k_1+p_1}{2}}\lambda^{-1} 2^{-(\frac{1}{2}+2\beta)l_1}2^{-2\beta p_1}2^{-4k_1^+}\eps^2\\
     &\lesssim 2^{(\frac{1}{2}-\beta+3\delta)m}2^{(\frac{3}{2}-\frac{1}{2}-6\beta)k}2^{(1+2\beta)k_1}2^{(-\frac{1}{2}-10\beta)p_1}2^{-4k_1^+}\eps^2\\
     &\lesssim 2^{(\frac{1}{2}-\frac{\beta}{2})m}\eps^2,
\end{align*}
where we have used $2^{k}\lesssim 2^{p_1+k_1}$. Next, we bound using Lemma \ref{lemma: bounds on dtSf}, $\abs{S}\lesssim 2^{\frac{k+k_1+p_1}{2}}$ and the $X$-norm :
\begin{align*}
     2^{4k^+}2^{(1+\beta+2\delta)m}\ltwonorm{P_{k,p}R_l\B_{\m^{nr}\Phi^{-1}}(\partial_tf_1,f_2)}&\lesssim 2^{4k^+}2^{(2+\beta+2\delta)m}2^{\frac{3k+k_1+p_1}{2}}\lambda^{-1}\ltwonorm{\partial_tf_1}\ltwonorm{f_2}\\
     &\lesssim 2^{(\frac{5}{4}+\beta+3\delta)m}2^{(\frac{3}{2}-4\beta)k}2^{(-\frac{1}{2}-8\beta)p_1}2^{-(\frac{1}{4}+2\beta)l_2}2^{(\frac{1}{4}-2\beta)p_1}\eps^3\\
     &\lesssim 2^{(1-\beta+3\delta)m}2^{(\frac{5}{4}-6\beta)k}2^{(-\frac{1}{2}-12\beta)p_1}\eps^3\\
     &\lesssim 2^{(1-\frac{\beta}{2})m}\eps^3.
\end{align*}
And finally for the third term, with Lemma \ref{lemma: bounds on dtSf} and $\abs{S}\lesssim 2^{\frac{k+k_1+p_1}{2}}$ there holds:
\begin{align*}
      2^{4k^+}2^{(1+\beta+2\delta)m}\ltwonorm{P_{k,p}R_l\B_{\m^{nr}\Phi^{-1}}(f_1,\partial_tf_2)}&\lesssim 2^{4k^+}2^{(2+\beta+2\delta)m}2^{\frac{3k+k_1+p_1}{2}}\lambda^{-1}\ltwonorm{f_1}\ltwonorm{\partial_tf_2}\\
      &\lesssim 2^{(\frac{5}{4}+\beta+3\delta)m}2^{(\frac{3}{2}-4\beta)k}2^{\frac{k_1}{2}}2^{(-\frac{1}{2}-8\beta)p_1}2^{-\frac{l_1}{2}}\eps^3\\
      &\lesssim 2^{(\frac{3}{4}+\beta+3\delta )m}2^{(1-4\beta)k}2^{k_1}2^{(-\frac{1}{2}-8\beta)p_1}\eps^3\\
      &\lesssim 2^{(\frac{3}{4}+2\beta)m}\eps^3.
\end{align*}
On the resonant set, we observe:
\begin{align*}
    \abs{\partial_{\xi_1}\Phi_{\pm}^{\mu\nu}}=\abs{\pm\frac{\xi_2^2}{\absxi^3}-\mu\frac{(\xi_2-\eta_2)^2}{\absxieta^3}}\gtrsim\abs{2^{2p-k}-2^{2p_1-k_1}}\gtrsim 2^{-k}=:K.
\end{align*}
Therefore, we can employ Lemma \ref{lemma: normal forms}\eqref{it3: lemma: normal forms} with $\lambda$ and $K=2^{-k}$ and \eqref{eqn: l_1,l_2 condition Case B.2(c)} on $l_2$ to obtain
\begin{align*}
    &2^{4k^+}2^{(1+\beta+2\delta)m}\ltwonorm{P_{k,p}R_l\B_{\m^{res}}(f_1,f_2)}\\
     &\hspace{10em}\lesssim 2^{4k^+}2^{(2+\beta+2\delta)m}2^{k} (\lambda K^{-1})^{\frac{1}{2}}2^{\frac{k_1+p_1}{2}}\ltwonorm{f_1}\ltwonorm{f_2}\\ 
     &\hspace{10em}\lesssim 2^{(2+\beta+2\delta)m} 2^{(\frac{3}{2}+2\beta)k}2^{(1+4\beta)p_1}2^{\frac{k_1}{2}}2^{-(1+\beta)l_1}2^{-(\frac{1}{2}+\beta)p_1} 2^{-4k_1^+}2^{-(\frac{1}{2}+\beta)l_2}2^{-2\beta p_1}\eps^2\\
     &\hspace{10em}\lesssim 2^{(\frac{1}{2}-\beta +4\delta)m}\eps^2.
\end{align*}
This concludes the proof of  \textbf{Case B.2(b.1)}.

\par\textbf{Case B.2(b.2):} Assume $\mu=-\nu$ or $p_1\ll p_2$. In this setting, we can split the analysis in the resonant and non-resonant parts as explained in Section \ref{sec: Normal forms} with $\lambda=2^{(\frac{3}{2}-6\beta)p_2}$. On the non-resonant part we have the three terms in \eqref{eqn: nonresonant multiplier bilinear estimate}.
With Lemma \ref{lemma: normal forms}\eqref{it1: lemma: normal forms}, Lemma \ref{lemma: set size estimate} with $\abs{S}\lesssim 2^{\frac{k+k_1+p_1}{2}}$ and \eqref{eqn: l_i bound X norm p1<=p_2} for the boundary term we obtain using the $X$-norms on $f_i$:
\begin{align*}
    2^{4k^+}2^{(1+\beta+2\delta)m}\ltwonorm{P_{k,p}R_l\mathcal{Q}_{\m^{nr}\Phi^{-1}}(f_1,f_2)}&\lesssim 2^{4k^+}2^{(1+\beta+2\delta)m}2^{\frac{3k+k_1+p_1}{2}}\lambda^{-1}\ltwonorm{f_1}\ltwonorm{f_2}\\
    &\lesssim 2^{(1+\beta+2\delta)m}2^{\frac{3k+k_1+p_1}{2}}\lambda^{-1} 2^{-\frac{l_1}{2}}2^{-4k_1^+}2^{-2\beta l_2}2^{(\frac{1}{2}-2\beta)p_2}\eps^2\\
    &\lesssim 2^{(\frac{1}{2}-\beta+3\delta)m}2^{(1-2\beta)k}2^{(1+2\beta)k_1}\lambda^{-1}2^{(\frac{1}{2}-4\beta)p_2}2^{-4k_1^+}\eps^2\\
    &\lesssim  2^{(\frac{1}{2}-\beta+3\delta)m} 2^{(\frac{3}{2}-6\beta)p_2}2^{2k_1}\lambda^{-1}2^{-4k^+_1}\eps^2.
\end{align*}
Note that we have used $2^{-\frac{l_2}{2}}\leq 2^{-2\beta l_2}2^{(\frac{1}{2}-2\beta)p_2}$ since $l_2+p_2\geq 0$. For the other two terms in \eqref{eqn: nonresonant multiplier bilinear estimate} we obtain with Lemma \ref{lemma: bounds on dtSf}, $\abs{S}\lesssim 2^{\frac{k+k_1+p_2}{2}}$ and \eqref{eqn: l_i bound X norm p1<=p_2}:
\begin{align*}
   2^{4k^+}2^{(1+\beta+2\delta)m}\ltwonorm{P_{k,p}R_l\B_{\m^{nr}\Phi^{-1}}(\partial_tf_1,f_2)}&\lesssim 2^{4k^+}2^{(2+\beta+2\delta)m}2^{\frac{3k+k_1+p_2}{2}}\lambda^{-1}\ltwonorm{\partial_tf_1}\ltwonorm{f_2}\\
   &\lesssim 2^{(\frac{5}{4}+\beta+3\delta)m}2^{\frac{3k+k_1+p_2}{2}}\lambda^{-1}2^{-(\frac{1}{4}+2\beta)l_2}2^{(\frac{1}{4}-2\beta)p_2}2^{-4k^+_1}\eps^3\\
   &\lesssim 2^{(1-\beta+3\delta)m}2^{\frac{3k+k_1+p_2}{2}}\lambda^{-1}2^{-4\beta p_2}2^{-(\frac{1}{4}+2\beta)k_1}2^{(\frac{1}{4}+2\beta)k_1}\eps^3\\
   &\lesssim 2^{(1-\beta+3\delta)m}2^{(\frac{5}{4}-2\beta)k}2^{(\frac{3}{4}+2\beta)k_1}2^{(-1+2\beta)p_2}\eps^3\\
   &\lesssim 2^{(1-\frac{\beta}{2})m}\eps^3,
\end{align*}
since $\delta_0\ll \delta\ll \beta$. The third term in \eqref{eqn: nonresonant multiplier bilinear estimate} is bounded analogously by using Lemma \ref{lemma: bounds on dtSf} on $f_2$ and the $X-$norm on $f_1$. 
\par On the resonant part, we first observe
\begin{align}\label{eqn: partial_xi1 L defn X norm}
    \abs{\partial_{\xi_1}\Phi_{\pm}^{\mu\nu}}=\abs{\pm\frac{\xi_2^2}{\absxi^3}-\mu\frac{(\xi_2-\eta_2)^2}{\absxieta^3}}\gtrsim\abs{2^{2p-k}-2^{2p_1-k_1}}\gtrsim 2^{-k}=:K.
\end{align}
Next recall
\begin{align*}
    \abs{\partial_{\eta_1}\Phi_\pm^{\mu\nu}}=\abs{\mu \frac{(\xi_2-\eta_2)^2}{\absxieta^3}-\nu \frac{\eta_2^2}{\abseta^3}},
\end{align*}
and observe that if $\mu=-\nu$ or $p_1\ll p_2$ there holds
\begin{align*}
    \abs{\partial_{\eta_1}\Phi^{\mu\nu}_\pm}\gtrsim 2^{2p_2-k_2}.
\end{align*}
In this case we can integrate by parts using Lemma \ref{lemma: ibp in D_eta} with  $L=2^{2p_2}$ and obtain the claim if $\max \{{l_1+p_1-2p_2},{l_2-p_2}\}<{(1-\delta)m}$.
\par Now we assume that 
\begin{align}\label{eqn: ibp in Deta doesn't work}
    \max \{{l_1+p_1-2p_2},{l_2-p_2}\}>{(1-\delta)m},
\end{align}
and note that this is an improvement compared to \eqref{eqn: l_i bound X norm p1<=p_2} since we do not have losses in $k$. We consider two cases based on which term in \eqref{eqn: ibp in Deta doesn't work} is the largest.
\par \textbf{1.} Assume $l_1>(1-\delta)m-p_1+2p_2$ and $l_2$ satisfies \eqref{eqn: l_i bound X norm p1<=p_2}. Therefore, with \eqref{eqn: partial_xi1 L defn X norm} we can bound the resonant part using Lemma \ref{lemma: normal forms}\eqref{it3: lemma: normal forms} with $\lambda=2^{(\frac{3}{2}-6\beta)p_2}$ and $K=2^{-k}$ :
\begin{align*}
    &2^{4k^+}2^{(1+\beta+2\delta)m}\ltwonorm{P_{k,p}R_l\B_{\m^{res}}(f_1,f_2)}\\
     &\hspace{8em}\lesssim 2^{4k^+}2^{(2+\beta+2\delta)m}2^{k}(\lambda K^{-1})^{\frac{1}{2}}2^{\frac{k_1+p_1}{2}}\ltwonorm{f_1}\ltwonorm{f_2}\\
     &\hspace{8em}\lesssim 2^{(2+\beta+2\delta)m}2^{\frac{3}{2}k}2^{(\frac{3}{4}-3\beta)p_2}2^{\frac{k_1+p_2}{2}}2^{-(1+\beta)l_1}2^{(-\frac{1}{2}-\beta)p_1} 2^{-(\frac{1}{2}+\frac{\beta}{2})l_2}2^{-\frac{\beta}{2}p_2}2^{-4k_1^+}\eps^2\\
     &\hspace{8em}\lesssim2^{(\frac{1}{2}-\frac{\beta}{2}+4\delta)m}2^{\frac{3k+k_1}{2}}2^{(\frac{5}{4}-3\beta)p_2}2^{\frac{p_1}{2}}2^{-2(1+\beta)p_2}2^{-(\frac{1}{2}+\beta)p_2}2^{-(\frac{1}{2}+\frac{\beta}{2})k}2^{(\frac{1}{2}+\frac{\beta}{2})k_1}2^{-4k_1^+}\eps^2\\
     &\hspace{8em}\lesssim2^{(\frac{1}{2}-\frac{\beta}{2}+4\delta)m}2^{(1-\frac{\beta}{2})k}2^{(1+\frac{\beta}{2})k_1}2^{(-\frac{3}{4}-6\beta)p_2}2^{4k_1^+}\eps^2\\
      &\hspace{8em}\lesssim2^{(\frac{1}{2}-\frac{\beta}{4})m}\eps^2.
\end{align*}
\par\textbf{2.} Assume $l_2-p_2>(1-\delta)m$ and $l_1$ satisfies \eqref{eqn: l_i bound X norm p1<=p_2}. In this case we do another splitting 
\begin{align*}
    \m^{res}(\xi,\eta)=\psi(\lambda_1^{-1}\Phi)\m^{res}(\xi,\eta)+(1-\psi(\lambda_1^{-1}\Phi))\m^{res}(\xi,\eta)=: \m^{res,res}(\xi,\eta)+\m^{res,nr}(\xi,\eta)
\end{align*}
with $\lambda_1:=\lambda 2^{-20\delta m}<\lambda$ and obtain a decomposition of the bilinear term
\begin{align*}
    \B_{\m^{res}}(f_1,f_2)=\B_{\m^{res,res}}(f_1,f_2)+\B_{\m^{res,nr}}(f_1,f_2).
\end{align*}
We can estimate the first term as follows using Lemma \ref{lemma: normal forms}\eqref{it3: lemma: normal forms} with $K=2^{-k}$, see  \eqref{eqn: partial_xi1 L defn X norm}, and $\lambda_1$:
\begin{align*}
    &2^{4k^+}2^{(1+\beta+2\delta)m}\ltwonorm{P_{k,p}R_l\B_{\m^{res,res}}(f_1,f_2)}\\
     &\hspace{12em}\lesssim 2^{4k^+}2^{(2+\beta+2\delta)m}2^{k}(\lambda_1 K^{-1})^{\frac{1}{2}}2^{\frac{k_1+p_1}{2}}\ltwonorm{f_1}\ltwonorm{f_2}\\
     &\hspace{12em}\lesssim 2^{(2+\beta+2\delta)m}2^{\frac{3k+k_1+p_1}{2}}2^{(\frac{3}{4}-3\beta)p_2}2^{-10\delta m}2^{-\frac{l_1}{2}} 2^{-4k_1^+}2^{-(1+\beta)l_2}2^{-(\frac{1}{2}+\beta)p_2}\eps^2\\
     &\hspace{12em}\lesssim 2^{(\frac{1}{2}-6\delta)m}2^{k+k_1}2^{(\frac{3}{4}-3\beta)p_2}2^{(-\frac{3}{2}-2\beta)p_2}2^{-4k_1^+}\eps^2\\
     &\hspace{12em}\lesssim 2^{(\frac{1}{2}-6\delta)m}2^{(\frac{1}{4}-5\beta)k}2^{(\frac{7}{4}+5\beta)k_1}2^{-4k_1^+}\eps^2\\
      &\hspace{12em}\lesssim 2^{(\frac{1}{2}-6\delta)m}\eps^2.
\end{align*}
We bound the terms arising in the non-resonant part as in \eqref{eqn: nonresonant multiplier bilinear estimate}. Using Lemma \ref{lemma: normal forms}\eqref{it1: lemma: normal forms} with $\lambda_1$ and \ref{lemma: m phi^-1 multiplier bound} we obtain using the $X$-norm on $f_1$ and $f_2$:
\begin{align*}
    &2^{4k^+}2^{(1+\beta+2\delta)m}\ltwonorm{P_{k,p}R_l\mathcal{Q}_{\Phi^{-1}\m^{res,nr}}(f_1,f_2)}\\
    &\hspace{15em}\lesssim 2^{4k^+}2^{(1+\beta+2\delta)m}2^{\frac{3k+k_1+p_1}{2}}\lambda_1^{-1}\ltwonorm{f_1}\ltwonorm{f_2}\\
     &\hspace{15em}\lesssim 2^{(1+\beta+22\delta)m}2^{\frac{3k+k_1+p_1}{2}}2^{-(\frac{3}{2}-6\beta)p_2}2^{-\frac{l_1+l_2}{2}}2^{-4k_1^+}\eps^2\\
     &\hspace{15em}\lesssim 2^{(\frac{1}{2}+\beta+23\delta)m}2^{\frac{3k+k_1+p_1}{2}}2^{-(2-6\beta)p_2} 2^{-4k_1^+}2^{-2\beta l_1}2^{(\frac{1}{2}-2\beta)p_1}\eps^2\\
      &\hspace{15em}\lesssim 2^{(\frac{1}{2}-\beta+24\delta)m}2^{(\frac{3}{2}-2\beta)k}2^{(\frac{1}{2}+2\beta)k_1} 2^{(1-4\beta )p_1}2^{(-2+6\beta)p_2}2^{-4k_1^+}\eps^2\\
    &\hspace{15em}\lesssim2^{(\frac{1}{2}-\beta+24\delta)m}2^{(\frac{3}{2}-2\beta)k}2^{(\frac{1}{2}+2\beta)k_1} 2^{(-1+2\beta)p_2}2^{-4k_1^+}\eps^2\\
    &\hspace{15em}\lesssim 2^{(\frac{1}{2}-\frac{\beta}{2})m}\eps^2.
\end{align*}
For the terms containing the time derivative we obtain with Lemma \ref{lemma: bounds on dtSf}:
\begin{align*}
    2^{4k^+}2^{(1+\beta+2\delta)m}\ltwonorm{P_{k,p}R_l\B_{\Phi^{-1}\m^{res,nr}}(\partial_tf_1,f_2)}&\lesssim 2^{4k^+}2^{(2+\beta+2\delta)m}2^{\frac{3k+k_1+p_1}{2}}\lambda_1^{-1}\ltwonorm{\partial_tf_1}\ltwonorm{f_2}\\
    &\lesssim 2^{(\frac{5}{4}+\beta+23\delta)m}2^{\frac{3k+k_1+p_1}{2}}2^{-(\frac{3}{2}-6\beta)p_2}2^{-\frac{l_2}{2}}\eps^3\\
    &\lesssim 2^{(\frac{3}{4}+\beta+24\delta)m}2^{\frac{3k+k_1}{2}}2^{-(\frac{3}{2}-6\beta)p_2}\eps^3\\
    &\lesssim 2^{(\frac{3}{4}+2\beta)m}\eps^3.
\end{align*}
 And finally there holds that
\begin{align*}
     &2^{4k^+}2^{(1+\beta+2\delta)m}\ltwonorm{P_{k,p}R_l\B_{\Phi^{-1}\m^{res,nr}}(f_1,\partial_tf_2)}\\
    &\hspace{15em}\lesssim 2^{4k^+}2^{(2+\beta+2\delta)m}2^{\frac{3k+k_1+p_1}{2}}\lambda_1^{-1}\ltwonorm{f_1}\ltwonorm{\partial_tf_2}\\
      &\hspace{15em}\lesssim 2^{(\frac{5}{4}+\beta+23\delta)m}2^{\frac{3k+k_1+p_1}{2}}2^{-(\frac{3}{2}-6\beta)p_2}2^{-\frac{l_1}{2}}\eps^3\\
      &\hspace{15em}\lesssim 2^{(\frac{5}{4}+\beta+23\delta)m}2^{\frac{3k+k_1+p_1}{2}}2^{-(\frac{3}{2}-6\beta)p_2}2^{-(\frac{1}{4}+2\beta)l_1}2^{(\frac{1}{4}-2\beta)p_1}\eps^3\\
      &\hspace{15em}\lesssim 2^{(1-\beta+24\delta)m}2^{(\frac{5}{4}-2\beta)k}2^{(\frac{3}{4}+2\beta)k_1}2^{-(\frac{3}{2}-6\beta)p_2}2^{(\frac{1}{2}-4\beta)p_1}\eps^3\\
      &\hspace{15em}\lesssim 2^{(1-\beta+24\delta)m}2^{\frac{p_2}{4}}2^{2k_1}\eps^3,
\end{align*}
which gives an acceptable contribution as $\delta_0\ll\delta\ll\beta\ll 1$.

\par \textbf{Case B.2(c): $2^{k_2}\ll 2^k\sim 2^{k_1}$ and $p_1\ll p\ll p_2\sim 0$.} By Lemma \ref{lemma: case organisation} we have  $2^{p_2+k_2}\sim 2^{p+k}$ and $k_{\max}+k_{\min}\sim k_2+k_1$. We obtain the claim if $l_1\leq (1-\delta)m +p_1+k_2-k$, or $l_2\leq (1-\delta)m$, see \eqref{eqn: ibp is working X norm}.
 Hence we may assume
\begin{align}\label{eqn: l_1,l_2 condition Case B.2(c)}
    -l_1<-(1-\delta)m-p_1+k_1-k_2 && \text{ and } && -l_2<-(1-\delta)m.
\end{align}
Using the set size estimate Lemma \ref{lemma: set size estimate} with $\abs{S}\lesssim 2^{\frac{k+p_1+k_2}{2}}$ we bound:
\begin{align*}
    \ltwonorm{P_{k,p}R_l\B_{\m}(f_1,f_2)}&\lesssim 2^m 2^{\frac{3}{2}k+\frac{p_1}{2}}2^{\frac{k_2}{2}}2^{-4k^+}2^{-\frac{l_1}{2}}2^{-4k_2^+}2^{-(1+\beta)l_2}2^{-\beta p_2-\frac{p_2}{2}}\eps^2\\
    &\lesssim 2^m2^{\frac{3}{2}k+\frac{k_2}{2}+\frac{p_1}{2}}2^{-\frac{(1-\delta)m}{2}-\frac{p_1}{2}+\frac{k-k_2}{2}}2^{-4k^+-4k_2^+}2^{-(1+\beta)(1-\delta)m}\eps^2\\
    &\lesssim 2^{(-\frac{1}{2}-\beta+2\delta)m}2^{2k}2^{-4k^+}\eps^2.
\end{align*}
Thus for the $X$-norm and with $2^p\sim 2^{k_2-k}$ there holds that
\begin{align*}
    2^{4k^+}2^{(1+\beta+2\delta)m}2^{\beta p} 2^{\frac{p}{2}}\ltwonorm{P_{k,p}R_l\B_{\m}(f_1,f_2)}&\lesssim 2^{(\frac{1}{2}+4\delta)m}2^{2k}2^{(\frac{1}{2}+\beta)(k_2-k)}\eps^2\lesssim 2^{(\frac{1}{2}+6\delta)m}2^{(\frac{1}{2}+\beta) k_2}\eps^2,
\end{align*}
and the claim follows if $k_2\leq-20\delta m$. Assume now that $k_2>-20\delta m$. We decompose the multiplier into the resonant and non-resonant part as in Section \ref{sec: Normal forms} with $\lambda=2^{-100\delta m}$. For $\Phi_{\pm}^{\mu\nu}$ on the support of the resonant set there holds:
\begin{align*}
    \abs{\partial_{\eta_1}\Phi_{\pm}^{\mu\nu}}&= \abs{\mu2^{2p_2-k_2}-\nu2^{2p_1-k_1}}\gtrsim \abs{2^{2p_2-k_2}-2^{2p_2+2k_2-3k}}\gtrsim2^{-k_2}=:K>0.
\end{align*}
Using Lemma \ref{lemma: normal forms}\eqref{it3: lemma: normal forms} with $K=2^{-k_2}$ and $\lambda$ we estimate
\begin{align*}
  2^{4k^+}2^{(1+\beta+2\delta)m}2^{\beta p} 2^{\frac{p}{2}}   \ltwonorm{P_{k,p}R_l\B_{\m^{res}}(f_1,f_2)}&\lesssim 2^{4k^+}2^{(2+\beta+2\delta)m}2^{\frac{3}{2}k}2^{\frac{p_1}{2}}(\lambda K^{-1})^{\frac{1}{2}}\ltwonorm{f_1}\ltwonorm{f_2}\\
    &\lesssim 2^{(2+\beta+2\delta)m}2^{\frac{3}{2}k}2^{\frac{p_1}{2}}(\lambda K^{-1})^{\frac{1}{2}} 2^{-\frac{l_1}{2}}2^{-(1+\beta)l_2}\eps^2\\
    &\lesssim 2^{(\frac{1}{2}-40\delta) m}\eps^2. 
\end{align*}
 Now we turn to the non-resonant term and do a normal form as in \eqref{eqn: nonresonant multiplier bilinear estimate}. For the boundary term we obtain with Lemma \ref{lemma: normal forms}\eqref{it1: lemma: normal forms} and \eqref{eqn: l_1,l_2 condition Case B.2(c)}:
\begin{align*}
     &2^{4k^+}2^{(1+\beta+2\delta)m}2^{\beta p} 2^{\frac{p}{2}}\ltwonorm{P_{k,p}R_l\mathcal{Q}_{\Phi^{-1}\m^{nr}}(f_1,f_2)}\\
     &\hspace{15em}\lesssim  2^{4k^+}2^{(1+\beta+2\delta)m}2^{\frac{3}{2}k}\lambda^{-1}2^{\frac{p_1+k_2}{2}}\ltwonorm{f_1}\ltwonorm{f_2}\\
    &\hspace{15em}\lesssim 2^{(1+\beta+2\delta)m}2^{\frac{3}{2}k}\lambda^{-1}2^{\frac{p_1+k_2}{2}}2^{-\frac{l_1}{2}}\norm{f_1}_Xt2^{-(1+\beta)l_2}2^{-(\frac{1}{2}+\beta)p_2}\norm{f_2}_X\\
    &\hspace{15em}\lesssim 2^{(-\frac{1}{2}+110\delta)m}\eps^2.
\end{align*}
 For the other terms in the non-resonant decomposition \eqref{eqn: nonresonant multiplier bilinear estimate} we estimate using Lemma \ref{lemma: bounds on dtSf}:
\begin{align*}
2^{4k^+}2^{(1+\beta+2\delta)m}2^{\beta p} 2^{\frac{p}{2}}\ltwonorm{P_{k,p}R_l\B_{\m^{nr}\Phi^{-1}}(f_1,\partial_tf_2)}&\lesssim2^{4k^+}2^{(2+\beta+2\delta)m} 2^{\frac{3}{2}k} \lambda^{-1}2^{\frac{p_1+k_2}{2}}\ltwonorm{f_1}\ltwonorm{\partial_tf_2}\\
    &\lesssim 2^{(2+\beta+2\delta)m} 2^{\frac{3}{2}k}\lambda^{-1}2^{\frac{p_1+k_2}{2}}2^{-\frac{l_1}{2}}\norm{f_1}_X2^{(-\frac{3}{4}+2\delta)m}\eps^2\\
    &\lesssim 2^{(\frac{3}{4}+2\beta)m}\eps^3.
\end{align*}
And finally with $\abs{S}\lesssim 2^{\frac{k+k_2}{2}}$ and the $X$-norm on $f_2$ we obtain that
\begin{align*}
  2^{4k^+}2^{(1+\beta+2\delta)m}  2^{\beta p} 2^{\frac{p}{2}} \ltwonorm{P_{k,p}R_l\B_{\m^{nr}\Phi^{-1}}(\partial_tf_1,f_2)}&\lesssim 2^{4k^+}2^{(2+\beta+2\delta)m} 2^{\frac{3}{2}k} 2^{\frac{k_2}{2}}\lambda^{-1}\ltwonorm{\partial_tf_1}\ltwonorm{f_2}\\
    &\lesssim 2^{4k^+}2^{(2+\beta+2\delta)m} 2^{\frac{3}{2}k}\lambda^{-1}2^{(-\frac{3}{4}+2\delta)m} 2^{-(1+\beta)l_2}\eps^3\\
    &\lesssim 2^{(\frac{1}{4}+110\delta)m}\eps^3,
\end{align*}
which is more than enough for the claim of the proposition.

\par \textbf{Case C: $p_{\max}\ll 0$.} In this case we have that the phase $\Phi$ is large $\abs{\Phi}>\frac{1}{10}$ and we can do a splitting as per Section \ref{sec: Normal forms} with $\lambda=\frac{1}{100}$ and thus $\m^{res}=0$. So we have $\B_{\m}(f_1,f_2)=\B_{\m^{nr}}(f_1,f_2)$ and we can split the bilinear term as in \eqref{eqn: nonresonant multiplier bilinear estimate}. For the last two terms, using Lemmas \ref{lemma: normal forms}\eqref{it2: lemma: normal forms} and \ref{lemma: linear decay many vector fields}, we have that
\begin{align*}
    2^{4k^+}2^{(1+\beta+2\delta)m}2^{\beta p}2^{\frac{p}{2}} \ltwonorm{P_{k,p}R_l\B_{\m^{nr}\Phi^{-1}}(\partial_tf_1,f_2)}&\lesssim 2^{(2+\beta+3\delta)m}2^{\beta p}2^{\frac{p}{2}}2^{k+p_{\max}}\ltwonorm{\partial_tf_1}\inftynorm{e^{it\Lambda}f_2}\\
    &\lesssim 2^{(\frac{3}{4}+2\beta)m}\eps^3,
\end{align*}
assuming $\kappa\ll \beta$ in Lemma \ref{lemma: linear decay many vector fields}. The other term is symmetric in this estimate and is bounded analogously. As for the boundary term, assume w.l.o.g. $p_1\leq p_2$, then we have the multiplier bound from Lemma \ref{lemma: multiplier bound}. We distinguish two cases:
\par \textbf{Case C.1:} If $f_2$ has fewer vector fields than $f_1$, then we can decompose $f_2$ according to Proposition \ref{proposition: linear decay} $P_{k_2,p_2}e^{it\Lambda}f_2=I_{k_2,p_2}(f_2)+II_{k_2,p_2}(f_2)$ with bounds as in \eqref{eqn: linfty decay of the profile}-\eqref{eqn: l2 decay of the profile} and use Lemma \ref{lemma: linear decay many vector fields} on $f_1$. With $\log(t)\lesssim 2^{\delta m}$ and using Lemmas \ref{lemma: linear decay many vector fields} and \ref{lemma: normal forms}\eqref{it2: lemma: normal forms} we obtain:
\begin{align*}
    &\ltwonorm{P_{k,p}R_l\mathcal{Q}_{\m\Phi^{-1}}(f_1,f_2)}\\
    &\hspace{5em}\lesssim 2^{k}\big[\ltwonorm{f_1}\inftynorm{I_{k_2,p_2}(f_2)}+\sinftynorm{e^{it\Lambda}f_1}\ltwonorm{II_{k_2,p_2}(f_2)} \big]\\
     &\hspace{5em}\lesssim 2^{k}\big[ 2^{-4k_1^+}2^{\frac{p_1}{2}}\norm{f_1}_B 2^{\frac{3}{4}k_2}2^{-\frac{15}{4}k_2^+}\min\{2^{-p_2}2^{-m},2^{p_2}\}2^{\delta m}\eps+2^{-2k_1^+}2^{(-\frac{1}{2}+\kappa)m}\eps 2^{-4k_2^+}2^{-\frac{m}{2}}\eps \big]\\
     &\hspace{5em}\lesssim 2^{k}[2^{-3k_2^+-4k_1^+}2^{(-\frac{3}{4}+\delta)m}+2^{-4k_2^+-2k_1^+}2^{(-1+\kappa) m}]\eps^2\\
     &\hspace{5em}\lesssim 2^{(-\frac{3}{4}+\delta)m} \eps^2.
\end{align*}
For the $X$-norm we then obtain that
\begin{align*}
     2^{4k^+}2^{(1+\beta+2\delta)m}2^{\beta p}2^{\frac{p}{2}}\ltwonorm{P_{k,p}R_l\mathcal{Q}_{\m\Phi^{-1}}(f_1,f_2)}&\lesssim 2^{(\frac{1}{4}+\beta+6\delta)m}\eps^2,
\end{align*}
which gives the claim.
\par \textbf{Case C.2} If $f_2$ has more vector fields than $f_1$, then by Proposition \ref{proposition: linear decay} there holds $\sinftynorm{e^{it\Lambda}f_1}\lesssim2^{\frac{3}{4}k_1-3k_1^+}2^{-\frac{m}{2}}\eps$. Moreover, notice that since $\abs{\Phi}>\frac{1}{10}$, there holds:
    \begin{align*}
        \abs{\frac{\m \chi S_{\xi-\eta}(\Phi^{-1})}{S_{\xi-\eta} \Phi}}=\abs{\frac{\m \chi \Phi^{-2}S_{\xi-\eta} \Phi}{S_{\xi-\eta} \Phi}}\lesssim \abs{\m \chi}.
    \end{align*}
Therefore, we can integrate by parts along $S_{\xi-\eta}$, see proof of Lemma \ref{lemma: ibp in bilinear expressions}. This gives the claim if
\begin{align}\label{eqn: ibp in Seta Case C X norm}
    2^{-p_2-p_{\max}}2^{2k_2-k_{\max}-k_{\min}}(1+2^{k_1-k_2}2^{l_2})<2^{(1-\delta)m}.
\end{align}
Otherwise we distinguish different cases and do an $L^2-L^{\infty}$ estimate:
\par \textbf{Case C.2(a): $k_{\min}+k_{\max}\sim k_1+k_2$.} Assume $k_2-k_1<l_2$, then \eqref{eqn: ibp in Seta Case C X norm} doesn't hold if $-l_2<-(1-\delta)m-p_2-p_{\max}$, then we have for the boundary term:
\begin{align*}
  2^{4k^+}2^{(1+\beta+2\delta)m}2^{\beta p}2^{\frac{p}{2}}   \ltwonorm{P_{k,p}R_l\mathcal{Q}_{\m\Phi^{-1}}(f_1,f_2)}&\lesssim  2^{4k^+}2^{(1+\beta+2\delta)m}2^{k+p_{\max}}2^{-3k_1^+}2^{-\frac{m}{2}}\eps\ltwonorm{f_2}\\
    &\lesssim 2^{(\frac{1}{2}+\beta+3\delta)m} 2^{k+p_{\max}}\eps2^{-\frac{l_2}{4}}2^{\frac{p_2}{4}}\norm{f_2}_X\\
    &\lesssim 2^{(\frac{1}{4}+\beta+6\delta)m}\eps^2.
\end{align*}
Otherwise, if $k_1-k_2<-(1-\delta)m-p_2-p_{\max}$ there holds:
\begin{align*}
     2^{4k^+}2^{(1+\beta+2\delta)m}2^{\beta p}2^{\frac{p}{2}}   \ltwonorm{P_{k,p}R_l\mathcal{Q}_{\m\Phi^{-1}}(f_1,f_2)}&\lesssim  2^{4k^+}2^{(1+\beta+2\delta)m}2^{k+p_{\max}}2^{\frac{3k_1}{4}}2^{-3k_1^+}2^{-\frac{m}{2}}\eps 2^{\frac{p_2}{2}}\norm{f_2}_B\\
     &\lesssim 2^{2\beta m}\eps^2.
\end{align*}
\par \textbf{Case C.2(b):  $k_{\min}\sim k$.} Then \eqref{eqn: ibp in Seta Case C X norm} doesn't holds if $-l_2<-(1-\delta)m-p_2-p_{\max}-k+k_1$. Otherwise:
\begin{align*}
    2^{4k^+}2^{(1+\beta+2\delta)m}2^{\beta p}2^{\frac{p}{2}}  \ltwonorm{P_{k,p}R_l\mathcal{Q}_{\m\Phi^{-1}}(f_1,f_2)}& \lesssim 2^{4k^+}2^{(1+\beta+2\delta)m}2^{k+p_{\max}}2^{-\frac{m}{2}}\eps2^{-6k_2^+}2^{-\frac{l_2}{4}}2^{\frac{p_2}{4}}\norm{f_2}_X\\
    &\lesssim 2^{(\frac{1}{2}+\beta+3\delta)m}2^{k+p_{\max}}2^{-\frac{(1-\delta)m}{4}-\frac{p_2+p_{\max}}{4}+\frac{k_2-k}{4}}2^{\frac{p_2}{4}}\eps^2\\
    &\lesssim 2^{(\frac{1}{4}+2\beta)m}\eps^2.
\end{align*}
This finished \textbf{Case C}.
\par \textbf{Case D: No gaps with $p\sim p_1\sim p_2\sim 0$.} As the linear decay is not enough to obtain the claim, in this instance we introduce the $q,q_1,q_2$ localizations:
\begin{align*}
    P_{k,p}R_lB_{\m}(f_1,f_2)=\sum_{q,q_1,q_2\in \Z^-}P_{k,p,q}R_lB_{\m}(P_{k_1,p_1,q_1}f_1,P_{k_2,p_2,q_2}f_2),
\end{align*}
where by abuse of notation we let $f_i=P_{k_i,p_i,q_i}R_{l_i}f_i$ for $i=1,2$. Moreover, recall the notation $\tchi$ from \eqref{eqn: chi localizations} and note that from Lemma \ref{lemma: multiplier bound} there holds $\inftynorm{\m\tchi}\lesssim 2^{k+p_{\max}+q_{\max}}$. 
\par Using the set size estimate \ref{lemma: set size estimate}, it suffices to bound $\norm{P_{k,p,q}B_{\m}(P_{k_1,p_1,q_1}f_1,P_{k_2,p_2,q_2}f_2)}_X$ for finitely many $q,\;q_1,\;q_2\in \Z^-$. Indeed there holds that
\begin{align*}
    \ltwonorm{ P_{k,p}R_lB_{\m}(f_1,f_2)}&\lesssim 2^m2^k \sum_{q,q_1,q_2\in \Z^-}\ltwonorm{P_{k,p,q}R_lB_{\m}(f_1,f_2)}\\
    &\lesssim 2^m2^{2k_{\max}} \sum_{q,q_1,q_2\in \Z^-} 2^{\frac{q_{\min}}{2}}2^{-N_0k_1^+}2^{-N_0k_2^+}\norm{f_1}_{H^{N_0}}\norm{f_2}_{H^{N_0}}.
\end{align*}
Therefore with the energy estimates obtained from the bootstrap assumption \eqref{eqn: bootstrap assumption} we obtain the claim if $q_{\min}<-12m$.
So in the following we assume $q, q_1, q_2>-12m$ and prove that
\begin{align*}
   \ltwonorm{ P_{k,p,q}R_lB_{\m}(f_1,f_2)}\lesssim 2^{(-\frac{1}{2}-\frac{3}{2}\beta +3\delta)m}\eps^2.
\end{align*}
Therefore, since $\log(t)\sim m\lesssim 2^{\delta m}$, we obtain the claim of the proposition:
\begin{align*}
     \ltwonorm{ P_{k,p}R_lB_{\m}(f_1,f_2)}&\lesssim \sum_{\substack{q,q_1,q_2\in \Z^-\\ q_{\min}<-12m}}  \ltwonorm{ P_{k,p,q}R_lB_{\m}(f_1,f_2)}+\sum_{\substack{q,q_1,q_2\in \Z^-\\ q_{\min}\geq-12m}} \ltwonorm{ P_{k,p,q}R_lB_{\m}(f_1,f_2)}\\
     &\lesssim 2^{-2m}\eps^2+ \sum_{q,q_1,q_2\geq-12m}2^{(-\frac{1}{2}-\frac{3}{2}\beta+3\delta)m}\eps^2\\
     &\lesssim 2^{(-\frac{1}{2}-\frac{5}{4}\beta)m}\eps^2.
\end{align*}
\par To begin with, we split the analysis $\m=\m^{res}+\m^{nr}$ as described in Section \ref{sec: Normal forms} with $\lambda= 2^{q_{\max}-10}$. On the non-resonant part we do a normal transform as in \eqref{eqn: nonresonant multiplier bilinear estimate} and treat each term separately. Observe that by Lemma \ref{lemma: m phi^-1 multiplier bound}, there holds:
\begin{align*}
    \abs{\m^{nr}\Phi^{-1}}\lesssim \norm{\m^{nr}}_W\snorm{\Phi^{-1}}_W\lesssim 2^{k+q_{\max}}2^{-q_{\max}}\lesssim 2^k.
\end{align*} Moreover, assume w.l.o.g. that $f_1$ has fewer vector fields than $f_2$. Then from Proposition \ref{proposition: linear decay} since $p\sim p_i\sim 0$, we have  $f_1=P_{q_1}I_{k_1,p_1}(f_1)+P_{q_1}II_{k_1,p_1}(f_1)$ with the following estimates:
\begin{align}\label{eqn: X norm Case D linear decay}
     \inftynorm{I_{k_1,p_1}(f_1)}\lesssim 2^{\frac{3k_1}{4}}2^{-\frac{15}{4}k_1^+}2^{(-1+\delta)m}\eps, && \ltwonorm{II_{k_1,p_1}(f_1)}\lesssim 2^{-4k_1^+}2^{-(\frac{1}{2}+\frac{\beta}{2})m}\eps.
\end{align}
Hence for the boundary term in the normal form \eqref{eqn: nonresonant multiplier bilinear estimate} by Lemma \ref{lemma: Hoelder with multiplier} and Lemma \ref{lemma: linear decay many vector fields} with $\kappa\ll \beta/2$ on $f_2$, we obtain:
\begin{equation}\label{eqn: nonresonant bdry term no gap X norm}
\begin{split}
     \ltwonorm{P_{k,p,q}R_l\mathcal{Q}_{\m^{nr}\Phi^{-1}}(f_1,f_2)}&\lesssim 2^k[\inftynorm{I_{k_1,p_1}(f_1)}\ltwonorm{f_2}+\ltwonorm{II_{k_1,p_1}(f_1)}\sinftynorm{e^{it\Lambda}f_2}]\\
    &\lesssim2^{k}[ 2^{-3k_1^+}2^{(-1+\delta)m}2^{-4k_2^+}\eps^2+ 2^{-4k_1^+}2^{-(\frac{1}{2}+\frac{\beta}{2})m} 2^{-2k_2^+}2^{(-\frac{1}{2}+\kappa)m}\eps^2]\\
    &\lesssim 2^{(-1+\delta)m}\eps^2.
\end{split}
\end{equation}
Thus for the $X$-norm we obtain:
\begin{align*}
    2^{4k^+}2^{(1+\beta+2\delta)m}\ltwonorm{P_{k,p,q}R_l\mathcal{Q}_{\m^{nr}\Phi^{-1}}(f_1,f_2)}&\lesssim 2^{2\beta m}\eps^2,
\end{align*}
which is an acceptable bound. Next using Lemma \ref{lemma: bounds on dtSf} we compute that
\begin{equation}\label{eqn: nonresonant dtSnf term no gap X norm}
\begin{split}
     2^{4k^+}2^{(1+\beta+2\delta)m}\ltwonorm{P_{k,p,q}R_l\B_{\m^{nr}\Phi^{-1}}(\partial_tf_1,f_2)}&\lesssim 2^{4k^+}2^{(2+\beta+2\delta)m}2^k\ltwonorm{\partial_tf_1}\sinftynorm{e^{it\Lambda}f_2}\\
    &\lesssim 2^{(\frac{3}{4}+2\beta)m}\eps^3.
\end{split}
\end{equation} 
 Similarly we obtain the claim for the other term in \eqref{eqn: nonresonant multiplier bilinear estimate}, where we have even a better bound using the decomposition \eqref{eqn: X norm Case D linear decay} on $f_1$.
This concludes the non-resonant part.
\par As for the resonant case, we observe that if $\abs{\Phi}<2^{q_{\max}-10}$ then by Proposition \ref{prop: lower bound on sigma} there holds $\abs{\sigma}>2^{k_{\min}+k_{\max}+q_{\max}}$. Here we consider several cases based on the sizes of $q,\; q_i$.
\par\textbf{Case D.1: $q_1\geq q_{\max}-50$ or $q_2\geq q_{\max}-50$.} We integrate by parts along $S$ using Lemma \ref{lemma: ibp in bilinear expressions}\eqref{it2: lemma: ibp in bilinear expressions} when feasible. Observe that 
\begin{align}\label{eqn: ibp no gap case S_eta on lambda}
    \abs{\frac{\m\tchi S_\eta\psi(\lambda^{-1}\Phi)}{s S_\eta\Phi}}\lesssim \abs{\m\tchi\psi'(\lambda^{-1}\Phi)\lambda^{-1}s^{-1}}.
\end{align}
Therefore, we can integrate by parts using Lemma \ref{lemma: ibp in bilinear expressions}\eqref{it2: lemma: ibp in bilinear expressions} and obtain the claim if 
\begin{align}
   S_\eta:\hspace{0.3em} \max\{2^{k_2-k_1-q_1},2^{2k_1}2^{-k_{\min}-k_{\max}-q_{\max}}(1+2^{k_2-k_1}(2^{q_2-q_1}+2^{l_1})),2^{-q_{\max}}\}<2^{(1-\delta)m}, \label{eqn: Case D ibp  S_eta res part}\\
   S_{\xi-\eta}: \hspace{0.3em} \max\{2^{k_1-k_2-q_2},2^{2k_2}2^{-k_{\min}-k_{\max}-q_{\max}}(1+2^{k_1-k_2}(2^{q_1-q_2}+2^{l_2})),2^{-q_{\max}}\}<2^{(1-\delta)m}.\label{eqn: ibp Sxi-eta q_1=max}
\end{align}
Otherwise we proceed with several cases \textbf{D.1(a)-(d)} below. 
In the cases \textbf{D.1(a)-(c)} we can assume w.l.o.g. that $q_1\geq q_{\max}-50$ and use \emph{only} the integration by parts in $S_\eta$ \eqref{eqn: Case D ibp  S_eta res part}. The claim for $q_2\geq q_{\max}-50$ follows analogously by integrating by parts in $S_{\xi-\eta}$ using \eqref{eqn: ibp Sxi-eta q_1=max} and the symmetric estimates are obtained with the roles of $q_1,q_2$ and $k_1,k_2$ interchanged. The \textbf{Case D.1(d)} is treated separately depending whether $q_1\geq q_{\max}-50$ or $q_2\geq q_{\max}-50$.
\par Let $q_1\geq q_{\max}-50$ and assume that \eqref{eqn: Case D ibp  S_eta res part} doesn't hold.
\par\textbf{Case D.1(a):} If $2^{q_1}<2^{-(1-\delta)m}$, then there holds:
\begin{align*}
    \ltwonorm{P_{k,p,q}R_l\B_{\m^{res}}(f_1,f_2)}&\lesssim 2^m2^{k+q_{1}}[\inftynorm{I_{k_1,p_1}(f_1)}\ltwonorm{f_2}+\abs{S}\ltwonorm{II_{k_1,p_1}(f_1)}\ltwonorm{f_2}]\\
    &\lesssim 2^{\delta m}[2^{(-1+\delta)m}+2^{k_{1}+\frac{q_{1}}{2}} 2^{(-\frac{1}{2}-\frac{\beta}{2})m}]\eps^2\\
    &\lesssim 2^{(-1+2\delta)m}\eps^2.
\end{align*}
This yields an admissible bound for the $X-$norm.
\par\textbf{Case D.1(b):} If $2^{k_1}<2^{-(1-\delta)m}2^{k_2-q_1}$ and $k_1\ll k_2\sim k$.
Then we do another splitting $\m^{res}=\m^{res,nr}+\m^{res,res}$ as per Section \ref{sec: Normal forms} with $\lambda_1=\lambda 2^{-4\beta m}<\lambda $. On the non-resonant part using Lemma \ref{lemma: m phi^-1 multiplier bound} we obtain that
\begin{align*}
    \abs{\m^{res,nr}\Phi^{-1}}\lesssim \norm{\m^{res,nr}}_W\norm{\Phi^{-1}}_W\lesssim 2^{k+4\beta m}.
\end{align*}
Then we can treat the terms arising from the normal form in \eqref{eqn: nonresonant multiplier bilinear estimate} via Lemma \ref{lemma: normal forms}\eqref{it1: lemma: normal forms} as in \eqref{eqn: nonresonant bdry term no gap X norm}-\eqref{eqn: nonresonant dtSnf term no gap X norm}. The boundary term is estimated with Lemmas \ref{lemma: Hoelder with multiplier}, \ref{lemma: linear decay many vector fields} and \eqref{eqn: X norm Case D linear decay} as follows:
\begin{equation}\label{eqn: nonresonant bdry term D.1(b)}
\begin{split}
      \ltwonorm{P_{k,p,q}R_l\mathcal{Q}_{\m^{res,nr}\Phi^{-1}}(f_1,f_2)}&\lesssim 2^{k+4\beta m}[\inftynorm{I_{k_1,p_1}(f_1)}\ltwonorm{f_2}+\ltwonorm{II_{k_1,p_1}(f_1)}\sinftynorm{e^{it\Lambda}f_2}]\\
    &\lesssim2^{k+4\beta m}[ 2^{-3k_1^+}2^{(-1+\delta)m}2^{-4k_2^+}\eps^2+ 2^{-4k_1^+}2^{-(\frac{1}{2}+\frac{\beta}{2})m} 2^{-2k_2^+}2^{(-\frac{1}{2}+\kappa)m}\eps^2]\\
    &\lesssim 2^{(-1+4\beta+\delta)m}\eps^2.
\end{split}
\end{equation}
This yields an admissible bound on the $X-$norm. For the terms in \eqref{eqn: nonresonant multiplier bilinear estimate} involving the time derivative we apply Lemma \ref{lemma: bounds on dtSf} to obtain an admissible bound as follows:
\begin{equation}\label{eqn: nonresontant dtf_1 term D.1(b)}
\begin{split}2^{4k^+}2^{(1+\beta+2\delta)m}\ltwonorm{P_{k,p,q}R_l\B_{\m^{res,nr}\Phi^{-1}}(\partial_tf_1,f_2)}&\lesssim 2^{4k^+}2^{(2+5\beta+2\delta)m}2^k\ltwonorm{\partial_tf_1}\sinftynorm{e^{it\Lambda}f_2}\\
    &\lesssim 2^{(\frac{3}{4}+6\beta)m}\eps^3.
\end{split}
\end{equation}
And similarly there holds:
\begin{equation}\label{eqn: nonresontant dtf_2 term D.1(b)}
\begin{split}
2^{4k^+}2^{(1+\beta+2\delta)m}\ltwonorm{P_{k,p,q}R_l\B_{\m^{res,nr}\Phi^{-1}}(f_1,\partial_tf_2)}&\lesssim 2^{4k^+}2^{(2+5\beta+2\delta)m}2^k\sinftynorm{e^{it\Lambda}f_1}\ltwonorm{\partial_tf_2}\\
    &\lesssim 2^{(\frac{3}{4}+6\beta)m}\eps^3.
\end{split}
\end{equation}
 On the resonant part we observe that since $k_1\ll k_2$, there holds:
\begin{align*}
    \sabs{\partial_{\eta_1}\Phi_\pm^{\mu\nu}}=\sabs{\mu 2^{2p_2-k_2}-\nu 2^{2p_1-k_1}}\gtrsim 2^{-k_1}=: K,
\end{align*}
and we can use Lemma \ref{lemma: normal forms}\eqref{it3: lemma: normal forms},\eqref{it5: lemma: normal forms} with $K=2^{-k_1}$ and $\lambda_1$ to obtain that
\begin{align*}
     &\ltwonorm{P_{k,p,q}R_l\B_{\m^{res,res}}(f_1,f_2)}\\
      &\hspace{10em}\lesssim 2^m2^{k+q_{1}}[\inftynorm{I_{k_1,p_1}(f_1)}\ltwonorm{f_2}+(\lambda_1 K^{-1})^{\frac{1}{2}}2^{\frac{k_1+q_1}{2}}\ltwonorm{II_{k_1,p_1}(f_1)}\ltwonorm{f_2}]\\
   &\hspace{10em}\lesssim 2^{(-\frac{1}{2}-2\beta)m}\eps^2.
\end{align*}
Observe that if $ q_2\geq q_{\max}-50$, then we can use Lemma \ref{lemma: normal forms}\eqref{it3: lemma: normal forms},\eqref{it5: lemma: normal forms} instead of the $L^\infty$ decay in the first term of the sum on the right-hand side above to obtain the claim.
\par\textbf{Case D.1(c):} Assume $2^{-k_1}<2^{-(1-\delta)m}2^{-k_{\min}-q_1}$ and $2^{k_2-k_1+l_1}\ll 1$. In particular, there holds $k_{\min}\sim k_2\ll k_1$. We additionally split the analysis in a resonant and non-resonant part with $\lambda_1=2^{-4\beta m}\lambda$. Indeed, the non-resonant part where $\abs{\Phi_\pm^{\mu\nu}}\gtrsim \lambda_1$ is handled as in \eqref{eqn: nonresonant bdry term D.1(b)}-\eqref{eqn: nonresontant dtf_2 term D.1(b)} using the linear decay \eqref{eqn: X norm Case D linear decay} and Lemma \ref{lemma: normal forms}\eqref{it1: lemma: normal forms}, while on the resonant part we have $\abs{\partial_{\eta_1}\Phi_\pm^{\mu\nu}}\gtrsim 2^{-k_2}$ and we can use Lemma \ref{lemma: normal forms}\eqref{it3: lemma: normal forms},\eqref{it5: lemma: normal forms} with $L=2^{-k_2}$ and $\lambda_1$ as we have done in \textbf{Case D.1(b)}.

\par\textbf{Case D.1(d):} First let $ q_2\geq q_{\max}-50$. The remaining case to treat if \eqref{eqn: ibp Sxi-eta q_1=max} doesn't hold is when $2^{-l_2}<2^{-(1-\delta)m}2^{k_1+k_2-k_{\min}-k_{\max}-q_2}$. Then we obtain the claim using \eqref{eqn: X norm Case D linear decay} and Lemma \ref{lemma: set size estimate}:
\begin{align*}
    \ltwonorm{P_{k,p,q}R_l\B_{\m^{res}}(f_1,f_2)}&\lesssim 2^m2^{k+q_{2}}[\inftynorm{I_{k_1,p_1}(f_1)}\ltwonorm{f_2}+\abs{S}\ltwonorm{II_{k_1,p_1}(f_1)}\ltwonorm{f_2}]\\
    &\lesssim 2^{m}2^{k+q_2}[2^{-(1-\delta)m}2^{-l_2}+2^{\frac{k_{\min}}{2}}2^{\frac{k_2+q_2}{2}}2^{(-\frac{1}{2}-\beta+\delta)m}2^{-(1+\beta)l_2}]\eps\norm{f_2}_X\\
    &\lesssim 2^{(-\frac{1}{2}-2\beta+4\delta)m}\eps^2.
\end{align*}
\par Let now $q_1\geq q_{\max}-50$ and assume $2^{-l_1}<2^{-(1-\delta)m}2^{k_1+k_2-k_{\min}-k_{\max}-q_1}$. Here we obtain the claim by additionally integrating by parts in $S_{\xi-\eta}$ using \eqref{eqn: ibp Sxi-eta q_1=max}. Observe that here in each of the terms in \eqref{eqn: ibp Sxi-eta q_1=max} we either have a ``loss" in the parameter $q_1$ \emph{or} $q_2$, cf.Lemma \ref{lemma: ibp in bilinear expressions}\eqref{it2: lemma: ibp in bilinear expressions}. 
\par Assume now that \eqref{eqn: ibp Sxi-eta q_1=max} doesn't hold.  We consider several cases based on the relative size of the parameters $k,k_1,k_2$.
\par\textbf{Case D.1(d.1): $2^{k_2}\ll 2^{k_1}$ or  $2^{k_1}\ll 2^{k_2}$.} In particular, there holds $2^{-l_1}<2^{-(1-\delta)m-q_1}$. Moreover:
\begin{align*}
\abs{\partial_{\eta_1}\Phi_\pm^{\mu\nu}}=\abs{\nu\partial_{\eta_1}\Lambda(\xi-\eta)\pm \nu \partial_{\eta_1}\Lambda(\eta)}\gtrsim 2^{-\min\{k_1,k_2\}}=:K.
\end{align*}
We split the analysis further into the resonant and non-resonant part with $\lambda_1=\lambda2^{-4\beta m}$. The non-resonant part is treated as in \eqref{eqn: nonresonant bdry term D.1(b)}-\eqref{eqn: nonresontant dtf_2 term D.1(b)} using the linear decay \eqref{eqn: X norm Case D linear decay} and Lemmas \ref{lemma: normal forms}\eqref{it1: lemma: normal forms}, \ref{lemma: bounds on dtSf} which are independent of the relative size of $k_1$ to $k_2$. On the resonant set, where $\abs{\Phi}\lesssim \lambda_1$, we can use Lemma \ref{lemma: normal forms}\eqref{it3: lemma: normal forms},\eqref{it5: lemma: normal forms} with $K=2^{-\min\{k_1,k_2\}}$, $\lambda_1$ and $q$-localizations to obtain the bound
\begin{equation}\label{eqn: resonant part q_1=max}
\begin{split}
      \ltwonorm{P_{k,p,q}R_l\B_{\m^{res,res}}(f_1,f_2)}&\lesssim 2^m2^{k+q_{1}}(\lambda_1K^{-1})^{\frac{1}{2}}2^{\frac{\min\{k_1+q_1,k_2+q_2\}}{2}}2^{-l_1}\norm{f_1}_X2^{-l_2}\norm{f_2}_X\\
    &\lesssim 2^{(1-2\beta)m}2^{k+\frac{3}{2}q_1}2^{\frac{\min\{k_1,k_2\}}{2}}2^{\frac{\min\{k_1+q_1,k_2+q_2\}}{2}}2^{-l_1-l_2}\eps^2.
\end{split}
\end{equation}
Now if $2^{q_1}<2^{-(1-\delta)m}$ the claim follows directly. If $2^{k_1-k_2-q_2}>2^{(1-\delta)m}$, then there holds $2^{\frac{k_2+q_2}{2}}\lesssim 2^{-\frac{1}{2}(1-\delta)m+\frac{k_1}{2}}$ and we can continue the estimate in \eqref{eqn: resonant part q_1=max} to obtain an admissible bound:
\begin{equation}\label{Case D.1(d.1) res res}
\begin{split}
    \ltwonorm{P_{k,p,q}R_l\B_{\m^{res,res}}(f_1,f_2)}&\lesssim 2^{(1-2\beta)m}2^{\frac{3}{2}k_{\max}+\frac{3}{2}q_{1}}2^{\frac{k_2+q_2}{2}}2^{-l_1}\norm{f_1}_X\norm{f_2}_X\\
    &\lesssim 2^{(-\frac{1}{2}-2\beta+3\delta)m}\eps^2.
\end{split}
\end{equation}
Similarly, if 
 $2^{k_1+k_2-k_{\min}-k_{\max}-q_{2}}>2^{(1-\delta)m}$ then $2^{\frac{q_2}{2}}<2^{-\frac{1}{2}(1-\delta)m}$ and the claim follows from \eqref{Case D.1(d.1) res res}. Lastly, if $2^{2k_2-k_{\min}-k_{\max}-q_{1}+l_2}>2^{(1-\delta)m}$, then continuing the estimate \eqref{eqn: resonant part q_1=max} we obtain:
\begin{align*}
      \ltwonorm{P_{k,p,q}R_l\B_{\m^{res,res}}(f_1,f_2)}&\lesssim2^{(1-2\beta)m}2^{k_{\max}+2q_{1}}2^{k_{\min}}2^{-l_1}2^{-l_2}\eps^2\\
      &\lesssim 2^{(-1-2\beta+3\delta)m}\eps^2.
\end{align*}
\par\textbf{Case D.1(d.2): $2^{k_2}\sim 2^{k_1}$.} Observe that $q_1\sim q_2\sim q_{\max}$ was treated at the beginning of \textbf{Case D.1(d)}. Thus we may assume $q_2< q_1-100$. The following bound holds for the resonant bilinear term:
\begin{align*}
     \ltwonorm{P_{k,p,q}R_l\B_{\m^{res}}(f_1,f_2)}&\lesssim 2^m2^{k+q_{1}}[\inftynorm{I_{k_1,p_1}(f_1)}\ltwonorm{f_2}\\
     &\qquad+\min\{\abs{S}\ltwonorm{II_{k_1,p_1}(f_1)}\ltwonorm{f_2},\ltwonorm{II_{k_1,p_1}(f_1)}\inftynorm{e^{it\Lambda}f_2}\}].
\end{align*}
By considering each possible maximum on the left-hand side of \eqref{eqn: ibp Sxi-eta q_1=max} with $k_{\min}\sim k$ and $q_2\ll q_1$ we observe that we have two possibilities. Either
$2^{-l_2}<2^{-(1-\delta)m}2^{-q_1-k+k_2}$ and we use \eqref{eqn: X norm Case D linear decay} on $f_1$, the set size estimate with $\abs{S}\lesssim 2^{\frac{k+k_1+q_1}{2}}$ and the $X-$norm on $f_2$ to obtain an admissible bound:
\begin{align*}
     \ltwonorm{P_{k,p,q}R_l\B_{\m^{res}}(f_1,f_2)}&\lesssim 2^{m}2^{k+q_1}[2^{(-1+\delta)m} 2^{-l_2}2^{-4k_2^+}+2^{\frac{k+k_1+q_1}{2}}2^{(-\frac{1}{2}-\frac{\beta}{2})m}2^{-(1+\beta)l_2}2^{-3k_2^+}]\eps\norm{f_2}_X\\
     &\lesssim (2^{(-1+2\delta)m}+2^{(-\frac{1}{2}-\frac{3}{2}\beta+2\delta)m})\eps^2\\
     &\lesssim 2^{(-\frac{1}{2}-\frac{3}{2}\beta+2\delta)m}\eps^2.
\end{align*}
Else there holds $2^{k+q_2}<2^{-(1-\delta)m+k_2}$. Recall that in this case we also have $2^{-l_1}<2^{-(1-\delta)m}2^{k_2-k-q_1}$. First of all, observe that if $-k<-4\beta m$ we obtain the claim using the set-size estimate Lemma \ref{lemma: set size estimate} and the $X$-norm on $f_1$:
\begin{align*}
     \ltwonorm{P_{k,p,q}R_l\B_{\m^{res}}(f_1,f_2)}&\lesssim 2^m2^{k+q_{1}}2^{\frac{k+k_2+q_2}{2}}\ltwonorm{f_1}\ltwonorm{f_2}\lesssim 2^m2^{k+q_{1}}2^{k_2+\frac{q_2}{2}}2^{-l_1}\eps^2\lesssim 2^{(-\frac{1}{2}-2\beta)m}\eps^2.
\end{align*}
Otherwise, if $k<4\beta m$ we do a further splitting $\m^{res}=\m^{res,nr}+\m^{res,res}$ at the scale ${\lambda_2}=2^{-4\beta m}2^{k}\lambda<\lambda$. We bound each term in the non-resonant part \eqref{eqn: nonresonant multiplier bilinear estimate} as in \textbf{Case D.1(b)} where $\sabs{\m^{res,nr}\Phi^{-1}}\lesssim \snorm{\m^{res,nr}}_W \snorm{\Phi^{-1}}_W\lesssim 2^{4\beta m}$ by Lemma \ref{lemma: m phi^-1 multiplier bound}. For the boundary term, using Lemma \ref{lemma: normal forms}\eqref{it1: lemma: normal forms} with $\lambda_2$ and \eqref{eqn: X norm Case D linear decay}, we obtain:
\begin{align*}
       \ltwonorm{P_{k,p,q}R_l\mathcal{Q}_{\m^{res,nr}\Phi^{-1}}(f_1,f_2)}&\lesssim 2^{4\beta m}[\inftynorm{I_{k_1,p_1}(f_1)}\ltwonorm{f_2}+\ltwonorm{II_{k_1,p_1}(f_1)}\sinftynorm{e^{it\Lambda}f_2}]\\
    &\lesssim2^{4\beta m}[ 2^{-3k_1^+}2^{(-1+\delta)m}2^{-4k_2^+}+ 2^{-4k_1^+}2^{-(\frac{1}{2}+\frac{\beta}{2})m} 2^{-2k_2^+}2^{(-\frac{1}{2}+\kappa)m}]\eps^2\\
    &\lesssim 2^{(-1+4\beta+\delta)m}\eps^2.
\end{align*}
The terms in \eqref{eqn: nonresonant multiplier bilinear estimate} involving the time derivative follow using Lemma \ref{lemma: bounds on dtSf}. Indeed:
\begin{align*}
2^{4k^+}2^{(1+\beta+2\delta)m}\ltwonorm{P_{k,p,q}R_l\B_{\m^{res,nr}\Phi^{-1}}(\partial_tf_1,f_2)}&\lesssim 2^{4k^+}2^{(2+5\beta+2\delta)m}\ltwonorm{\partial_tf_1}\sinftynorm{e^{it\Lambda}f_2}\\
    &\lesssim 2^{(\frac{3}{4}+6\beta)m}\eps^3.
\end{align*}
The analogous estimate holds for the term involving $\partial_t f_2$.
On the resonant set observe that:
\begin{align*}
\abs{\partial_{\eta_2}\Phi_{\pm}^{\mu\nu}}=\abs{\mu\partial_{\eta_2}\Lambda(\xi-\eta)\pm\nu \partial_{\eta_2}\Lambda(\eta)}=\abs{\mu \frac{\eta_1\eta_2}{\abseta^3}\mp\nu\frac{(\xi_1-\eta_1)(\xi_2-\eta_2)}{\absxieta^3}}\gtrsim2^{-k_2+q_1}=:K.
\end{align*} 
Thus, we can apply Lemma \ref{lemma: normal forms}\eqref{it4: lemma: normal forms} with $K=2^{-k_2+q_1}$ and $\lambda_2$ to obtain the claim on the $X-$norm as follows
\begin{align*}
    \ltwonorm{P_{k,p,q}R_l\B_{\m^{res,res}}(f_1,f_2)}&\lesssim 2^m2^{k+q_{1}}(\lambda_2K^{-1})^{\frac{1}{2}}2^{\frac{\min\{k_1+q_1,k_2+q_2\}}{2}}\ltwonorm{f_1}\ltwonorm{f_2}\\
    &\lesssim 2^{(1-2\beta)m}2^{\frac{3}{2}k+q_1}2^{k_2+\frac{q_2}{2}}2^{-l_1}2^{-8k_2^+}\eps^2\\
    &\lesssim 2^{(-\frac{1}{2}-2\beta+2\delta)m}\eps^2.
\end{align*}

\par \textbf{Case D.3: $q_{1}<q_{\max}-50$ and $q_{2}<q_{\max}-50$.} In particular, there holds $q_{\max}=q-50> q_1,q_2$ and therefore  $\abs{\Phi_{\pm}^{\mu\nu}}\gtrsim 2^{q-10}$. Thus the splitting described at the beginning of the \textbf{Case D} with $\lambda=2^{q-10}$ contains only the non-resonant part $\m=\m^{nr}$ which was handled in \eqref{eqn: nonresonant bdry term no gap X norm}-\eqref{eqn: nonresonant dtSnf term no gap X norm}. 
\end{proof}

\addtocontents{toc}{\protect\setcounter{tocdepth}{0}}
\section*{Acknowledgments}
The authors gratefully acknowledge support of the SNSF through grant PCEFP2\_203059. \blue{They thank the anonymous referee for their careful reading of the manuscript and the many valuable suggestions.}

\addtocontents{toc}{\protect\setcounter{tocdepth}{1}}

\appendix 
\section{}\label{appendix}
\begin{subsection}{Control of the Fourier transform in  \texorpdfstring{$L^\infty$}{Linfty}}\label{ssec:Linftybd}
   \begin{proof}[Proof of Lemma \ref{lemma: control of Fourier transform}]
By abuse of notation we consider $f$ to be already localized, that is $\widehat{f}=\widehat{P_{k,p}f}=\varphi_{k,p}(\xi)\widehat{f}(\xi)$ and consider polar coordinates as in \eqref{eqn: polar coordinates}. Then with \eqref{eqn: phi_kp in rho lambda} we can write for any $(\rho_0,\tau_0) \in \supp{\varphi_{k,p}\widehat{f}}$ using the fundamental theorem of calculus:
\begin{equation}\label{HDIR}
    \begin{split}
    \phikp\widehat{f}(\rho,\tau)&=\phikp\widehat{f}(\rho,\tau_0)+\int_{\tau_0}^\tau\partial_\alpha(\phikp\widehat{f})(\rho,\alpha)d\alpha\\
&=\phikp\widehat{f}(\rho_0,\tau_0)+\int_{\rho_0}^\rho \partial_s(\phikp\widehat{f})(s,\tau_0)ds+\int_{\tau_0}^\tau\partial_\alpha(\phikp\widehat{f})(\rho_0,\alpha)d\alpha+\\
&\hspace{1cm}+\int_{\rho_0}^\rho \int_{\tau_0}^\tau \partial_s\partial_\alpha(\phikp\widehat{f})(s,\alpha)dsd\alpha.
\end{split}
\end{equation}
Now bound each of these terms on the right-hand side. Let $\overline{\varphi}_{k,p}$ be a function with similar support properties as $\phikp$ (cf.\ Remark \ref{rk: notation similar supp properties}) and observe that for the last term in \eqref{HDIR} there holds: 
\begin{align*}
\sabs{\partial_\rho\partial_\tau(\phikp\widehat{f})(s,\alpha)}&=\sabs{\partial_\rho(\partial_\tau\phikp\widehat{f}+\phikp\partial_\tau\widehat{f})(s,\alpha)}\\
&=\sabs{\partial_\rho\partial_\tau\phikp \widehat{f}(s,\alpha) +\partial_\tau\phikp\partial_\rho\widehat{f}(s,\alpha)+\partial_\rho\phikp \partial_\tau\widehat{f}(s,\alpha)+\phikp \partial_\tau\partial_\rho\widehat{f}(s,\alpha)}\\
&\lesssim \overline{\varphi}_{k,p}\Big[ 2^{-k-p}\sabs{\widehat{f}(s,\alpha)}+2^{-p}\sabs{\partial_\rho\widehat{f}(s,\alpha)}+2^{-k}\sabs{\partial_\tau\widehat{f}(s,\alpha)}+\sabs{\partial_\rho\partial_\tau\widehat{f}(s,\alpha)}  \Big]
\end{align*}
Also for any $g\in L^2$, we observe that by the Cauchy-Schwarz inequality there holds
\begin{equation}\label{absolute value of double integral}
\begin{split}
\abs{\int_{\tau_0}^\tau\int_{\rho_0}^{\rho}\overline{\varphi}_{k,p}\widehat{g}(s,\alpha)dsd\alpha}^2&\lesssim  \int_{\supp \overline{\varphi}_{k,p}}s^{-1}dsd\alpha \int_{\supp\overline{\varphi}_{k,p}}\abs{\widehat{g}(s,\alpha)}^2sdsd\alpha \lesssim 2^p \ltwonorm{g}^2.
\end{split}\end{equation}
With this and the previous observation, we can directly bound the last term in \eqref{HDIR}
\begin{align*}
    \abs{\int_{\rho_0}^\rho \int_{\tau_0}^\tau \partial_s\partial_\alpha(\phikp\widehat{f})(s,\alpha)dsd\alpha}&\lesssim 2^{-k-p}\abs{\int_\alpha\int_s\overline{\varphi}_{k,p}\widehat{f} dsd\alpha}+2^{-p}\abs{\int_\alpha\int_s \overline{\varphi}_{k,p} \partial_s \widehat{f} dsd\alpha}\\ 
    &\hspace{1cm}+2^{-k}\abs{\int_\alpha\int_s \overline{\varphi}_{k,p}\partial_\alpha \widehat{f} ds d\alpha} +\abs{\int_\alpha \int_s \overline{\varphi}_{k,p} \partial_s\partial_\alpha \widehat{f} dsd\alpha}\\ 
    &\lesssim 2^{-k-\frac{p}{2}}[\ltwonorm{f}+\ltwonorm{Sf}]+2^{-k+\frac{p}{2}}[\ltwonorm{Wf}+\ltwonorm{WSf}],
\end{align*}
where we have used $S=s\partial_s$ and $W=\partial_\alpha$. Similarly, we can average in $\tau$ and obtain:
\begin{align*}
    \abs{\int_{\rho_0}^\rho \partial_s(\phikp\widehat{f})(s,\tau_0)ds}&=\abs{\int_{\rho_0}^\rho\frac{1}{\abs{\tau \in \supp\phikp}} \int_{\tau\in \supp \phikp} \partial_s(\phikp\widehat{f})(s,\tau_0)d\tau ds}\\
    &\hspace{-2cm}=\abs{\int_{\rho_0}^\rho\frac{1}{\abs{\tau \in \supp\phikp}} \int_{\tau}\Big[\int_{\tau_0}^\tau \partial_\alpha \partial_s(\phikp\widehat{f})(s,\alpha)d\alpha +\partial_s(\phikp\widehat{f})(s,\tau)\Big]d\tau ds }\\
    &\hspace{-2cm}\lesssim\abs{\int_{\rho_0}^\rho\int_{\tau_0}^\tau\partial_\alpha \partial_s(\phikp\widehat{f})(s,\alpha)d\alpha ds} + 2^{-p}\abs{\int_s\int_\tau \partial_s(\phikp\widehat{f})(s,\tau)d\tau ds}.
\end{align*}
The first term in the sum on the right-hand side above is handled in the previous estimates. As for the second term, we can use \eqref{absolute value of double integral} to obtain
\begin{align*}
    \abs{\iint \partial_s(\phikp\widehat{f})(s,\tau)d\tau ds}&\lesssim 2^{-k}\abs{\iint \overline{\varphi}_{k,p} \widehat{f}(s,\alpha) d\alpha ds}+ 2^{-k}\abs{\iint \overline{\varphi}_{k,p} \partial_s\widehat{f}(s,\alpha)s d\alpha ds }\\
    &\lesssim 2^{-k+\frac{p}{2}}[\ltwonorm{f}+\ltwonorm{Sf}].
\end{align*}
So altogether we obtain
\[\abs{\int_{\rho_0}^\rho \partial_s(\phikp\widehat{f})(s,\tau_0)ds} \lesssim 2^{-k-\frac{p}{2}}[\ltwonorm{f}+\ltwonorm{Sf}]+2^{-k+\frac{p}{2}}[\ltwonorm{Wf}+\ltwonorm{WSf}].\]
Similarly we estimate the remaining integral in \eqref{HDIR}:
\begin{align*}
    \abs{\int_{\tau_0}^\tau \partial_\alpha(\phikp\widehat{f})(\rho_0,\alpha)d\alpha}&=\abs{\int_{\tau_0}^\tau \frac{1}{\abs{\rho \in \supp \phikp}}\int_{\rho \in \supp \phikp}\partial_\alpha(\phikp\widehat{f})(\rho_0,\alpha)d\rho d\alpha}\\
    &\hspace{-2cm}=\abs{ \int_{\tau_0}^\tau \frac{1}{\abs{\rho \in \supp \phikp}}\int_\rho\Big[ \int_{\rho_0}^{\rho}\partial_s\partial_\alpha(\phikp\widehat{f})(s,\alpha)ds +\partial_\alpha(\phikp\widehat{f})(\rho,\alpha) \Big] d\rho d\alpha  }\\
    &\hspace{-2cm}\lesssim \abs{\int_{\rho_0}^\rho\int_{\tau_0}^\tau\partial_\alpha \partial_s(\phikp\widehat{f})(s,\alpha)d\alpha ds} +2^{-k}\abs{\int_\alpha \int_\rho \partial_\alpha(\phikp\widehat{f})(\rho,\alpha)d\rho d\alpha}\\
    &\hspace{-2cm} \lesssim \abs{\int_{\rho_0}^\rho\int_{\tau_0}^\tau\partial_\alpha \partial_s(\phikp\widehat{f})(s,\alpha)d\alpha ds} + 2^{-k-p}\abs{\int_\alpha\int_\rho \overline{\varphi}_{k,p} \widehat{f}(\rho,\alpha)d\rho d\alpha}\\
    &+2^{-k}\abs{\int_\alpha \int_\rho \overline{\varphi}_{k,p} \partial_\alpha\widehat{f}(\rho,\alpha)d\rho d\alpha}.
\end{align*}
So with \eqref{absolute value of double integral} we obtain
\begin{align*}
    \abs{\int_{\tau_0}^\tau \partial_\alpha(\phikp\widehat{f})(\rho_0,\alpha)d\alpha}\lesssim 2^{-k-\frac{p}{2}}[\ltwonorm{f}+\ltwonorm{Sf}]+2^{-k+\frac{p}{2}}[\ltwonorm{Wf}+\ltwonorm{WSf}].
\end{align*}
And lastly, we estimate
\begin{align*}
    \sabs{\phikp\widehat{f}(\rho_0,\tau_0)}&=\abs{\frac{1}{\abs{\tau \in \supp \phikp}} \int_{\tau \in \supp \phikp}\phikp\widehat{f}(\rho_0,\tau_0)d\tau}\\
    &=\frac{1}{\abs{\tau \in \supp \phikp}}\abs{\int_\tau\Bigg[ \int_{\tau_0}^\tau \partial_\alpha(\phikp\widehat{f})(\rho_0,\alpha)d\alpha +(\phikp\widehat{f})(\rho_0,\tau) \Big]d\tau}\\
    &\lesssim \abs{\int_{\tau_0}^\tau \partial_\tau(\phikp\widehat{f})(\rho_0,\alpha)d\alpha}+\\
    &\hspace{0.12cm}+\frac{1}{\abs{\tau \in \supp \phikp}} \abs{ \int_\tau\frac{1}{\abs{\rho \in \supp \phikp}}\Big[\int_{\rho_0}^\rho \partial_s(\phikp\widehat{f})(s,\tau)ds +\phikp\widehat{f}(\rho,\tau) \Big] d\rho d\tau}\\
    &\lesssim 2^{-k-\frac{p}{2}}[\ltwonorm{f}+\ltwonorm{Sf}]+2^{-k+\frac{p}{2}}[\ltwonorm{Wf}+\ltwonorm{WSf}] +\\
    &\hspace{1cm}+ 2^{-k-p}\abs{\int_\rho\int_\tau \phikp\widehat{f}(\rho,\tau)d\tau d\rho}\\
    &\lesssim 2^{-k-\frac{p}{2}}[\ltwonorm{f}+\ltwonorm{Sf}]+2^{-k+\frac{p}{2}}[\ltwonorm{Wf}+\ltwonorm{WSf}].
\end{align*}

Recalling the definition of the $B-$norm \eqref{B norm}, we obtain from \eqref{HDIR}
\begin{align*}
    \sinftynorm{\widehat{P_{k,p}f}}&\lesssim 2^{-k-\frac{p}{2}}[\ltwonorm{P_{k,p}f}+\ltwonorm{SP_{k,p}f}]+2^{-k+\frac{p}{2}}[\ltwonorm{WP_{k,p}f}+\ltwonorm{WSP_{k,p}f}]\\
    &\lesssim 2^{-4k^+}2^{\frac{k^-}{2}}2^{-k}[\norm{f}_B+\norm{Sf}_B]+2^{-k+\frac{p}{2}}\sum_{l\in \Z^+,l+p\geq 0}[\ltwonorm{WP_{k,p}R_lf}+\ltwonorm{WSP_{k,p}R_lf}].
\end{align*}
The claim follows by noting that with Proposition \ref{prop: angular localization properties} and the definition of the $X-$norm \eqref{X norm} there holds
\begin{align*}
    2^{\frac{p}{2}}\sum_{l\in \Z^+,l+p\geq 0}\ltwonorm{WP_{k,p}R_lf}&\lesssim  \sum_{l\in \Z^+,l+p\geq 0}2^{l}2^{\frac{p}{2}}\ltwonorm{P_{k,p}R_lf}\\
    &\lesssim \sum_{l\in \Z^+,l+p\geq 0} 2^{l}2^{\frac{p}{2}}2^{-3k^+}2^{-(1+\beta)l}2^{-\beta p}2^{-\frac{p}{2}}\norm{f}_X\\
    &\lesssim 2^{-4k^+} \sum_{l\in \Z^+,l+p\geq 0} 2^{-\beta(l+p)}\norm{f}_X\lesssim 2^{-4k^+} \norm{f}_X.
\end{align*}
\end{proof} 
\end{subsection}
\begin{subsection}{Symbol bounds}
    \begin{lemma}\label{lemma: Hoelder with multiplier}
    Consider a multiplier $\m\in L^{1}_{loc}(\R^2\times \R^2)$ and $\chi$ as in \eqref{eqn: chi localizations}. Then for bilinear expressions as in \eqref{eqn: bilinear form Qm} there holds
\begin{align*}
   & \norm{\mathcal{Q}_{\m\chi}(f,g)}_{L^r}\lesssim \snorm{\m}_{W}\lpnorm{f}{p}\lpnorm{g}{q}, &&\frac{1}{r}=\frac{1}{p}+\frac{1}{q},
\end{align*}
and where $\snorm{\m}_W:=\sup_{k,p,k_i,p_i, i=1,2}\norm{\mathcal{F}(\m\chi)}_{L^1(\R^2\times\R^2)}$.
\end{lemma}
The proof of Lemma \ref{lemma: Hoelder with multiplier} follows by the Minkowski and H\"older inequalities.
 Moreover, note the following property. Let $\m_1$, $\m_2\in L^{1}_{loc}(\R^2\times \R^2)$, then
\begin{align}\label{eqn: algebra proeprty W}
    \norm{\m_1\cdot \m_2}_{W}\lesssim \norm{\m_1}_W\norm{\m_2}_W.
\end{align}
This follows from the convolution property of the Fourier transform. Next we prove that we have the same bound for $\norm{\m}_W$ as in Lemma \ref{lemma: multiplier bound}.
\begin{lemma}\label{lemma: multiplier bound W norm}
    For $\m\in \{\m_0,\m_\pm^{\mu\nu}\}$ there holds
    \begin{align*}
        \norm{\m}_W\lesssim 2^{k+p_{\max}}.
    \end{align*}
\end{lemma}

This follows by establishing an $L^1$-bound on $\F(\m\chi)=\iint e^{-ix\cdot \xi}e^{-iy\cdot\eta}\m(\xi,\eta)\chi(\xi,\eta)d\xi d\eta$ using a suitable change of variables and integration by parts, see \cite[\textcolor{MidnightBlue}{Lemma A.13}]{guo2020stabilizing} for an analogous computation in three dimensions.

\par In Section \ref{sec: Normal forms} and in particular in Lemma \ref{lemma: normal forms}\eqref{it1: lemma: normal forms}, we want to bound $\m^{nr}\Phi^{-1}\chi$. By the algebra property \eqref{eqn: algebra proeprty W} and Lemma \ref{lemma: multiplier bound W norm} it suffices to establish the following lemma, whose proof is a straightforward adaptation to two dimensions of the proof of  \cite[\textcolor{MidnightBlue}{Lemma A.15}]{guo2020stabilizing}.
\begin{lemma} \label{lemma: m phi^-1 multiplier bound}
    Let $\Phi\in\set{\Phi_{\pm}^{\mu\nu}}{\mu,\nu\in\{+,-\}}$ and $\chi$ as in \eqref{eqn: chi localizations}. Then there holds
    \begin{align*}
        \norm{\phi^{-1}(1-\psi(\lambda^{-1}\Phi))\chi}_W\lesssim \lambda^{-1}.
    \end{align*}
\end{lemma}

\end{subsection}

\medskip
\addtocontents{toc}{\protect\setcounter{tocdepth}{0}}
\section*{Declarations}
\textbf{Data Availability Statement.} No datasets were generated or analyzed in the preparation of this manuscript.

\textbf{Conflict of interest.}
The authors have no relevant financial or non-financial interests to disclose.
\addtocontents{toc}{\protect\setcounter{tocdepth}{1}}

\bibliographystyle{alpha}
\bibliography{bib}

@article{guo2020stabilizing,
Author = {Guo, Yan and Huang, Chunyan and Pausader, Benoit and Widmayer, Klaus},
 Title = {On the stabilizing effect of rotation in the {3D} {Euler} equations},
 FJournal = {Communications on Pure and Applied Mathematics},
 Journal = {Commun. Pure Appl. Math.},
 Volume = {76},
 Number = {12},
 Pages = {3553--3641},
 Year = {2023},
 Language = {English},}

@article{EC2022,
   Author = {Guo, Yan and Pausader, Benoit and Widmayer, Klaus},
 Title = {Global axisymmetric {Euler} flows with rotation},
 FJournal = {Inventiones Mathematicae},
 Journal = {Invent. Math.},
 ISSN = {0020-9910},
 Volume = {231},
 Number = {1},
 Pages = {169--262},
 Year = {2023},
 Language = {English},}

@article{LWP_Chae,
 Author = {Chae, Dongho and Nam, Hee-Seok},
 Title = {Local existence and blow-up criterion for the {Boussinesq} equations},
 FJournal = {Proceedings of the Royal Society of Edinburgh. Section A. Mathematics},
 Journal = {Proc. R. Soc. Edinb., Sect. A, Math.},
 ISSN = {0308-2105},
 Volume = {127},
 Number = {5},
 Pages = {935--946},
 Year = {1997},
 Language = {English},}

@InCollection{gallay2019stabilityvorticesidealfluids,
 Author = {Gallay, Thierry},
 Title = {Stability of vortices in ideal fluids: the legacy of {Kelvin} and {Rayleigh}},
 BookTitle = {Hyperbolic problems: theory, numerics, applications. Proceedings of the 17th international conference, HYP2018, Pennsylvania State University, University Park, PA, USA, June 25--29, 2018},
 ISBN = {978-1-60133-023-9},
 Pages = {42--59},
 Year = {2020},
 Language = {English},
 zbMATH = {7315452},
 Zbl = {1462.35268}
}

@article{Elgindi_2015,
   Author = {Elgindi, Tarek M. and Widmayer, Klaus},
 Title = {Sharp decay estimates for an anisotropic linear semigroup and applications to the surface quasi-geostrophic and inviscid {Boussinesq} systems},
 FJournal = {SIAM Journal on Mathematical Analysis},
 Journal = {SIAM J. Math. Anal.},
 ISSN = {0036-1410},
 Volume = {47},
 Number = {6},
 Pages = {4672--4684},
 Year = {2015},
 Language = {English}, }

@article{wan2016,
   Author = {Wan, Renhui and Chen, Jiecheng},
 Title = {Global well-posedness for the {2D} dispersive {SQG} equation and inviscid {Boussinesq} equations},
 FJournal = {ZAMP. Zeitschrift f{\"u}r angewandte Mathematik und Physik},
 Journal = {Z. Angew. Math. Phys.},
 ISSN = {0044-2275},
 Volume = {67},
 Number = {4},
 Pages = {22},
 Note = {Id/No 104},
 Year = {2016},
 Language = {English},
  }

@book{majda2002vorticity,
   Author = {Majda, Andrew J. and Bertozzi, Andrea L.},
 Title = {Vorticity and incompressible flow},
 FSeries = {Cambridge Texts in Applied Mathematics},
 Series = {Camb. Texts Appl. Math.},
 ISBN = {0-521-63057-6; 0-521-63948-4},
 Year = {2002},
 Publisher = {Cambridge: Cambridge University Press},
 Language = {English},
}

@book{vallis2017atmospheric,
 Author = {Vallis, Geoffrey K.},
 Title = {Atmospheric and oceanic fluid dynamics. {Fundamentals} and large-scale circulation},
 Edition = {2nd edition},
 ISBN = {978-1-107-06550-5; 978-1-107-58841-7},
 Year = {2017},
 Publisher = {Cambridge: Cambridge University Press},
 Language = {English},
}

@article{bedrossian2021nonlinearinvisciddampingshearbuoyancy,
 Author = {Bedrossian, Jacob and Bianchini, Roberta and Coti Zelati, Michele and Dolce, Michele},
 Title = {Nonlinear inviscid damping and shear-buoyancy instability in the two-dimensional {Boussinesq} equations},
 FJournal = {Communications on Pure and Applied Mathematics},
 Journal = {Commun. Pure Appl. Math.},
 ISSN = {0010-3640},
 Volume = {76},
 Number = {12},
 Pages = {3685--3768},
 Year = {2023},
 Language = {English},
}

@article{bianchini2024strongillposednesslinfty2D,
 author = {Bianchini, Roberta and Hientzsch, Lars Eric and Iandoli, Felice},
 title = {Strong ill-{Posedness} in {{\(L^\infty\)}} of the 2D {Boussinesq} equations in vorticity form and application to the 3D axisymmetric {Euler} equations},
 fjournal = {SIAM Journal on Mathematical Analysis},
 journal = {SIAM J. Math. Anal.},
 issn = {0036-1410},
 volume = {56},
 number = {5},
 pages = {5915--5968},
 year = {2024}}

@article{P_Constantin_1994,
Author = {Constantin, Peter and Majda, Andrew J. and Tabak, Esteban},
 Title = {Formation of strong fronts in the 2-{D} quasigeostrophic thermal active scalar},
 FJournal = {Nonlinearity},
 Journal = {Nonlinearity},
 ISSN = {0951-7715},
 Volume = {7},
 Number = {6},
 Pages = {1495--1553},
 Year = {1994},
 Language = {English},
}

@article{germain2009globalsolutionsgravitywater,
Author = {Germain, Pierre and Masmoudi, Nader and Shatah, Jalal},
 Title = {Global solutions for the gravity water waves equation in dimension 3},
 FJournal = {Annals of Mathematics. Second Series},
 Journal = {Ann. Math. (2)},
 ISSN = {0003-486X},
 Volume = {175},
 Number = {2},
 Pages = {691--754},
 Year = {2012},
 Language = {English},
}

@article{Pusateri_2018,
  Author = {Pusateri, Fabio and Widmayer, Klaus},
 Title = {On the global stability of a beta-plane equation},
 FJournal = {Analysis \& PDE},
 Journal = {Anal. PDE},
 ISSN = {2157-5045},
 Volume = {11},
 Number = {7},
 Pages = {1587--1624},
 Year = {2018},
 Language = {English},}

@article{Deng_2017,
    Author = {Deng, Yu and Ionescu, Alexandru D. and Pausader, Benoit},
 Title = {The {Euler}-{Maxwell} system for electrons: global solutions in {2D}},
 FJournal = {Archive for Rational Mechanics and Analysis},
 Journal = {Arch. Ration. Mech. Anal.},
 ISSN = {0003-9527},
 Volume = {225},
 Number = {2},
 Pages = {771--871},
 Year = {2017},
 Language = {English},}

@article{Germain_Masmoudi_Shatah_CWW2015,
 Author = {Germain, Pierre and Masmoudi, Nader and Shatah, Jalal},
 Title = {Global existence for capillary water waves},
 FJournal = {Communications on Pure and Applied Mathematics},
 Journal = {Commun. Pure Appl. Math.},
 ISSN = {0010-3640},
 Volume = {68},
 Number = {4},
 Pages = {625--687},
 Year = {2015},
 Language = {English},}

@article{Ionescu_Pausader_EP2011,
 Author = {Ionescu, Alexandru D. and Pausader, Benoit},
 Title = {The {Euler}-{Poisson} system in {2D}: global stability of the constant equilibrium solution},
 FJournal = {IMRN. International Mathematics Research Notices},
 Journal = {Int. Math. Res. Not.},
 ISSN = {1073-7928},
 Volume = {2013},
 Number = {4},
 Pages = {761--826},
 Year = {2013},
 Language = {English},}

@article{Ionescu_Pausader_KG2012,
Author = {Ionescu, Alexandru D. and Pausader, Benoit},
 Title = {Global solutions of quasilinear systems of {Klein}-{Gordon} equations in {3D}},
 FJournal = {Journal of the European Mathematical Society (JEMS)},
 Journal = {J. Eur. Math. Soc. (JEMS)},
 ISSN = {1435-9855},
 Volume = {16},
 Number = {11},
 Pages = {2355--2431},
 Year = {2014},
 Language = {English},}

@article{EulerMaxwell3DIonescuPausader,
 Author = {Guo, Yan and Ionescu, Alexandru D. and Pausader, Benoit},
 Title = {Global solutions of the {Euler}-{Maxwell} two-fluid system in {3D}},
 FJournal = {Annals of Mathematics. Second Series},
 Journal = {Ann. Math. (2)},
 ISSN = {0003-486X},
 Volume = {183},
 Number = {2},
 Pages = {377--498},
 Year = {2016},
 Language = {English},
}

@article{Local_Nonlocal_Dispersive_Turbulence,
Author = {Sukhatme, Jai and Smith, Leslie M.},
 Title = {Local and nonlocal dispersive turbulence},
 FJournal = {Physics of Fluids},
 Journal = {Phys. Fluids},
 ISSN = {1070-6631},
 Volume = {21},
 Number = {5},
 Pages = {9},
 Note = {Id/No 056603},
 Year = {2009},
 Language = {English},
}

@article{3DBQ_Widmayer,
Author = {Widmayer, Klaus},
 Title = {Convergence to stratified flow for an inviscid {3D} {Boussinesq} system},
 FJournal = {Communications in Mathematical Sciencesx},
 Journal = {Commun. Math. Sci.},
 ISSN = {1539-6746},
 Volume = {16},
 Number = {6},
 Pages = {1713--1728},
 Year = {2019},
 Language = {English},
}

@article{TAKADA,
 Author = {Takada, Ryo},
 Title = {Long time existence of classical solutions for the {3D} incompressible rotating {Euler} equations},
 FJournal = {Journal of the Mathematical Society of Japan},
 Journal = {J. Math. Soc. Japan},
 ISSN = {0025-5645},
 Volume = {68},
 Number = {2},
 Pages = {579--608},
 Year = {2016},
 Language = {English},
}

@book{Math_Geophysics_Gallagher-Chemin,
 Author = {Chemin, Jean-Yves and Desjardins, Benoit and Gallagher, Isabelle and Grenier, Emmanuel},
 Title = {Mathematical geophysics. {An} introduction to rotating fluids and the {Navier}-{Stokes} equations.},
 FSeries = {Oxford Lecture Series in Mathematics and its Applications},
 Series = {Oxf. Lect. Ser. Math. Appl.},
 Volume = {32},
 ISBN = {0-19-857133-X},
 Year = {2006},
 Publisher = {Oxford: Clarendon Press},
 Language = {English},
}

@article{GP-DefocusingNLS-GNT,
Author = {Gustafson, Stephen and Nakanishi, Kenji and Tsai, Tai-Peng},
 Title = {Scattering theory for the {Gross}-{Pitaevskii} equation in three dimensions},
 FJournal = {Communications in Contemporary Mathematics},
 Journal = {Commun. Contemp. Math.},
 ISSN = {0219-1997},
 Volume = {11},
 Number = {4},
 Pages = {657--707},
 Year = {2009},
 Language = {English},
}

@article{ren2024globalsolutionseulercoriolis,
      title={Global solutions to the {E}uler-{C}oriolis system}, 
      author={Xiao Ren and Gang Tian},
      year={2024},
      journal={arXiv e-prints},
volume={arXiv:2405.18390},
      eprint={2405.18390},
      archivePrefix={arXiv},
      primaryClass={math.AP},
      url={https://arxiv.org/abs/2405.18390}, 
}

@Article{EW_betaplane,
 Author = {Elgindi, Tarek M. and Widmayer, Klaus},
 Title = {Long time stability for solutions of a {{\(\beta\)}}-plane equation},
 FJournal = {Communications on Pure and Applied Mathematics},
 Journal = {Commun. Pure Appl. Math.},
 ISSN = {0010-3640},
 Volume = {70},
 Number = {8},
 Pages = {1425--1471},
 Year = {2017},
 Language = {English},}

@Article{CZDZW,
      title={Stability of viscous three-dimensional stratified {C}ouette flow via dispersion and mixing}, 
      journal = {arXiv e-prints},
      Volume = {arXiv:2402.15312},
      author={Coti Zelati, Michele and Del Zotto, Augusto and Widmayer, Klaus},
      year={2024},
      eprint={2402.15312},
      doi={10.48550/arXiv.2402.15312},
      archivePrefix={arXiv},
      primaryClass={math.AP},
      url={https://arxiv.org/abs/2402.15312}, 
}

@Article{Zil23-1,
 Author = {Zillinger, Christian},
 Title = {On stability estimates for the inviscid {Boussinesq} equations},
 FJournal = {Journal of Nonlinear Science},
 Journal = {J. Nonlinear Sci.},
 ISSN = {0938-8974},
 Volume = {33},
 Number = {6},
 Pages = {38},
 Note = {Id/No 106},
 Year = {2023},
 Language = {English},
 DOI = {10.1007/s00332-023-09958-2},
 zbMATH = {7752353},
 Zbl = {1529.35419}
}

@Article{Zil23-2,
 Author = {Zillinger, Christian},
 Title = {On echo chains in the linearized {Boussinesq} equations around traveling waves},
 FJournal = {SIAM Journal on Mathematical Analysis},
 Journal = {SIAM J. Math. Anal.},
 ISSN = {0036-1410},
 Volume = {55},
 Number = {5},
 Pages = {5127--5188},
 Year = {2023},
 Language = {English},
 DOI = {10.1137/21M1458053},
 zbMATH = {7757939},
 Zbl = {1529.35420}
}

@Article{Elgindi-Jeong-SingFormBQ,
 Author = {Elgindi, Tarek M. and Jeong, In-Jee},
 Title = {Finite-time singularity formation for strong solutions to the {Boussinesq} system},
 FJournal = {Annals of PDE},
 Journal = {Ann. PDE},
 ISSN = {2524-5317},
 Volume = {6},
 Number = {1},
 Pages = {50},
 Note = {Id/No 5},
 Year = {2020},
 Language = {English},}

@Article{Chen-Hou-Blowup2DBQ,
 author = {Chen, Jiajie and Hou, Thomas Yizhao},
 title = {Finite time blowup of {2D} {Boussinesq} and {3D} {Euler} equations with {{\({C}^{1, {{\alpha}}}\)}} velocity and boundary},
 fjournal = {Communications in Mathematical Physics},
 journal = {Commun. Math. Phys.},
 issn = {0010-3616},
 volume = {383},
 number = {3},
 pages = {1559--1667},
 year = {2021},}

@Article{DIPP17,
 Author = {Deng, Yu and Ionescu, Alexandru D. and Pausader, Beno{\^{\i}}t and Pusateri, Fabio},
 Title = {Global solutions of the gravity-capillary water-wave system in three dimensions},
 FJournal = {Acta Mathematica},
 Journal = {Acta Math.},
 ISSN = {0001-5962},
 Volume = {219},
 Number = {2},
 Pages = {213--402},
 Year = {2017},
 Language = {English},
 DOI = {10.4310/ACTA.2017.v219.n2.a1},
 zbMATH = {6871315},
 Zbl = {1397.35190}
}

@Article{GP11,
 Author = {Guo, Yan and Pausader, Benoit},
 Title = {Global smooth ion dynamics in the {Euler}-{Poisson} system},
 FJournal = {Communications in Mathematical Physics},
 Journal = {Commun. Math. Phys.},
 ISSN = {0010-3616},
 Volume = {303},
 Number = {1},
 Pages = {89--125},
 Year = {2011},
 Language = {English},
 DOI = {10.1007/s00220-011-1193-1},
 zbMATH = {5883605},
 Zbl = {1220.35129}
}

@Article{GM14,
 Author = {Germain, Pierre and Masmoudi, Nader},
 Title = {Global existence for the {Euler}-{Maxwell} system},
 FJournal = {Annales Scientifiques de l'{\'E}cole Normale Sup{\'e}rieure. Quatri{\`e}me S{\'e}rie},
 Journal = {Ann. Sci. {\'E}c. Norm. Sup{\'e}r. (4)},
 ISSN = {0012-9593},
 Volume = {47},
 Number = {3},
 Pages = {469--503},
 Year = {2014},
 Language = {English},
 DOI = {10.24033/asens.2219},
 URL = {smf4.emath.fr/en/Publications/AnnalesENS/4_47/html/ens_ann-sc_47_469-503.php},
 zbMATH = {6380112},
 Zbl = {1311.35195}
}

@Article{IP15,
 Author = {Ionescu, Alexandru D. and Pusateri, Fabio},
 Title = {Global solutions for the gravity water waves system in {2D}},
 FJournal = {Inventiones Mathematicae},
 Journal = {Invent. Math.},
 ISSN = {0020-9910},
 Volume = {199},
 Number = {3},
 Pages = {653--804},
 Year = {2015},
 Language = {English},
 DOI = {10.1007/s00222-014-0521-4},
 zbMATH = {6418043},
 Zbl = {1325.35151}
}

@Article{RksOnSmoothness2DBQ3DE-Chen,
 Author = {Chen, Jiajie},
 Title = {Remarks on the smoothness of the {{\(C^{1,\alpha}\)}} asymptotically self-similar singularity in the {3D} {Euler} and {2D} {Boussinesq} equations},
 FJournal = {Nonlinearity},
 Journal = {Nonlinearity},
 ISSN = {0951-7715},
 Volume = {37},
 Number = {6},
 Pages = {32},
 Note = {Id/No 065018},
 Year = {2024},
 Language = {English},
 DOI = {10.1088/1361-6544/ad45a2},
 zbMATH = {7867499}
}

@article{kiselev2022smallscaleformation2d,
 author = {Kiselev, Alexander and Park, Jaemin and Yao, Yao},
 title = {Small scale formation for the 2-dimensional {Boussinesq} equation},
 fjournal = {Analysis \& PDE},
 journal = {Anal. PDE},
 issn = {2157-5045},
 volume = {18},
 number = {1},
 pages = {171--198},
 year = {2025}}

@Article{Carstro-Cordoba-Gomez-Serrano-VstatesSQG,
 Author = {Castro, Angel and C{\'o}rdoba, Diego and G{\'o}mez-Serrano, Javier},
 Title = {Existence and regularity of rotating global solutions for the generalized surface quasi-geostrophic equations},
 FJournal = {Duke Mathematical Journal},
 Journal = {Duke Math. J.},
 ISSN = {0012-7094},
 Volume = {165},
 Number = {5},
 Pages = {935--984},
 Year = {2016},
 Language = {English},
 DOI = {10.1215/00127094-3449673},
 zbMATH = {6580686},
 Zbl = {1339.35234}
}

@InCollection{Kiselev-Nazarov-NormInflSQG,
 Author = {Kiselev, Alexander and Nazarov, Fedor},
 Title = {A simple energy pump for the surface quasi-geostrophic equation},
 BookTitle = {Nonlinear partial differential equations. The Abel symposium 2010. Proceedings of the Abel symposium, Oslo, Norway, September 28--October 2, 2010},
 ISBN = {978-3-642-25360-7; 978-3-642-25361-4},
 Pages = {175--179},
 Year = {2012},
 Publisher = {Berlin: Springer},
 Language = {English},
 DOI = {10.1007/978-3-642-25361-4_9},
 zbMATH = {6082185},
 Zbl = {1250.35034}
}

@article{Gomez-Serrano-Ionescu-Park-gSQG,
      title={Quasiperiodic solutions of the generalized {SQG} equation}, 
      author={Javier Gómez-Serrano and Alexandru D. Ionescu and Jaemin Park},
      year={2023},
      journal={arXiv e-prints},
      volume={arXiv:2303.03992},
      eprint={2303.03992},
      archivePrefix={arXiv},
      primaryClass={math.AP},
      url={https://arxiv.org/abs/2303.03992}, 
}

@Article{Wan2020,
 Author = {Wan, Renhui},
 Title = {Long time stability for the dispersive {SQG} equation and {Boussinesq} equations in {Sobolev} space {{\(H^s\)}}},
 FJournal = {Communications in Contemporary Mathematics},
 Journal = {Commun. Contemp. Math.},
 ISSN = {0219-1997},
 Volume = {22},
 Number = {3},
 Pages = {13},
 Note = {Id/No 1850063},
 Year = {2020},
 Language = {English},
 DOI = {10.1142/S0219199718500633},
 zbMATH = {7190493},
 Zbl = {1434.35125}
}

@Article{TakBQ3D,
 Author = {Takada, Ryo},
 Title = {Strongly stratified limit for the {3D} inviscid {Boussinesq} equations},
 FJournal = {Archive for Rational Mechanics and Analysis},
 Journal = {Arch. Ration. Mech. Anal.},
 ISSN = {0003-9527},
 Volume = {232},
 Number = {3},
 Pages = {1475--1503},
 Year = {2019},
 Language = {English},
 DOI = {10.1007/s00205-018-01347-4},
 zbMATH = {7045538},
 Zbl = {1447.35278}
}

@article{dauxois-ChallaengesInEnvironmentalFM,
  TITLE = {{Confronting Grand Challenges in environmental fluid mechanics}},
  AUTHOR = {Dauxois, Thierry and Peacock, T and Bauer, P and Caulfield, C P and Cenedese, C and Gorl{\'e}, C and Haller, G and Ivey, G N and Linden, P F and Meiburg, E and Pinardi, N and Vriend, N M and Woods, A W},
  URL = {https://hal.science/hal-03381624},
  JOURNAL = {{Physical Review Fluids}},
  PUBLISHER = {{American Physical Society}},
  VOLUME = {6},
  YEAR = {2021},
  MONTH = Feb,
  DOI = {10.1103/physrevfluids.6.020501},
  HAL_ID = {hal-03381624},
  HAL_VERSION = {v1},
}

@Article{EJ2019,
 Author = {Elgindi, Tarek M. and Jeong, In-Jee},
 Title = {Finite-time singularity formation for strong solutions to the axi-symmetric {{\({3D}\)}} {Euler} equations},
 FJournal = {Annals of PDE},
 Journal = {Ann. PDE},
 ISSN = {2524-5317},
 Volume = {5},
 Number = {2},
 Pages = {51},
 Note = {Id/No 16},
 Year = {2019},
 Language = {English},
 DOI = {10.1007/s40818-019-0071-6},
 zbMATH = {7186449},
 Zbl = {1436.35055}
}

@Article{CCW14,
 Author = {Chae, Dongho and Constantin, Peter and Wu, Jiahong},
 Title = {An incompressible {2D} didactic model with singularity and explicit solutions of the {2D} {Boussinesq} equations},
 FJournal = {Journal of Mathematical Fluid Mechanics},
 Journal = {J. Math. Fluid Mech.},
 ISSN = {1422-6928},
 Volume = {16},
 Number = {3},
 Pages = {473--480},
 Year = {2014},
 Language = {English},
 DOI = {10.1007/s00021-014-0166-5},
 zbMATH = {6407900},
 Zbl = {1307.35213}
}

@article{EP23,
      title={From Instability to {S}ingularity {F}ormation in {I}ncompressible {F}luids}, 
      author={Tarek M. Elgindi and Federico Pasqualotto},
      year={2023},
      journal={arXiv e-prints},
      volume={arXiv:2310.19780},
      archivePrefix={arXiv},
      primaryClass={math.AP},
      url={https://arxiv.org/abs/2310.19780}, 
}

@article{CH22,
      title={Stable nearly self-similar blowup of the {2D Boussinesq} and {3D Euler} equations with smooth data {I}: {A}nalysis}, 
      author={Jiajie Chen and Thomas Y. Hou},
      year={2023},
      volume={arXiv:2210.07191},
      journal={arXiv e-prints},
      archivePrefix={arXiv},
      primaryClass={math.AP},
      url={https://arxiv.org/abs/2210.07191}, 
}

@article{CH24,
 author = {Chen, Jiajie and Hou, Thomas Y.},
 title = {Stable nearly self-similar blowup of the {2D} {Boussinesq} and {3D} {Euler} equations with smooth data. {II}: {Rigorous} numerics},
 fjournal = {Multiscale Modeling \& Simulation},
 journal = {Multiscale Model. Simul.},
 issn = {1540-3459},
 volume = {23},
 number = {1},
 pages = {25--130},
 year = {2025}}

@article{IT17,
 author = {Ifrim, Mihaela and Tataru, Daniel},
 title = {The lifespan of small data solutions in two dimensional capillary water waves},
 fjournal = {Archive for Rational Mechanics and Analysis},
 journal = {Arch. Ration. Mech. Anal.},
 issn = {0003-9527},
 volume = {225},
 number = {3},
 pages = {1279--1346},
 year = {2017}}

@book{DM22,
 author = {Delort, Jean-Marc and Masmoudi, Nader},
 title = {Long-time dispersive estimates for perturbations of a kink solution of one-dimensional cubic wave equations},
 fseries = {Memoirs of the European Mathematical Society},
 series = {Mem. Eur. Math. Soc.},
 issn = {2747-9080},
 volume = {1},
 isbn = {978-3-98547-020-4; 978-3-98547-520-9},
 year = {2022},
 publisher = {Berlin: European Mathematical Society (EMS)}}

@article{IT24,
 author = {Ifrim, Mihaela and Tataru, Daniel},
 title = {Long time solutions for 1D cubic dispersive equations. {II}: {The} focusing case},
 fjournal = {Vietnam Journal of Mathematics},
 journal = {Vietnam J. Math.},
 issn = {2305-221X},
 volume = {52},
 number = {3},
 pages = {597--614},
 year = {2024}}

@article{MP17,
 author = {Pusateri, Fabio and Murphy, Jason},
 title = {Almost global existence for cubic nonlinear {Schr{\"o}dinger} equations in one space dimension},
 fjournal = {Discrete and Continuous Dynamical Systems},
 journal = {Discrete Contin. Dyn. Syst.},
 issn = {1078-0947},
 volume = {37},
 number = {4},
 pages = {2077--2102},
 year = {2017}}

@article{DIP25,
 author = {Deng, Yu and Ionescu, Alexandru D. and Pusateri, Fabio},
 title = {On the wave turbulence theory of 2D gravity waves. {I}: {Deterministic} energy estimates},
 fjournal = {Communications on Pure and Applied Mathematics},
 journal = {Commun. Pure Appl. Math.},
 issn = {0010-3640},
 volume = {78},
 number = {2},
 pages = {211--322},
 year = {2025},
 language = {English},
 doi = {10.1002/cpa.22224},
 zbMATH = {7960943},
 Zbl = {1555.35226}
}

@book{AD15,
 author = {Alazard, Thomas and Delort, Jean-Marc},
 title = {Sobolev estimates for two dimensional gravity water waves},
 fseries = {Ast{\'e}risque},
 series = {Ast{\'e}risque},
 issn = {0303-1179},
 volume = {374},
 isbn = {978-2-85629-821-3},
 year = {2015},
 publisher = {Paris: Soci{\'e}t{\'e} Math{\'e}matique de France (SMF)},
 language = {English},
 zbMATH = {6538885},
 Zbl = {1360.35002}
}

@article{HIT16,
 author = {Hunter, John K. and Ifrim, Mihaela and Tataru, Daniel},
 title = {Two dimensional water waves in holomorphic coordinates},
 fjournal = {Communications in Mathematical Physics},
 journal = {Commun. Math. Phys.},
 issn = {0010-3616},
 volume = {346},
 number = {2},
 pages = {483--552},
 year = {2016},
 language = {English},
 doi = {10.1007/s00220-016-2708-6},
 zbMATH = {6679493},
 Zbl = {1358.35121}
}

@book{IP18,
 author = {Ionescu, Alexandru D. and Pusateri, Fabio},
 title = {Global regularity for 2D water waves with surface tension},
 fseries = {Memoirs of the American Mathematical Society},
 series = {Mem. Am. Math. Soc.},
 issn = {0065-9266},
 volume = {1227},
 isbn = {978-1-4704-3103-7; 978-1-4704-4917-9},
 year = {2018},
 publisher = {Providence, RI: American Mathematical Society (AMS)},
 language = {English},
 doi = {10.1090/memo/1227},
 zbMATH = {7025815},
 Zbl = {1435.76002}
}

@article{Wu20,
 author = {Wu, Sijue},
 title = {The quartic integrability and long time existence of steep water waves in 2D},
 year = {2020},
volume={arXiv:2010.09117},
journal={arXiv e-prints},
 url = {https://arxiv.org/abs/2010.09117},
}
\end{document}